\documentclass[a4paper,11pt]{article}
\usepackage[OT1,T1]{fontenc}
\usepackage{a4wide}
\usepackage{color}
\usepackage{bm}
\usepackage{amsmath,amssymb,amsthm}
\usepackage{mathrsfs}
\usepackage{textcomp}
\usepackage{multirow}
\usepackage[all]{xy}
\usepackage{tikz}
\usepgflibrary{shapes.geometric}
\usepackage[pdftex,pdfborder={0 0 0}]{hyperref}
    \usepackage{graphicx}

2
\makeatletter
\newcounter{statement}

\DeclareMathOperator\End{End}
\DeclareMathOperator\im{im}
\DeclareMathOperator\coker{coker}
\DeclareMathOperator\Hom{Hom}

\DeclareMathOperator\Ext{Ext}
\newcommand\SVec{{\mathop{\mathrm{SVect}}}}
\newcommand\Comod{{\mathop{\mathrm{Comod}}}}
\newcommand\Perv{{\mathord{\mathrm P_{\mathscr S}(\Gr_G,\bk)}}}
\newcommand\Gr{\mathord{\mathrm{Gr}}}
\newcommand\Fl{\mathord{\mathrm{Fl}}}
\newcommand\GO{{G_{\mathcal O}}}
\newcommand\GK{{G_{\mathcal K}}}
\newcommand\NK{{N_{\mathcal K}}}
\newcommand\IC{{\mathbf{IC}}}
\newcommand\supp{\mathop{\mathrm{supp}}}
\newcommand\aff{{\mathrm{aff}}}
\newcommand\fin{{\mathord{\mathrm{fin}}}}
\newcommand\Homint{{\mathop{\underline{\mathrm{Hom}}}}}
\newcommand\forget{{\mathord{\mathrm{forget}}}}
\newcommand\SL{\mathord{\mathbf{SL}}}

\newcommand{\Per}{\mathrm{P}}
\newcommand{\id}{\mathrm{id}}
\newcommand{\Spec}{\mathrm{Spec}}
\newcommand\coim{\mathop{\mathrm{coim}}}
\newcommand\Conv{\mathop{\mathrm{Conv}}}

\newcommand\Vect{{\mathord{\mathrm{Vect}}}}
\newcommand\Rep{{\mathord{\mathrm{Rep}}}}
\newcommand\Mod{{\mathord{\mathrm{Mod}}}}

\newcommand{\Sat}{\mathsf{S}}
\newcommand{\F}{\mathsf{F}}
\newcommand{\bk}{\mathbf{k}}
\newcommand{\Z}{\mathbf{Z}}
\newcommand{\Q}{\mathbf{Q}}
\newcommand{\C}{\mathbf{C}}
\newcommand{\FF}{\mathbf{F}}

\newcommand{\cJ}{\mathcal{J}}
\newcommand{\pH}{{}^p \hspace{-1pt} \mathscr{H}}

\newcommand{\ad}{\mathrm{ad}}
\newcommand{\coH}{\mathsf{H}}

\newcommand{\Db}{D^{\mathrm{b}}}

\makeatletter
\def\lotimes{\@ifnextchar_{\@lotimessub}{\@lotimesnosub}}
\def\@lotimessub_#1{\mathchoice{\mathbin{\mathop{\otimes}^L}_{#1}}%
  {\otimes^L_{#1}}{\otimes^L_{#1}}{\otimes^L_{#1}}}
\def\@lotimesnosub{\mathbin{\mathop{\otimes}^L}}
\makeatother

\makeatletter
\def\Tenint{\@ifnextchar_{\@Tenintsub}{\@Tenintnosub}}
\def\@Tenintsub_#1{\mathchoice{\mathbin{\underline{\mathop{\otimes}}}_{#1}}%
  {\underline{\otimes}_{#1}}{\underline{\otimes}_{#1}}{\underline{\otimes}^L_{#1}}}
\def\@Tenintnosub{\mathbin{\underline{\mathop{\otimes}}}}
\makeatother

\makeatletter
\def\lboxtimes{\@ifnextchar_{\@lboxtimessub}{\@lboxtimesnosub}}
\def\@lboxtimessub_#1{\mathchoice{\mathbin{\mathop{\boxtimes}^L}_{#1}}%
  {\boxtimes^L_{#1}}{\boxtimes^L_{#1}}{\boxtimes^L_{#1}}}
\def\@lboxtimesnosub{\mathbin{\mathop{\boxtimes}^L}}
\makeatother

\numberwithin{equation}{section}
\newtheorem{thm}{Theorem}[section]
\newtheorem{lem}[thm]{Lemma}
\newtheorem{prop}[thm]{Proposition}
\newtheorem{cor}[thm]{Corollary}

\theoremstyle{definition}
\newtheorem{defn}[thm]{Definition}

\theoremstyle{remark}
\newtheorem{rmk}[thm]{Remark}
\newtheorem{ex}[thm]{Example}
\newtheorem{exo}[thm]{Exercise}

\begin{document}

\title{Notes on the geometric Satake equivalence}

\author{Pierre Baumann \thanks{Institut de Recherche Math\'ematique Avanc\'ee, CNRS UMR 7501, Universit\'e de Strasbourg, F-67084 Strasbourg Cedex, France, p.baumann@unistra.fr} \and Simon Riche \thanks{Universit\'e Clermont Auvergne, CNRS, LMBP, F-63000 Clermont-Ferrand, France, simon.riche@uca.fr}}

\date{}
\maketitle

\setcounter{tocdepth}{1}
\tableofcontents

\newpage

\section{Introduction}

\subsection{Description}

These notes are devoted to a detailed exposition of the proof of the Geometric Satake Equivalence
 by Mirkovi\'c--Vilonen~\cite{mv}. This celebrated result provides, for $G$ a complex connected reductive group and $\bk$ a Noetherian commutative ring of finite global dimension, an equivalence of categories between the category $\Per_{\GO}(\Gr_G,\bk)$ of $\GO$-equivariant perverse sheaves on the affine Grassmannian $\Gr_G$ of $G$ (where $\GO$ is the loop group of $G$) and the category $\Rep_\bk(G^\vee_\bk)$ of representations of the Langlands dual split reductive group over $\bk$ on finitely generated $\bk$-modules. Under this equivalence, the tensor product of $G^\vee_\bk$-modules corresponds to a geometric construction on perverse sheaves called \emph{convolution}.

This result can be considered on the one hand as giving a geometric description of the category of representations of $G^\vee_\bk$, and on the other hand as giving a ``concrete'' construction of the dual reductive group $G^\vee_\bk$ out of the original (complex) reductive group $G$.

\subsection{History and idea of proof}

The first evidence of a close relationship between perverse sheaves on $\Gr_G$ and representations of $G^\vee_\bk$ was found in work of Lusztig~\cite{lusztig}, where a ``combinatorial shadow'' of the equivalence was proved (for $\bk$ a field of characteristic $0$). The equivalence itself was first proved, in the case when $\bk$ is a field of characteristic $0$, by Ginzburg~\cite{ginzburg}. (In this case, the existence of an equivalence of abelian categories $\Per_{\GO}(\Gr_G,\bk) \cong \Rep_\bk(G^\vee_\bk)$ is obvious, since both categories are semisimple with isomorphism classes of simple objects parametrized by the same set. The content of the theorem is thus only the description of the tensor product in geometric terms.) A new proof, valid for general coefficients, was later given by Mirkovi\'c--Vilonen \cite{mv}. This is the proof that we consider here; the main new ingredient of their approach is the definition of the \emph{weight functors}, which give a geometric construction of the decomposition of $G^\vee_\bk$-representations into weight spaces for a maximal torus. A later proof in the case of characteristic-$0$ fields (which applies for $\ell$-adic sheaves, when $G$ is defined over a more general field) was given by Richarz~\cite{richarz}. (The main difference with the approaches of~\cite{ginzburg, mv} lies in the identification of the group scheme, which relies on work of Kazhdan--Larsen--Varshavsky~\cite{klv}.)

All proofs 
are based on ideas from \emph{Tannakian formalism}. The strategy is to construct enough structure on the category $\Per_{\GO}(\Gr_G,\bk)$ so as to guarantee that this category is equivalent to the category of representations of a $\bk$-group scheme. In the case when $\bk$ is a field, one can apply general results due to Saavedra Rivano~\cite{sr} and Deligne--Milne~\cite{dm} to prove this; for general coefficients no such theory is available, and Mirkovi\'c--Vilonen construct the group scheme ``by hand'' using their weight functors.
The next step is to identify this group scheme with $G^\vee_\bk$. The case of fields of characteristic~$0$ is relatively easy. Then, in~\cite{mv}, the general case is deduced from this one using a detailed analysis of the group scheme in the case $\bk$ is an algebraic closure of a finite field, and a general result on reductive group schemes due to Prasad--Yu~\cite{py}.

\subsection{Applications}

The geometric Satake equivalence has found numerous applications in Representation Theory, Algebraic Geometry and Number Theory. For the latter applications (see in particular~\cite{lafforgue}; see also~\cite[\S 5.5]{zhu} for other examples and references), it is important to have a version of this equivalence where the affine Grassmannian is defined not over $\C$ (as we do here) but rather over an algebraically closed field of positive characteristic (and where the sheaves for the classical topology are replaced by \'etale sheaves). We will not consider this variant, but will only mention that the analogues in this setting of all results that we use on the geometry of the affine Grassmannian are known; see~\cite{zhu} for details and references.\footnote{There is an additional subtlety in this setting if the characteristic of the base field is ``small,'' namely that the neutral connected component of the Grassmannian might not be isomorphic to the affine Grassmannian of the simply-connected cover of the derived subgroup; see~\cite[Remark~6.4]{pr} for an example. However, as was explained to us by X.~Zhu, in any case the natural morphism from the latter to the former is a universal homeomorphism (again, see~\cite[Remark~6.4]{pr} for a special case) and hence is as good as an isomorphism, as far as \'etale sheaves are concerned.} With these results at hand, our considerations adapt in a straightforward way to this setting to prove the desired equivalence of categories. (Here of course the coefficients of sheaves cannot be arbitrary, and the role played by $\Z$ in Section~\ref{sec:identification} should be played by $\Z_\ell$, where $\ell$ is a prime number different from the characteristic of the field of definition of the affine Grassmannian.)

The applications to Number Theory have also motivated a number of generalizations of the geometric Satake equivalence (so far mainly in the case of characteristic-$0$ coefficients) which will not be reviewed here; see in particular~\cite{richarz, zhu-ram, rz, zhu-mixed}.

\subsection{Contents}

The notes consist of two parts with different purposes. The first one is a gentle introduction to the proof of Mirkovi\'c--Vilonen in the special case where $\bk$ is a field of characteristic $0$. This case allows for important simplifications, but at the same time plays a crucial role in the proof for general coefficients. It is well understood, but (to our knowledge) has not been treated in detail in the literature from the point of view of Mirkovi\'c--Vilonen (except of course in their paper).
We follow their arguments closely, adding only a few details where their proofs might be considered a little bit sketchy. We also treat certain prerequisites (e.g.~Tannakian formalism) in detail. On the other hand, most ``standard'' results on the affine Grassmannian are stated without proof; for details and references we refer e.g.~to~\cite{zhu}.

Part~\ref{pt:arbitrary} is devoted to the proof for general coefficients. Some people have expressed doubts about the proof in this generality, so we have tried to make all the arguments explicit, and to clarify the proofs as much as possible. In this process, Geordie Williamson suggested a direct proof of the fact that the group scheme constructed by Mirkovi\'c--Vilonen is of finite type in the case of field coefficients. This proof is reproduced in Lemma~\ref{lem:G-algebraic-connected}, and allows to simplify the arguments a little bit.

Finally, Appendix~\ref{sec:appendix} provides proofs of some ``well-known'' results on equivariant perverse sheaves.

\subsection{Acknowledgements}

These notes grew out of a $2\frac{1}{2}$-days mini-course given during the workshop ``Geometric methods and Langlands functoriality in positive characteristic'' held in Luminy in January 2016. This mini-course (which only covered the contents of Part~\ref{pt:char-0}) also comprised reminders on constructible sheaves and equivariant derived categories (by D. Fratila), and on perverse sheaves (by V. Heiermann), which are not reproduced in the notes.

We thank Dragos Fratila for many discussions which helped clarify various constructions and proofs in Part~\ref{pt:char-0}. We thank Volker Heiermann, Vincent Lafforgue and an anymous referee for insisting that we should treat the case of general coefficients, which led to the work in Part~\ref{pt:arbitrary}.
We also thank Geordie Williamson for very helpful discussions on Part~\ref{pt:arbitrary}, and for allowing us to reproduce his proof of Lemma~\ref{lem:G-algebraic-connected}. Finally, we thank Brian Conrad and Gopal Prasad for kindly answering some questions related to the application of the results of~\cite{py}, Julien Bichon for helpful discussions and providing some references, Xinwen Zhu for answering various questions, and Vincent Lafforgue and an (other?) anonymous referee for their comments on a previous version of these notes.

Both authors were partially supported by the ANR Grant No.~ANR-13-BS01-0001-01. This project has received
funding from the European Research Council (ERC) under the European Union's Horizon 2020
research and innovation programme (grant agreement No. 677147).

\newpage

\part{The case of characteristic-\texorpdfstring{$0$}{0} coefficients}
\label{pt:char-0}

\section{Tannakian reconstruction}
\label{sec:Tannakian}

In this section (where we follow closely \cite[\S II]{dm}), $\bk$ is an arbitrary field, and we denote by $\Vect_\bk$ the category of finite-dimensional $\bk$-vector spaces. All categories are tacitly assumed to be essentially small. By a commutative diagram of functors we will mean a diagram commutative \emph{up to isomorphism}.

Some important ideas of Tannakian reconstruction are already contained in the following easy exercise.

\begin{exo}
\label{exo:reconstruction}
Let $A$ be a $\bk$-algebra, $X$ be an $A$-module which is finite-dimensional over $\bk$, and
$\alpha\in\End_\bk(X)$. Show that
$$\alpha\in\im\bigl(A\to\End_\bk(X)\bigr)\ \Longleftrightarrow\
\forall n\geq0,\ \forall Y\subset X^{\oplus n}\ \text{$A$-submodule},\
\alpha^{\oplus n}(Y)\subset Y.$$
\end{exo}

(\emph{Hint} : Of course, the implication $\Rightarrow$ is obvious. To prove the reverse direction,
assume the condition in the
right-hand side holds. Pick a $\bk$-basis $(e_1,\cdots,e_n)$ of $X$, take
for $Y$ the $A$-submodule generated by $(e_1,\cdots,e_n) \in X^{\oplus n}$, and write
that $Y$ contains $\alpha^{\oplus n}(e_1,\cdots,e_n)=(\alpha(e_1),
\cdots,\alpha(e_n))$.)

Tannakian reconstruction actually amounts to veneer this exercise first
with the language of categories and then with the language of Hopf
algebras (i.e.~affine group schemes).

\subsection{A first reconstruction theorem}
\label{ss:reconstruction-1}

Let us denote the category of finite-dimensional $\bk$-vector spaces by
$\Vect_\bk$. Given a $\bk$-algebra $A$, we denote the category of
finite-dimensional left $A$-modules by $\Mod_A$.

Recall that a category $\mathscr C$ is called \emph{additive} if
\begin{itemize}
\item
each set $\Hom_{\mathscr C}(X,Y)$ is an abelian group;
\item
the composition of morphisms is a bilinear operation;
\item
$\mathscr{C}$ has a zero object;
\item
finite products and coproducts exist in
$\mathscr C$. 
\end{itemize}
Such a category is called $\bk$-linear if each $\Hom_{\mathscr C}(X,Y)$ is a $\bk$-vector space, and if the composition is $\bk$-bilinear.
An additive category $\mathscr C$ is called \emph{abelian} if
\begin{itemize}
\item
each
morphism has a kernel and a cokernel;
\item
for any morphism $f$, the natural morphism from the cokernel of the kernel (a.k.a.~the coimage) of $f$
to the kernel of the
cokernel (a.k.a.~the image) of $f$ is an isomorphism.
\end{itemize}

Given an object $X$ in an abelian category $\mathscr C$, we will denote
by $\langle X\rangle$ the full subcategory of $\mathscr C$ formed by
all objects that are isomorphic to a subquotient of a direct sum
$X^{\oplus n}$ for some $n \in \Z_{\geq 0}$.

\begin{prop}
\label{prop:reconstruction}
Let $\mathscr C$ be an abelian $\bk$-linear category and let
$\omega:\mathscr C\to\Vect_\bk$ be a $\bk$-linear exact faithful functor.
Fix an object $X$ in $\mathscr C$ and introduce the finite-dimensional
$\bk$-algebra
$$A_X:=\bigl\{\alpha\in\End_\bk(\omega(X))\bigm|\forall n\geq0,\
\forall Y\subset X^{\oplus n}\ \text{subobject},\
\alpha^{\oplus n}(\omega(Y))\subset\omega(Y)\bigr\}.$$
Then $\omega\bigl|_{\langle X\rangle}$ admits a canonical factorization
$$\xymatrix@=40pt@!0{\langle X\rangle\ar[rr]^-{\overline\omega_X}
\ar[dr]_-{\omega}&&\Mod_{A_X}\ar[dl]^-{\forget}\\&\Vect_\bk,&}$$
and $\overline\omega_X$ is an equivalence of categories.
In addition $A_X$ is the endomorphism algebra of the functor
$\omega\bigl|_{\langle X\rangle}$.
\end{prop}

If $A$ is a $\bk$-algebra, and if we apply this proposition to the
category $\mathscr C=\Mod_A$ of finite-dimensional $A$-modules with $\omega$ the forgetful functor (which
keeps the $\bk$-vector space structure but forgets the structure of
$A$-module), then Exercise~\ref{exo:reconstruction} shows that the algebra $A_X$
is precisely the image of $A$ in $\End_\bk(X)$. The proposition is thus
mainly saying that the exercise can be stated within the language
of abelian categories.

For the proof of Proposition~\ref{prop:reconstruction} we will need the following standard facts from Category Theory.

\begin{lem}\phantomsection
\label{lem:cat}
\begin{enumerate}
\item
\label{it:lem-cat-1}
An exact additive functor $F:\mathscr A\to\mathscr B$ between two abelian
categories preserves kernels and cokernels. It thus preserves
finite intersections and finite sums (in an ambient object).
\item
\label{it:lem-cat-2}
A faithful functor $F:\mathscr A\to\mathscr B$ between two abelian
categories does not kill any nonzero object.
\item
\label{it:lem-cat-3}
Let $F:\mathscr A\to\mathscr B$ be an exact faithful additive functor between
two abelian categories and let $u:X\to Y$ be a morphism in $\mathscr A$.
Then $u$ is an monomorphism (respectively, epimorphism) if and only
if $F(u)$ is so.
\item
\label{it:lem-cat-4}
Let $F:\mathscr A\to\mathscr B$ be an exact faithful additive functor
between two abelian categories. Assume that $\mathscr B$ is Artinian
and Noetherian: any monotone sequence of subobjects becomes
eventually constant. Then arbitrary intersections and arbitrary
sums (in an ambient object) exist in both $\mathscr A$ and $\mathscr B$, and $F$ preserves
intersections and sums.
\end{enumerate}
\end{lem}

\begin{proof}
\eqref{it:lem-cat-1}
Any morphism $u:X\to Y$ in $\mathscr A$ gives rise to two short
exact sequences
$$\xymatrix@!0{&&X\ar[dr]\ar[rr]^u&&Y\ar[dr]&&\\&\ker u\ar[ur]&&
\im u\ar[ur]\ar[dr]&&\coker u\ar[dr]&\\0\ar[ur]&&0\ar[ur]&&0&&0.}$$
Applying $F$ to this diagram and using the exactness assumption,
we see that this functor preserves kernels and cokernels. The last assertion
comes from the fact that the intersection (respectively, sum) of
two subobjects can be expressed by a pull-back (respectively,
push-forward) diagram, that is, as a kernel (respectively, cokernel).

\eqref{it:lem-cat-2}
Assume that $X$ is a nonzero object in $\mathscr A$. Then
$\id_X\neq0$ in $\End_{\mathscr A}(X)$. The faithfulness assumption
then implies that $\id_{F(X)}=F(\id_X)\neq0$ in
$\End_{\mathscr B}(F(X))$, whence $F(X)\neq0$.

\eqref{it:lem-cat-3}
It suffices to note that
$$\ker u=0\ \Longleftrightarrow\ F(\ker u)=0\
\Longleftrightarrow\ \ker F(u)=0$$
and
$$\coker u=0\ \Longleftrightarrow\ F(\coker u)=0\
\Longleftrightarrow\ \coker F(u)=0.$$

\eqref{it:lem-cat-4}
We first claim that $\mathscr A$ is Artinian and Noetherian.
Indeed given a monotone sequence of subobjects in $\mathscr A$, its
image by $F$ is a monotone sequence of subobjects in $\mathscr B$,
so becomes eventually constant; \eqref{it:lem-cat-3} then implies that the
sequence in $\mathscr A$ also becomes eventually constant. Thus
arbitrary intersections and sums exist in $\mathscr A$ as well as
in $\mathscr B$ and are in fact finite intersections or sums (by the Artinian or Noetherian property, respectively). We
conclude with the help of \eqref{it:lem-cat-1}.
\end{proof}

We can now give the proof of Proposition~{\rm \ref{prop:reconstruction}}.

\begin{proof}
By definition, for any $\alpha\in A_X$, the endomorphism $\alpha^{\oplus n}$
of $\omega(X)^{\oplus n}$ leaves stable $\omega(Y)$ for all subobjects
$Y\subset X^{\oplus n}$, and thus induces an endomorphism of $\omega(Z)$
for all subquotients $Z$ of $X^{\oplus n}$. In this way, for each
object $Z$ in $\langle X\rangle$, the $\bk$-vector space $\omega(Z)$
becomes an $A_X$-module. If $Z$ is a subquotient of $X^{\oplus n}$ and $Z'$ is a subquotient of $X^{\oplus m}$, and if $f : Z \to Z'$ is a morphism in $\mathscr{C}$, then $Z \oplus Z'$ is a subquotient of $X^{\oplus (n+m)}$, and the image $\mathrm{gr}(f)$ of the morphism $(\id,f) : Z \to Z \oplus Z'$ (in other words the graph of $f$) is a subobject of $Z \oplus Z'$, hence also a subquotient of $X^{\oplus (n+m)}$. The fact that $A_X$ stabilizes $\omega(\mathrm{gr}(f))$  means that $\omega(f)$ is a morphism of $A_X$-modules.
In summary, we have proved that $\omega\bigl|_{\langle X\rangle}$
factorizes through the category of finite-dimensional $A_X$-modules,
as stated.


By definition, an endomorphism $\alpha$ of the functor
$\omega\bigl|_{\langle X\rangle}$ is the datum of an endomorphism
$\alpha_Z\in\End_\bk(\omega(Z))$ for each $Z\in\langle X\rangle$, such
that the diagram
\[
\xymatrix@C=2cm{\omega(Z)\ar[r]^{\alpha_Z}\ar[d]_{\omega(f)}&\omega(Z)
\ar[d]^{\omega(f)}\\\omega(Z')\ar[r]^{\alpha_{Z'}}&\omega(Z')}
\]
commutes for any morphism $f:Z\to Z'$ in $\langle X\rangle$.
This compatibility condition and the definition of $\langle X\rangle$
forces $\alpha$ to be determined by $\alpha_X\in\End_\bk(\omega(X))$,
and forces $\alpha_X$ to belong to $A_X$. Conversely, any element
in $A_X$ gives rise to an endomorphism of
$\omega\bigl|_{\langle X\rangle}$. This discussion shows the last
assertion in the proposition.

It remains to show that the functor $\overline\omega_X$ is an
equivalence of categories. We already know that it is faithful,
so we must show that it is full and essentially surjective.
We will do that by
constructing an inverse functor.


We will denote by $\mathscr C^\fin$ the category opposite to the category
of $\bk$-linear functors from $\mathscr C$ to $\Vect_\bk$. Yoneda's lemma says that
the functor $Z\mapsto\Hom_{\mathscr C}(Z,-)$ from $\mathscr C$ to
$\mathscr C^\fin$ is fully faithful, so $\mathscr C$ is a full
subcategory of $\mathscr C^\fin$. Given an object $Y\in\mathscr C$ and
a finite-dimensional $\bk$-vector space $V$, we define two objects
in $\mathscr C^\fin$ by
$$\Homint(V,Y):=\Hom_{\mathscr C}(Y,-)\otimes_\bk V\quad\text{and}
\quad Y\Tenint V:=\Hom_\bk(V,\Hom_{\mathscr C}(Y,-)).$$
These functors are representable: if $V=\bk^n$, then both functors are
represented by $Y^{\oplus n}$. So we will regard $\Homint(V,Y)$ and
$Y\Tenint V$ as being objects in $\mathscr C$ and forget everything
about $\mathscr C^\fin$.
Note however that we gained functoriality in $V$ in the process:
given two $\bk$-vector spaces $V$ and $W$ and an object $Y\in\mathscr C$,
there is a linear map
\begin{equation}
\label{eqn:uHom-functor}
\Hom_\bk(W,V)\to\Hom_{\mathscr C}(\Homint(V,Y),\Homint(W,Y))
\end{equation}
that sends an element $f\in\Hom_\bk(W,V)$ to the image of the identity
by the map
$$\xymatrix@C=7pt{\End_{\mathscr C}(\Homint(W,Y))\ar@{=}[r]&
\Hom_{\mathscr C}(Y,\Homint(W,Y))\otimes_\bk W\ar[d]^{\id\otimes f}&\\
&\Hom_{\mathscr C}(Y,\Homint(W,Y))\otimes_\bk V\ar@{=}[r]&
\Hom_{\mathscr C}(\Homint(V,Y),\Homint(W,Y)).}$$

For two $\bk$-vector spaces $W\subset V$ and two objects $Z\subset Y$
in $\mathscr C$, we define the transporter of $W$ into $Z$ as the
subobject
$$(Z:W):=\ker\bigl(\Homint(V,Y)\to\Homint(W,Y/Z)\bigr)$$
of $\Homint(V,Y)$, where the morphism $\Homint(V,Y)\to\Homint(W,Y/Z)$ is the obvious one.

Now we define
$$P_X=\bigcap_{\substack{n\geq0\\Y\subset X^{\oplus n}}}
\Bigl(\Homint(\omega(X),X)\cap(Y:\omega(Y))\Bigr).$$
Here the small intersection is computed in the ambient object
$\Homint\left(\omega(X)^{\oplus n},X^{\oplus n}\right)$,
the space $\Homint(\omega(X),X)$ being embedded diagonally,
and the large intersection, taken over all $n \geq 0$ and all subobjects
$Y\subset X^{\oplus n}$, is computed in the ambient object
$\Homint(\omega(X),X)$. The existence of this intersection is guaranteed by Lemma~\ref{lem:cat}\eqref{it:lem-cat-4}; moreover, as a subobject of $\Homint(\omega(X),X)
\cong X^{\oplus\dim(\omega(X))}$, the object $P_X$ belongs to
$\langle X\rangle$.

Equation~\eqref{eqn:uHom-functor} provides us with an algebra map
$$\End_\bk(\omega(X))\to\End_{\mathscr C}(\Homint(\omega(X),X)),$$
which induces an algebra map $A_X\to\End_{\mathscr C}(P_X)$.
This map can be seen as a morphism $P_X\Tenint A_X\to P_X$
in $\mathscr C$, and we can thus define the coequalizer
$$P_X\Tenint_{A_X}V:=\mathrm{coeq}\,\bigl(P_X\Tenint (A_X\otimes_\bk V)
\rightrightarrows P_X\Tenint V\bigr)$$
for each $A_X$-module $V$. (Here, one of the maps is induced by the $A_X$-action on $V$ via~\eqref{eqn:uHom-functor}, and the other one by the map $P_X\Tenint A_X\to P_X$ we have just constructed.) We will prove that the functor
$$P_X\Tenint_{A_X}-:\Mod_{A_X}\to\langle X\rangle$$
is an inverse to $\overline\omega_X$.

First, we remark that for any $\bk$-vector space $V$ and any object $Y\in\mathscr C$ there exists a canonical identification
\[
\omega(Y\Tenint V)=\omega(Y)
\otimes_\bk V.
\]
Indeed $\id_{Y\Tenint V}$ defines an element in
\[
\Hom_{\mathscr{C}}(Y\Tenint V, Y \Tenint V) = \Hom_\bk(V, \Hom_{\mathscr{C}}(Y,Y\Tenint V)).
\]
The image of this element under the map
\[
\Hom_\bk(V, \Hom_{\mathscr{C}}(Y,Y\Tenint V)) \to \Hom_\bk(V, \Hom_{\bk}(\omega(Y),\omega(Y\Tenint V)))
\]
induced by $\omega$ provides a canonical element in
\[
\Hom_\bk(V, \Hom_\bk(\omega(Y), \omega(Y\Tenint V))) \cong \Hom_\bk(\omega(Y) \otimes_\bk V, \omega(Y \Tenint V)),
\]
or in other words a canonical morphism $\omega(Y) \otimes_\bk V \to \omega(Y \Tenint V)$. To check that this morphism is invertible one can assume that $V=\bk^n$, in which case the claim is obvious.
%
Likewise, we have an identification
$\omega(\Homint(V,Y))=\Hom_\bk(V,\omega(Y))$.

Using these identifications, the exactness
of $\omega$ implies that given two $\bk$-vector spaces $W\subset V$
and two objects $Z\subset Y$ in $\mathscr C$, 
$$\omega \bigl( (Z:W) \bigr)=\bigl\{\alpha\in\Hom_\bk(V,\omega(Y))\bigm|\alpha(W)
\subset\omega(Z)\bigr\}.$$
From Lemma~\ref{lem:cat}\eqref{it:lem-cat-4}, it then follows that $\omega(P_X)=A_X$ (as a right
$A_X$-module), and therefore, that for each $V\in\Mod_{A_X}$ we have
$$\overline\omega_X(P_X\Tenint_{A_X}V)=\overline\omega_X(P_X)\otimes_{A_X}V
\cong V.$$
Hence $\overline\omega_X\bigl(P_X\Tenint_{A_X}-\bigr)$ is
naturally isomorphic to the identity functor of $\Mod_{A_X}$.

For the other direction, we start by checking that
$$\Hom_\bk(V,\End_{\mathscr C}(Y)\otimes_\bk V)=\Hom_\bk(V,\Hom_{\mathscr C}
(\Homint(V,Y),Y))=\Hom_{\mathscr C}(\Homint(V,Y)\Tenint V,Y).$$
To the canonical element in the left-hand side (defined by $v \mapsto \id_Y \otimes v$)
corresponds a canonical morphism $\Homint(V,Y)\Tenint V\to Y$ in
$\mathscr C$. Considering the latter for $V=\omega(X)$ and $Y=X^{\oplus n}$,
we obtain a canonical map
$$\Homint(\omega(X),X)\Tenint\omega(X^{\oplus n})\to X^{\oplus n},$$
whence by restriction
$$P_X\Tenint_{A_X}\omega(Y)\to Y$$
for any subobject $Y\subset X^{\oplus n}$.
The latter map is an isomorphism because, as we saw above, its image
by $\omega$ is an isomorphism (see Lemma~\ref{lem:cat}\eqref{it:lem-cat-3}). The right exactness
of $\omega$ and of $P_X\Tenint_{A_X}-$ then imply that
$P_X\Tenint_{A_X}\omega(Z)\xrightarrow\sim Z$ for each
subquotient $Z$ of $X^{\oplus n}$, and we conclude that
$P_X\Tenint_{A_X}\overline\omega_X(-)$ is naturally isomorphic to
the identity functor of $\langle X\rangle$.
\end{proof}


In the setup of the proposition, if $X$ and $X'$ are two objects
of $\mathscr C$ such that $\langle X\rangle\subset\langle X'\rangle$
(for instance if $X'$ is of the form $X\oplus Y$), then we have a
restriction morphism
$$A_{X'}\cong\End\bigl(\omega\bigl|_{\langle X'\rangle}\bigr)\to
\End\bigl(\omega\bigl|_{\langle X\rangle}\bigr)\cong A_X.$$
One would like to embrace the whole category $\mathscr C$ by taking
larger and larger subcategories $\langle X\rangle$ and going to the
limit, but the category of finite-dimensional modules over the
inverse limit of a system of algebras is not the union of the
categories of finite-dimensional modules over the algebras. Things work much better if one looks at
comodules over coalgebras, mainly because tensor products commute
with direct limits.

\subsection{Algebras and coalgebras}
\label{ss:alg-coalg}
Let us recall how the dictionary between finite-dimensional algebras
and finite-dimensional coalgebras works (see for instance~\cite[Chap.~III]{kassel}):
\begin{center}
\begin{tabular}{lcl}
$A$ finite-dimensional $\bk$-algebra&
$\longleftrightarrow$& its $\bk$-dual
$B=A^\vee$, a finite-dim. $\bk$-coalgebra;\\[8pt]
$m:A\otimes_\bk A\to A$ multiplication&
$\longleftrightarrow$&
$\Delta:B\to B\otimes_\bk B$ comultiplication (coassociative)\\
(associative)&&
$\varepsilon:B\to \bk$ counit\\
$\eta:\bk\to A$ unit&&(obtained by transposing $m$ and $\eta$);\\[8pt]
left $A$-module structure&
$\longleftrightarrow$&
right $B$-comodule structure on $M$ with\\
on a space $M$ with action  &&
coaction map $\delta:M\to M\otimes_\bk B$\\
map $\mu:A\otimes_\bk M\to M$ &&defined by $\mu(a\otimes m)=(\id_M\otimes\mathrm{ev}_a)\circ\delta(m)$,
where\\
&&$\mathrm{ev}_a:B\to\bk$ is the evaluation at $a$.
\end{tabular}
\end{center}
In the context of this dictionary, one can identify the category
$\Mod_A$ of finite-dimensional left $A$-modules with the category
$\Comod_B$ of finite-dimensional right $B$-comodules.

\subsection{A second reconstruction theorem}
\label{ss:reconstruction-2}

Going back to the setting of~\S\ref{ss:reconstruction-1}, we see that whenever $\langle X\rangle
\subset\langle X'\rangle$, we get a morphism of coalgebras
$$B_X:=A_X^\vee\to A_{X'}^\vee=:B_{X'}.$$
Now we can choose $X$ with more and more direct summands, so that
$\langle X\rangle$ grows larger and larger. Our running assumption that
all categories are essentially small allows us to take the direct limit of the coalgebras $B_X$ over the set of isomorphism classes of objects of $\mathscr{C}$, for the order determined by the inclusions $\langle X\rangle
\subset\langle X'\rangle$. We then obtain the following statement.

\begin{thm}
\label{thm:reconstruction-coalgebras}
Let $\mathscr C$ and $\omega$ be as in Proposition~{\rm \ref{prop:reconstruction}}. Set
\[
B:=\varinjlim_X B_X.
\]
Then $\omega$ admits a canonical factorization
$$\xymatrix@=40pt@!0{\mathscr C\ar[rr]^(.46){\overline\omega}
\ar[dr]_(.4)\omega&&\Comod_B\ar[dl]^(.4){\forget}\\&\Vect_\bk&}$$
where $\overline\omega$ is an equivalence of categories.
\end{thm}

Here, the fact that $\omega(X)$ admits a structure of $B$-comodule (for $X$ in $\mathscr C$) means that there exists a canonical morphism $\omega(X) \to \omega(X) \otimes_\bk B$ satisfying the appropriate axioms. In other words, we have obtained a canonical morphism of functors $\omega \to \omega \otimes_\bk B$, where the right-hand side means the functor $X \mapsto \omega(X) \otimes_\bk B$ (and where we omit the natural functor from $\Vect_\bk$ to the category of \emph{all} $\bk$-vector spaces).

\begin{ex}\phantomsection
\label{ex:reconstruction}
\begin{enumerate}
\item
Let $V$ be a finite-dimensional $\bk$-vector space, and take
$\mathscr C=\Vect_\bk$ and $\omega=V\otimes_\bk-$. Then $B=\End_\bk(V)^\vee$.
Indeed, the category of finite-dimensional left $\End_\bk(V)$-modules
is semisimple, with just one simple object up to isomorphism, namely~$V$.
\item
\label{it:ex-recontruction-2}
Let $M$ be a set, let $\mathscr C = \Vect_\bk(M)$ be the category of finite-dimensional
$M$-graded $\bk$-vector spaces, and $\omega:\mathscr C\to\Vect_\bk$ be the
functor that forgets the $M$-grading. Then $B=\bk M$, the $\bk$-vector
space with basis $M$, with the coalgebra structure given by
$$\Delta(m)=m\otimes m,\quad\varepsilon(m)=1$$
for all $m\in M$. For each $X\in\mathscr C$, the coaction of $B$ on
$\omega(X)=X$ is the map
$$X\to X\otimes_\bk B,\quad x\mapsto\sum_{m\in M}x_m\otimes m,$$
where $x=\sum_{m\in M}x_m$ is the decomposition of $x$ into its
homogeneous components.
\item
\label{it:ex-recontruction-3}
Let $C$ be a coalgebra, and take $\mathscr{C}=\Comod_C$, with $\omega$ the forgetful functor. Then there exists a canonical isomorphism $C \cong B$. Indeed, if $X \in \Comod_C$, there exists a finite-dimensional
subcoalgebra $C'\subset C$ such that the $C$-comodule $X$ is actually
a $C'$-comodule.\footnote{In view of its importance, let us briefly recall the proof of this classical fact. Let $\delta:X\to X\otimes_\bk C$ be
the structure map of the $C$-comodule $X$ and let $(e_1,\ldots,e_n)$ be
a $\bk$-basis of $X$. Write $\delta(e_j)=\sum_ie_i\otimes c_{i,j}$. Then
$\Delta(c_{i,j})=\sum_kc_{i,k}\otimes c_{k,j}$ and $\varepsilon(c_{i,j})=
\delta_{i,j}$ (Kronecker's symbol), so $C'$ can be chosen as the $\bk$-span in $C$ of the
elements $c_{i,j}$.} Then the coaction morphism $X \to X \otimes_\bk C'$ defines an algebra morphism $(C')^\vee \to A_X$, hence a coalgebra morphism $B_X \to C'$. Composing with the embedding $C' \hookrightarrow C$ and passing to the limit we deduce a coalgebra morphism $B \to C$. In the reverse direction, if $C' \subset C$ is a finite-dimensional subcoalgebra, then $C'$ is an object in $\Comod_C$, hence it acquires a canonical $B$-comodule structure, i.e.~a coalgebra map $C' \to C' \otimes_\bk B$. Composing with the map induced by $\varepsilon$ we deduce a coalgebra morphism $C' \to B$. Since $C$ is the direct limit of its finite-dimensional subcoalgebras, we deduce a coalgebra morphism $C \to B$. It is easily seen that the morphisms we constructed are inverse to each other, proving our claim.
\end{enumerate}
\end{ex}

An homomorphism of $\bk$-coalgebras $f:B\to C$ induces a functor
$f_*:\Comod_B\to\Comod_C$. Specifically, given a $B$-comodule
$M$ with structure map $\delta:M\to M\otimes_\bk B$, the $C$-comodule
$f_*M$ has the same underlying $\bk$-vector space as $M$ and has
structure map $(\id_M\otimes f)\circ\delta:M\to M\otimes_\bk C$.

\begin{prop}\phantomsection
\label{prop:reconstruction-morph}
\begin{enumerate}
\item
\label{it:reconstruction-morph-1}
Let $\mathscr C$ be an abelian $\bk$-linear category, let $C$ be a
$\bk$-coalgebra, and let $F:\mathscr C\to\Comod_C$ be a $\bk$-linear
exact faithful functor. 
If $B$ is the coalgebra provided by Theorem~{\rm \ref{thm:reconstruction-coalgebras}} (for the functor given by the composition of $F$ with the forgetful functor $\Comod_C \to \Vect_\bk$) and $\overline{F} : \mathscr{C} \xrightarrow{\sim} \Comod_B$ the corresponding equivalence, then
there exists 
a unique morphism of $\bk$-coalgebras $f:B\to C$ such that 
the following diagram commutes:
$$\xymatrix@=40pt@!0{\mathscr C\ar[rr]^(.46){\overline F}
\ar[dr]_(.4)F&&\Comod_B\ar[dl]^(.4){f_*}\\&\Comod_C.&}$$
\item
\label{it:reconstruction-morph-2}
Let $B$ and $C$ be two $\bk$-coalgebras. Any $\bk$-linear functor
$F:\Comod_B\to\Comod_C$ such that
$$\xymatrix@=40pt@!0{\Comod_B\ar[rr]^F
\ar[dr]_(.4)\forget&&\Comod_C\ar[dl]^(.4)\forget\\&\Vect_\bk&}$$
commutes is of the form $F=f_*\,$ for a unique morphism of
coalgebras $f:B\to C$.
\end{enumerate}
\end{prop}

\begin{proof}
\eqref{it:reconstruction-morph-1}
Let $X$ be an object in $\mathscr C$. As seen in Example~\ref{ex:reconstruction}\eqref{it:ex-recontruction-3}, there exists a finite-dimensional
subcoalgebra $C'\subset C$ such that the $C$-comodule $F(X)$ is actually
a $C'$-comodule.
The restriction to the
category $\langle X\rangle$ of the functor $F$ then factorizes through
$\Comod_{C'}=\Mod_{(C')^\vee}$. Let $\omega:\Mod_{(C')^\vee}\to\Vect_\bk$ be
the forgetful functor and let $A_X$ be the endomorphism algebra of the
functor $\omega\circ F\bigl|_{\langle X\rangle}$.

Consider the diagram
$$\xymatrix@C=64pt@R=34pt@!0{\langle X\rangle\ar[rr]^{\overline F_X}
\ar[dr]_(.4)F&&\Mod_{A_X}\ar@{-->}[dl]\ar[ddl]\\&\Mod_{(C')^\vee}
\ar[d]_(.46)\omega&\\&\Vect_\bk.&}$$
Any $\alpha\in (C')^\vee$ can be seen as an endomorphism of the
functor $\omega$, so induces by restriction an endomorphism of
$\omega\circ F\bigl|_{\langle X\rangle}$, or in other words an element
of $A_X$. Our situation thus gives us a morphism of algebras
$(C')^\vee\to A_X$, that is, a morphism of coalgebras $A_X^\vee\to C'$.
Further, $\overline F_X$ is an equivalence of categories, because the
$\bk$-linear functor $\omega\circ F\bigl|_{\langle X\rangle}$ is exact
and faithful. 
Taking as before the limit over $\langle X\rangle$ yields
the desired coalgebra $B=\varinjlim A_X^\vee$, the morphism of
coalgebras $f:B\to C$, and the equivalence of categories $\overline F$.

\eqref{it:reconstruction-morph-2}
Statement~\eqref{it:reconstruction-morph-2} is essentially the special case of~\eqref{it:reconstruction-morph-1} in the context of Example~\ref{ex:reconstruction}\eqref{it:ex-recontruction-3}. More concretely,
the coalgebra $B$ is a right comodule over itself. The functor $F$ maps
it to a $C$-comodule with the same underlying vector space. We thus get
a structure map $\delta:B\to B\otimes C$. Composing with the augmentation
$\varepsilon:B\to\bk$, we get a map $f=(\varepsilon\otimes\id_C)\circ\delta$
from $B$ to $C$. Abstract nonsense arguments (the functoriality of $F$
and the axioms of comodules) imply that $f:B\to C$ is a coalgebra map
and that $F=f_*$.
\end{proof}

\subsection{Tannakian reconstruction}

As we saw in~\S\ref{ss:reconstruction-2}, a $\bk$-linear abelian category equipped with
an exact faithful $\bk$-linear functor to $\Vect_\bk$ is equivalent to the category of
right comodules over a $\bk$-coalgebra $B$ equipped with the forgetful
functor. On the other hand, an affine group scheme $G$ over $\bk$ is a scheme
represented by a commutative $\bk$-Hopf algebra $H=\bk[G]$, and representations of
$G$ are the same as right $H$-comodules (see e.g.~\cite[\S 1.4 and \S 3.2]{waterhouse} or~\cite[\S I.2]{jantzen}).
A commutative Hopf algebra is a coalgebra with the extra datum of an
associative and commutative multiplication with unit, plus the existence
of the antipode. Striving to translate this setup into the language of
categories, we look for the extra structures on an abelian $\bk$-linear
category that characterize categories of representations of affine
group schemes.

The adequate notion is called rigid abelian tensor categories.
Rather than studying this notion in the greatest possible
generality, which would take too much space for the expected benefit,
we will state a theorem tailored to our goal of understanding the
geometric Satake equivalence. For a more thorough (and formal)
treatment, the reader is referred to~\cite{sr} and~\cite{deligne}, or to~\cite{kassel} for a more leisurely walk.

A last word before stating the main theorem of this subsection: a
multiplication map
\[
 \mathrm{mult} : B\otimes_\bk B\to B
\]
on a coalgebra $B$ which is a coalgebra morphism allows
to define a structure of $B$-comodule on the tensor product over
$\bk$ of two $B$-comodules. Specifically, if $M$ and $M'$ are
$B$-comodules with structure maps $\delta_M:M\to M\otimes_\bk B$ and
$\delta_{M'}:M'\to M'\otimes_\bk B$, then the structure map on
$M\otimes_\bk M'$ is defined by the composition depicted on the diagram
$$\xymatrix{M\otimes_\bk M'\ar[r]^{\delta_{M\otimes M'}}
\ar[d]_{\delta_M\otimes\delta_{M'}}&M\otimes_\bk M'\otimes_\bk B\\
M\otimes_\bk B\otimes_\bk M'\otimes_\bk B\ar[r]&M\otimes_\bk M'\otimes_\bk
B\otimes_\bk B\ar[u]_{\id_{M\otimes M'}\otimes\mathrm{mult}}}$$
where the bottom arrow is the usual commutativity constraint for
tensor products of $\bk$-vector spaces that swaps the second and third
factors.

Given an affine group scheme $G$ over $\bk$, we denote the category of
finite-dimensional representations of $G$ (or equivalently finite dimensional right $\bk[G]$-comodules) by $\Rep_\bk(G)$.

\begin{thm}
\label{thm:tannakina-reconstruction}

Let $\mathscr C$ be an abelian $\bk$-linear category equipped with the
following data:
\begin{itemize}
\item
an exact $\bk$-linear faithful functor $\omega:\mathscr C\to\Vect_\bk$
(called the fiber functor);
\item
a $\bk$-bilinear functor $\otimes:\mathscr C\times\mathscr C\to\mathscr C$
(the tensor product);
\item
an object $U\in\mathscr C$ (the tensor unit);
\item
an isomorphism $\phi_{X,Y,Z}:X\otimes(Y\otimes Z)\xrightarrow{\sim}(X\otimes Y)\otimes
Z$, natural in $X$, $Y$ and $Z$ (the associativity constraint);
\item
isomorphisms $U\otimes X\xrightarrow[\sim]{\lambda_X}X\xleftarrow[\sim]{\rho_X}
X\otimes U$, both natural in $X$ (the unit constraints);
\item
an isomorphism $\psi_{X,Y}:X\otimes Y\xrightarrow{\sim} Y\otimes X$ natural in $X$
and $Y$ (the commutativity constraint).
\end{itemize}
We also assume we are given isomorphisms $\upsilon:\bk\xrightarrow{\sim}\omega(U)$ and
\begin{equation}
\label{eqn:tauXY}
 \tau_{X,Y}:\omega(X)\otimes_\bk \omega(Y)\xrightarrow{\sim}
\omega(X\otimes Y)
\end{equation}
in $\Vect_\bk$, with $\tau_{X,Y}$ natural in $X,Y \in \mathscr C$.
Finally, we assume the following conditions hold:
\begin{enumerate}
\item
\label{it:compatibility-2}
Taking into account the identifications provided by $\tau$ and $\upsilon$,
the isomorphisms $\omega(\phi_{X,Y,Z})$, $\omega(\lambda_X)$,
$\omega(\rho_X)$ and $\omega(\psi_{X,Y})$ are the usual associativity,
unit and commutativity constraints in $\Vect_\bk$.
\item
\label{it:inverses}
If $\dim_\bk \bigl( \omega(X) \bigr)=1$, then there exists $X^*\in\mathscr C$ such that
$X\otimes X^*\cong U$.
\end{enumerate}
Under these assumptions, there exists an affine group scheme $G$ such
that $\omega$ admits a canonical factorization
$$\xymatrix@=40pt@!0{\mathscr C\ar[rr]^(.46){\overline\omega}
\ar[dr]_(.4)\omega&&\Rep_\bk(G)\ar[dl]^(.4){\forget}\\&\Vect_\bk&}$$
where $\overline\omega$ is an equivalence of categories that respects
the tensor product and the unit in the sense of the compatibility
condition~\eqref{it:compatibility-2}.
\end{thm}

\begin{rmk}\phantomsection
\label{rmk:Tannaka}
\begin{enumerate}
\item
\label{it:G-end-fiber-functor}
It will be clear from the proof below that
the group scheme $G$ is the ``automorphism group of the fiber functor.''
This sentence means that for any commutative $\bk$-algebra $R$, an element
$\alpha\in G(R)$ is a collection of elements
$\alpha_X\in\End_R(\omega(X)\otimes_\bk R)$, natural in $X \in \mathscr{C}$, and compatible
with $\otimes$ and $U$ via the isomorphisms $\tau$ and $\upsilon$. There
is no need to specifically ask for invertibility: this will automatically
follow from the compatibility condition~\eqref{it:inverses}.
\item
\label{it:fiber}
The datum of isomorphisms~\eqref{eqn:tauXY} satisfying condition~~\eqref{it:compatibility-2}
are usually worded as: ``the functor $\omega$ is a tensor functor.''
\item
The faithfulness of $\omega$ and the compatibility condition~\eqref{it:compatibility-2} imply
that the associativity constraint $\phi$ (respectively, the unit
constraints $\lambda$ and $\rho$, the commutativity constraint $\psi$)
of $\mathscr C$ satisfies MacLane's pentagon axiom (respectively, the
triangle axiom, the hexagon axiom). Together, these coherence axioms
imply that any diagram built from the constraints commutes. This makes multiple
tensor products in $\mathscr C$ non-ambiguous, see~\cite[\S VII.2]{maclane}.
\item
Our formulation dropped completely the ``rigidity condition'' in the
usual formulation of the Tannakian reconstruction theorem. This condition
demands that each object $X$ has a dual $X^\vee$ characterized by an
evaluation map $X^\vee\otimes X\to U$ and a coevaluation map
$U\to X\otimes X^\vee$. Its purpose is to guarantee the existence
of inverses in $G$---without it, $G$ would only be an affine monoid
scheme. In Theorem~\ref{thm:tannakina-reconstruction}, it has been replaced by condition~\eqref{it:inverses}, which
is easier to check in the case we have in mind. See~\cite[Proposition~1.20 and Remark~2.18]{dm} for a more precise study of the relationship between these conditions.
\item
As in Example~\ref{ex:reconstruction}\eqref{it:ex-recontruction-3}, if we start with the category $\mathscr{C}=\Rep_\bk(G)$ for some $\bk$-group scheme $G$, with $\omega$ being the natural forgetful functor, then the group scheme reconstructed in Theorem~\ref{thm:tannakina-reconstruction} identifies canonically with $G$.
\end{enumerate}
\end{rmk}

\begin{proof}
We first remark that the bifunctor $\otimes:\mathscr C\times\mathscr
C\to\mathscr C$ is exact in each variable: this follows from the
analogous fact in the category $\Vect_\bk$ together with Lemma~\ref{lem:cat}\eqref{it:lem-cat-3}.

We reuse the notation $\langle X\rangle$, $A_X$ and $B_X$ from
\S\S\ref{ss:reconstruction-1}--\ref{ss:reconstruction-2}. The direct limit of the coalgebras $B_X$ is the coalgebra
$B$, with comultiplication $\Delta$ and counit $\varepsilon$.

Let $X$ and $X'$ two objects in $\mathscr C$. The isomorphism
$\tau_{X,X'}:\omega(X)\otimes_\bk \omega(X')\xrightarrow{\sim}\omega(X\otimes X')$
induces an isomorphism of algebras
\begin{equation}
\label{eqn:End-tensor}
\End_\bk(\omega(X\otimes X'))\cong\End_\bk(\omega(X))\otimes_\bk
\End_\bk(\omega(X')).
\end{equation}
The opening remark in this proof implies that given subobjects
$Y\subset X^{\oplus n}$ and $Y'\subset(X')^{\oplus n'}$, the tensor
products $Y\otimes X'$ and $X\otimes Y'$ are subobjects of
respectively $(X\otimes X')^{\oplus n}$ and $(X\otimes X')^{\oplus n'}$.
It follows that the isomorphism~\eqref{eqn:End-tensor} takes $A_{X\otimes X'}$ into
$A_X\otimes_\bk A_{X'}$. Taking the duals, we get a morphism of coalgebras
$B_{X'}\otimes_\bk B_X\to B_{X\otimes X'}$, and taking the direct limit
over $\langle X\rangle$ and $\langle X'\rangle$, we obtain a morphism
of coalgebras $m:B\otimes_\bk B\to B$.

On the other hand, the $\bk$-vector space $\omega(U)$ has dimension~$1$,
so the algebra $A_U$ is reduced to $\End_\bk(\omega(U))=\bk$. Thus
$B_U$ is the trivial one-dimensional $\bk$-coalgebra, and the definition
of $B$ as a direct limit of the $B_X$ (including $B_U$) leads to a
morphism of coalgebras $\eta:\bk\to B$.

Our coalgebra $B$ is thus equipped with a multiplication
$m:B\otimes_\bk B\to B$ and a unit $\eta:\bk\to B$. The compatibility
condition~\eqref{it:compatibility-2} implies that $(B,m,\eta)$ is an associative and
commutative $\bk$-algebra with unit.

Let us call $G$ the spectrum of the commutative $\bk$-algebra $(B,m,\eta)$;
this is an affine scheme over $\bk$. The commutative diagrams that express
the fact that $m$ and $\eta$ are morphisms of coalgebras also say that
$\Delta$ and $\varepsilon$ are morphisms of algebras. The latter thus
define morphisms of schemes
$$\Delta^*:G\times_{\Spec(\bk)} G\to G\quad\text{and}\quad\varepsilon^*:\Spec
(\bk)\to G.$$
The coassociativity of $\Delta$ and the counit property then imply that
$(G,\Delta^*,\varepsilon^*)$ is an affine
monoid scheme. It thus only remains to show that the elements of $G$ are
invertible.

Unwinding the construction that led to the definition of $G$, we see
that for any commutative $\bk$-algebra $R$, an element $\alpha\in G(R)$
is a collection of elements $\alpha_X\in\End_R(\omega(X)\otimes_\bk R)$,
natural in $X$, and compatible with $\otimes$ and $U$. We want to show
that $\alpha_X$ is invertible for all objects $X$.

First, if $X$ is such that $\dim\omega(X)=1$, then
by condition~\eqref{it:inverses} there exists $X^*$ such that $X\otimes X^*\cong U$,
and therefore $\alpha_X\otimes\alpha_{X'}$ is conjugate to $\alpha_U
=\id_{\omega(U)\otimes_\bk R}$, so $\alpha_X$ (an endomorphism of a free
$R$-module of rank~$1$) is invertible.

To deal with the general case, one constructs the exterior power
$\bigwedge^dX$ as the quotient of the $d$-th tensor power of $X$ by the appropriate relations (defined with the help of the commutativity constraint of
$\mathscr C$), for $d=\dim\omega(X)$. Since $\omega$ is compatible
with the commutativity constraint, $\omega\bigl(\bigwedge^dX\bigr)
\cong\bigwedge^d\omega(X)$ is $1$-dimensional. As we saw in our
particular case, this implies that $\alpha_{\bigwedge^dX}$ is
invertible. But this is $\bigwedge^d\alpha_X$ (in other words the determinant of $\alpha_X$), so we eventually
obtain that~$\alpha_X$ (an endomorphism of a free $R$-module of
rank $d$) is invertible.
\end{proof}

\begin{ex}\phantomsection
\label{ex:tannakian-reconstruction}
\begin{enumerate}
\item
\label{it:graded-Vect}
Continue with Example~{\rm \ref{ex:reconstruction}\eqref{it:ex-recontruction-2}}, and
suppose now that our category $\mathscr C$ of finite-dimensional
$M$-graded $\bk$-vector spaces is endowed with a tensor product
$\otimes$. There is then a law $*$ on $M$ such that
$$\bk[m] \otimes \bk[n]=\bk[m*n]$$
for all $m,n\in M$ (where $\bk[p]$ means $\bk$ placed in
degree $p$). The constraints~\eqref{it:compatibility-2} in the theorem impose
that $M$ is a commutative monoid, and then $B=\bk M$ is the associated monoid
algebra. The condition~\eqref{it:inverses}, if verified, implies that $M$ is
indeed a group. The affine group scheme $G=\Spec (B)$ given by the
theorem is then the Cartier dual of $M$ (see~\cite[\S2.4]{waterhouse}).
\item
Let $X$ be a connected topological manifold, let $\mathscr C$ be the
category of local systems on $X$ with coefficients in $\bk$, let $x\in X$,
and let $\omega$ be the functor $\mathcal L\mapsto\mathcal L_x$, the
fiber at point $x$. Then $G$ is the constant group scheme equal to
the fundamental group $\pi_1(X,x)$. On this example, we see how the
choice of a fiber functor subtly changes the group.
\item
\label{it:SVec}
We define the category $\SVec_\bk$ of supersymmetric $\bk$-vector spaces
as the category of $\mathbf Z/2\mathbf Z$-graded vector spaces, equipped
with the usual tensor product, with the usual associativity and unit
constraints, but with the supersymmetric commutativity constraint:
for $V=V_{\bar0}\oplus V_{\bar1}$ and $W=W_{\bar0}\oplus W_{\bar1}$,
the isomorphism $\psi_{V,W}:V\otimes_\bk W\to W\otimes_\bk V$ is defined as
$$\psi_{V,W}(v\otimes w)=\begin{cases}
w\otimes v&\text{if $v\in V_{\bar0}$ or $w\in W_{\bar0}$,}\\
-(w\otimes v)&\text{if $v\in V_{\bar1}$ and $w\in W_{\bar1}$.}
\end{cases}$$
Then the forgetful functor from $\SVec_\bk$ to $\Vect_\bk$ does not
respect the commutativity constraints, so one cannot apply the
theorem to this situation.
\end{enumerate}
\end{ex}

Proposition~\ref{prop:reconstruction-morph} also has a tensor analog. We state
without proof the assertion that is needed in the proof of the
geometric Satake equivalence. Observe that a homomorphism of
$\bk$-group schemes $f:H\to G$ induces a restriction functor
$f^*:\Rep_\bk(G)\to\Rep_\bk(H)$.

\begin{prop}\phantomsection
\label{prop:tannakian-morph}
\begin{enumerate}
\item
Let $\mathscr C$ be an abelian $\bk$-linear category with tensor
product, tensor unit, and associativity, commutativity and unit constraints.
Let $H$ be an
affine group scheme over $\bk$. Let $F:\mathscr C\to\Rep_\bk(H)$ be a
$\bk$-linear exact faithful functor, compatible with the monoidal structure and the
various constraints (in the same sense as in Theorem~{\rm \ref{thm:tannakina-reconstruction}}). Let $G$ be the affine $\bk$-group scheme provided by Theorem~{\rm \ref{thm:tannakina-reconstruction}} (for the composition of $F$ with the forgetful functor $\Rep_\bk(H) \to \Vect_\bk$) and $\overline{F} : \mathscr{C} \xrightarrow{\sim} \Rep_\bk(G)$ be the corresponding equivalence. Then there exists
a unique morphism of group schemes $f:H\to G$ such that the following diagram commutes:
$$\xymatrix@=40pt@!0{\mathscr C\ar[rr]^(.46){\overline F}
\ar[dr]_(.4)F&&\Rep_\bk(G)\ar[dl]^(.4){f^*}\\&\Rep_\bk(H).&}$$
 \item 
 If $G$ and $H$ are affine $\bk$-group schemes, any $\bk$-linear functor $F : \Rep_\bk(G) \to \Rep_\bk(H)$ compatible with tensor products, the forgetful functors, and the various constraints is of the form $f^*$ for a unique $\bk$-group scheme morphism $f : H \to G$.
\end{enumerate}
\end{prop}

\subsection{Properties of $G$ visible on $\Rep_\bk(G)$}

Recall that an affine $\bk$-group scheme $G$ is called \emph{algebraic} if the $\bk$-algebra of regular functions on $G$ is finitely generated. In this case, the construction of the group of connected components $\pi_0(G)$ of $G$ is recalled in~\cite[\S XIII.3]{milne} or~\cite[\S 6.7]{waterhouse}; this is an affine \'etale $\bk$-group scheme endowed with a canonical morphism $G \to \pi_0(G)$, and $G$ is connected iff $\pi_0(G)$ is trivial.

Recall also that an affine algebraic group scheme $G$ over $\bk$ is called \emph{reductive}\footnote{Sometimes, the definition of reductive groups allows disconnected groups. All the reductive groups we will consider in these notes will be (sometimes tacitly) assumed to be connected. Note that connectedness can be checked after base change to an algebraic closure of the base field; see~\cite[Proposition~XIII.3.8]{milne}.\label{fn:connectedness}} if it is smooth\footnote{Recall that this condition is automatic if $G$ is algebraic and $\mathrm{char}(\bk)=0$; see~\cite[Theorem~VI.9.3]{milne}.\label{fn:smooth-char-0}} (hence in particular algebraic) and connected and, for an algebraic closure $\overline{\bk}$ of $\bk$, the group $\Spec(\overline{\bk}) \times_{\Spec(\bk)} G$ is reductive in the usual sense, i.e.~does not contain any nontrivial smooth connected normal unipotent subgroup; see~\cite[Definition~XVII.2.1]{milne}.

\begin{prop}\phantomsection
\label{prop:properties-G}
\begin{enumerate}
\item
\label{it:properties-G-1}
Let $G$ be an affine group scheme over $\bk$. Then $G$ is algebraic if and
only if there exists $X\in\Rep_\bk(G)$ such that $X$ generates $\Rep_\bk(G)$
by taking direct sums, tensor products, duals, and subquotients.
\item
\label{it:properties-G-2}
Let $G$ be an algebraic affine group scheme over $\bk$. If $G$ is not connected,
then there exists a nontrivial representation $X\in\Rep_\bk(G)$ such that
the subcategory $\langle X\rangle$ (with the notation of~\S{\rm \ref{ss:reconstruction-1}}) is stable under $\otimes$.
\item
\label{it:properties-G-3}
Let $G$ be a connected algebraic affine group scheme over $\bk$. Assume that $\bk$ has
characteristic $0$, and that $\Rep_{\overline{\bk}}(G_{\overline{\bk}})$ is semisimple, where $\overline{\bk}$ is an algebraic closure of $\bk$ and $G_{\overline{\bk}} = \Spec(\overline{\bk}) \times_{\Spec(\bk)} G$. Then $G$ is reductive.
\end{enumerate}
\end{prop}

\begin{proof}
\eqref{it:properties-G-1} Suppose $G$ is algebraic. Then $G$ admits a faithful representation
(see \cite[\S3.4]{waterhouse}), i.e.~$G$ can be viewed as a closed subgroup of
some $\mathbf{GL}_n$. It is then a classical result that any finite dimensional
representation of $G$ can be obtained from the representation on $\bk^n$
by the processes of forming tensor products, direct sums,
subrepresentations, quotients and duals (see~\cite[\S3.5]{waterhouse}).

Conversely, suppose the existence of a representation $X$ that
generates $\Rep_\bk(G)$ in the sense explained in the statement of
the proposition. Then $X$ is necessarily a faithful representation
of $G$, so $G$ embeds as a closed subgroup in $\mathbf{GL}(X)$ (see \cite[\S15.3]{waterhouse}), and
is therefore algebraic.

\eqref{it:properties-G-2} 
The quotient $G\to\pi_0(G)$ induces a fully faithful functor
$\Rep_\bk(\pi_0(G))\to\Rep_\bk(G)$. Taking for $X$ the image of the regular
representation of $\pi_0(G)$, we see that
$\langle X\rangle$ (which coincides with the essential image of $\Rep_\bk(\pi_0(G))$) is stable under tensor products.
If $G$ is not connected, then $X$ is not trivial.

\eqref{it:properties-G-3} 
As explained in Footnote~\ref{fn:smooth-char-0}, $G$ is automatically smooth since $\mathrm{char}(\bk)=0$. Hence the only thing we have to check is that the unipotent radical $R_u(G_{\overline{\bk}})$ is trivial.
In a simple representation $X$ of $G_{\overline{\bk}}$,
$R_u(G_{\overline{\bk}})$ acts trivially; indeed the subspace of points 
fixed by $R_u(G_{\overline{\bk}})$ is nontrivial by Kolchin's fixed point theorem
(see \cite[\S8.2]{waterhouse}) and is $G_{\overline{\bk}}$-stable because $R_u(G_{\overline{\bk}})$ is a normal
subgroup. This result immediately extends to semisimple representations
of $G_{\overline{\bk}}$. Now if $\Rep_{\overline{\bk}}(G_{\overline{\bk}})$ is semisimple, then $G_{\overline{\bk}}$ admits
a semisimple faithful representation. On this representation, $R_u(G_{\overline{\bk}})$
acts trivially and faithfully. Therefore $R_u(G_{\overline{\bk}})$ is trivial.
\end{proof}

\begin{rmk}
\begin{enumerate}
\item
An object which satisfies the conditions in Proposition~\ref{prop:properties-G}\eqref{it:properties-G-1} will be called a \emph{tensor generator} of the category $\Rep_\bk(G)$.
\item
An algebraic affine group scheme is called \emph{strongly connected} if it admits no nontrivial finite quotient. (This property is in general stronger than being connected---which is equivalent to having no nontrivial finite \emph{\'etale} quotient---but these notions are equivalent if $\mathrm{char}(\bk)=0$.)
If $G$ is an algebraic affine group scheme over $\bk$, the condition appearing in
Proposition~\ref{prop:properties-G}\eqref{it:properties-G-2} is equivalent to $G$ being strongly connected, see~\cite[\S XVII.7]{milne}.
\item
A more precise version of Proposition~\ref{prop:properties-G}\eqref{it:properties-G-3} is proved in~\cite[Proposition~2.23]{dm}. (But the simpler version we stated will be sufficient for our purposes.)
\end{enumerate}
\end{rmk}

\section{The affine Grassmannian}
\label{sec:Gr}


%
%

In this section we provide a brief introduction to the affine Grassmannian of a complex connected reductive algebraic group. For more details, examples and references, the reader can e.g.~consult~\cite[\S 2]{goertz} or~\cite{zhu}. (All of these properties are often considered as ``well known,'' and we have not tried to give the original references for them, but rather the most convenient one.)

\subsection{Definition}
\label{ss:def-Gr}

We set $\mathcal{O}:=\mathbf{C}[ \hspace{-1pt} [t] \hspace{-1pt}]$ and $\mathcal{K} := \mathbf{C}(\hspace{-1pt}(t)\hspace{-1pt})$, where $t$ is an indeterminate.
If $H$ is a linear complex algebraic group, we denote by $H_\mathcal{O}$, resp.~$H_\mathcal{K}$, the functor from $\mathbf{C}$-algebras to groups defined by
\[
R \mapsto H \bigl( R[ \hspace{-1pt} [t] \hspace{-1pt}] \bigr), \quad \text{resp.} \quad R \mapsto H \bigl( R( \hspace{-1pt} (t) \hspace{-1pt}) \bigr).
\]
It is not difficult to check that $H_{\mathcal{O}}$ is represented by a
$\mathbf{C}$-group scheme (not of finite type in most cases), and that $H_{\mathcal{K}}$ is represented by an ind-group scheme (i.e.~an inductive limit of schemes parametrized by $\mathbf{Z}_{\geq 0}$, with closed embeddings as transition maps). We will still denote these (ind-)group schemes by $H_{\mathcal{O}}$ and $H_\mathcal{K}$.

We now fix a standard triple $G\supset B\supset T$ of a connected
complex reductive algebraic group, a Borel subgroup, and a maximal torus. We will denote by $N$
the unipotent radical of $B$. We will denote by $\Delta(G,T)$ the root system of $(G,T)$, by $\Delta_+(G,B,T) \subset \Delta(G,T)$ the subset of positive roots (consisting of the $T$-weights in the Lie algebra of $B$), and by $\Delta_{\mathrm{s}}(G,B,T)$ the corresponding subset of simple roots. For $\alpha$ a root, we will denote the corresponding coroot by $\alpha^\vee$.

Let $X_*(T)$ be the lattice of
cocharacters of $T$; it contains the coroot lattice $Q^\vee$
and is endowed with the standard order $\leq$ (such that nonnegative elements are
nonnegative integral combinations of positive coroots). We will denote by $X_*(T)^+\subset X_*(T)$ the cone of dominant
cocharacters. We define $\rho$ as the halfsum of the positive roots and regard it as a linear form $X_*(T)\to\frac12\mathbf Z$.


If $L^{\leq 0} G$ denotes the ind-group scheme which represents the functor
\[
R \mapsto G \bigl( R[t^{-1}] \bigr)
\]
and if $L^{<0}G$ is the kernel of the natural morphism $L^{\leq 0} G \to G$ (sending $t^{-1}$ to $0$), then $L^{<0} G$ is a subgroup of $G_{\mathcal{K}}$ in a natural way, and the multiplication morphism
\[
L^{<0} G \times G_{\mathcal{O}} \to G_{\mathcal{K}}
\]
is an open embedding by~\cite[Lemma~3]{faltings} (see also~\cite[Lemme~2.1]{ngo-polo} or~\cite[Lemma~2.3.5]{zhu}). In view of this property, the quotient
\[
\widetilde{\Gr}_G:=G_{\mathcal{K}}/G_{\mathcal{O}}
\]
has a natural structure of ind-scheme. 
In fact, one can check that this ind-scheme is ind-proper, and of ind-finite type.

\begin{rmk}\phantomsection
\label{rmk:def-Gr}
\begin{enumerate}
 \item 
 In many references (but not~\cite{faltings}), $\widetilde{\Gr}_G$ is rather defined as the object representing a certain presheaf on the category of affine $\C$-schemes (see e.g.~\cite[Theorem~1.2.2]{zhu}) and then identified with a fpqc quotient $G_{\mathcal{K}}/G_{\mathcal{O}}$, see~\cite[Proposition~1.3.6]{zhu}. Finally, it is realized that the quotient map $G_{\mathcal{K}} \to \widetilde{\Gr}_G$ is Zariski locally trivial. In these notes the ``moduli interpretation'' of $\widetilde{\Gr}_G$ will be introduced in Section~\ref{sec:convolution-BD} below.
 \item
 \label{it:def-Fl}
 Consider the group scheme $\mathfrak{G} := G \times_{\Spec(\C)} \Spec(\mathcal{O})$ over $\Spec(\mathcal{O})$.
 Then $G_{\mathcal{O}}$ is the \emph{arc space} of $\mathfrak{G}$, and
 $G_{\mathcal{K}}$ is the \emph{loop space} of $\mathfrak{G} \times_{\Spec(\mathcal{O})} \Spec(\mathcal{K})$ in the sense of~\cite[Definition~1.3.1]{zhu}. From this point of view one can consider the ``affine Grassmannian'' of more general (smooth, affine) group schemes over $\Spec(\mathcal{O})$. In particular, 
 replacing $\mathfrak{G}$ by the Iwahori group scheme constructed in Bruhat--Tits theory,
 then we obtain an ind-scheme $\Fl_G$ which is often called the \emph{affine flag variety} of $G$.
 \item
 See also~\cite[\S 1.6]{zhu} for a description of $\widetilde{\Gr}_G$ in terms of the loop group of a maximal compact subgroup (which only makes sense for the case of \emph{complex} reductive groups, unlike the other descriptions considered above). This approach is crucial in the proof of Ginzburg~\cite{ginzburg}; it also shows that the torsor $G_\mathcal{K} \to \widetilde{\Gr}_G$ is topologically trivial.
\end{enumerate}
\end{rmk}

In general, the quotient $G_{\mathcal{K}}/G_{\mathcal{O}}$ is not reduced. (This can already be seen when $G$ is the multiplicative group $\mathbb{G}_{\mathrm{m}}$.) Since we will only consider constructible sheaves on this quotient, this non-reduced structure can be forgotten, and we will denote by $\Gr_G$ the (reduced) ind-variety associated with the ind-scheme $\widetilde{\Gr}_G$.

Any cocharacter $\nu\in X_*(T)$ defines a morphism $\mathcal{K}^\times \to T_{\mathcal{K}}$. The image
of $t$ under this morphism will be denoted by $t^\nu$.
The coset $t^\nu\GO$ is a point in $\Gr_G$, which will be denoted by $L_\nu$.

The Cartan decomposition describes the $\GO$-orbits in $\Gr_G$, in the following way (see~\cite[\S 2.1]{zhu} for more details and references).

\begin{prop}
\label{prop:Cartan-dec}
We have a decomposition
\begin{equation}
\label{eqn:Cartan-dec}
\Gr_G=\bigsqcup_{\lambda\in X_*(T)^+}
\Gr_G^\lambda,\qquad\text{where}\quad\Gr_G^\lambda:=\GO\cdot L_\lambda.
\end{equation}
Moreover, this decomposition is a stratification of $\Gr_G$ and, for any $\lambda \in X_*(T)$, $\Gr_G^\lambda$ is an affine
bundle over the partial flag variety $G/P_\lambda$ where $P_\lambda$ is the parabolic subgroup of $G$ containing $B$ and associated with the subset of simple roots $\{\alpha \in \Delta_{\mathrm{s}}(G,B,T) \mid \langle \lambda, \alpha \rangle = 0\}$. We also have
\[
\dim(\Gr^\lambda_G) = \langle 2\rho, \lambda \rangle
\]
and
\begin{equation}
\label{eqn:closure-orbit}
\overline{\Gr_G^\lambda}=\bigsqcup_{\substack{\eta\in X_*(T)^+ \\ \eta\leq\lambda}}\Gr_G^\eta.
\end{equation}
\end{prop}

The stratification
of $\Gr_G$ by $\GO$-orbits will be denoted by $\mathscr S$.

%

Finally, we will need a description of the connected components of $\Gr_G$. For any $c \in X_*(T)/Q^\vee$, let us set
\[
\Gr_G^c := \bigsqcup_{\substack{\lambda \in X_*(T)^+ \\ \lambda + Q^\vee = c}} \Gr_G^\lambda.
\]
Then the connected components of $\Gr_G$ are exactly the subvarieties $\Gr_G^c$ for $c \in X_*(T)/Q^\vee$ (see~\cite[Comments after Theorem 1.3.11]{zhu} for references). In particular, since $\langle \rho,\lambda \rangle \in \Z$ for any $\lambda \in Q^\vee$, the parity of the dimensions of the
Schubert varieties $\Gr_G^\lambda$ is constant on each connected
component. A connected component will be called even, resp.~odd, if these dimensions are even, resp.~odd.

\subsection{Semi-infinite orbits}
\label{ss:semi-infinite-orbits}

The Iwasawa decomposition describes (set-theoretically) the $\NK$-orbits in $\Gr_G$
as follows:
\begin{equation}
\label{eqn:Iwasawa-dec}
\Gr_G=\bigsqcup_{\mu\in X_*(T)}S_\mu,\qquad\text{where}\quad S_\mu:=\NK
\cdot L_\mu.
\end{equation}
Each orbit $S_\mu$ is ``infinite dimensional,'' 
i.e.~not contained in any finite type subscheme
of $\Gr_G$.

The closure of these orbits for the inductive limit topology on $\Gr_G$
can be described in the following way:
\begin{equation*}
\overline{S_\mu}=\bigsqcup_{\substack{\nu\in X_*(T)\\\nu\leq\mu}}S_\nu.
\end{equation*}

We will soon sketch a formal proof of this equality (see Proposition~\ref{prop:closure-semiinfinite} below), but
let us first try to make this result intuitive, at least in the case
$G=\SL_2$. For that, we denote by $\alpha$ the unique positive root, and consider the standard Iwahori subgroup
$$I=\left\{\left.\left(\begin{smallmatrix}a&b\\c&d\end{smallmatrix}\right)
\in\SL_{2,\mathcal K}\right|a,b,t^{-1}c,d\in\mathcal O\right\}$$
and the two maximal parahoric subgroups
$$P_0=\SL_{2,\mathcal O},\quad P_1=\left\{\left.\left(\begin{smallmatrix}
a&b\\c&d\end{smallmatrix}\right)\in\SL_{2,\mathcal K}\right|a,tb,t^{-1}c,d
\in\mathcal O\right\}$$
that contain $I$.

A parahoric subgroup of $\SL_{2,\mathcal K}$ is a subgroup conjugated to
one of these three standard subgroups $I$, $P_0$ or $P_1$. Parahoric subgroups form a poset
for the inclusion. The Serre tree is the simplicial realization of the
opposite poset. Figure~\ref{fig:Serre-tree} shows a small part of this tree; namely
we just pictured the parahoric sugroups that are conjugated to the standard
ones by elements of the affine Weyl group (see~\S\ref{ss:parity}). Note here that the inclusions
$P_1\supset I\subset P_0$ translate to the fact that the vertices
corresponding to $P_1$ and $P_0$ are incident to the edge corresponding
to $I$.

\begin{figure}
\begin{center}
\begin{tikzpicture}
  [bnode/.style={circle,draw=black,fill=black,thick,
                 inner sep=0pt,minimum size=1mm},
   rnode/.style={diamond,draw=black,fill=white,thick,
                 inner sep=0pt,minimum size=2mm},
   gnode/.style={circle,draw=black,fill=white,thin,
                 inner sep=0pt,minimum size=1mm}]
  \path (0,0) node (t0) [bnode] {};
  \path (t0) ++(0:1) node (t1) [gnode] {};
  \path (t1) ++(0:1) node (t2) [bnode] {};
  \path (t2) ++(0:1) node (t3) [gnode] {};
  \path (t3) ++(0:1) node (t4) [bnode] {};
  \path (t4) ++(0:1) node (t5) [gnode] {};
  \path (t0) ++(180:1) node (t6) [gnode] {};
  \path (t6) ++(180:1) node (t7) [bnode] {};
  \path (t7) ++(180:1) node (t8) [gnode] {};
  \path (t8) ++(180:1) node (t9) [bnode] {};
  \path (t9) ++(180:1) node (tZ) [gnode] {};
  \node [below] at (t0) {$P_0$};
  \node [below] at (t2) {$P'_0$};
  \node [below] at (t4) {$P''_0$};
  \node [below] at (t6) {$P_1$};
  \node [above] at (-.5,0) {$I$};
  \node [above] at (1.5,0) {$I'$};
  \draw (t0) edge (t1) edge (t6);
  \draw (t2) edge (t1) edge (t3);
  \draw (t4) edge (t3) edge (t5);
  \draw (t7) edge (t6) edge (t8);
  \draw (t9) edge (t8) edge (tZ);
  \draw (5.5,-2) node {$\begin{aligned}
    I'&=t^{\alpha^{\!\vee}}I\,t^{-\alpha^{\!\vee}}\\
    P'_0&=t^{\alpha^{\!\vee}}P_0\,t^{-\alpha^{\!\vee}}\\
    P''_0&=t^{2\alpha^{\!\vee}}P_0\,t^{-2\alpha^{\!\vee}}
    \end{aligned}$};
  \draw (-7,-2) node {};
\end{tikzpicture}
\end{center}
\caption{A small part of the Serre tree}\label{fig:Serre-tree}
\end{figure}
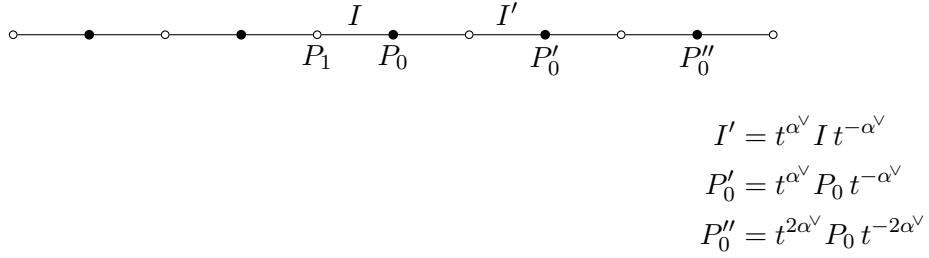

Since Iwahori subgroups are conjugated in $\SL_{2,\mathcal K}$ and since
$I$ is its own normalizer, the set of edges in the tree incident to a given edge is in bijection
with the so-called affine flag variety $\SL_{2,\mathcal K}/I \cong \mathbb{P}^1_{\C}$. Likewise,
the set of parahoric subgroups conjugated to $P_0$, depicted as black
dots on the tree, is in bijection with the affine Grassmannian
$\Gr_{\SL_2}=\SL_{2,\mathcal K}/P_0$. One can rephrase this by saying that the
group $\SL_{2,\mathcal K}$ acts on the tree (transitively on the edges,
on the black vertices, and on the white vertices) and that the stabilizer
of the simplex associated to a parahoric subgroup is the subgroup itself. (The white dots form the second connected component in the affine Grassmannian $\Gr_{\mathbf{PGL}_2}$.) A slightly more accurate picture of the Serre tree is given in Figure~\ref{fig:GrSL2}. (But here again we only draw a finite number of edges incident to each vertex, while as explained above such edges are in fact in bijection with $\mathbb{P}^1_{\C}$.)

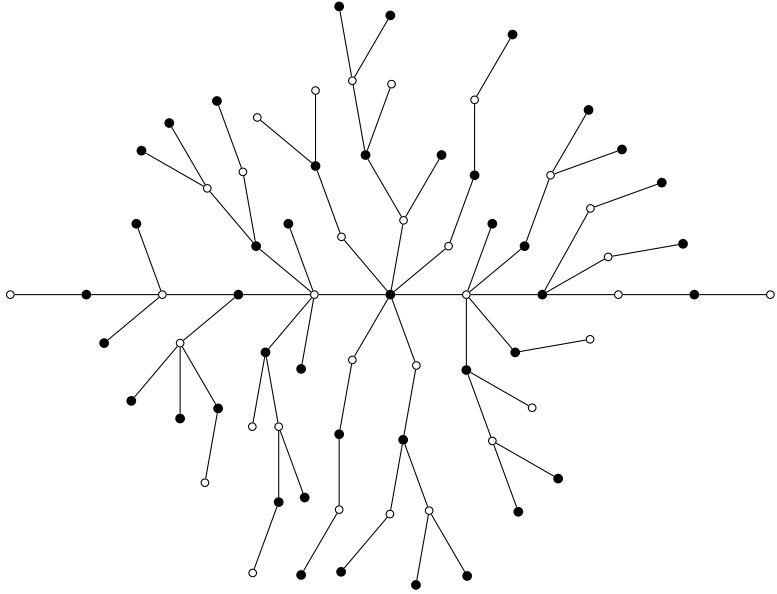
\begin{figure}
\begin{center}
\begin{tikzpicture}
  [bnode/.style={circle,draw=black,fill=black,thick,
                 inner sep=0pt,minimum size=1mm},
   rnode/.style={diamond,draw=black,fill=white,thick,
                 inner sep=0pt,minimum size=2mm},
   gnode/.style={circle,draw=black,fill=white,thin,
                 inner sep=0pt,minimum size=1mm}]
  \path (0,0) node (t0) [bnode] {};
  \path (t0) ++(0:1) node (t1) [gnode] {};
  \path (t1) ++(0:1) node (t2) [bnode] {};
  \path (t2) ++(0:1) node (t3) [gnode] {};
  \path (t3) ++(0:1) node (t4) [bnode] {};
  \path (t4) ++(0:1) node (t5) [gnode] {};
  \path (t0) ++(180:1) node (t6) [gnode] {};
  \path (t6) ++(180:1) node (t7) [bnode] {};
  \path (t7) ++(180:1) node (t8) [gnode] {};
  \path (t8) ++(180:1) node (t9) [bnode] {};
  \path (t0) ++(40:1) node (ta) [gnode] {};
  \path (ta) ++(70:1) node (tb) [bnode] {};
  \path (tb) ++(90:1) node (tc) [gnode] {};
  \path (tc) ++(60:1) node (td) [bnode] {};
  \path (t0) ++(80:1) node (te) [gnode] {};
  \path (te) ++(60:1) node (tf) [bnode] {};
  \path (te) ++(120:1) node (tg) [bnode] {};
  \path (tg) ++(70:1) node (th) [gnode] {};
  \path (tg) ++(100:1) node (ti) [gnode] {};
  \path (ti) ++(60:1) node (tj) [bnode] {};
  \path (ti) ++(100:1) node (tk) [bnode] {};
  \path (t0) ++(130:1) node (tl) [gnode] {};
  \path (tl) ++(110:1) node (tm) [bnode] {};
  \path (tm) ++(90:1) node (tn) [gnode] {};
  \path (tm) ++(140:1) node (to) [gnode] {};
  \path (t6) ++(110:1) node (tp) [bnode] {};
  \path (t6) ++(140:1) node (tq) [bnode] {};
  \path (t8) ++(110:1) node (tr) [bnode] {};
  \path (t1) ++(70:1) node (ts) [bnode] {};
  \path (t1) ++(40:1) node (tt) [bnode] {};
  \path (tt) ++(30:1) node (tu) [gnode] {};
  \path (tu) ++(20:1) node (tv) [bnode] {};
  \path (tt) ++(70:1) node (tw) [gnode] {};
  \path (tw) ++(20:1) node (tx) [bnode] {};
  \path (tw) ++(60:1) node (ty) [bnode] {};
  \path (t2) ++(30:1) node (tz) [gnode] {};
  \path (tz) ++(10:1) node (tA) [bnode] {};
  \path (t8) ++(-140:1) node (tB) [bnode] {};
  \path (t7) ++(-140:1) node (tC) [gnode] {};
  \path (tC) ++(-130:1) node (tD) [bnode] {};
  \path (tC) ++(-90:1) node (tE) [bnode] {};
  \path (tC) ++(-60:1) node (tF) [bnode] {};
  \path (t6) ++(-130:1) node (tG) [bnode] {};
  \path (tG) ++(-80:1) node (tH) [gnode] {};
  \path (tG) ++(-100:1) node (tI) [gnode] {};
  \path (t0) ++(-120:1) node (tJ) [gnode] {};
  \path (tJ) ++(-100:1) node (tK) [bnode] {};
  \path (t0) ++(-70:1) node (tL) [gnode] {};
  \path (tL) ++(-100:1) node (tM) [bnode] {};
  \path (tM) ++(-100:1) node (tN) [gnode] {};
  \path (tN) ++(-130:1) node (tO) [bnode] {};
  \path (tM) ++(-70:1) node (tP) [gnode] {};
  \path (tP) ++(-100:1) node (tQ) [bnode] {};
  \path (tP) ++(-60:1) node (tR) [bnode] {};
  \path (t1) ++(-90:1) node (tS) [bnode] {};
  \path (tS) ++(-70:1) node (tT) [gnode] {};
  \path (tT) ++(-70:1) node (tU) [bnode] {};
  \path (tT) ++(-30:1) node (tV) [bnode] {};
  \path (tS) ++(-30:1) node (tW) [gnode] {};
  \path (t1) ++(-50:1) node (tX) [bnode] {};
  \path (tX) ++(10:1) node (tY) [gnode] {};
  \path (t9) ++(180:1) node (tZ) [gnode] {};
  \path (tq) ++(100:1) node (taa) [gnode] {};
  \path (tq) ++(130:1) node (tbb) [gnode] {};
  \path (taa) ++(110:1) node (tcc) [bnode] {};
  \path (tbb) ++(120:1) node (tdd) [bnode] {};
  \path (tbb) ++(150:1) node (tee) [bnode] {};
  \path (tK) ++(-90:1) node (tff) [gnode] {};
  \path (tH) ++(-70:1) node (tgg) [bnode] {};
  \path (tH) ++(-90:1) node (thh) [bnode] {};
  \path (thh) ++(-110:1) node (tii) [gnode] {};
  \path (tff) ++(-120:1) node (tjj) [bnode] {};
  \path (tF) ++(-100:1) node (tkk) [gnode] {};
  \path (t6) ++(-100:1) node (tll) [bnode] {};
  \draw (t0) edge (t1) edge (t6) edge (ta) edge (te)
             edge (tl) edge (tJ) edge (tL);
  \draw (tf) edge (te);
  \draw (tg) edge (te) edge (th) edge (ti);
  \draw (tm) edge (tl) edge (tn) edge (to);
  \draw (tb) edge (ta) edge (tc);
  \draw (td) edge (tc);
  \draw (tt) edge (t1) edge (tw);
  \draw (ts) edge (t1);
  \draw (tx) edge (tw);
  \draw (ty) edge (tw);
  \draw (tr) edge (t8);
  \draw (tB) edge (t8);
  \draw (t9) edge (t8) edge (tZ);
  \draw (t7) edge (t6) edge (t8) edge (tC);
  \draw (tD) edge (tC);
  \draw (tE) edge (tC);
  \draw (tF) edge (tC);
  \draw (tG) edge (t6) edge (tH) edge (tI);
  \draw (tK) edge (tJ);
  \draw (tM) edge (tL) edge (tN) edge (tP);
  \draw (tO) edge (tN);
  \draw (tQ) edge (tP);
  \draw (tR) edge (tP);
  \draw (tS) edge (t1) edge (tT) edge (tW);
  \draw (tU) edge (tT);
  \draw (tV) edge (tT);
  \draw (tX) edge (t1) edge (tY);
  \draw (t2) edge (t1) edge (t3) edge (tu) edge (tz);
  \draw (t4) edge (t3) edge (t5);
  \draw (t6) edge (tp) edge (tq) edge (tll);
  \draw (tq) edge (taa) edge (tbb);
  \draw (taa) edge (tcc);
  \draw (tbb) edge (tdd) edge (tee);
  \draw (tj) edge (ti);
  \draw (tk) edge (ti);
  \draw (tv) edge (tu);
  \draw (tA) edge (tz);
  \draw (tK) edge (tff);
  \draw (tH) edge (tgg) edge (thh);
  \draw (thh) edge (tii);
  \draw (tff) edge (tjj);
  \draw (tF) edge (tkk);
\end{tikzpicture}
\end{center}
\caption{Intuitive picture for $\Gr_{\SL_2}$}\label{fig:GrSL2}
\end{figure}

Likewise, the Iwahori subgroups contained in $P_0$ can be obtained by
letting the normalizer of $P_0$ act on $I$; in other words, the set of
edges incident to the vertex $P_0$ is in bijection with $P_0/I$. This
is a complex projective line. So the set of edges incident to a given
black vertex is a complex projective line. (The same thing holds also
for white vertices.) Our drawings are thus quite incomplete, because
a lot of edges were omitted.

Again, our affine Grassmannian is the set of all black vertices. Here
it is worth noting that the tree metric is related to the description
of the ind-structure of $\Gr_{\SL_2}$: one can take for the $n$-th
finite-dimensional piece of $\Gr_{\SL_2}$ the set $\Gr_n$ of all vertices at distance $\leq2n$
from $P_0$. Further, the analytic (respectively, Zariski) topology of the
variety $\Gr_n$ can also be seen on the tree: it comes from the analytic
(respectively, Zariski) topology on all the projective lines mentioned
in the previous paragraph. Thus, we can for instance see that $\Gr_n$ is
dominated by a tower of $2n$ projective lines, because each point
at distance $\leq2n$ from the origin can be reached by choosing first
an edge around the origin, then another edge around the white vertex
at the end of this edge, and so on $2n$ times.\footnote{From the algebro-geometric point of view,
this process is a particular case of a Bott--Samelson resolution, as
explained in~\cite{gaussent-littelmann}.}

Now let us see how our orbits $\Gr^\lambda_G$ and $S_\mu$ are depicted
in this model. The point $L_\nu$ corresponds to the Iwahori subgroup
$t^\nu P_0\,t^{-\nu}$, see Figure~\ref{fig:Serre-tree-points}.

\begin{figure}
\begin{center}
\begin{tikzpicture}
  [bnode/.style={circle,draw=black,fill=black,thick,
                 inner sep=0pt,minimum size=1mm},
   rnode/.style={diamond,draw=black,fill=white,thick,
                 inner sep=0pt,minimum size=2mm},
   gnode/.style={circle,draw=black,fill=white,thin,
                 inner sep=0pt,minimum size=1mm}]
  \path (0,0) node (t0) [bnode] {};
  \path (t0) ++(0:1) node (t1) [gnode] {};
  \path (t1) ++(0:1) node (t2) [bnode] {};
  \path (t2) ++(0:1) node (t3) [gnode] {};
  \path (t3) ++(0:1) node (t4) [bnode] {};
  \path (t4) ++(0:1) node (t5) [gnode] {};
  \path (t0) ++(180:1) node (t6) [gnode] {};
  \path (t6) ++(180:1) node (t7) [bnode] {};
  \path (t7) ++(180:1) node (t8) [gnode] {};
  \path (t8) ++(180:1) node (t9) [bnode] {};
  \draw (t0) edge (t1) edge (t6);
  \draw (t2) edge (t1) edge (t3);
  \draw (t4) edge (t3) edge (t5);
  \draw (t7) edge (t6) edge (t8);
  \draw (t9) edge (t8) edge (tZ);
  \path (t9) ++(180:1) node (tZ) [gnode] {};
  \node [below] at (t0) {$L_0$};
  \node [below] at (t2) {$L_{\alpha^{\!\vee}}$};
  \node [below] at (t4) {$L_{2\alpha^{\!\vee}}$};
  \node [below] at (t7) {$L_{-{\alpha^{\!\vee}}}$};
  \node [below] at (t9) {$L_{-2\alpha^{\!\vee}}$};
\end{tikzpicture}
\end{center}
\caption{More points in $\Gr_{\SL_2}$}\label{fig:Serre-tree-points}
\end{figure}
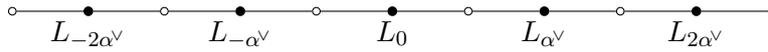

The Schubert cell $\Gr_G^\lambda$ is the orbit of $L_\lambda$ under the
stabilizer of the base point $L_0$; it therefore looks like the sphere
with center $L_0$ going through $L_\lambda$. On Figure~\ref{fig:GrSL2-Schubert}, the
diamonds are points in $\Gr_G^{{\alpha^{\!\vee}}}$.

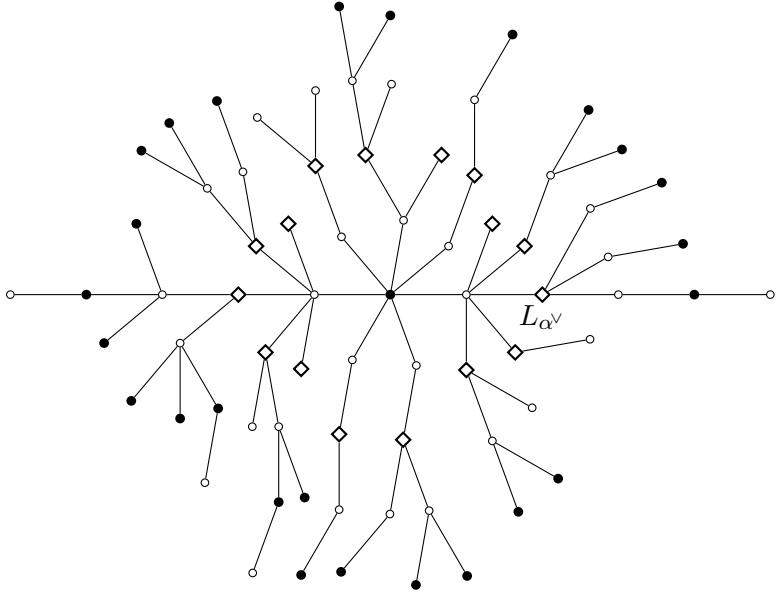
\begin{figure}[h!]
\begin{center}
\begin{tikzpicture}
  [bnode/.style={circle,draw=black,fill=black,thick,
                 inner sep=0pt,minimum size=1mm},
   rnode/.style={diamond,draw=black,fill=white,thick,
                 inner sep=0pt,minimum size=2mm},
   gnode/.style={circle,draw=black,fill=white,thin,
                 inner sep=0pt,minimum size=1mm}]
  \path (0,0) node (t0) [bnode] {};
  \path (t0) ++(0:1) node (t1) [gnode] {};
  \path (t1) ++(0:1) node (t2) [rnode] {};
  \path (t2) ++(0:1) node (t3) [gnode] {};
  \path (t3) ++(0:1) node (t4) [bnode] {};
  \path (t4) ++(0:1) node (t5) [gnode] {};
  \path (t0) ++(180:1) node (t6) [gnode] {};
  \path (t6) ++(180:1) node (t7) [rnode] {};
  \path (t7) ++(180:1) node (t8) [gnode] {};
  \path (t8) ++(180:1) node (t9) [bnode] {};
  \path (t0) ++(40:1) node (ta) [gnode] {};
  \path (ta) ++(70:1) node (tb) [rnode] {};
  \path (tb) ++(90:1) node (tc) [gnode] {};
  \path (tc) ++(60:1) node (td) [bnode] {};
  \path (t0) ++(80:1) node (te) [gnode] {};
  \path (te) ++(60:1) node (tf) [rnode] {};
  \path (te) ++(120:1) node (tg) [rnode] {};
  \path (tg) ++(70:1) node (th) [gnode] {};
  \path (tg) ++(100:1) node (ti) [gnode] {};
  \path (ti) ++(60:1) node (tj) [bnode] {};
  \path (ti) ++(100:1) node (tk) [bnode] {};
  \path (t0) ++(130:1) node (tl) [gnode] {};
  \path (tl) ++(110:1) node (tm) [rnode] {};
  \path (tm) ++(90:1) node (tn) [gnode] {};
  \path (tm) ++(140:1) node (to) [gnode] {};
  \path (t6) ++(110:1) node (tp) [rnode] {};
  \path (t6) ++(140:1) node (tq) [rnode] {};
  \path (t8) ++(110:1) node (tr) [bnode] {};
  \path (t1) ++(70:1) node (ts) [rnode] {};
  \path (t1) ++(40:1) node (tt) [rnode] {};
  \path (tt) ++(30:1) node (tu) [gnode] {};
  \path (tu) ++(20:1) node (tv) [bnode] {};
  \path (tt) ++(70:1) node (tw) [gnode] {};
  \path (tw) ++(20:1) node (tx) [bnode] {};
  \path (tw) ++(60:1) node (ty) [bnode] {};
  \path (t2) ++(30:1) node (tz) [gnode] {};
  \path (tz) ++(10:1) node (tA) [bnode] {};
  \path (t8) ++(-140:1) node (tB) [bnode] {};
  \path (t7) ++(-140:1) node (tC) [gnode] {};
  \path (tC) ++(-130:1) node (tD) [bnode] {};
  \path (tC) ++(-90:1) node (tE) [bnode] {};
  \path (tC) ++(-60:1) node (tF) [bnode] {};
  \path (t6) ++(-130:1) node (tG) [rnode] {};
  \path (tG) ++(-80:1) node (tH) [gnode] {};
  \path (tG) ++(-100:1) node (tI) [gnode] {};
  \path (t0) ++(-120:1) node (tJ) [gnode] {};
  \path (tJ) ++(-100:1) node (tK) [rnode] {};
  \path (t0) ++(-70:1) node (tL) [gnode] {};
  \path (tL) ++(-100:1) node (tM) [rnode] {};
  \path (tM) ++(-100:1) node (tN) [gnode] {};
  \path (tN) ++(-130:1) node (tO) [bnode] {};
  \path (tM) ++(-70:1) node (tP) [gnode] {};
  \path (tP) ++(-100:1) node (tQ) [bnode] {};
  \path (tP) ++(-60:1) node (tR) [bnode] {};
  \path (t1) ++(-90:1) node (tS) [rnode] {};
  \path (tS) ++(-70:1) node (tT) [gnode] {};
  \path (tT) ++(-70:1) node (tU) [bnode] {};
  \path (tT) ++(-30:1) node (tV) [bnode] {};
  \path (tS) ++(-30:1) node (tW) [gnode] {};
  \path (t1) ++(-50:1) node (tX) [rnode] {};
  \path (tX) ++(10:1) node (tY) [gnode] {};
  \path (t9) ++(180:1) node (tZ) [gnode] {};
  \path (tq) ++(100:1) node (taa) [gnode] {};
  \path (tq) ++(130:1) node (tbb) [gnode] {};
  \path (taa) ++(110:1) node (tcc) [bnode] {};
  \path (tbb) ++(120:1) node (tdd) [bnode] {};
  \path (tbb) ++(150:1) node (tee) [bnode] {};
  \path (tK) ++(-90:1) node (tff) [gnode] {};
  \path (tH) ++(-70:1) node (tgg) [bnode] {};
  \path (tH) ++(-90:1) node (thh) [bnode] {};
  \path (thh) ++(-110:1) node (tii) [gnode] {};
  \path (tff) ++(-120:1) node (tjj) [bnode] {};
  \path (tF) ++(-100:1) node (tkk) [gnode] {};
  \path (t6) ++(-100:1) node (tll) [rnode] {};
  \draw (t0) edge (t1) edge (t6) edge (ta) edge (te)
             edge (tl) edge (tJ) edge (tL);
  \draw (tf) edge (te);
  \draw (tg) edge (te) edge (th) edge (ti);
  \draw (tm) edge (tl) edge (tn) edge (to);
  \draw (tb) edge (ta) edge (tc);
  \draw (td) edge (tc);
  \draw (tt) edge (t1) edge (tw);
  \draw (ts) edge (t1);
  \draw (tx) edge (tw);
  \draw (ty) edge (tw);
  \draw (tr) edge (t8);
  \draw (tB) edge (t8);
  \draw (t9) edge (t8) edge (tZ);
  \draw (t7) edge (t6) edge (t8) edge (tC);
  \draw (tD) edge (tC);
  \draw (tE) edge (tC);
  \draw (tF) edge (tC);
  \draw (tG) edge (t6) edge (tH) edge (tI);
  \draw (tK) edge (tJ);
  \draw (tM) edge (tL) edge (tN) edge (tP);
  \draw (tO) edge (tN);
  \draw (tQ) edge (tP);
  \draw (tR) edge (tP);
  \draw (tS) edge (t1) edge (tT) edge (tW);
  \draw (tU) edge (tT);
  \draw (tV) edge (tT);
  \draw (tX) edge (t1) edge (tY);
  \draw (t2) edge (t1) edge (t3) edge (tu) edge (tz);
  \draw (t4) edge (t3) edge (t5);
  \draw (t6) edge (tp) edge (tq) edge (tll);
  \draw (tq) edge (taa) edge (tbb);
  \draw (taa) edge (tcc);
  \draw (tbb) edge (tdd) edge (tee);
  \draw (tj) edge (ti);
  \draw (tk) edge (ti);
  \draw (tv) edge (tu);
  \draw (tA) edge (tz);
  \draw (tK) edge (tff);
  \draw (tH) edge (tgg) edge (thh);
  \draw (thh) edge (tii);
  \draw (tff) edge (tjj);
  \draw (tF) edge (tkk);
  \node [below] at (t2) {$L_{\alpha^{\!\vee}}$};
\end{tikzpicture}
\end{center}
\caption{Illustration of a Schubert cell in $\Gr_{\SL_2}$}\label{fig:GrSL2-Schubert}
\end{figure}

Now look at the white vertex between the origin and $L_{\alpha^\vee}$:
edges starting from this vertex form a projective line, and the other
vertices of these edges belong to $\Gr_G^{\alpha^\vee}$, with one exception,
namely $L_0$. This point $L_0$ appears thus as a limit (on the projective
line) of points in $\Gr_G^{\alpha^\vee}$, that is, belongs to the closure
of $\Gr_G^{\alpha^\vee}$. This provides an intuitive interpretation of the
inclusion $\Gr_G^0\subset\overline{\Gr_G^{\alpha^\vee}}$.

In the same line of ideas, the semi-infinite orbit $S_\mu$ can be depicted
as the sphere centered at~$-\infty$ (also called ``horosphere'') and going
through $L_\mu$. In Figure~\ref{fig:GrSL2-semiinfinite}, the diamonds are points in
$S_{\alpha^\vee}$.

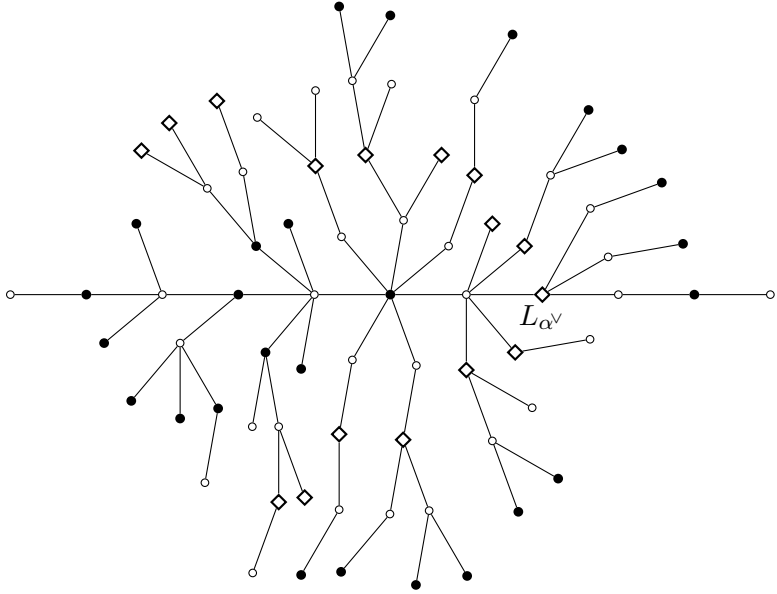
\begin{figure}
\begin{center}
\begin{tikzpicture}
  [bnode/.style={circle,draw=black,fill=black,thick,
                 inner sep=0pt,minimum size=1mm},
   rnode/.style={diamond,draw=black,fill=white,thick,
                 inner sep=0pt,minimum size=2mm},
   gnode/.style={circle,draw=black,fill=white,thin,
                 inner sep=0pt,minimum size=1mm}]
  \path (0,0) node (t0) [bnode] {};
  \path (t0) ++(0:1) node (t1) [gnode] {};
  \path (t1) ++(0:1) node (t2) [rnode] {};
  \path (t2) ++(0:1) node (t3) [gnode] {};
  \path (t3) ++(0:1) node (t4) [bnode] {};
  \path (t4) ++(0:1) node (t5) [gnode] {};
  \path (t0) ++(180:1) node (t6) [gnode] {};
  \path (t6) ++(180:1) node (t7) [bnode] {};
  \path (t7) ++(180:1) node (t8) [gnode] {};
  \path (t8) ++(180:1) node (t9) [bnode] {};
  \path (t0) ++(40:1) node (ta) [gnode] {};
  \path (ta) ++(70:1) node (tb) [rnode] {};
  \path (tb) ++(90:1) node (tc) [gnode] {};
  \path (tc) ++(60:1) node (td) [bnode] {};
  \path (t0) ++(80:1) node (te) [gnode] {};
  \path (te) ++(60:1) node (tf) [rnode] {};
  \path (te) ++(120:1) node (tg) [rnode] {};
  \path (tg) ++(70:1) node (th) [gnode] {};
  \path (tg) ++(100:1) node (ti) [gnode] {};
  \path (ti) ++(60:1) node (tj) [bnode] {};
  \path (ti) ++(100:1) node (tk) [bnode] {};
  \path (t0) ++(130:1) node (tl) [gnode] {};
  \path (tl) ++(110:1) node (tm) [rnode] {};
  \path (tm) ++(90:1) node (tn) [gnode] {};
  \path (tm) ++(140:1) node (to) [gnode] {};
  \path (t6) ++(110:1) node (tp) [bnode] {};
  \path (t6) ++(140:1) node (tq) [bnode] {};
  \path (t8) ++(110:1) node (tr) [bnode] {};
  \path (t1) ++(70:1) node (ts) [rnode] {};
  \path (t1) ++(40:1) node (tt) [rnode] {};
  \path (tt) ++(30:1) node (tu) [gnode] {};
  \path (tu) ++(20:1) node (tv) [bnode] {};
  \path (tt) ++(70:1) node (tw) [gnode] {};
  \path (tw) ++(20:1) node (tx) [bnode] {};
  \path (tw) ++(60:1) node (ty) [bnode] {};
  \path (t2) ++(30:1) node (tz) [gnode] {};
  \path (tz) ++(10:1) node (tA) [bnode] {};
  \path (t8) ++(-140:1) node (tB) [bnode] {};
  \path (t7) ++(-140:1) node (tC) [gnode] {};
  \path (tC) ++(-130:1) node (tD) [bnode] {};
  \path (tC) ++(-90:1) node (tE) [bnode] {};
  \path (tC) ++(-60:1) node (tF) [bnode] {};
  \path (t6) ++(-130:1) node (tG) [bnode] {};
  \path (tG) ++(-80:1) node (tH) [gnode] {};
  \path (tG) ++(-100:1) node (tI) [gnode] {};
  \path (t0) ++(-120:1) node (tJ) [gnode] {};
  \path (tJ) ++(-100:1) node (tK) [rnode] {};
  \path (t0) ++(-70:1) node (tL) [gnode] {};
  \path (tL) ++(-100:1) node (tM) [rnode] {};
  \path (tM) ++(-100:1) node (tN) [gnode] {};
  \path (tN) ++(-130:1) node (tO) [bnode] {};
  \path (tM) ++(-70:1) node (tP) [gnode] {};
  \path (tP) ++(-100:1) node (tQ) [bnode] {};
  \path (tP) ++(-60:1) node (tR) [bnode] {};
  \path (t1) ++(-90:1) node (tS) [rnode] {};
  \path (tS) ++(-70:1) node (tT) [gnode] {};
  \path (tT) ++(-70:1) node (tU) [bnode] {};
  \path (tT) ++(-30:1) node (tV) [bnode] {};
  \path (tS) ++(-30:1) node (tW) [gnode] {};
  \path (t1) ++(-50:1) node (tX) [rnode] {};
  \path (tX) ++(10:1) node (tY) [gnode] {};
  \path (t9) ++(180:1) node (tZ) [gnode] {};
  \path (tq) ++(100:1) node (taa) [gnode] {};
  \path (tq) ++(130:1) node (tbb) [gnode] {};
  \path (taa) ++(110:1) node (tcc) [rnode] {};
  \path (tbb) ++(120:1) node (tdd) [rnode] {};
  \path (tbb) ++(150:1) node (tee) [rnode] {};
  \path (tK) ++(-90:1) node (tff) [gnode] {};
  \path (tH) ++(-70:1) node (tgg) [rnode] {};
  \path (tH) ++(-90:1) node (thh) [rnode] {};
  \path (thh) ++(-110:1) node (tii) [gnode] {};
  \path (tff) ++(-120:1) node (tjj) [bnode] {};
  \path (tF) ++(-100:1) node (tkk) [gnode] {};
  \path (t6) ++(-100:1) node (tll) [bnode] {};
  \draw (t0) edge (t1) edge (t6) edge (ta) edge (te)
             edge (tl) edge (tJ) edge (tL);
  \draw (tf) edge (te);
  \draw (tg) edge (te) edge (th) edge (ti);
  \draw (tm) edge (tl) edge (tn) edge (to);
  \draw (tb) edge (ta) edge (tc);
  \draw (td) edge (tc);
  \draw (tt) edge (t1) edge (tw);
  \draw (ts) edge (t1);
  \draw (tx) edge (tw);
  \draw (ty) edge (tw);
  \draw (tr) edge (t8);
  \draw (tB) edge (t8);
  \draw (t9) edge (t8) edge (tZ);
  \draw (t7) edge (t6) edge (t8) edge (tC);
  \draw (tD) edge (tC);
  \draw (tE) edge (tC);
  \draw (tF) edge (tC);
  \draw (tG) edge (t6) edge (tH) edge (tI);
  \draw (tK) edge (tJ);
  \draw (tM) edge (tL) edge (tN) edge (tP);
  \draw (tO) edge (tN);
  \draw (tQ) edge (tP);
  \draw (tR) edge (tP);
  \draw (tS) edge (t1) edge (tT) edge (tW);
  \draw (tU) edge (tT);
  \draw (tV) edge (tT);
  \draw (tX) edge (t1) edge (tY);
  \draw (t2) edge (t1) edge (t3) edge (tu) edge (tz);
  \draw (t4) edge (t3) edge (t5);
  \draw (t6) edge (tp) edge (tq) edge (tll);
  \draw (tq) edge (taa) edge (tbb);
  \draw (taa) edge (tcc);
  \draw (tbb) edge (tdd) edge (tee);
  \draw (tj) edge (ti);
  \draw (tk) edge (ti);
  \draw (tv) edge (tu);
  \draw (tA) edge (tz);
  \draw (tK) edge (tff);
  \draw (tH) edge (tgg) edge (thh);
  \draw (thh) edge (tii);
  \draw (tff) edge (tjj);
  \draw (tF) edge (tkk);
  \node [below] at (t2) {$L_{\alpha^{\!\vee}}$};
\end{tikzpicture}
\end{center}
\caption{Illustration of a semi-infinite orbit in $\Gr_{\SL_2}$}\label{fig:GrSL2-semiinfinite}
\end{figure}

For the same reason as before, we see that $L_0$ belongs to the
closure of $S_{\alpha^\vee}$. The reader can however feel cheated
here, since we relied on geometrical intuition. For a more formal
proof, one computes\footnote{In fact, this computation is precisely
our observation on the tree.}
$$\begin{pmatrix}1&at^{-1}\\0&1\end{pmatrix}=\begin{pmatrix}t^{-1}&0\\
a^{-1}&t\end{pmatrix}\begin{pmatrix}t&a\\-a^{-1}&0\end{pmatrix}
\quad\text{for $a\in\C^\times$},$$
that is 
$$\begin{pmatrix}1&at^{-1}\\0&1\end{pmatrix}\GO=\begin{pmatrix}t^{-1}&0\\
a^{-1}&t\end{pmatrix}\GO,\hspace{83pt}$$
and therefore 
$$\underbrace{\begin{pmatrix}1&at^{-1}\\0&1\end{pmatrix}\GO}_{\in S_0}
\to\underbrace{\begin{pmatrix}t^{-1}&0\\0&t\end{pmatrix}\GO}_{=L_{
-\alpha^\vee}}\quad\text{when $a\to\infty$.\;\,}$$

Multiplying on the left by $t^\mu$ leads to $L_{\mu-\alpha^\vee}\in
\overline{S_\mu}$, whence $S_{\mu-\alpha^\vee}\subset\overline{S_\mu}$.
This justifies (at least in the case of $\SL_2$, but the general case can be deduced from this special case) the inclusion
\begin{equation}
\label{eqn:closure-S-inclusion}
\overline{S_\mu}\supset\bigsqcup_{\substack{\nu\in X_*(T)\\
\nu\leq\mu}}S_\nu.
\end{equation}

The proof of the reverse inclusion requires another tool, which is the
subject of the next section.

\subsection{Projective embeddings}
\label{ss:proj-embeddings}

We now want to embed the affine Grassmannian $\Gr_G$ in an (infinite dimensional)
projective space $\mathbf P(V)$ in order to get more control over its geometry. Replacing
$G$ by a simply connected cover of its derived subgroup may kill connected components, but has
the advantage that the resulting group is a product of simple groups.
Therefore in this subsection we assume that $G$ is quasi-simple (i.e.~that it is semisimple and that the quotient by its center is simple) and simply connected. (For the applications we consider, the general case will be reduced to this one.)

The character lattice $X^*(T)$ of $T$ is the $\mathbf{Z}$-dual of $X_*(T)$. Let $W$
be the Weyl group of $(G,T)$, and let $\tau:X_*(T)\to\mathbf Z$ be the $W$-invariant
quadratic form that takes the value $1$ at each short coroot. The polar
form of $\tau$ defines a map $\iota:X_*(T)\to X^*(T)$; from the
$W$-invariance of $\tau$, one deduces that
\begin{equation}
\label{eqn:iota-coroot}
\iota(\alpha^\vee)=\tau(\alpha^\vee)\alpha \quad \text{for each coroot $\alpha^\vee$.}
\end{equation}

Let $\mathfrak g$ be the Lie algebra of $G$. The Lie algebra of $T$ is a
Cartan subalgebra $\mathfrak h$ of $\mathfrak g$. Then $\tau$ can be seen
as the restriction to $\mathfrak h$ of the Killing form of $\mathfrak g$
(suitably rescaled), and $X^*(T)$ is a lattice in the dual space
$\mathfrak h^\vee$.

With the help of the Killing form of $\mathfrak g$, one defines a
$2$-cocycle of the Lie algebra $\mathfrak g
\otimes_{\mathbf C}\mathbf C[t,t^{-1}]$, and thus a central extension
$$0\to\mathbf CK\to\tilde{\mathfrak g}\xrightarrow p\mathfrak
g\otimes_{\mathbf C}\mathbf C[t,t^{-1}]\to0$$
of this algebra by a one-dimensional Lie algebra $\mathbf CK$
(see \cite[pp.~97--98]{kac}).
The affine Kac--Moody Lie algebra is then obtained by taking a
semidirect product
$$\hat{\mathfrak g}=\tilde{\mathfrak g}\rtimes\mathbf Cd$$
with a one-dimensional Lie algebra $\mathbf Cd$, where $d$ acts as
$t\frac d{dt}$ on $\mathfrak g\otimes_{\mathbf C}\mathbf C[t,t^{-1}]$.

Further, $\mathfrak h\subset\mathfrak g\otimes_{\mathbf C}\mathbf C[t,t^{-1}]$ can be canonically
lifted in $\tilde{\mathfrak g}$. Then $\hat{\mathfrak h}=\mathfrak
h\oplus\mathbf CK\oplus\mathbf Cd$ is a Cartan subalgebra of
$\hat{\mathfrak g}$. Let $\Lambda_0\in(\hat{\mathfrak h})^\vee$ be the
linear form that vanishes on $\mathfrak h\oplus\mathbf Cd$ and that
maps $K$ to $1$. Let $V(\Lambda_0)$ be the irreducible integrable
highest weight representation of $\hat{\mathfrak g}$ with highest
weight $\Lambda_0$. It is generated by a highest weight vector
$v_{\Lambda_0}$, and the stabilizer of the line $[v_{\Lambda_0}]$
in $\mathbf P(V(\Lambda_0))$ is the parabolic subalgebra
$p^{-1}(\mathfrak g[t])\rtimes\mathbf Cd$.

Thanks to Garland's work \cite{garland}, we know that the representation
$V(\Lambda_0)$ can be integrated to the Kac--Moody group $\hat G$
that corresponds to the Lie algebra $\hat{\mathfrak g}$. This group
is the semi-direct product of a central extension
$$1\to\mathbf C^\times\to\tilde G\to G(\mathbf C[t,t^{-1}])\to1$$
by another copy of $\mathbf C^\times$, acting by loop rotations. The
central $\mathbf C^\times$ in $\tilde G$ acts by 
scalar multiplication on $V(\Lambda_0)$, so
$G(\mathbf C[t,t^{-1}])$ acts on $\mathbf P(V(\Lambda_0))$. Since
the stabilizer of the line $[v_{\Lambda_0}]$ for this action is the
subgroup $G(\mathbf C[t])$, the map $g\mapsto g \cdot [v_{\Lambda_0}]$
defines an embedding
$$\Psi:G(\mathbf C[t,t^{-1}])/G(\mathbf C[t])\hookrightarrow
\mathbf P(V(\Lambda_0)).$$
Further, using for instance the Iwasawa decomposition, one can show
that on the level of $\mathbf C$-points, the obvious map
$$G(\mathbf C[t,t^{-1}])/G(\mathbf C[t])\to
\GK/\GO=\Gr_G$$
is bijective. We eventually obtain a closed embedding
$$\Psi:\Gr_G\to\mathbf P(V),$$
where (here and below) we write $V$ instead of $V(\Lambda_0)$ to shorten the notation.

Certainly, $\mathbf P(V)$ has the structure of an ind-variety: the
finite-dimensional pieces are all the finite-dimensional projective
subspaces $\mathbf P(W)$ inside $\mathbf P(V)$. Then $\Psi$ is a
morphism of ind-varieties. Even better: thanks to the work of Kumar
(see~\cite[Chap.~7]{kumar}), we know that the ind-variety structure of $\Gr_G$
is induced via $\Psi$ by that of $\mathbf P(V)$.

Lastly, \cite[(6.5.4)]{kac}
implies that
\begin{equation}
\label{eqn:Psi-weights}
\Psi(L_\nu)\in\mathbf P(V_{-\iota(\nu)}),
\end{equation}
where $V_{-\iota(\nu)}$ is the subspace of $V$ of weight
$-\iota(\nu)$ for the action of $\mathfrak h\subset\hat{\mathfrak g}$.

\begin{rmk}
 See~\cite[Remark~10.2(ii)]{pr} for a comparison between the group $\tilde{G}$ considered above and a central extension of $\GK$ considered by Faltings in~\cite{faltings}.
\end{rmk}

\subsection{Consequences}

After these lengthy preliminaries, we can go back to our problem. We drop our assumption that $G$ is quasi-simple and simply connected.

\begin{prop}
\label{prop:closure-semiinfinite}
Let $\mu\in X_*(T)$. Then
\[
\overline{S_\mu}=\bigsqcup_{\nu\leq\mu}S_\nu.
\]
Moreover, there exists a $\C$-vector space $V$ and a closed embedding $\Psi : \Gr_G^{\mu + Q^\vee} \to \mathbf{P}(V)$ such that the boundary of $S_\mu$ is the set-theoretic intersection of $\overline{S_\mu}$
with a hyperplane $H_\mu$ of $\mathbf P(V)$:
$$\partial S_\mu=\overline{S_\mu}\cap\Psi^{-1}(H_\mu).$$
\end{prop}

\begin{proof}
First we assume that $G$ is quasi-simple and simply-connected, and choose $V$ and $\Psi$ as in~\S\ref{ss:proj-embeddings}.

Let $\lambda\in X_*(T)$. Writing $\Psi(L_\lambda)=\mathbf Cv$,
the vector $v$ belongs to the weight subspace $V_{-\iota(\lambda)}$
of $V$ by~\eqref{eqn:Psi-weights}. The action on $v$ of an element $u\in\NK$ can only increase
weights,\footnote{Certainly, above we have described $V$ only as a projective representation
of $G(\C[t,t^{-1}])$, so it seems hazardous to let $\NK$ act on $V$. To be more precise, we observe that the action of $\hat{\mathfrak{g}}$ on $V$ can be extended to its completion considered e.g.~in~\cite[\S 13.1]{kumar} (because the part one needs to complete acts in a locally nilpotent way). Then, using~\cite[Theorem~6.2.3 \& Theorem~13.2.8]{kumar} one sees that this action integrates to an action of a central extension of $\GK$. Finally, one observes
that the cocycle that defines the central extension is
trivial on $\NK$, so that this subgroup can be lifted to the central extension.}
hence $uv-v\in\sum_{\chi>-\iota(\lambda)}V_\chi$. (The
order $\geq$ on $\mathfrak h^\vee$ used here is the dominance order:
nonnegative elements in $\mathfrak h^\vee$ are nonnegative integral
combinations of positive roots.) It follows that
$$\Psi(u\cdot L_\lambda)\in\mathbf P\left(\sum_{\chi\geq-\iota(\lambda)}V_\chi
\right)\smallsetminus\mathbf P\left(\sum_{\chi>-\iota(\lambda)}V_\chi\right).$$
Letting $u$ run over $\NK$, we deduce
$$\Psi(S_\lambda)\subset\mathbf P\left(\sum_{\chi\geq-\iota(\lambda)}V_\chi\right)
\smallsetminus\mathbf P\left(\sum_{\chi>-\iota(\lambda)}V_\chi\right).$$
Writing these inclusions for all possible $\lambda$, we conclude
that\footnote{Specifically, one must here observe that $\nu\leq\mu$
in $X_*(T)$ implies $-\iota(\nu)\geq-\iota(\mu)$ in $\mathfrak h^\vee$.
This follows from the equality~\eqref{eqn:iota-coroot}.}
$$\bigsqcup_{\nu\leq\mu}S_\nu=\Psi^{-1}\left(\mathbf P\left(\sum_{\chi\geq-\iota(\mu)}
V_\chi\right)\right).$$
This implies that $\bigsqcup_{\nu\leq\mu}S_\nu$ is closed in $\Gr_G$,
whence (in view of~\eqref{eqn:closure-S-inclusion}) the first equality in the statement. For the second one, one
chooses a linear form $h\in V^\vee$ that vanishes on
$\sum_{\chi>-\iota(\mu)}V_\chi$ but does not vanish on $\Psi(L_\mu)$ and
takes $H_\mu=\mathbf{P}(\ker h)$.

The case of simply connected semisimple (but not necessarily quasi-simple) groups reduces to the preceding case since such a group is a product of simply connected quasi-simple groups. Finally, for $G$ general, the action of $t^{-\mu}$ identifies $\Gr_G^{\mu + Q^\vee}$ with $\Gr_G^{Q^\vee}$ (which itself identifies with the affine Grassmannian of a simply connected cover of the derived subgroup of $G$) and sends $S_\mu$ to $S_0$; this reduces the proof to the case of simply connected semisimple groups, and allows to conclude.
\end{proof}

\begin{rmk}
 See~\cite[Corollary~5.3.8]{zhu} for a different proof of Proposition~\ref{prop:closure-semiinfinite}, which avoids the use of Kac--Moody groups.
\end{rmk}

For symmetry reasons, one should also consider the Borel subgroup
$B^-$ opposite to $B$ with respect to $T$ and its unipotent radical $N^-$.
One then has an Iwasawa decomposition
$$\Gr_G=\bigsqcup_{\mu\in X_*(T)}T_\mu,\qquad\text{where}\quad T_\mu=
N^-_{\mathcal K}\cdot L_\mu,$$
and the closure of these orbits is given by
\begin{equation}
\label{eqn:closure-Tmu}
\overline{T_\mu}=\bigsqcup_{\substack{\nu\in X_*(T)\\\nu\geq\mu}}T_\nu.
\end{equation}

On the Serre tree (see~\S\ref{ss:semi-infinite-orbits}), $T_\mu$ is seen as the horosphere centered at
$+\infty$ going through $L_\mu$. This makes the following lemma
quite intuitive.

\begin{lem}
\label{lem:semiinfinite-intersection}
Let $\mu,\nu\in X_*(T)$. Then $\overline{S_\mu}\cap\overline{T_\nu}=
\varnothing$ except if $\nu\leq\mu$, and
$\overline{S_\mu}\cap\overline{T_\mu}=\{L_\mu\}$.
\end{lem}

(For a formal proof in the general case, one uses the projective embedding and weights
arguments, as in the proof of Proposition~\ref{prop:closure-semiinfinite}.)

\section{Semisimplicity of the Satake category}
\label{sec:semisimplicity}

From now on in this part, we fix a field $\bk$
of characteristic~$0$.
Our goal in this section is to show that the category $\Perv$ of perverse sheaves
on $\Gr_G$ with coefficients in $\bk$ and  with $\mathscr S$-constructible
cohomology is semisimple. Since every object of this abelian category has
finite length, this result means that there are no non-trivial extensions between
simple objects.

\subsection{The Satake category}
\label{ss:satake-category}

Recall the notion of t-structure introduced in~\cite{bbd}.

\begin{defn}
\label{defn:$t$-structure}
Let $\mathcal{D}$ be a triangulated category. A \emph{$t$-structure} on $\mathcal{D}$ is a pair $(\mathcal{D}^{\leq 0}, \mathcal{D}^{\geq 0})$ of strictly full subcategories of $\mathcal{D}$ which satisfy the following properties:
\begin{enumerate}
\item
If $X \in \mathcal{D}^{\leq 0}$ and $Y \in \mathcal{D}^{\geq 0}$, then $\Hom_{\mathcal{D}}(X,Y[-1])=0$.
\item
\label{it:cond-2-$t$-structure}
We have $\mathcal{D}^{\leq 0} \subset \mathcal{D}^{\leq 0}[-1]$ and $\mathcal{D}^{\geq 0} \supset \mathcal{D}^{\geq 0}[-1]$.
\item
For all $X \in \mathcal{D}$, there exists a distinguished triangle
\[
A \to X \to B \xrightarrow{[1]}
\]
in $\mathcal{D}$ with $A \in \mathcal{D}^{\leq 0}$ and $B \in \mathcal{D}^{\geq 0}[-1]$.
\end{enumerate}
\end{defn}

We will say that
an object in $\mathcal{D}^{\leq0}$ (respectively, $\mathcal{D}^{\geq0}$) is
concentrated in nonpositive (respectively, nonnegative) degrees with
respect to the t-structure. By axiom~\eqref{it:cond-2-$t$-structure} in Definition~\ref{defn:$t$-structure}, these notions are compatible with
the cohomological shift, so we may as well consider for instance the
subcategory $\mathcal{D}^{\geq1}=(\mathcal{D}^{\geq0})[-1]$ of objects concentrated in
positive degrees. We also recall that the \emph{heart} of the t-structure is
the full subcategory $\mathcal{A} := \mathcal{D}^{\leq0}\cap \mathcal{D}^{\geq0}$ of $\mathcal D$;
this is an abelian category, whose exact sequences are the distinguished triangles
\[
X \to Y \to Z \xrightarrow{[1]}
\]
in $\mathcal{D}$ where $X$, $Y$ and $Z$ belong to $\mathcal{A}$. In particular, this means that for any $X,Y$ in $\mathcal{A}$ we have a canonical identification
\begin{equation}
\label{eqn:Ext-heart}
\Ext^1_{\mathcal{A}}(X,Y) \cong \Hom_{\mathcal{D}}(X,Y[1]).
\end{equation}
For instance, the bounded derived category
$\Db(\mathcal A)$ of an abelian category $\mathcal A$ has a natural
t-structure, called the ordinary t-structure, whose heart is $\mathcal A$.

Let now $X$ be a topological space and $\mathscr S$ be a stratification which satisfies certain technical conditions; see~\cite[\S 2.1.3]{bbd}. (These conditions will always tacitly be assumed to be satisfied when we consider perverse sheaves. They are obvious in the concrete cases we study.)
Given $S\in\mathscr S$, we denote
by $i_S:S\hookrightarrow X$ the inclusion map. We denote by
$\Db_{\mathscr S}(X, \bk)$ the bounded derived category of
sheaves of $\bk$-vector spaces on $X$
which are constructible with respect to $\mathscr S$. Thus, a complex
$\mathscr F$ of $\bk$-sheaves on $X$ belongs to
$\Db_{\mathscr S}(X, \bk)$ if the cohomology sheaves
$\mathscr H^n\,\mathscr F$ vanish for $|n|\gg0$ and if each
restriction $i_S^*\,\mathscr H^n\,\mathscr F$ is a local system (i.e.~a locally free sheaf of finite rank).

In this setting, we define
\begin{align*}
{}^p \hspace{-1pt} D^{\leq0}&=\bigl\{\mathscr F\in \Db_{\mathscr S}(X,\bk)
\mid \forall S\in\mathscr S,\ \forall n>-\dim S,\
\mathscr H^n \bigl( (i_S)^*\mathscr F\bigr)=0 \bigr\},\\[2pt]
{}^p \hspace{-1pt} D^{\geq0}&=\bigl\{\mathscr F\in \Db_{\mathscr S}(X,\bk)
\mid \forall S\in\mathscr S,\ \forall n<-\dim S,\
\mathscr H^n \bigl( (i_S)^!\mathscr F \bigr) =0\bigr\}.
\end{align*}
It is known (see~\cite[\S 2.1.13]{bbd}) that $({}^p \hspace{-1pt} D^{\leq0},{}^p \hspace{-1pt} D^{\geq0})$ is
a t-structure on $\Db_{\mathscr S}(X,\bk)$, called the
\emph{perverse $t$-structure}. The simplest example is the case where
$\mathscr S$ contains only one stratum (which requires that $X$
is smooth); then the perverse t-structure is just the ordinary
t-structure (restricted to $\Db_{\mathscr S}(X,\bk)$), shifted to the left by $\dim X$. Objects in the heart
$\Per_{\mathscr{S}}(X,\bk) := {}^p \hspace{-1pt} D^{\leq0}\cap {}^p \hspace{-1pt} D^{\geq0}$ of this t-structure are called
\emph{perverse sheaves}. The truncation functors for this t-structure will we denoted ${}^p  \tau_{\leq i}$ and ${}^p \tau_{\geq i}$, and the corresponding cohomology functors will be denoted ${}^p \hspace{-1pt} \mathscr{H}^i = {}^p \tau_{\leq i} \circ {}^p \tau_{\geq i} = {}^p \tau_{\geq i} \circ {}^p \tau_{\leq i}$.

It is known that every object in $\Per_{\mathscr{S}}(X,\bk)$ has finite length, see~\cite[Th\'eor\`eme~4.3.1]{bbd}. Moreover, 
the simple objects in this category are
classified by pairs $(S,\mathscr L)$, with $S\in\mathscr S$ and
$\mathscr L$ a simple local system on $S$. Specifically, to
$(S,\mathscr L)$ corresponds a unique object $\mathscr F\in
\Db_{\mathscr S}(X,\bk)$ characterized by the conditions
\begin{equation}
\label{eqn:characterization-IC}
\mathscr F\bigl|_{X\smallsetminus\overline S}=0,\quad
\mathscr F\bigl|_S=\mathscr L[\dim S],\quad
i^*\mathscr F\in{}^p \hspace{-1pt} D^{\leq-1}(\overline S\smallsetminus S, \bk),\quad
i^!\mathscr F\in{}^p \hspace{-1pt} D^{\geq1}(\overline S\smallsetminus S, \bk),
\end{equation}
where $i:\overline S\smallsetminus S\hookrightarrow X$ is the inclusion
map. This $\mathscr F$ is a simple perverse sheaf and is usually called the
intersection cohomology sheaf on $\overline S$ with coefficients in
$\mathscr L$, and denoted $\IC(S,\mathscr{L})$. Then the assignment $(S,\mathscr{L}) \mapsto \IC(S,\mathscr{L})$ induces a bijection between equivalence classes of pairs $(S,\mathscr{L})$ as above (where $(S,\mathscr{L}) \sim (S,\mathscr{L}')$ if $\mathscr{L} \cong \mathscr{L}'$) and isomorphism classes of simple objects in $\Per_{\mathscr{S}}(X,\bk)$.

We can finally define the main object of study of these notes. Consider the affine Grassmannian $\Gr_G$, and its 
stratification $\mathscr{S}$ by $\GO$-orbits, see~\S\ref{ss:def-Gr}. Then we can consider the constructible derived category $\Db_{\mathscr S}(\Gr_G,\bk)$, and its full subcategory
$\Perv$ of perverse sheaves. The main result of this section is the following.

\begin{thm}
\label{thm:semisimplicity}
The category $\Perv$ is semisimple.
\end{thm}

\begin{rmk}
Note that the assumption that $\mathrm{char}(\bk)=0$ is crucial here. The category $\Perv$ with $\bk$ a field of positive characteristic is \emph{not} semisimple.
\end{rmk}

The strata $\Gr_G^\lambda$ of $\mathscr{S}$ are simply connected, because they are affine bundles over partial flag varieties (see Proposition~\ref{prop:Cartan-dec}). Thus, the only simple
local system on $\Gr_G^\lambda$ is the trivial local system $\underline{\bk}$. We denote by
$\IC_\lambda$ the corresponding intersection cohomology sheaf. Then,
the simple objects in $\Perv$ are (up to isomorphism) these complexes $\IC_\lambda$,
for $\lambda\in X_*(T)^+$. Since every object in $\Perv$ has finite length, and in view of~\eqref{eqn:Ext-heart}, Theorem~\ref{thm:semisimplicity} follows from the following claim.

\begin{prop}
\label{prop:Ext-vanishing}
For any $\lambda, \mu \in X_*(T)^+$, we have
\[
\Hom_{\Db_{\mathscr S}(\Gr_G,\bk)}(\IC_\lambda, \IC_\mu[1])=0.
\]
\end{prop}

The main ingredients in the proof of Proposition~\ref{prop:Ext-vanishing} are the following facts:
\begin{itemize}
\item
the cohomology
sheaves $\mathscr H^k(\IC_\lambda)$ vanish unless $k$ and
$\dim(\Gr_G^\lambda)$ have the same parity (see Lemma~\ref{lem:parity-IC} below);
\item
if $\Gr_G^\mu \subset \overline{\Gr_G^\lambda}$, then $\mathrm{codim}_{\overline{\Gr_G^\lambda}}(\Gr^\mu)$ is even (see~\S\ref{ss:def-Gr}).
\end{itemize}

\subsection{Parity vanishing}
\label{ss:parity}

As explained above, a key point in the proof of Proposition~\ref{prop:Ext-vanishing} is the following result.

\begin{lem}
\label{lem:parity-IC}
For any $\lambda \in X_*(T)^+$, we have
\[
\mathscr H^n(\IC_\lambda)=0 \quad \text{unless $n \equiv \dim(\Gr_G^\lambda) \pmod 2$.}
\]
\end{lem}


A similar property in fact holds for Iwahori-constructible perverse sheaves on the affine flag variety.
In this section, we argue that
this property can be deduced from the existence of resolutions of
closures of Iwahori orbits whose fibers are paved by affine spaces. (These arguments are sketched in~\cite[\S A.7]{gaitsgory}. A different proof of this property can be given by imitating the case of the finite flag variety treated in~\cite{springer}.)

As in Section~\ref{sec:Gr},
let $W$ be the Weyl group of $(G,T)$ and let $Q^\vee\subset X_*(T)$ be
the coroot lattice. The affine Weyl group and the extended affine Weyl
group are defined as
\[
W_\aff=W\ltimes Q^\vee \quad \text{and} \quad
\widetilde W_\aff=W\ltimes X_*(T)
\]
respectively.
As is well known, $W_\aff$ is
generated by a set $S_\aff$ of simple reflections, and $(W_\aff,S_\aff)$
is a Coxeter system with length function $\ell$ which satisfies
\[
\ell(w \cdot \lambda)=\sum_{\substack{\alpha \in \Delta_+(G,B,T) \\ w(\alpha) \in \Delta_+(G,B,T)}} |\langle \lambda,
\alpha \rangle | + \sum_{\substack{\alpha \in \Delta_+(G,B,T) \\ w(-\alpha) \in \Delta_+(G,B,T)}}
  |1 + \langle \lambda, \alpha \rangle |
\]
for $w \in W$ and $\lambda \in Q^\vee$.
This formula makes sense more generally  for $\lambda \in X_*(T)$, which allows to extend $\ell$ to
$\widetilde
W_\aff$. Then if $\Omega=\{w \in \widetilde{W}_\aff \mid \ell(w)=0\}$, the conjugation action of the subgroup $\Omega$ on $\widetilde{W}_\aff$ preserves $S_\aff$, hence also $W_\aff$, and we have
$\widetilde W_\aff=W_\aff\rtimes\Omega$.

As in~\S\ref{ss:proj-embeddings},
let $B^-\subset G$ be the Borel subgroup opposite to $B$ with respect to $T$,
and let $I\subset\GO$ be the corresponding Iwahori subgroup, defined as the preimage
of $B^-$ under the evaluation map $\GO\to G$ given by $t\mapsto 0$ (i.e.~the arc space of the group scheme considered in Remark~\ref{rmk:def-Gr}\eqref{it:def-Fl}, for the Borel subgroup $B^-$ instead of $B$).
The Bruhat decomposition then yields
$$\Gr_G=\bigsqcup_{w\in\widetilde W_\aff/W}Iw\GO/\GO,$$
and $Iw\GO/\GO$ is an affine space of dimension $\ell(w)$ if
$w$ is of minimal length in the coset $wW$. (Here, if $w=v \cdot \lambda$ with $v \in W$ and $\lambda \in X_*(T)$, by $Iw\GO/\GO$ we mean the $I$-orbit of $\dot v t^{\lambda} \GO/\GO$, where $\dot{v}$ is any lift of $v$ in $N_G(T) \subset G$.)

Let $\lambda\in X_*(T)^+$.
Then
$$\Gr_G^\lambda=\bigsqcup_{w\in Wt_\lambda W/W}Iw\GO/\GO$$
is a union of Schubert cells. One of these cells is open dense in
$\Gr_G^\lambda$; we denote by $w_\lambda$ the unique element in $W t_\lambda W$ which is minimal in $w_\lambda W$ and such that $Iw_\lambda\GO/\GO$ is open in $\Gr_G^\lambda$.
Certainly then we have
\[
\IC_\lambda = \IC(Iw_\lambda\GO/\GO, \underline{\bk}).
\]
Hence Lemma~\ref{lem:parity-IC} follows from the claim that for any $w \in \widetilde W_\aff$ which is minimal in $wW$ we have
\begin{equation}
\label{eqn:parity-IC}
\mathscr H^n \bigl( \IC(Iw\GO/\GO,\underline \bk) \bigr)\neq0
\quad\Longrightarrow\quad n\equiv\ell(w)\pmod2.
\end{equation}

To prove~\eqref{eqn:parity-IC} we introduce the affine flag variety
\[
\Fl_G := \GK/I
\]
(see also Remark~\ref{rmk:def-Gr}\eqref{it:def-Fl}).
As for $\Gr_G$, this variety has a natural complex ind-variety structure, and a Bruhat decomposition
\[
\Fl_G = \bigsqcup_{w \in \widetilde{W}_\aff} I w I/I,
\]
see~\cite{goertz} for details and references. This decomposition provides a 
stratification of $\Fl_G$, which we denote by $\mathscr{T}$. Then we can consider the constructible derived category $\Db_{\mathscr{T}}(\Fl_G, \bk)$ and the corresponding category $\Per_{\mathscr{T}}(\Fl_G,\bk)$ of perverse sheaves.

Let $\pi : \Fl_G \to \Gr_G$ be the natural projection. This morphism is smooth; in fact it is a locally trivial fibration\footnote{This fibration is in fact topologically trivial, as follows from the realization of $\Gr_G$ as a topological group, see~\cite[\S 1.2]{ginzburg}.} with fiber $G/B^-$. From this property and the characterization of the intersection cohomology complex given in~\eqref{eqn:characterization-IC}, it is not difficult to check that for any $w \in \widetilde{W}_\aff$ which is minimal in $wW$, we have
\[
\pi^* \IC (Iw\GO/\GO,\underline \bk) [\ell(w_0)] \cong \IC(Iww_0 I/I,\underline \bk),
\]
where $w_0 \in W$ is the longest element (so that $\ell(w_0)=\dim(G/B^-)$). This shows that~\eqref{eqn:parity-IC} (hence also Lemma~\ref{lem:parity-IC}) follows from the following claim.

\begin{lem}
\label{lem:IC-parity-Fl}
For any $w \in \widetilde{W}_\aff$ we have
\[
\mathscr H^n \bigl( \IC(IwI/I,\underline \bk) \bigr)\neq0
\quad\Longrightarrow\quad n\equiv\ell(w)\pmod2.
\]
\end{lem}

\begin{proof}
For any $s \in S_\aff$,
denote by $J_s=IsI\cup I$ the minimal parahoric subgroup of
$\GK$ associated with $s$. Fix $w \in \widetilde{W}_\aff$, and choose
a reduced expression $\underline{w}=(s_1, \cdots, s_r, \omega)$ for $w$ (with $s_j\in S_\aff$
and $\ell(\omega)=0$). We can then consider the Bott-Samelson resolution
$$\pi_{\underline{w}}:J_{s_1}\times^I\cdots\times^IJ_{s_r}\times^I\bigl(\underbrace{
I\omega I/I\strut}_{\text{(a point)}}\bigr)\to\overline{IwI/I}$$
induced by multiplication in $\GK$. It is known that $\pi_{\underline{w}}$
is proper and is an isomorphism over $IwI/I$.
It is known also that each fiber $\pi_{\underline{w}}^{-1}(x)$ is paved by affine spaces. (For this claim in the case of finite flag varieties, see~\cite{gaussent}. See also~\cite{haines} for a different proof, which works 
mutatis mutandis in the affine setting.)
Therefore
$$\coH_c^{n+\ell(w)}\bigl(\pi_{\underline{w}}^{-1}(x);\bk\bigr)$$
is nonzero only if $n+\ell(w)$ is even. By proper base change, this
cohomology group is the stalk at $x$ of the cohomology sheaf $\mathscr H^n \bigl( (\pi_{\underline{w}})_!\,\underline \bk[\ell(w)] \bigr)$,
so that
$$\mathscr H^n \bigl( (\pi_{\underline{w}})_!\,\underline \bk[\ell(w)] \bigr)\neq0
\quad\Rightarrow\quad n\equiv\ell(w)\pmod2.$$
Our desired parity vanishing property
then follows from the celebrated Decomposition Theorem (see~\cite[Theorem~6.2.5]{bbd}), which here implies that
$\IC(IwI/I;\underline \bk)$ is a direct summand of
the complex $(\pi_{\underline{w}})_!\,\underline \bk[\ell(w)]$.
\end{proof}

\subsection{Proof of Proposition~\ref{prop:Ext-vanishing}}


We follow the arguments in~\cite[Proof of Proposition~1]{gaitsgory} (but adding more details). We distinguish 3 cases (of which only the third one will use Lemma~\ref{lem:parity-IC}).

\subsubsection{First case: $\lambda=\mu$.}

Consider the diagram
$$\xymatrix{\Gr_G^\lambda\ar[r]^j\ar[dr]_(.44){j_\lambda}&\overline{\Gr_G^\lambda}
\ar[d]^(.46){i_\lambda}&\overline{\Gr_G^\lambda}\smallsetminus\Gr_G^\lambda,
\ar[l]_(.58)i\\&\Gr_G&}$$
where all maps are the obvious embeddings.
Set $\mathscr F=(i_\lambda i)^*\IC_\lambda$; by~\eqref{eqn:characterization-IC},
this complex of sheaves is concentrated in negative perverse degrees.
Likewise, the complex of sheaves $(i_\lambda i)^!\IC_\lambda$ is
concentrated in positive perverse degrees. It follows that
\begin{equation}
\label{eqn:ext1-vanishing}
\Hom_{\Db_{\mathscr S}(\overline{\Gr_G^\lambda}\smallsetminus\Gr_G^\lambda,\bk)}\bigl(\mathscr F,(i_\lambda i)^!\IC_\lambda[1]\bigr)=0.
\end{equation}

Applying the cohomological functor $\Hom_{\Db_{\mathscr S}(\Gr_G,\bk)}^\bullet\bigl((i_\lambda)_!\,\bm?,
\IC_\lambda[1]\bigr)$ to the distinguished triangle
$$j_!j^*\Bigl(\IC_\lambda\bigl|_{\overline{\Gr_G^\lambda}}\Bigr)
\to\Bigl(\IC_\lambda\bigl|_{\overline{\Gr_G^\lambda}}\Bigr)
\to i_!i^*\Bigl(\IC_\lambda\bigl|_{\overline{\Gr_G^\lambda}}\Bigr)
\xrightarrow{[1]},$$
we get an exact sequence
\begin{multline}
\label{eqn:exact-sequence-case1}
\Hom_{\Db_{\mathscr S}(\Gr_G,\bk)}\bigl((i_\lambda i)_!\;\mathscr F,\IC_\lambda[1]\bigr)\to
\Hom_{\Db_{\mathscr S}(\Gr_G,\bk)}\bigl(\IC_\lambda,\IC_\lambda[1]\bigr)\\
\to\Hom_{\Db_{\mathscr S}(\Gr_G,\bk)}\bigl((j_\lambda)_!\;
\underline \bk_{\Gr_G^\lambda}[\dim\Gr_G^\lambda],\IC_\lambda[1]\bigr).
\end{multline}
The first space in~\eqref{eqn:exact-sequence-case1} is zero, thanks to~\eqref{eqn:ext1-vanishing} and because by adjunction we have
$$\Hom_{\Db_{\mathscr S}(\Gr_G,\bk)}\bigl((i_\lambda i)_!\,\mathscr F,\IC_\lambda[1]\bigr)\cong
\Hom_{\Db_{\mathscr S}(\overline{\Gr_G^\lambda}\smallsetminus\Gr_G^\lambda,\bk)}\bigl(\mathscr F,(i_\lambda i)^!\,\IC_\lambda[1]\bigr).$$
By adjunction again, the third space in~\eqref{eqn:exact-sequence-case1} is
\begin{align*}
\Hom_{\Db_{\mathscr S}(\Gr_G,\bk)}\bigl((j_\lambda)_!\,\underline \bk_{\Gr_G^\lambda}[\dim\Gr_G^\lambda],
\IC_\lambda[1]\bigr)&\cong\Hom_{\Db_{\mathscr S}(\Gr^\lambda_G,\bk)}\bigl(\underline \bk_{\Gr_G^\lambda}
[\dim\Gr_G^\lambda],(j_\lambda)^!\,\IC_\lambda[1]\bigr)\\[2pt]
&\cong\Hom_{\Db_{\mathscr S}(\Gr^\lambda_G,\bk)}\bigl(\underline \bk_{\Gr_G^\lambda},
\underline \bk_{\Gr_G^\lambda}[1]\bigr)\\[2pt]
&=\coH^1(\Gr_G^\lambda;\bk).
\end{align*}
This last space is again zero since $\Gr_G^\lambda$ is an affine bundle
over a partial flag variety (see Proposition~\ref{prop:Cartan-dec}), so has only cohomology in even degrees.

We conclude that
$\Hom_{\Db_{\mathscr S}(\Gr_G,\bk)}(\IC_\lambda,\IC_\lambda[1])=0.$

\subsubsection{Second case: Neither $\Gr_G^\lambda\subset\overline{\Gr_G^\mu}$
nor $\Gr_G^\mu\subset\overline{\Gr_G^\lambda}$.}

Consider the inclusion $i_\mu:\overline{\Gr_G^\mu}\hookrightarrow\Gr_G$.
Since $\IC_\mu$ is supported on $\overline{\Gr_G^\mu}$, we have
$\IC_\mu=(i_\mu)_*(i_\mu)^*\IC_\mu$ and therefore by adjunction
$$\Hom_{\Db_{\mathscr S}(\Gr_G,\bk)}(\IC_\lambda,\IC_\mu[1])\cong
\Hom_{\Db_{\mathscr{S}}(\overline{\Gr_G^\mu},\bk)}((i_\mu)^*\IC_\lambda,(i_\mu)^*\IC_\mu[1]).$$

Now set $Z=\overline{\Gr_G^\lambda}\cap\overline{\strut\Gr_G^\mu}$ and consider
the inclusion $f:Z\hookrightarrow\overline{\Gr_G^\mu}$. Since $(i_\mu)^*
\IC_\lambda$ is supported on $Z$, it is of the form $f_!\mathscr F$
for some complex of sheaves $\mathscr F\in \Db_{\mathscr{S}}(Z,\bk)$. Arguing as in the first case, we see that
$\mathscr F$ is concentrated in negative perverse degrees and that
$f^!(i_\mu)^*\IC_\mu \cong (i_\mu f)^!\,\IC_\mu$ is concentrated in positive perverse degrees.
Therefore
\begin{align*}
\Hom_{\Db_{\mathscr S}(\Gr_G,\bk)}\bigl(\IC_\lambda,\IC_\mu[1]\bigr) &\cong \Hom_{\Db_{\mathscr{S}}(\overline{\Gr_G^\mu},\bk)}
\bigl(f_!\,\mathscr F,(i_\mu)^*\IC_\mu[1]\bigr)\\
&\cong\Hom_{\Db_{\mathscr{S}}(Z,\bk)}\bigl(
\mathscr F,f^!(i_\mu)^*\IC_\mu[1]\bigr)=0,
\end{align*}
as desired.

\subsubsection{Third case: $\lambda\neq\mu$ and either $\Gr_G^\lambda\subset\overline{\Gr_G^\mu}$
or $\Gr_G^\mu\subset\overline{\Gr_G^\lambda}$.}

Since Verdier duality is an anti-autoequivalence of $\Db_{\mathscr S}(\Gr_G,\bk)$ which fixes $\IC_\lambda$ and $\IC_\mu$, we can assume that
$\Gr_G^\mu\subset\overline{\Gr_G^\lambda}$. Let
$j_\mu:\Gr_G^\mu\hookrightarrow\Gr_G$ be the inclusion, and let
$\mathscr G\in \Db_{\mathscr{S}}(\Gr_G,\bk)$ be the cone of the adjunction map
$\IC_\mu\to (j_\mu)_*(j_\mu{})^*\IC_\mu \cong (j_\mu)_*\underline \bk_{\Gr_G^\mu}[\dim\Gr_G^\mu]$. It follows from the definition of the perverse t-structure that $(j_\mu)_*\underline \bk_{\Gr_G^\mu}[\dim\Gr_G^\mu]$ is concentrated in nonnegative perverse degrees, and it is a classical fact that the morphism $\IC_\mu\to {}^p \hspace{-1pt} \mathscr H^0 \bigl( (j_\mu)_*\underline \bk_{\Gr_G^\mu}[\dim\Gr_G^\mu] \bigr)$ induced by the adjunction map considered above (where ${}^p \hspace{-1pt} \mathscr H^0 ( \bm ?)$ means the degree-$0$ perverse cohomology) is injective, see e.g.~\cite[(1.4.22.1)]{bbd}. Therefore, $\mathscr{G}$ is concentrated in nonnegative perverse degrees.

From the triangle
$$\IC_\mu\to (j_\mu)_*\underline \bk_{\Gr_G^\mu}[\dim\Gr_G^\mu]\to
\mathscr G\xrightarrow{[1]}$$
we get an exact sequence
\begin{multline}
\label{eqn:exact-sequence-Ext-vanishing}
\Hom_{\Db_{\mathscr S}(\Gr_G,\bk)}\bigl(\IC_\lambda,\mathscr G\bigr)\to\Hom_{\Db_{\mathscr S}(\Gr_G,\bk)}\bigl(\IC_\lambda,
\IC_\mu[1]\bigr)\\
\to\Hom_{\Db_{\mathscr S}(\Gr_G,\bk)}\bigl(\IC_\lambda,(j_\mu)_*\;\underline
\bk_{\Gr_G^\mu}[\dim\Gr_G^\mu+1]\bigr).
\end{multline}
As in the second case (but now using the $((-)^*,(-)_*)$ adjunction), using the fact that $\mathscr G$ is concentrated in
nonnegative perverse degrees and supported on $\overline{\Gr_G^\mu}$, which is included in $\overline{\Gr_G^\lambda} \smallsetminus \Gr_G^\lambda$, one checks that the left $\Hom$ space is zero.

By~\eqref{eqn:characterization-IC}, $(j_\mu)^*\IC_\lambda$ is concentrated in
degrees $<-\dim\Gr_G^\mu$. On the other hand, by Lemma~\ref{lem:parity-IC}, this complex has cohomology only in degrees of the same parity as $\dim(\Gr_G^\lambda)$. Noting that $\dim(\Gr_G^\lambda) \equiv \dim(\Gr_G^\mu) \pmod 2$ (because these orbits belong to the same connected component of $\Gr_G$), this implies that in fact $(j_\mu)^*\IC_\lambda$ is concentrated in
degrees $\leq -\dim\Gr_G^\mu-2$.
It follows that
$$\Hom_{\Db_{\mathscr S}(\Gr_G,\bk)}\bigl(\IC_\lambda,(j_\mu)_*\underline \bk_{\Gr_G^\mu}[\dim\Gr_G^\mu+1]
\bigr)=\Hom_{\Db_{\mathscr{S}}(\Gr_G^\mu,\bk)}\bigl((j_\mu)^*\IC_\lambda,\underline \bk_{\Gr_G^\mu}
[\dim\Gr_G^\mu+1]\bigr)$$
vanishes.

Our exact sequence~\eqref{eqn:exact-sequence-Ext-vanishing} then yields the desired equality
$\Hom_{\Db_{\mathscr S}(\Gr_G,\bk)}(\IC_\lambda,\IC_\mu[1])=0$.

\begin{rmk}
 One can give a slightly shorter proof of Proposition~\ref{prop:Ext-vanishing} as follows. Lemma~\ref{lem:parity-IC} and the Verdier self-duality of the objects $\IC_\lambda$ show that these objects are \emph{parity complexes} in the sense of~\cite[Definition~2.4]{jmw} (for the constant pariversity). More precisely, $\IC_\lambda$ is even if $\dim(\Gr_G^\lambda)$ is even, and odd if $\dim(\Gr_G^\lambda)$ is odd. Now Proposition~\ref{prop:Ext-vanishing} is obvious if $\dim(\Gr_G^\lambda)$ and $\dim(\Gr_G^\mu)$ do not have the same parity (because then $\IC_\lambda$ and $\IC_\mu$ live on different connected components of $\Gr_G$) and follows from~\cite[Corollary~2.8]{jmw} if they do have the same parity.
\end{rmk}

\subsection{Consequence on equivariance}
\label{ss:equivariance}

Consider the category
\[
\Per_{\GO}(\Gr_G,\bk)
\]
of 
$\GO$-equivariant perverse sheaves on $\Gr_G$; see~\S\ref{ss:equiv-perv}. (Here the stratification we consider is $\mathscr{S}$.) We have a forgetful functor
\[
\Per_{\GO}(\Gr_G,\bk) \to \Perv,
\]
which is fully faithful by construction.
As a consequence of Theorem~\ref{thm:semisimplicity}, each object in
$\Perv$ is isomorphic to a direct sum of the simple objects $\IC_\lambda$, hence belongs to the essential image of this functor. We deduce the following.

\begin{cor}
\label{cor:equivariance}
The forgetful functor
\[
\Per_{\GO}(\Gr_G,\bk) \to \Perv
\]
is an equivalence of categories.
\end{cor}

\begin{rmk}
 See~\S\ref{ss:equiv-const} below for a different proof of Corollary~\ref{cor:equivariance} which does not use the semisimplicity of $\Perv$ (but requires much more sophisticated tools).
\end{rmk}


\section{Dimension estimates and the weight functors}
\label{sec:weight-functors}

\subsection{Overview}
\label{ss:overview}

Recall that if $\FF$ is a field, the split\footnote{A reductive group is called \emph{split} if it admits a maximal torus which is split, i.e.~isomorphic to a product of copies of the multiplicative group over $\FF$. Here a maximal torus of a reductive group $H$ is a closed subgroup which is a torus and whose base change to an algebraic closure $\overline{\FF}$ of $\FF$ is a maximal torus of $\Spec(\overline{\FF}) \times_{\Spec(\FF)} H$ in the ``traditional'' sense, see e.g.~\cite{humphreys-gps}.} reductive groups over $\FF$ are classified, up to isomorphism, by their root datum\footnote{If $H$ is a split reductive group and $K \subset H$ is a maximal torus, then the root datum of $H$ with respect to $K$ is the quadruple $(X^*(K_{\overline{\FF}}), X_*(K_{\overline{\FF}}), \Delta(H_{\overline{\FF}},K_{\overline{\FF}}), \Delta^\vee(H_{\overline{\FF}},K_{\overline{\FF}}))$ where $\overline{\FF}$ is an algebraic closure of $\FF$, $H_{\overline{\FF}} := \Spec(\overline{\FF}) \times_{\Spec(\FF)} H$, $K_{\overline{\FF}} := \Spec(\overline{\FF}) \times_{\Spec(\FF)} K$, $\Delta(H_{\overline{\FF}},K_{\overline{\FF}})$, resp.~$\Delta^\vee(H_{\overline{\FF}},K_{\overline{\FF}})$, is the root system, resp.~coroot system, of $H_{\overline{\FF}}$ with respect to $K_{\overline{\FF}}$, together with the bijection $\Delta(H_{\overline{\FF}},K_{\overline{\FF}}) \xrightarrow{\sim} \Delta^\vee(H_{\overline{\FF}},K_{\overline{\FF}})$ given by $\alpha \mapsto \alpha^\vee$.} (see e.g.~\cite[Expos\'e XXIII, Corollaire 5.4 and Expos\'e XXII, Proposition~2.2]{sga3}). In particular, we can consider the reductive $\bk$-group $G^\vee_\bk$ which is Langlands dual to $G$, i.e.~whose root datum is dual to that of $G$ (which means that it is obtained from that of $G$ by exchanging weights and coweights and roots and coroots); this group is defined up to isomorphism.

The \emph{geometric Satake equivalence} is the statement that the category $\Per_{\GO}(\Gr_G,\bk)$ is 
equivalent to the category of finite-dimensional representations of $G^\vee_\bk$, in such a way that the tensor product of representations corresponds to a natural operation in $\Per_{\GO}(\Gr_G,\bk)$ called \emph{convolution}.\footnote{Note that if we drop this requirement, the statement becomes vacuous, because the categories $\Per_{\GO}(\Gr_G,\bk)$ and $\Rep_\bk(G^\vee_\bk)$ are both semisimple with simple objects parametrized by $X_*(T)^+$. This weaker statement might be already nontrivial, however, for more general coefficients (see Part~\ref{pt:arbitrary}).} In fact, we will even explain how to construct a \emph{canonical} group scheme $G^\vee_\bk$ which is split reductive (with a canonical maximal torus) and whose root datum is dual to that of $G$, and a canonical equivalence of monoidal categories $\Per_{\GO}(\Gr_G,\bk) \xrightarrow{\sim} \Rep_\bk(G^\vee_\bk)$.

To achieve this goal, the method is to define a convolution product on the abelian category $\Per_{\GO}(\Gr_G,\bk)$
so that this category satisfies the conditions of Theorem~\ref{thm:tannakina-reconstruction} with respect to the
functor
\[
\coH^\bullet(\Gr_G,\bm?) : \Per_{\GO}(\Gr_G,\bk) \to \Vect_\bk.
\]
We will then need to identify the
affine group scheme provided by Theorem~\ref{thm:tannakina-reconstruction}. The construction of a (split) maximal torus in this group scheme, which is the first step in this direction, is based on Mirkovi\'c
and Vilonen's \emph{weight functors}, which we introduce in this section.

Recall that
we have chosen a maximal torus and a Borel subgroup $T\subset B\subset G$.
Then $T\subset\GK$ acts on $\Gr_G=\GK/\GO$ with fixed points
$$(\Gr_G)^T=\bigl\{L_\mu : \mu\in X_*(T)\bigr\}.$$
The choice of a dominant regular cocharacter $\eta\in X_*(T)$ provides
a one-parameter subgroup $\mathbb G_{\mathbf{m}} \subset T$, whence a
$\mathbf C^\times$-action on $\Gr_G$ with fixed points $(\Gr_G)^T$.
The attractive and repulsive varieties relative to the fixed point
$L_\mu$ coincide with the semi-infinite orbits $S_\mu$ and $T_\mu$
defined in Section~\ref{sec:Gr}:
\[
S_\mu=\bigl\{x\in\Gr_G\bigm| \eta(a) \cdot x\to L_\mu\text{ when }a\to0\bigr\}
\]
and
\[
T_\mu=\bigl\{x\in\Gr_G\bigm| \eta(a) \cdot x \to L_\mu\text{ when }a\to\infty\bigr\}
\]
(see~\S\ref{ss:dim-estimates} for details).
With these notations, the weight functor $\F_\mu$ is defined either as
the cohomology with compact support of the restriction to $S_\mu$, or as
the cohomology with support in $T_\mu$. These two definitions are
equivalent, thanks to Braden's theorem on hyperbolic localization.

\begin{rmk}
 In Ginzburg's approach to the geometric Satake equivalence~\cite{ginzburg}, the maximal torus in $G^\vee_\bk$ is instead constructed using equivariant cohomology and the functors of co-restriction to points $L_\lambda$. For a comparison between these points of view, the reader may consult~\cite{gr}.
\end{rmk}

\subsection{Dimension estimates}
\label{ss:dim-estimates}

Recall from~\S\ref{ss:def-Gr} that $\rho$ denotes
the halfsum of the positive roots, considered as a linear form
$X_*(T)\to\frac12\mathbf Z$, and that $Q^\vee \subset X_*(T)$ denotes the coroot lattice.

\begin{thm}
\label{thm:orbits}
Let $\lambda,\mu\in X_*(T)$ with $\lambda$ dominant.
\begin{enumerate}
\item
\label{it:thm-orbits-1}
We have
\[
\overline{\Gr_G^\lambda}\cap S_\mu\neq\varnothing\ \Longleftrightarrow\
L_\mu\in\overline{\Gr_G^\lambda}\ \Longleftrightarrow\
\mu\in\Conv(W\lambda) \cap (\lambda + Q^\vee),
\]
where $\Conv$ denotes the convex hull.
\item
\label{it:thm-orbits-2}
If $\mu$ satisfies the condition in~\eqref{it:thm-orbits-1},
the intersection $\overline{\Gr_G^\lambda}\cap S_\mu$ has pure
dimension\footnote{By this, we mean that all the irreducible components of this variety have dimension $\langle \rho, \lambda+\mu\rangle$.} $\langle \rho, \lambda+\mu\rangle$.
\item
\label{it:thm-orbits-3}
If $\mu$ satisfies the condition in~\eqref{it:thm-orbits-1},
then $\Gr_G^\lambda \cap S_\mu$ is open dense in $\overline{\Gr_G^\lambda}\cap S_\mu$; in particular, the irreducible components of $\overline{\Gr_G^\lambda}\cap S_\mu$ and $\Gr_G^\lambda \cap S_\mu$ are in a canonical bijection.
\end{enumerate}
\end{thm}

\begin{proof}
\eqref{it:thm-orbits-1}
Let $\eta\in X_*(T)$ be regular dominant. If $g\in\NK$, then
$\eta(a) g \eta(a)^{-1} \to1$ when $a\to0$. Therefore, looking at the induced action of $\mathbf C^\times$
on $\Gr_G$, we obtain that for any $\mu\in X_*(T)$,
$$S_\mu\subset\{x\in\Gr_G\mid \eta(a) \cdot x\to L_\mu\text{ when $a\to0$}\}.$$
In view of the Iwasawa decomposition~\eqref{eqn:Iwasawa-dec}, this inclusion
is in fact an equality. Then the stability of $\overline{\Gr_G^\lambda}$
by the action of $T$ implies the first equivalence.

On the other hand, we have
$$W\lambda\subset\{\mu\in X_*(T)\mid L_\mu\in\Gr_G^\lambda\},$$
and using the Cartan decomposition~\eqref{eqn:Cartan-dec}, we see that this inclusion is
in fact an equality. The description of $\overline{\Gr_G^\lambda}$
recalled in~\eqref{eqn:closure-orbit} then implies that $L_\mu \in \overline{\Gr_G^\lambda}$ if and only if the dominant $W$-conjugate $\mu^+$ of $\mu$ satisfies $\mu^+ \leq \lambda$, that is, if and only if
$$W\mu\subset\{\nu\in X_*(T)\mid\nu\leq\lambda\}.$$
Using \cite[chap.~VIII, \S7, exerc.~1]{bourbaki-lie}, we see that this
condition is equivalent to
$$\mu\in\Conv(W\lambda)\cap(\lambda+Q^\vee).$$

\eqref{it:thm-orbits-2}
We start with the following remarks.
From~\eqref{it:thm-orbits-1}, we deduce that $\overline{\Gr_G^\lambda}$ meets only
those $S_\mu$ such that $\mu\leq\lambda$, therefore
$\overline{\Gr_G^\lambda}\subset\overline{\strut S_\lambda}$.\footnote{To prove the inclusion $\overline{\Gr_G^\lambda}
\subset\overline{\strut S_\lambda}$, one can also argue as follows.
The open cell $N_{\mathcal O}B^-_{\mathcal O}$ is dense in $\GO$ and
$B^-_{\mathcal O}$ stabilizes $L_\lambda$, therefore $\Gr_G^\lambda=
\GO\cdot L_\lambda$ contains $N_{\mathcal O}\cdot L_\lambda$ as a
dense subset, whence $\Gr_G^\lambda\subset\overline{N_{\mathcal O}\cdot
L_\lambda}\subset\overline{S_\lambda}$.} If $w_0 \in W$ is the longest element (so that $w_0\lambda$ is the unique antidominant element in $W\lambda$), then conjugating by a lift of $w_0$ we deduce that $\overline{\Gr_G^\lambda}
\subset\overline{\strut T_{w_0\lambda}}$.

Note that if $\mu$ satisfies the condition in~\eqref{it:thm-orbits-1}, then $w_0 \lambda \leq \mu \leq \lambda$. We will now prove, by induction on $\langle \rho, \mu-w_0 \lambda \rangle$, that 
\begin{equation}
\label{eqn:dim-estimates-ineq}
\dim \left( \overline{\Gr_G^\lambda} \cap \overline{\strut S_\mu} \right) \leq \langle \rho, \lambda+\mu \rangle.
\end{equation}
If $\mu=w_0 \lambda$, then from the remarks above we have $\overline{\Gr_G^\lambda}\cap\overline{\strut S_{w_0\lambda}} \subset \overline{S_{w_0\lambda}} \cap \overline{T_{w_0\lambda}} = \{L_{w_0\lambda}\}$ (see Lemma~\ref{lem:semiinfinite-intersection}), so that the claim holds in this case.

Assume now that $\mu > w_0 \lambda$, and 
choose a hyperplane $H_\mu$ as in Proposition~\ref{prop:closure-semiinfinite}.
Let $C$ be an irreducible component of $\overline{\Gr_G^\lambda} \cap \overline{\strut S_\mu}$ and let $D$ be an irreducible component of $C\cap\Psi^{-1}(H_\mu)$. Then $\dim(D) \geq \dim(C)-1$, and $D$ contained in
\[
\Psi^{-1}(H_\mu) \cap \overline{\Gr_G^\lambda} \cap \overline{\strut S_\mu} = \partial S_\mu \cap \overline{\Gr_G^\lambda} = \bigcup_{\nu < \mu} \overline{\strut S_\nu} \cap \overline{\Gr_G^\lambda},
\]
so by induction $\dim D\leq\max_{\nu < \mu} \langle \rho, \lambda + \nu \rangle = \langle \rho, \lambda+\mu \rangle - 1$. We deduce that $\dim C\leq \dim D+1\leq\langle \rho, \lambda+\mu \rangle$, which finishes the proof of~\eqref{eqn:dim-estimates-ineq}.

The inequality~\eqref{eqn:dim-estimates-ineq} implies that each irreducible component $C$ of $\overline{\Gr_G^\lambda} \cap S_\mu$ has dimension at most $\langle \rho, \lambda+\mu \rangle$.
We will now prove that this dimension is always exactly $\langle \rho, \lambda+\mu \rangle$. First, if $\mu=\lambda$ then as observed above we have $\overline{\Gr_G^\lambda} \cap \overline{\strut S_\lambda} = \overline{\Gr_G^\lambda}$. Since this variety is irreducible of dimension $\langle 2\rho, \lambda \rangle = \langle \rho, \lambda + \lambda \rangle$ by Proposition~\ref{prop:Cartan-dec}, this implies that its (nonempty) open subset $\overline{\Gr_G^\lambda} \cap S_\lambda$ has the same properties. Now, assume that $\mu < \lambda$, and fix an irreducible component $C$ of $\overline{\Gr_G^\lambda} \cap S_\mu$. Set $d:=\langle \rho, 2\lambda \rangle - \dim(C)$, and let $H_\lambda$ be as in Proposition~\ref{prop:closure-semiinfinite}. Then we have
\[
\overline{\strut C} \subset \overline{\Gr_G^\lambda} \cap \partial S_\lambda = \overline{\Gr_G^\lambda} 
\cap \Psi^{-1}(H_\lambda).
\]
Hence there exists an irreducible component $D_1$ of the right-hand side containing $\overline{C}$. Then $\dim(D_1)=\langle \rho, 2\lambda \rangle-1$, and $D_1$ is the disjoint union of its locally closed intersections with the orbits $S_\nu$ with $w_0 \lambda \leq \nu < \lambda$; hence there exists such a $\nu_1$ such that $C_1:= D_1 \cap S_{\nu_1}$ is open dense in $D_1$. We necessarily have $\nu_1 \geq \mu$ since otherwise $\overline{C}$ would be contained in $\overline{\Gr_G^\lambda} \cap \partial S_\mu$, which is not the case. Now $C_1$ is an irreducible component of $\overline{\Gr_G^\lambda} \cap S_{\nu_1}$ of dimension $\langle \rho, 2\lambda \rangle-1$ such that $\overline{C_1}$ contains $\overline{C}$. If $d>1$ we must have $\mu < \nu_1$; in fact, otherwise from the facts that $\overline{C} \subset \overline{C_1}$ and that
\[
C_1 = \overline{\strut C_1} \cap S_{\nu_1}  \quad\text{and}\quad C = \overline{\strut C} \cap S_\mu
\]
we would deduce that $C \subset C_1$, so that $C=C_1$ (which is impossible for reasons of dimension) since both of these varieties are 
irreducible components of $\overline{\Gr^\lambda_G} \cap S_\mu$.

Repeating this argument we find coweights $\nu_1, \cdots, \nu_d$ which satisfy
\begin{equation}
\label{eqn:weights-semiinfinite}
\mu \leq \nu_d < \nu_{d-1} < \cdots < \nu_1 < \lambda
\end{equation}
and irreducible components $C_i$ of $\overline{\Gr^\lambda_G} \cap S_{\nu_i}$ such that $\overline{C} \subset \overline{C_i}$ and $\dim(C_i)=\langle \rho, 2\lambda \rangle-i$. Then~\eqref{eqn:weights-semiinfinite} implies that
$\langle \rho,\mu \rangle \leq \langle \rho,\lambda \rangle -d$,
or in other words that $d \leq \langle \rho,\lambda \rangle - \langle \rho,\mu \rangle$; this implies that
\[
\dim(C) \geq \langle \rho, \lambda + \mu \rangle,
\]
as expected.

\eqref{it:thm-orbits-3}
Let $Z$ be an irreducible component of
$\overline{\Gr_G^\lambda}\cap S_\mu$.
Then $Z$ must meet $\Gr_G^\lambda$, otherwise by \eqref{eqn:closure-orbit}
it would be contained in some $\overline{\Gr_G^\eta}$ with $\eta<\lambda$,
and the inequality $\dim Z=\langle\rho,\lambda+\mu\rangle>\langle\rho,
\eta+\mu\rangle$ would contradict~\eqref{it:thm-orbits-2}. Therefore
$Z\cap\Gr_G^\lambda$ is open dense in $Z$. 
\end{proof}

\begin{rmk}
The irreducible components of the intersections $\overline{\Gr_G^\lambda} \cap S_\mu$, or sometimes those of the intersections $\Gr_G^\lambda \cap S_\mu$, are called \emph{Mirkovi\'c--Vilonen cycles}; they have been studied and used extensively in various fields since their introduction in~\cite{mv}, see e.g.~\cite{bg, gaussent-littelmann, bag, kamnitzer}.
\end{rmk}

The following corollary will prove to be useful.

\begin{cor}
\label{cor:dimension}
Let $\lambda\in X_*(T)^+$, and let $X\subset\overline{\Gr_G^\lambda}$
be a closed $T$-invariant subvariety. Then
$$\dim (X)\leq\max_{\substack{\mu \in X_*(T) \\ L_\mu\in X}} \langle \rho, \lambda+\mu \rangle.$$
\end{cor}

\begin{proof}
Let $\eta\in X_*(T)$ be regular dominant. We saw during the proof of
Theorem~\ref{thm:orbits}\eqref{it:thm-orbits-1} that
$$S_\mu=\{x\in\Gr_G\mid \eta(a) \cdot x\to L_\mu\text{ when $a\to0$}\}.$$
Therefore $X$ meets $S_\mu$ if and only if $L_\mu\in X$, whence
\[
X\subset\bigcup_{\substack{\mu\in X_*(T)\\L_\mu\in X}}
S_\mu,
\]
and therefore
\[
X\subset\bigcup_{\substack{\mu\in X_*(T)\\L_\mu\in X}}
\bigl(\overline{\Gr_G^\lambda}\cap S_\mu\bigr).
\]
The corollary now follows from Theorem~\ref{thm:orbits}\eqref{it:thm-orbits-2}.
\end{proof}

The following theorem is the analogue of Theorem~\ref{thm:orbits} for the Borel subgroup $B^-$ in place of~$B$.

\begin{thm}
\label{thm:orbits-T}
Let $\lambda,\mu\in X_*(T)$ with $\lambda$ dominant.
\begin{enumerate}
\item
\label{it:thm-orbits-T-1}
We have
\[
\overline{\Gr_G^\lambda}\cap T_\mu\neq\varnothing\ \Longleftrightarrow\
L_\mu\in\overline{\Gr_G^\lambda}\ \Longleftrightarrow\
\mu\in\Conv(W\lambda) \cap (\lambda + Q^\vee).
\]
\item
\label{it:thm-orbits-T-2}
If $\mu$ satisfies the condition in~\eqref{it:thm-orbits-T-1},
the intersection $\overline{\Gr_G^\lambda}\cap T_\mu$ has pure
dimension
$\langle \rho, \lambda-\mu\rangle$.
\item
\label{it:thm-orbits-T-3}
If $\mu$ satisfies the condition in~\eqref{it:thm-orbits-T-1},
then $\Gr_G^\lambda \cap T_\mu$ is open dense in $\overline{\Gr_G^\lambda}\cap T_\mu$; in particular, the irreducible components of $\overline{\Gr_G^\lambda}\cap T_\mu$ and $\Gr_G^\lambda \cap T_\mu$ are in a canonical bijection.
\end{enumerate}
\end{thm}

\subsection{Weight functors}
\label{ss:weight-functors}

Recall that if $X$ is a topological space, $i : Y \to X$ is the inclusion of a locally closed subspace and $\mathscr{F} \in \Db_c(X,\bk)$, then the local cohomology groups $\coH^k_Y(X,\mathscr{F})$ are defined as $\coH^k(Y, i^! \mathscr{F})$.

\begin{prop}
\label{prop:weight-functors}
For each $\mathscr A\in\Per_\GO(\Gr_G,\mathbf k)$, $\mu\in X_*(T)$
and $k\in\mathbf Z$, there exists a canonical isomorphism
$$\coH^k_{T_\mu}(\Gr_G,\mathscr A)\xrightarrow\sim \coH^k_c(S_\mu,\mathscr A),$$
and both terms vanish if $k\neq \langle 2\rho,\mu \rangle$.
\end{prop}

\begin{proof}
For all $\lambda\in X_*(T)^+$, we have
$\mathscr A\bigl|_{\Gr_G^\lambda}\in D^{\leq- \langle 2\rho, \lambda \rangle}
(\Gr_G^\lambda,\bk)$ by the perversity conditions (see~\S\ref{ss:satake-category}).
Further, the dimension estimates from Theorem~\ref{thm:orbits}\eqref{it:thm-orbits-2} imply that
$\coH^k_c(\Gr_G^\lambda\cap S_\mu;\bk)=0$ for $k> \langle 2\rho, \lambda+\mu \rangle$, see~\cite[Proposition~X.1.4]{iversen}.
Using an easy d\'evissage argument, we deduce that
\[
\coH^k_c(\Gr_G^\lambda\cap S_\mu,\mathscr A)=0 \qquad \text{for $k> \langle 2\rho, \mu \rangle$}.
\]
Filtering the support of $\mathscr A$ by the closed subsets
$\overline{\Gr_G^\lambda}$, we deduce that
\[
\coH^k_c(S_\mu,\mathscr A)=0 \qquad \text{for $k> \langle 2\rho, \mu \rangle$.}
\]
(To prove this formally, one can either use
a spectral sequence or write down distinguished triangles associated
to inclusions of an open subset and its closed complement.
With both methods, in order to deal with a sequence of closed
subsets, it is convenient to enumerate the dominant weights as $(\lambda_n)_{n\geq0}$ in such
a way that $(\lambda_i\leq\lambda_j)
\Rightarrow(i\leq j)$.)

An analogous (dual) argument, using~\cite[Theorem~X.2.1]{iversen}, shows that
\[
\coH^k_{T_\mu}(\Gr_G,\mathscr A)=0 \qquad \text{for $k< \langle 2\rho, \mu \rangle$}.
\]
Lastly, Braden's
hyperbolic localization theorem \cite[Theorem~1]{braden} provides a canonical isomorphism
\[
\coH^k_{T_\mu}(\Gr_G,\mathscr A)\cong \coH^k_c(S_\mu,\mathscr A)
\]
for any $k \in \mathbf{Z}$. The claim follows.
%
\end{proof}

\begin{rmk}\phantomsection
\begin{enumerate}
 \item
 See~\cite[\S 1.8.1]{xue} for a discussion of the validity of the normality assumption needed to apply Braden's theorem, and for an alternative proof using~\cite{drg} instead on~\cite{braden} (and which therefore avoids this normality question).
 \item
 Explicity, the isomorphism in Proposition~\ref{prop:weight-functors} is constructed as follows. Let
 \[
  t_\mu : T_\mu \to \Gr_G \quad \text{and} \quad s_\mu : S_\mu \to \Gr_G
 \]
be the embeddings, and consider also the natural maps
\[
 \pi_\mu^T : T_\mu \to \{L_\mu\}, \quad \pi_\mu^S : S_\mu \to \{L_\mu\}, \quad i_\mu^T : \{L_\mu\} \to T_\mu, \quad i_\mu^S : \{L_\mu\} \to S_\mu.
\]
By adjunction and the base change theorem, there exist canonical isomorphisms
\begin{multline*}
 \Hom \bigl( (i_\mu^T)^* (t_\mu)^!(-), (i_\mu^S)^! (s_\mu)^* (-) \bigr) \cong \Hom \bigl( (t_\mu)^!(-), (i_\mu^T)_* (i_\mu^S)^! (s_\mu)^* (-) \bigr) \\
 \cong \Hom \bigl( (t_\mu)^!(-), (t_\mu)^! (s_\mu)_* (s_\mu)^* (-) \bigr);
\end{multline*}
hence the adjunction morphism $\id \to (s_\mu)_* (s_\mu)^*$ induces a morphism of functors
\[
 (i_\mu^T)^* (t_\mu)^! \to (i_\mu^S)^! (s_\mu)^*.
\]
Finally, one identifies the functors $(i_\mu^T)^*$ and $(\pi^T_\mu)_*$, resp.~$(i_\mu^S)^!$ and $(\pi^S_\mu)_!$, when applied to ``weakly equivariant'' objects; see~\cite[Equation~(1)]{braden}.
 \end{enumerate}
\end{rmk}


In view of this proposition, for any $\mu \in X_*(T)$ we consider the functor
\[
\F_\mu : \Per_{\GO}(\Gr_G,\bk) \to \Vect_\bk
\]
defined by
\[
\F_\mu(\mathscr{A}) = \coH^{\langle 2\rho, \mu \rangle}_{T_\mu}(\Gr_G,\mathscr A) \cong \coH^{\langle 2\rho, \mu \rangle}_c(S_\mu,\mathscr A).
\]
Since the category $\Per_{\GO}(\Gr_G,\bk)$ is semisimple (see Theorem~\ref{thm:semisimplicity} and Corollary~\ref{cor:equivariance}), this functor is automatically exact.

\begin{rmk}\phantomsection
\label{rmk:weight-functors}
\begin{enumerate}
\item
The $\GO$-invariance is not used in the proof of Proposition~\ref{prop:weight-functors} (only the constructibility with respect to $\GO$-orbits matters).
\item
The same arguments show more generally that if $\mathscr{F}$ is in $\Per_{\mathscr{S}}(Z,\bk)$, where $Z \subset \Gr_G$ is a locally closed union of $\GO$-orbits (and where by abuse we still denote by $\mathscr{S}$ the restriction of this stratification to $Z$), then for any $\lambda \in X_*(T)^+$ such that $L_\lambda \in Z$ and any $k \in \Z$ there exists a canonical isomorphism
\[
 \coH^k_{T_\lambda \cap Z}(Z,\mathscr{F}) \xrightarrow{\sim} \coH^k_c(S_\lambda \cap Z, \mathscr{F}),
\]
and that these spaces vanish unless $k=\langle 2\rho, \lambda \rangle$. (Note that if $Z$ is not closed, the condition $L_\lambda \in Z$ is \emph{not} equivalent to the condition $S_\lambda \cap Z \neq \varnothing$. In particular, $Z$ might not be covered by the intersections $Z \cap S_\lambda$ where $\lambda \in X_*(T)$ is such that $L_\lambda \in Z$.)
\end{enumerate}
\end{rmk}

\subsection{Total cohomology and weight functors}

We now consider the functor
\[
\F : \Per_\GO(\Gr_G,\mathbf k)\to\Vect_{\mathbf k}
\]
defined by
\[
\F(\mathscr{A}) = \coH^\bullet(\Gr_G,\mathscr{A}).
\]

\begin{thm}\phantomsection
\label{thm:fiber-functor}
\begin{enumerate}
\item
\label{it:fiber-functor-1}
There exists a canonical isomorphism of functors
$$\F\cong\bigoplus_{\mu\in X_*(T)}\F_\mu:
\Per_\GO(\Gr_G,\mathbf k)\to\Vect_{\mathbf k}.$$
\item
\label{it:fiber-functor-2}
The functor $\F$ is exact and faithful.
\end{enumerate}
\end{thm}

\begin{proof}
\eqref{it:fiber-functor-1}
Let $\mathscr A\in\Per_\GO(\Gr_G,\mathbf k)$. Our aim is to construct a canonical isomorphism
$$\coH^\bullet(\Gr_G,\mathscr A)\cong\bigoplus_{\mu\in X_*(T)}\F_\mu(\mathscr A),$$
and more precisely to construct a canonical isomorphism
$$\coH^k(\Gr_G,\mathscr A)\cong\bigoplus_{\substack{\mu\in X_*(T)\\[2pt]
\langle 2\rho, \mu \rangle=k}}\F_\mu(\mathscr A)$$
for each $k\in\mathbf Z$.

Without loss of generality, we may assume that $\mathscr A$ is
indecomposable, and in particular that the support of $\mathscr A$ is
connected.

For $n\in\frac12\mathbf Z$, set
$$Z_n=\bigsqcup_{\substack{\mu\in X_*(T)\\ \langle \rho, \mu \rangle=n}}T_\mu.$$
Then both
$$\bigcup_{n\in\mathbf Z}Z_n\quad\text{and}\quad
\bigcup_{n\in\frac12+\mathbf Z}Z_n$$
are unions of connected components of $\Gr_G$. As $\supp\mathscr A$
was assumed to be connected, it is contained in one of these
subsets. Let us assume that it is contained in the first one, the
reasoning in the other case being entirely similar.

We endow $Z_n$ with the topology induced from that of $\Gr_G$. Then
$Z_n$ is the topological disjoint union of the $T_\mu$ contained
in it, and it follows that
\begin{equation}
\label{eqn:local-cohom}
\coH^k_{Z_n}(\Gr_G,\mathscr A)= \begin{cases}
0 & \text{if $k\neq2n$;}\\[6pt]
\displaystyle\bigoplus_{\langle \rho, \mu \rangle=n}\F_\mu(\mathscr A) & \text{if $k=2n$.}
\end{cases}
\end{equation}

By~\eqref{eqn:closure-Tmu}, the closure of $Z_n$ is
$$\overline{Z_n}=Z_n\sqcup Z_{n+1}\sqcup Z_{n+2}\sqcup\cdots
=Z_n\sqcup\overline{Z_{n+1}},$$
so there is a diagram of complementary open and closed inclusions
$$\overline{Z_{n+1}}\xrightarrow i\overline{Z_n}\xleftarrow jZ_n.$$
Applying the cohomological functor $\coH^\bullet(\overline{Z_n},\bm ?)$ to the
distinguished triangle
$$i_*i^!\mathscr A_n\to\mathscr A_n\to j_*j^!\mathscr A_n\xrightarrow{[1]}$$
where $\mathscr{A}_n$ is the corestriction of $\mathscr{A}$ to $\overline{Z_n}$, we obtain a long exact sequence
$$\cdots\to
\coH^k_{\overline{Z_{n+1}}}(\Gr_G,\mathscr A)\to
\coH^k_{\overline{Z_n}}(\Gr_G,\mathscr A)\to
\coH^k_{Z_n}(\Gr_G,\mathscr A)\to
\coH^{k+1}_{\overline{Z_{n+1}}}(\Gr_G,\mathscr A)\to\cdots.$$

For $n$ large enough, $\supp\mathscr A$ is disjoint from
$\overline{Z_n}$, because $\supp\mathscr A$ is compact and
$\overline{Z_n}$ is far away from the origin
of $\Gr_G$.\footnote{The reader may here have in mind the Serre tree considered in~\S\ref{ss:semi-infinite-orbits}:
$\overline{Z_n}$ is a union of horospheres centered at $+\infty$ and
going through $L_{n\alpha^\vee}$; for $n$ large enough, this is located
far away on the right.} Consequently
$\coH^\bullet_{\overline{Z_n}}(\Gr_G,\mathscr A)=0$ for $n$ large enough.
Using the long exact sequence above and~\eqref{eqn:local-cohom}, a decreasing induction on
$n$ leads to
\begin{align*}
&\coH^k_{\overline{Z_n}}(\Gr_G,\mathscr A)=0\hspace{108pt}\text{if $k$
is odd or if $n > \frac{k}{2}$,}\\
&\xymatrix{\coH^k_{\overline{Z_{k/2}}}(\Gr_G,\mathscr A)\ar[r]^\sim
\ar[d]^\wr&\coH^k_{\overline{Z_n}}(\Gr_G,\mathscr A)\\
\coH^k_{Z_{k/2}}(\Gr_G,\mathscr A)}\qquad\text{if $k$ is even and $n \leq \frac{k}{2}$.}
\end{align*}

One concludes by taking $n$ small enough so that $\supp\mathscr A
\subset\overline{Z_n}$.

\eqref{it:fiber-functor-2}
The exactness is automatic since the category $\Per_{\GO}(\Gr_G,\bk)$ is semisimple (see the comments before the theorem). Given the exactness, the faithfulness means that
$\F$ does not kill any nonzero object in $\Per_\GO(\Gr_G,\mathbf k)$.
So let us take a nonzero perverse sheaf $\mathscr A$ in our category.
Then $\supp\mathscr A$ is a finite union of Schubert cells
$\Gr_G^\lambda$. Let us choose $\lambda$ maximal for this property.
Then $\mathscr{A}|_{\Gr_G^\lambda} \cong \underline{\bk}[\dim \Gr_G^\lambda]$ and as in the proof of Theorem~\ref{thm:orbits}\eqref{it:thm-orbits-2} we have
\[
\bigl( (\supp\mathscr A) \smallsetminus \Gr_G^\lambda \bigr) \cap T_\lambda = \varnothing \quad \text{and} \quad
\Gr_G^\lambda \cap T_\lambda=\{L_\lambda\},
\]
and therefore
$\F_\lambda(\mathscr A)\neq0$, which implies that $\F(\mathscr{A}) \neq 0$.
\end{proof}

\begin{rmk}
\label{rmk:equiv-def-wf}
The proof of Theorem~\ref{thm:fiber-functor} has broken the symmetry between the
two sides of hyperbolic localization, so let us try to restore it.
Given $\mu\in X_*(T)$, let us define the inclusion maps
$$\xymatrix{T_\mu\ar@/_1pc/[rr]_{t_\mu}\ar[r]^{t'_\mu}&
\overline{T_\mu}\ar[r]^{t''_\mu}&\Gr_G}
\quad\text{and}\quad
\xymatrix{S_\mu\ar@/_1pc/[rr]_{s_\mu}\ar[r]^{s'_\mu}&
\overline{S_\mu}\ar[r]^{s''_\mu}&\Gr_G}.$$
Then for each $\mathscr A\in\Per_\GO(\Gr_G,\mathbf k)$, we have
\begin{align*}
\coH^k_{\overline{T_\mu}}(\Gr_G,\mathscr A)&=\coH^k(\Gr_G,
(t''_{\mu})_! (t''_\mu)^!\mathscr A), \\
\coH^k(\overline{S_\mu},\mathscr A)&=\coH^k(\Gr_G,
(s''_{\mu})_* (s''_\mu)^*\mathscr A),\\
\coH^k_{T_\mu}(\Gr_G,\mathscr A)&=\coH^k(\Gr_G,
\underbrace{(t''_{\mu})_! (t'_{\mu})_* (t'_\mu)^*
(t''_\mu)^!}_{\textstyle \cong t_{\mu*}t_\mu^!}\mathscr A),\\
\coH^k_c(S_\mu,\mathscr A)&=\coH^k(\Gr_G,
\underbrace{(s''_{\mu})_*(s'_{\mu})_! (s'_\mu)^!
(s''_\mu)^*}_{\textstyle \cong s_{\mu!}s_\mu^*}\mathscr A).
\end{align*}
One can check that the adjunction maps and hyperbolic localization give
rise to a commutative diagram
$$\xymatrix@C=0em{&&\coH^k(\Gr_G,\mathscr A)
\ar[drr]|{\id\to(s''_{\mu})_*(s''_\mu)^*}&&\\
\coH^k_{\overline{T_\mu}}(\Gr_G,\mathscr A)
\ar[urr]|{(t''_{\mu})_!(t''_\mu)^!\to\id}
\ar[dr]|{\id\to(t'_{\mu})_*(t'_\mu)^*}&&&&
\coH^k(\overline{S_\mu},\mathscr A).\\
&\coH^k_{T_\mu}(\Gr_G,\mathscr A)\ar[rr]_-{\sim}^-{\text{hyperbolic loc.}}&&
\coH^k_c(S_\mu,\mathscr A)
\ar[ur]|{(s'_{\mu})_!(s'_\mu)^!\to\id}&}$$
If $k=\langle 2\rho, \mu \rangle$, then the three bottom arrows are isomorphisms,
so the four bottom spaces can be identified: they define the functor
$\F_\mu$.
At this point, let us write
$$\F_\mu(\mathscr A)\xrightarrow{i_\mu}
\coH^k(\Gr_G,\mathscr A)\xrightarrow{p_\mu}\F_\mu(\mathscr A)$$
for the two top arrows of the diagram above. Theorem~\ref{thm:fiber-functor} shows
that for each $k\in\mathbf Z$,
\begin{equation}
\label{eqn:decomposition-F-1}
\coH^k(\Gr_G,\mathscr A)=\bigoplus_{\substack{\mu\in X_*(T)\\ \langle 2\rho, \mu \rangle=k}}
\im (i_\mu),
\end{equation}
and likewise
\begin{equation}
\label{eqn:decomposition-F-2}
\coH^k(\Gr_G,\mathscr A)\xrightarrow\sim\bigoplus_{\substack{\nu\in X_*(T)\\ \langle 2\rho,\nu \rangle=k}}
\coim (p_\nu),
\end{equation}
where $\coim (p_\nu)=\coH^k(\Gr_G,\mathscr A)/\ker (p_\nu)$ is the coimage of
$p_\nu$. Further, if $\mu\neq\nu$ and $\langle 2\rho, \mu \rangle=\langle 2\rho,\nu \rangle$, then $\mu\not\leq\nu$,
so
$\overline{S_\nu}\cap\overline{T_\mu}=\varnothing$ by Lemma~\ref{lem:semiinfinite-intersection}, and therefore
$p_\nu\circ i_\mu=0$. This implies that the decompositions~\eqref{eqn:decomposition-F-1} and~\eqref{eqn:decomposition-F-2}
of $\coH^k(\Gr_G,\mathscr A)$ coincide.
The decomposition
$$\coH^k(\Gr_G,\mathscr A)=\bigoplus_{\substack{\mu\in X_*(T)\\\langle 2\rho,\mu \rangle=k}}
\F_\mu(\mathscr A)$$
is therefore defined without ambiguity.
\end{rmk}

\subsection{Independence of the choice of Torel}
\label{ss:independence-Torel}

To define the functors $\F_\mu$ we started by choosing a Torel (or Borus) $T \subset B$. In this subsection we show that these functors are in fact independent of this choice, in the following way. If we fix a Torel $T \subset B$, then any other Torel will be of the form $gTg^{-1} \subset gBg^{-1}$ for some $g \in G$, whose class $gT \in G/T$ is uniquely determined. Then there exists a canonical isomorphism $X_*(T) \xrightarrow{\sim} X_*(gTg^{-1})$ sending $\lambda : \C^\times \to T$ to the cocharacter $z \mapsto g\lambda(z)g^{-1}$. We use this operation to identify $X_*(T)$ and $X_*(gTg^{-1})$.\footnote{Note that this construction depends on the choice of Borel subgroup: for two general maximal tori in $G$, there is no canonical identification of their cocharacter lattices!} Then for $\lambda \in X_*(T)$ we can consider both the $\lambda$-weight functor $\F_\lambda$ constructed out of the Torel $T \subset B$, and the $\lambda$-weight functor $\F_\lambda^{gT}$ constructed out of the Torel $gTg^{-1} \subset gBg^{-1}$.

\begin{lem}
 \label{lem:indep}
 In the setting considered above, for any $gT \in G/T$
 there exists a canonical isomorphism of functor $\F_\lambda \xrightarrow{\sim} \F_\lambda^{gT}$.
\end{lem}

\begin{proof}
 Set $X_\lambda:=\{(x,gT) \in \Gr_G \times G/T \mid x \in g \cdot S_\lambda\}$, and consider the diagram
 \[
  \xymatrix{
  \Gr_G & \Gr_G \times G/T \ar[l]_-{a} \ar[rd]^-{b} && X_\lambda \ar[ll]_-{c} \ar[ld]_-{d} \\
  && G/T, &
  }
 \]
where $a,b,d$ are the natural projections, and $c$ is the embedding. Let $\mathscr{F}$ in $\Per_{\mathscr{S}}(\Gr_G,\bk)$, and consider the complex of sheaves $b_* c_! c^* a^* \mathscr{F} \cong d_! (ac)^*\mathscr{F}$. By the base change theorem, the fiber of this complex over $gT$ is $R\Gamma_c(g \cdot S_\lambda, \mathscr{F})$. Applying Proposition~\ref{prop:weight-functors} for the choice of Torel $gTg^{-1} \subset gBg^{-1}$, we see that this fiber is concentrated in degree $\langle \lambda, 2\rho \rangle$. Hence the complex $b_* c_! c^* a^* \mathscr{F}$ itself is concentrated in degree $\langle \lambda, 2\rho \rangle$.

Next, the proof of Theorem~\ref{thm:fiber-functor} (and the comments in Remark~\ref{rmk:equiv-def-wf}) can also be written ``in family'' over $G/T$; this shows that $\mathscr{H}^{\langle 2\rho, \lambda \rangle} \bigl( b_* c_! c^* a^* \mathscr{F} \bigr)$ is a direct factor of
\[
 \mathscr{H}^{\langle 2\rho, \lambda \rangle} \bigl( b_* a^* \mathscr{F} \bigr) \cong \coH^{\langle 2\rho,\lambda \rangle}(\Gr_G,\mathscr{F}) \otimes_\bk \underline{\bk}_{G/T}
\]
(where the isomorphism follows from the projection formula). Since this sheaf is a constant local system, we deduce that $\mathscr{H}^{\langle 2\rho, \lambda \rangle} \bigl( b_* c_! c^* a^* \mathscr{F} \bigr)$ is also a constant local system. Hence its fibers over any two points can be identified canonically (because they both identify with global sections); in particular we deduce a canonical isomorphism $\F_\lambda \xrightarrow{\sim} \F_\lambda^{gT}$ for any $gT \in G/T$.
\end{proof}

\begin{rmk}
 Note that the proof of Lemma~\ref{lem:indep} only relies on the $\mathscr{S}$-constructibility of $\mathscr{F}$, and not on its $\GO$-equivariance; in particular, this proof is independent of Corollary~\ref{cor:equivariance}. (This fact does not play any role in the present case when $\bk$ is a field of characteristic $0$, but will be important in the case of general coefficients considered in Part~\ref{pt:arbitrary}.)
\end{rmk}

\subsection{Weight spaces of simple objects}

\begin{prop}
\label{prop:dim-weight-spaces}
Let $\lambda,\mu\in X_*(T)$ with $\lambda$ dominant. 
Then $\dim \F_\mu(\IC_\lambda)$
is the number of irreducible components of $\Gr_G^\lambda\cap S_\mu$.
In particular, it is nonzero if and only if $\mu\in\Conv(W\lambda)
\cap\bigl(\lambda+Q^\vee\bigr)$.
\end{prop}

\begin{proof}
For each $\eta\in X_*(T)^+$, one of the following three
possibilities hold:\\[-20pt]
\begin{itemize}
\item
$\Gr_G^\eta$ does not meet $\supp\IC_\lambda$,
and $\IC_\lambda |_{\Gr_G^\eta}=0$;
\item
$\eta=\lambda$ and $\IC_\lambda |_{\Gr_G^\eta}\in D^{\leq-\langle2\rho,\eta\rangle}(\Gr_G^\eta,\mathbf k)$;
\item
$\eta<\lambda$ and $\IC_\lambda |_{\Gr_G^\eta}\in D^{\leq- \langle2\rho, \eta\rangle-1}(\Gr_G^\eta,\mathbf k)$
\end{itemize}
(see~\eqref{eqn:characterization-IC}).
In the last case, we can in fact replace $-\langle 2\rho,\eta \rangle-1$ by
$-\langle 2\rho,\eta \rangle-2$ because of Lemma~\ref{lem:parity-IC} (and the fact that $\eta < \lambda \Rightarrow \langle 2\rho, \lambda \rangle \equiv \langle 2\rho, \eta \rangle \pmod 2$).

When we gather these facts to reconstruct $\coH^{\langle 2\rho,\mu \rangle}_c(S_\mu,\IC_\lambda)$ using the same method as in the proof
of Proposition~\ref{prop:weight-functors}, only the stratum $\Gr_G^\lambda$
contributes, and we obtain an isomorphism
$$\coH_c^{\langle 2\rho,\mu \rangle}(S_\mu,\IC_\lambda)\cong
\coH_c^{\langle 2\rho,\mu \rangle}\bigl(\Gr_G^\lambda\cap S_\mu,\IC_\lambda |_{\Gr_G^\lambda}\bigr).$$
Therefore
$$\F_\mu(\IC_\lambda)\cong
\coH_c^{\langle 2\rho,\mu \rangle}\bigl(\Gr^\lambda_G\cap S_\mu,\IC_\lambda |_{\Gr_G^\lambda}\bigr)=
\coH_c^{\langle 2\rho,\lambda+\mu \rangle}(\Gr_G^\lambda\cap S_\mu;\mathbf k).$$
The right-hand side is the top cohomology group with compact support of
$\Gr_G^\lambda\cap S_\mu$ by Theorem~\ref{thm:orbits}; it therefore has a natural basis
indexed by the irreducible components of top dimension of this intersection.\footnote{This property is a classical fact about the top cohomology with compact supports of algebraic varieties, which follows e.g.~from the considerations in~\cite[\S X.1]{iversen}.}

The last claim then follows from Theorem~\ref{thm:orbits}.
\end{proof}

\begin{rmk}
\begin{enumerate}
 \item 
 See Proposition~\ref{prop:can-basis-k} below for a proof, based on slightly different ideas, of a statement which reduces to Proposition~\ref{prop:dim-weight-spaces} in the case $\bk$ is a field of characteristic $0$.
 \item
 In a similar vein, one can describe the multiplicity space of a simple object $\IC_\nu$ as a direct summand of a product $\IC_\lambda \star \IC_\mu$ (where $\star$ is the convolution product introduced in~\S\ref{ss:def-convolution} below) in terms of cohomology of a certain variety, see~\cite[Corollary~5.1.5]{zhu} for details.
\end{enumerate}

\end{rmk}

\section{Convolution product: ``classical'' point of view}
\label{sec:convolution}

Our goal in Sections~\ref{sec:convolution}--\ref{sec:convolution-BD} is to endow the category $\Per_{\GO}(\Gr_G,\bk)$ of $\GO$-equivariant
perverse sheaves on $\Gr_G$ with the structure of a symmetric monoidal category.
We first define the convolution product of two equivariant perverse
sheaves, and with the help of the notion of stratified semismall map,
we show that the result of the operation is a perverse sheaf. We also define
an associativity constraint. To proceed further, we will need a different point of view on convolution, which uses an important
auxiliary construction, known as the Be{\u\i}linson--Drinfeld Grassmannian. This is considered in Section~\ref{sec:convolution-BD}.

\subsection{Stratified semismall maps}

We first consider a general result, which guarantees that the direct image of a perverse sheaf under a \emph{stratified semismall morphism}\footnote{This notion is a refinement of the more familiar notion of \emph{semismall morphism} (see e.g.~\cite{gm,bm}) which takes into account the stratifications on the varieties under consideration.} is a perverse sheaf.

Let $(X,\mathscr T)$ and $(Y,\mathscr U)$ be two stratified algebraic
varieties, 
and let $f:Y\to X$ be a
proper map such that for each $U\in\mathscr U$, the set $f(U)$ is a
union of strata. We say that $f$ is \textit{stratified semismall}
if for any stratum $T\subset f(U)$ and any $x\in T$, we have
$$\dim\bigl(f^{-1}(x)\cap U\bigr)\leq\frac12\bigl(\dim U-\dim T\bigr).$$
We say that $f$ is \textit{locally trivial} if for any
$(T,U)\in\mathscr T\times\mathscr U$ such that $T \subset f(U)$, the map $U\cap f^{-1}(T)\to T$
induced by $f$ is a Zariski locally trivial fibration.

\begin{prop}
\label{prop:semismall-exact}
If $f$ is stratified semismall and locally trivial and if $\mathscr F$
is a perverse sheaf on $Y$ constructible with respect to $\mathscr U$, 
then $f_*\mathscr F$ is a perverse sheaf on $X$ constructible with
respect to $\mathscr T$.
\end{prop}

\begin{proof}
For any stratum $T \in\mathscr T$, we can consider the restriction
$$\xymatrix@R=4pt{f^{-1}(T)\ar[r]^(.56){f_T} \ar@{=}[d]&T.\\
\bigsqcup\limits_{U\in\mathscr U}f^{-1}(T)\cap U&}$$
We denote by $f_{T,U}:f^{-1}(T)\cap U\to T$ the restriction of $f$ (which is a Zariski locally trivial fibration by assumption if $T \subset f(U)$).
Note here that since $f(U)$ is a union of strata in $\mathscr T$,
the assertions that $T\subset f(U)$ and that $f^{-1}(T)\cap U\neq\varnothing$
are equivalent.

First, let us prove that for any $\mathscr{F}$ in the $\mathscr{U}$-constructible derived category $\Db_{\mathscr{U}}(Y,\bk)$, the complex $f_*\mathscr{F} = f_!\mathscr{F}$ belongs to the $\mathscr{T}$-constructible derived category $\Db_{\mathscr{T}}(X,\bk)$. We proceed by induction on the smallest number of strata whose union is a closed subvariety $Z$ of $Y$ such that $\mathscr{F}|_{Y \smallsetminus Z}=0$. So, let us consider such a closed union of strata, and choose some $U \in \mathscr{U}$ which is open in $Z$. We can consider $\mathscr{F}$ as a complex in $\Db_{\mathscr{U}}(Z,\bk)$. Then, if we denote by $j : U \to Z$ and $i : Z \smallsetminus U \to Z$ the embeddings, we have a standard distinguished triangle
\[
j_! j^* \mathscr{F} \to \mathscr{F} \to i_* i^* \mathscr{F} \xrightarrow{[1]}.
\]
Applying $f_!$, we deduce a distinguished triangle
\[
(f \circ j)_! j^* \mathscr{F} \to \mathscr{F} \to (f \circ i)_! i^* \mathscr{F} \xrightarrow{[1]}.
\]
By induction, the third term in this triangle belongs to $\Db_{\mathscr{T}}(X,\bk)$. Since $\Db_{\mathscr{T}}(X,\bk)$ is a triangulated subcategory of the derived category of $\bk$-sheaves on $X$, we are reduced to prove that $(f \circ j)_! j^* \mathscr{F}$ belongs to $\Db_{\mathscr{T}}(X,\bk)$. Using truncation triangles, for this it suffices to prove that for each $n \in \mathbf{Z}$, $(f \circ j)_! \mathscr H^n( j^* \mathscr{F})$ belongs to $\Db_{\mathscr{T}}(X,\bk)$. Let $T \in \mathscr{T}$ such that $T \subset f(U)$, and let $g : T \to X$ be the embedding.
By the base change theorem, we have
\begin{equation}
\label{eqn:proof-semismall}
g^* (f \circ j)_! \mathscr{H}^n( j^* \mathscr{F}) \cong (f_{T,U})_! \mathscr{H}^n(\mathscr{F})|_{f^{-1}(T) \cap U}.
\end{equation}
Now since $\mathscr{F}$ is $\mathscr{U}$-constructible, 
$\mathscr H^n( \mathscr{F})|_{f^{-1}(T) \cap U}$ is a local system; since $f_{T,U}$ is a locally trivial fibration we deduce that the cohomology sheaves of $g^* (f \circ j)_! \mathscr H^n( j^* \mathscr{F})$ are local systems on $T$, and finally that $\mathscr{F}$ belongs to $\Db_{\mathscr{T}}(X,\bk)$.\footnote{In this argument we use the compatibility of external products with $!$-pushforwards; see~\cite[Proposition~2.9.II]{lyubashenko} for a precise statement.}


Next, we prove that if $\mathscr{F}$ is in nonpositive perverse degrees, then $f_! \mathscr{F}$ is in nonpositive perverse degrees. Let as above $T \in \mathscr{T}$ be a stratum, and 
consider the Cartesian diagram
\[
\xymatrix{
f^{-1}(T) \ar[r]^-{h} \ar[d]_-{f_T} & Y \ar[d]^-{f} \\
T \ar[r]^-{g} & X,
}
\]
where $g$ and $h$ are the embeddings.
Then we need to prove that
\[
g^* f_! \mathscr{F} \cong (f_T)_! h^* \mathscr{F}
\]
is concentrated in degrees $\leq-\dim T$. By the same arguments as above, it suffices to prove that for any $U \in \mathscr{U}$ such that $U \cap f^{-1}(T) \neq \varnothing$, the complex $(f_{T,U})_! \mathscr{F}|_{U \cap f^{-1}(T)}$ satisfies this property. This follows from a classical vanishing result for cohomology with compact supports already used in the proof of Proposition~\ref{prop:weight-functors}, see~\cite[Proposition~X.1.4]{iversen}.



Finally, we need to prove that if $\mathscr{F}$ is in nonnegative perverse degrees, then $f_* \mathscr{F}$ is in nonnegative perverse degrees. This can be deduced from what we proved above using Verdier duality, or alternatively by an argument ``dual'' to the preceding one: for $T$, $h$, $g$ as above we need to prove that
\[
g^! f_* \mathscr{F} \cong (f_T)_* h^! \mathscr{F}
\]
is concentrated in degrees $\geq-\dim T$. Again, this can be reduced to proving that for any $U \in \mathscr{U}$ the complex
\[
(f_{T,U})_* k^! \mathscr{F}
\]
is concentrated in degrees $\geq-\dim T$, where $k : f^{-1}(T) \cap U \to Y$ is the embedding. If $x \in T$ is any point and $i_x : \{x\} \to T$ is the embedding, for this it suffices to prove that $(i_x)^! (f_{T,U})_* k^! \mathscr{F}$ is concentrated in degrees $\geq \dim(T)$. In turn, this follows from a classical result for cohomology with support, see e.g.~\cite[Lemma~4.12]{small2}.\footnote{In the cases of interest to us here, the local system appearing in~\cite[Lemma~4.12]{small2} will be constant; then the claim we need is the statement~\cite[Theorem~X.2.1]{iversen} already used in the proof of Proposition~\ref{prop:weight-functors}.}
\end{proof}

\subsection{Definition of convolution on $\Gr_G$}
\label{ss:def-convolution}

To define the convolution operation on $\Per_{\GO}(\Gr_G,\bk)$, we will identify this category with the heart of the perverse t-structure on the \emph{constructible equivariant derived category}
\[
 \Db_{c,\GO}(\Gr_G, \bk)
\]
in the sense of Bernstein--Lunts~\cite{bernstein-lunts}, see~\S\ref{ss:equiv-perv}. (See also~\S\S\ref{ss:appendix-conv}--\ref{ss:appendix-Gr} for details on the definition of $\Db_{c,\GO}(\Gr_G, \bk)$ and of convolution in a more general context.)

We denote by $[h]\in\Gr_G$ the coset $h\GO$ of an element $h\in\GK$.
Likewise, letting the group $\GO$ act on $\GK\times\Gr_G$ by
$k\cdot(g,[h])=(gk^{-1},[kh])$, we denote by $[g,h]$ the orbit of
$(g,[h])$. We form the diagram
\begin{equation}
\label{eqn:diagram-convolution}
\Gr_G\times\Gr_G\xleftarrow{\,p}\GK\times\Gr_G
\xrightarrow{q\,}\GK\times^\GO\Gr_G\xrightarrow{m\,}\Gr_G,
\end{equation}
where $p$ is the map $(g,[h])\mapsto([g],[h])$, $q$ is the map
$(g,[h])\mapsto[g,h]$, and $m$ is the map $[g,h]\mapsto[gh]$.

Let $\mathscr F$ and $\mathscr G$ be two complexes of sheaves in the
equivariant derived category $\Db_{c,\GO}(\Gr_G,\bk)$. 
Since the $\GO$-action on $\GK \times \Gr_G$ considered above is free, the functor $q^*$ induces an equivalence of categories
\[
\Db_{c,\GO}(\GK \times^{\GO} \Gr_G, \bk) \xrightarrow{\sim} \Db_{c,\GO \times \GO}(\GK \times \Gr_G,\bk),
\]
see~\cite[Theorem~2.6.2]{bernstein-lunts}.
(Here, $\GO$ acts on $\GK \times^{\GO} \Gr_G$ via multiplication on the left on $\GK$; for the action of $\GO\times\GO$ on $\GK\times\Gr_G$, the first copy of $\GO$ acts via left multiplication on $\GK$ and the second copy acts as above.) The complex $p^*(\mathscr{F} \boxtimes \mathscr{G})$ defines an object of $\Db_{c,\GO \times \GO}(\GK \times \Gr_G)$. Therefore,
we can
consider the unique object $\mathscr F\, \widetilde\boxtimes \,\mathscr G\in
\Db_{c,\GO}\bigl(\GK\times^\GO\Gr_G,\bk\bigr)$ such that
$$q^*(\mathscr F \, \widetilde\boxtimes \,\mathscr G)=
p^*(\mathscr F\boxtimes\mathscr G).$$
We then set
$$\mathscr F \star \mathscr G := m_*(\mathscr F \, \widetilde\boxtimes \, \mathscr G) \quad \in \Db_{c,\GO}(\Gr_G,\bk).$$

\begin{rmk}
When stating this construction in these terms we cheat a little bit; see~\S\ref{ss:appendix-Gr}.
\end{rmk}

\subsection{Exactness of convolution}
\label{ss:exactness-convolution}

The first important property of the convolution product $\star$ on $\Db_{c,\GO}(\Gr_G,\bk)$ is the following.

\begin{prop}
\label{prop:exactness-convolution}
Assume that $\mathscr{F}$ and $\mathscr{G}$ belong to $\Per_{\GO}(\Gr_G,\bk)$. Then $\mathscr{F} \star \mathscr{G}$ also belongs to $\Per_{\GO}(\Gr_G,\bk)$.
\end{prop}

To prove this result we will need an auxiliary lemma. Here, for $\lambda, \mu \in X_*(T)^+$ we set
\[
\widetilde\Gr^{\lambda,\mu}_G:=q(p^{-1}(\Gr_G^\lambda\times\Gr_G^\mu)).
\]

\begin{lem}
\label{lem:dimensions-semismall}
For any $\lambda, \mu \in X_*(T)^+$ and $\nu \in -X_*(T)^+$, we have
$$\dim\bigl(\widetilde\Gr^{\lambda,\mu}_G\cap m^{-1}(L_\nu)\bigr)\leq
\langle\rho,\lambda+\mu+\nu\rangle.$$
\end{lem}

\begin{proof}
We consider the $T$-action on $\GK\times^\GO\Gr_G$ induced by left multiplication on $\GK$,
and the diagonal $T$-action on $\Gr_G\times\Gr_G$. Then the map
$$\phi:\GK\times^\GO\Gr_G\to\Gr_G\times\Gr_G$$
that sends $[g,h]$ to $([g],[gh])$ is a $T$-equivariant isomorphism.
We deduce that the $T$-fixed points in $\GK\times^\GO\Gr_G$ are of
the form $[t^\alpha,t^\beta]$, with $\alpha,\,\beta\in X_*(T)$;
indeed $\phi([t^\alpha,t^\beta])=(L_\alpha,L_{\alpha+\beta})$.
Further, $[t^\alpha,t^\beta]$ belongs to
$$X_{\lambda,\mu} := \overline{\widetilde\Gr_G^{\lambda,\mu}}=q\bigl(p^{-1}\bigl(
\overline{\Gr_G^\lambda}\times\overline{\strut\Gr_G^\mu}\bigr)\bigr)$$
if and only if the dominant $W$-conjugate $\alpha^+$ of $\alpha\in X_*(T)$ is
$\leq\lambda$ and the dominant $W$-conjugate $\beta^+$ of $\beta$ is
$\leq\mu$ with respect to the dominance order.

The morphism $\phi$ maps $m^{-1}(L_\nu)$ to $\Gr_G\times\{L_\nu\}$.
This allows (by projecting onto the first factor) to regard
$X_{\lambda,\mu}\cap m^{-1}(L_\nu)$ as a
closed subvariety of $\overline{\Gr_G^\lambda}$. Now 
by Corollary~\ref{cor:dimension} we have
$$\dim\Bigl(X_{\lambda,\mu}\cap m^{-1}(L_\nu)\Bigr)
\leq\max_{\substack{\alpha, \beta \in X_*(T) \\ [t^\alpha,t^\beta] \in X_{\lambda,\mu} \cap m^{-1}(L_\nu)}} 
\langle\rho,\lambda+\alpha\rangle.
$$
The pairs $(\alpha,\beta)$ occurring here satisfy $\alpha+\beta=\nu$ and
\[
\langle \rho, \mu+\beta \rangle = \langle \rho, \mu-w_0(\beta) \rangle \geq 0
\]
since $w_0(\beta) \leq \beta^+ \leq \mu$;
hence they satisfy
$$\langle\rho,\lambda+\alpha\rangle \leq \langle\rho,\lambda+\alpha\rangle + \langle \rho, \mu+\beta \rangle =
\langle\rho,\lambda+\mu+\nu\rangle,$$
which entails the desired result.
\end{proof}

We can now give the proof of Proposition~\ref{prop:exactness-convolution}.

\begin{proof}
We consider the situation
$$\xymatrix@R=4pt{\GK\times^\GO\Gr_G\ar[r]^(.6)m\ar@{=}[d]&\Gr_G.
\ar@{=}[d],\\
\bigsqcup\limits_{\lambda,\mu\in X_*(T)^+}\!\!\!\!\widetilde\Gr_G^{\lambda,\mu}
&\bigsqcup\limits_{\nu\in X_*(T)^+}\!\!\!\Gr_G^\nu}$$
Here certainly $m$ is ind-proper. It is locally trivial, because
the whole situation is $\GO$-equivariant. Also, it follows from the
definitions that the complex $\mathscr F \, \widetilde\boxtimes \, \mathscr G \in \Db_{c,\GO}(\GK \times^{\GO} \Gr_G,\bk)$
defined in~\S\ref{ss:def-convolution} is perverse and is constructible
with respect to the stratification given by the subsets
$\widetilde\Gr_G^{\lambda,\mu}$. To show that $\mathcal F \star \mathcal G$
is perverse, using Proposition~\ref{prop:semismall-exact} it thus suffices to prove that $m$ is stratified
semismall. This is exactly the content of Lemma~\ref{lem:dimensions-semismall} (since $\dim(\Gr_G^{w_0(\nu)}) = \langle 2\rho, w_0(\nu) \rangle = -\langle 2\rho,\nu \rangle$ if $\nu \in -X_*(T)$).
\end{proof}

\begin{rmk}
A different proof of Proposition~\ref{prop:exactness-convolution} is due to
Gaitsgory. In fact, the convolution
$\mathscr F\star\mathscr G$ 
makes sense for any $\mathscr{F}$ in $\Db_c(\Gr_G,\bk)$ and $\mathscr{G}$ in $\Db_{c,\GO}(\Gr_G,\bk)$.
It follows from~\cite[Proposition~6]{gaitsgory} that, in this generality, $\mathscr F\star\mathscr G$ is perverse as soon as $\mathscr{F}$ and $\mathscr{G}$ are perverse. This approach uses an interpretation
of convolution in terms of nearby cycles. (See also~\cite[\S 5.4]{zhu} for an exposition of closely related ideas, based on the notion of universal local acyclicity.)
\end{rmk}

\subsection{Associativity of convolution}
\label{ss:associativity}

For $\mathscr F_1$, $\mathscr F_2$, $\mathscr F_3$ in
$\Per_{\GO}(\Gr_G,\bk)$, one can define
$$\mathrm{Conv}_3(\mathscr F_1,\mathscr F_2,\mathscr F_3)=
(m_3)_*\,\bigl(\mathscr F_1\, \widetilde\boxtimes \, \mathscr F_2
\, \widetilde\boxtimes \, \mathscr F_3\bigr),$$
where $m_3:\GK\times^\GO\GK\times^\GO\Gr_G\to\Gr_G$ is the map
$[g_1,g_2,g_3]\mapsto[g_1g_2g_3]$, with an obvious notation, and the twisted product $\mathscr F_1\, \widetilde\boxtimes \, \mathscr F_2\, \widetilde\boxtimes \, \mathscr F_3$ is defined in the obvious way.
Then base change yields natural isomorphisms
$$(\mathscr F_1\star \mathscr F_2) \star \mathscr F_3\xleftarrow\sim
\mathrm{Conv}_3(\mathscr F_1,\mathscr F_2,\mathscr F_3)
\xrightarrow\sim\mathscr F_1 \star (\mathscr F_2 \star \mathscr F_3).$$
The composition of these isomorphisms provides an associativity constraint that turns the pair
$(\Per_{\GO}(\Gr_G,\bk),\star)$ into a monoidal category.

\section{Convolution and fusion}
\label{sec:convolution-BD}

In this section we describe a different construction of the convolution product on $\Per_{\GO}(\Gr_G,\bk)$. This construction uses the Be{\u\i}linson--Drinfeld Grassmannian, hence ultimately the moduli interpretation of $\Gr_G$. It plays a crucial role in the definition of the commutativity constraint for $\star$. (The ideas behind all of this go back to work of Be{\u\i}linson--Drinfeld~\cite{beilinson-drinfeld}. For more details and references on this point of view, the reader might consult~\cite{zhu}.)

\subsection{A moduli interpretation of the affine Grassmannian}
\label{ss:moduli-Gr}

In this section, we adopt the following setup. We consider a smooth
curve $X$ over $\C$, and for any point $x\in X$, we denote
by $\mathcal O_x$ the completion of the local ring of $X$ at $x$
and by $\mathcal K_x$ the fraction field of $\mathcal O_x$; the choice
of a local coordinate $t$ on $X$ around $x$ leads to isomorphisms
$\mathcal O_x\cong\C[ \hspace{-1pt} [t] \hspace{-1pt} ]$ and $\mathcal K_x\cong\C( \hspace{-1pt} (t) \hspace{-1pt} )$. Using these data we can define a ``local'' version of $\Gr_G$ at $x$ by $\Gr_{G,x}:= \bigl( G_{\mathcal K_x}/G_{\mathcal O_x} )_{\mathrm{red}}$, where $G_{\mathcal K_x}$ and $G_{\mathcal O_x}$ are defined in the obvious way.

\begin{rmk}
 Below, to lighten notation (and since this does not play any role for us) we will not distinguish between the ind-scheme $G_{\mathcal K_x}/G_{\mathcal O_x}$ and the associated ind-variety $\Gr_{G,x}$. We leave it to the attentive (and interested) reader to check which version is more appropriate in each statement.
\end{rmk}

We define
$$\mathcal D_x=\Spec(\mathcal O_x)\quad\text{and}\quad
\mathcal D_x^\times=\Spec(\mathcal K_x).$$
For a $\C$-algebra $R$, we consider the completed tensor products
$R \, \widehat{\otimes} \,\mathcal O_x$ and $R \, \widehat{\otimes} \, \mathcal K_x$,
so that
$$R \, \widehat\otimes \, \mathcal O_x\cong R[ \hspace{-1pt} [t] \hspace{-1pt} ]\quad\text{and}\quad
R \, \widehat\otimes \, \mathcal K_x\cong R( \hspace{-1pt} (t) \hspace{-1pt} ).$$
We set
$$\mathcal D_{x,R}=\Spec(R \, \widehat\otimes \, \mathcal O_x)
\quad\text{and}\quad
\mathcal D_{x,R}^\times=\Spec(R \, \widehat\otimes \, \mathcal K_x).$$

For a $\C$-algebra $R$, we set
$$X_R=X\times_{\Spec(\C)}\Spec(R)\quad\text{and}\quad
X_R^\times=(X\smallsetminus\{x\})\times_{\Spec(\C)}\Spec(R).$$

\begin{rmk}
 Note that the subscript ``$R$'' does not have the same meaning in the notation ``$\mathcal D_{x,R}$'' and ``$X_R$,'' in that it is \emph{not} true that $\mathcal D_{x,R} \cong \mathcal{D}_x \otimes_{\Spec(\C)} \Spec(R)$.
\end{rmk}

The following proposition gives a first description of $\Gr_{G,x}$ in terms of moduli of bundles on $X$.

\begin{prop}
\begin{enumerate}
\item
The ind-scheme $G_{\mathcal{K}_x}$ represents
the functor
$$R\mapsto\left\{(\mathcal F,\nu,\mu)\left|
\begin{aligned}
\;&\mathcal F\text{ $G$-bundle on }
X_R\\
&\nu:G\times X_R^\times \xrightarrow{\sim}
\mathcal F|_{X_R^\times}
\text{ trivialization on }X_R^\times \\
&\mu:G\times \mathcal{D}_{x,R} \xrightarrow{\sim}
\mathcal F|_{\mathcal{D}_{x,R}}
\text{ trivialization on } \mathcal{D}_{x,R}
\end{aligned}\right\}\right.\Bigr/\text{isomorphism}.$$
\item
The ind-scheme $\Gr_{G,x}$ represents
the functor
$$R\mapsto\left\{(\mathcal F,\nu)\left|
\begin{aligned}
\;&\mathcal F\text{ $G$-bundle on }
X_R\\
&\nu:G\times X_R^\times \xrightarrow{\sim}
\mathcal F|_{X_R^\times}
\text{ trivialization on }X_R^\times
\end{aligned}\right\}\right.\Bigr/\text{isomorphism}.$$
\end{enumerate}
\end{prop}

Here, a \emph{$G$-bundle} on a scheme $Z$ is a scheme $\mathcal{F} \to Z$ equipped with a right $G$-action and which, locally in the fpqc topology, is isomorphic to the product $G \times Z$ as a $G$-scheme. (In fact, since $G$ is smooth here, a $G$-bundle will also be locally trivial in the \'etale tolopology; see~\cite[Remark~2.1.2]{sorger} for more comments and references.)
The proof of this proposition is given in~\cite[Propositions~3.8 and~3.10]{laszlo-sorger}. The main ingredients are:
\begin{enumerate}
\item
The Beauville--Laszlo theorem~\cite{beauville-laszlo}, which says that the datum of
a $G$-bundle on $X_R$ is equivalent to the datum of a $G$-bundle
on $X^\times_R$, of a $G$-bundle on $\mathcal D_{x,R}$, and of a gluing
datum on $\mathcal{D}_{x,R}^\times = \mathcal{D}_{x,R} \cap X_R^\times$.
\item
The fact that any $G$-bundle on $\mathcal D_{x,R}$ 
becomes trivial when pulled back to $\mathcal D_{x,R'}$ for some faithfully flat extension $R \to R'$.\footnote{See also~\cite[Lemma~1.3.7]{zhu} for a slightly different statement in the same vein.}
\end{enumerate}

The Beauville--Laszlo theorem also shows that restriction
induces an isomorphism
$$\left\{(\mathcal F,\nu)\left|
\begin{aligned}
\;&\mathcal F\text{ $G$-bundle on }X_R\\
&\nu\text{ trivialization on }X^\times_R
\end{aligned}\right\}\right.\Bigr/\text{isom.}\xrightarrow{\sim}
\left\{(\mathcal F,\nu)\left|
\begin{aligned}
\;&\mathcal F\text{ $G$-bundle on }
\mathcal D_{x,R}\\
&\nu\text{ trivialization on }\mathcal D_{x,R}^\times
\end{aligned}\right\}\right.\Bigr/\text{isom.}$$
In particular, we deduce that $\Gr_{G,x}$ also represents the functor 
%
\begin{equation}
\label{eqn:Gr-moduli}
R\mapsto\left\{(\mathcal F,\nu)\left|
\begin{aligned}
\;&\mathcal F\text{ $G$-bundle on }
\mathcal D_{x,R}\\
&\nu:G\times \mathcal D_{x,R}^\times \xrightarrow{\sim}
\mathcal F|_{\mathcal D_{x,R}^\times}
\text{ trivialization on }\mathcal D_{x,R}^\times
\end{aligned}\right\}\right.\Bigr/\text{isomorphism}.
\end{equation}

\begin{rmk}
 The description of $\Gr_G$ (or in fact more precisely $\widetilde{\Gr}_G$) in terms of $G$-bundles on $\Spec(\C[ \hspace{-1pt} [t] \hspace{-1pt} ])$ as in~\eqref{eqn:Gr-moduli} is in fact often taken as the definition of this ind-scheme, see e.g.~\cite[\S 1.2]{zhu}. The identification with the quotient $G_{\mathcal{K}} / G_{\mathcal{O}}$ is ``purely local'' and does not require the Beauville--Laszlo theorem.
\end{rmk}

\subsection{Moduli interpretation of the convolution diagram}
\label{ss:moduli-convolution}

We now give a similar geometric interpretation of the diagram
\begin{equation}
\label{eqn:convolution-local}
\Gr_{G,x}\times\Gr_{G,x}\xleftarrow{\,p}G_{\mathcal K_x}\times\Gr_{G,x}\xrightarrow
{q\,}G_{\mathcal K_x}\times^{G_{\mathcal O_x}}\Gr_{G,x}\xrightarrow{m\,}\Gr_{G,x},
\end{equation}
which is the ``local version at $x$'' of the
diagram~\eqref{eqn:diagram-convolution}. 
We first remark that
$G_{\mathcal K_x}\times^{G_{\mathcal O_x}}\Gr_{G,x}$ represents the
functor
$$R\mapsto\left\{(\mathcal F_1,\mathcal F,\nu_1,\eta)\left|
\begin{aligned}
\;&\mathcal F_1,\,\mathcal F\;\text{ $G$-bundles on }X_R\\
&\nu_1\text{ trivialization of $\mathcal F_1$ on }X^\times_R\\
&\eta:\mathcal F_1|_{X^\times_R}\xrightarrow\sim\mathcal F|_{X^\times_R}
\text{ isomorphism}\end{aligned}\right\}\right.\Bigr/\text{isom.}$$
To check this, one observes that the datum of
$(\mathcal F_1,\mathcal F,\nu_1,\eta)$ is equivalent to the datum
of $\bigl((\mathcal F_1,\nu_1),(\mathcal F,\eta\circ\nu_1)\bigr)$,
and one notes that this transformation is completely similar to
the isomorphism $\GK\times^\GO\Gr_G\xrightarrow{\sim}\Gr_G\times\Gr_G$
 used in the proof of Lemma~\ref{lem:dimensions-semismall}.

Likewise, $G_{\mathcal K_x}\times\Gr_{G,x}$ represents the functor
$$R\mapsto\left\{(\mathcal F_1,\mathcal F_2,\nu_1,\nu_2,\mu_1)\left|
\begin{aligned}
\;&\mathcal F_1,\,\mathcal F_2\;\text{ $G$-bundles on }X_R\\
&\nu_1,\,\nu_2\text{ trivializations of $\mathcal F_1,\,
\mathcal F_2$ on }X^\times_R\\
&\mu_1\text{ trivialization of $\mathcal F_1$ on }\mathcal D_{x,R}
\end{aligned}\right\}\right.\Bigr/\text{isom.}$$

With these identifications, the maps $m$ and $p$ in diagram~\eqref{eqn:convolution-local}
are given by
\begin{align*}
m(\mathcal F_1,\mathcal F,\nu_1,\eta)&=
(\mathcal F,\eta\circ\nu_1),\\
p(\mathcal F_1,\mathcal F_2,\nu_1,\nu_2,\mu_1)&=
\bigl((\mathcal F_1,\nu_1),(\mathcal F_2,\nu_2)\bigr),
\end{align*}
and the map $q$ associates to $(\mathcal F_1,\mathcal F_2,\nu_1,
\nu_2,\mu_1)$ the quadruple $(\mathcal F_1,\mathcal F,\nu_1,\eta)$,
where $\mathcal F$ is obtained by gluing $\mathcal F_1|_{X^\times_R}$
and $\mathcal F_2|_{\mathcal D_{x,R}}$ along the isomorphism
$$\mathcal F_1|_{\mathcal D_{x,R}^\times}\xleftarrow[\mu_1]\sim
G\times \mathcal D_{x,R}^\times \xrightarrow[\nu_2]\sim
\mathcal F_2|_{\mathcal D_{x,R}^\times}$$
and $\eta$ is the natural isomorphism obtained in the process. (This gluing datum indeed defines a $G$-bundle on $X_R$ thanks to the Beauville--Laszlo theorem, see~\S\ref{ss:moduli-Gr}.)

\subsection{The Be{\u\i}linson--Drinfeld Grassmannian}
\label{ss:BD-Gr}

The idea behind the fusion procedure is to regard the geometric
situation described in~\S\S\ref{ss:moduli-Gr}--\ref{ss:moduli-convolution} as the degeneration of a simpler
situation. This involves the Be{\u\i}linson--Drinfeld Grassmannian.

Specifically, we define $\Gr_{G,X}$ as the ind-scheme over $X$ that
represents the functor
$$R\mapsto\left\{(\mathcal F,\nu,x)\left|
\begin{aligned}
\;&x\in X(R)\\
&\mathcal F\text{ $G$-bundle on }X_R\\
&\nu\text{ trivialization of $\mathcal{F}$ on }X_R\smallsetminus x
\end{aligned}\right\}\right.\Bigr/\text{isom.,}$$
where the symbol $X_R\smallsetminus x$ indicates the complement
in $X_R$ of the graph of $x:\Spec(R)\to X$ (a closed
subscheme of $X_R=X \times \Spec(R)$).

In the same way, we define $\Gr_{G,X^2}$ as the ind-scheme over $X^2$ that
represents the functor
$$R\mapsto\left\{(\mathcal F,\nu,x_1,x_2)\left|
\begin{aligned}
\;&(x_1,x_2)\in X^2(R)\\
&\mathcal F\text{ $G$-bundle on }X_R\\
&\nu\text{ trivialization of $\mathcal{F}$ on }X_R\smallsetminus(x_1\cup x_2)
\end{aligned}\right\}\right.\Bigr/\text{isom.}$$

By definition
there is an obvious morphism $\Gr_{G,X^2}\to X^2$. Plainly, the
restriction of $\Gr_{G,X^2}$ to the diagonal $\Delta_X$ of $X^2$,
namely $\Gr_{G,X^2}\times_{X^2}\Delta_X$, is isomorphic to $\Gr_{G,X}$.
Away from the diagonal, we have an isomorphism
\begin{equation}
\label{eqn:GrBD-away-diag}
\Gr_{G,X^2}\bigl|_{X^2\smallsetminus\Delta_X}\cong
\bigl(\Gr_{G,X}\times\Gr_{G,X}\bigr)\bigl|_{X^2\smallsetminus\Delta_X}
\end{equation}
given by $(\mathcal F,\nu,x_1,x_2)\mapsto\bigl((\mathcal F_1,\nu_1,x_1),
(\mathcal F_2,\nu_2,x_2)\bigr)$, with $\mathcal F_i$ obtained by
gluing the trivial $G$-bundle on $X_R \smallsetminus x_i$ and 
the bundle $\mathcal F|_{X_R \smallsetminus x_j}$ along the map $\nu$ (where $\{i,j\}=\{1,2\}$).
Under the
converse isomorphism
$\bigl((\mathcal F_1,\nu_1,x_1),(\mathcal F_2,\nu_2,x_2)\bigr)
\mapsto(\mathcal F,\nu,x_1,x_2)$, the $G$-bundle $\mathcal F$
is obtained by gluing $\mathcal F_1|_{X_R \smallsetminus x_2}$ and
$\mathcal F_2|_{X_R \smallsetminus x_1}$ along the isomorphism
$$\mathcal F_1|_{X_R\smallsetminus (x_1 \cup x_2)}\xleftarrow[\nu_1]\sim
G\times(X_R \smallsetminus (x_1 \cup x_2))\xrightarrow[\nu_2]\sim
\mathcal F_2|_{X_R\smallsetminus (x_1 \cup x_2)}.$$

\begin{rmk}
\label{rmk:BL-curve}
\begin{enumerate}
\item
Of course one can define more generally Be{\u\i}linson--Drinfeld Grassmannians 
over arbitrary powers of $X$,
which satisfy appropriate analogues of the isomorphism~\eqref{eqn:GrBD-away-diag}. More formally this collection satisfies the ``factorization'' properties spelled out e.g.~in~\cite[Theorem~3.2.1]{zhu}; see also~\cite[\S\S 5.3.10--16]{beilinson-drinfeld} and~\cite[\S 3]{richarz}.
\item
One can also consider Be{\u\i}linson--Drinfeld Grassmannians associated with more general affine smooth group schemes over $X$; see~\cite{zhu-conj} for references and applications.
\end{enumerate}
\end{rmk}

\subsection{Global version of the convolution diagram}
\label{ss:global-conv}

We can also define global analogues of $G_{\mathcal K_x}\times\Gr_{G,x}$
and $G_{\mathcal K_x}\times^{G_{\mathcal O_x}}\Gr_{G,x}$.
For that, we define $\widetilde{\Gr_{G,X}\times\Gr_{G,X}}$ as the ind-scheme
that represents the functor
$$R\mapsto\left\{(\mathcal F_1,\nu_1,\mu_1,\mathcal F_2,\nu_2,x_1,x_2)
\left|
\begin{aligned}
\;&(x_1,x_2)\in X^2(R)\\
&\mathcal F_1,\,\mathcal F_2\;\text{ $G$-bundles on }X_R\\
&\nu_i\text{ trivialization of $\mathcal F_i$ on }X_R\smallsetminus x_i\\
&\mu_1\text{ trivialization of $\mathcal F_1$ on } \mathcal D_{x_2,R}
\end{aligned}\right\}\right.\Bigr/\text{isomorphism.}$$
(Here and below, $\mathcal{D}_{x_2,R}$ means the formal neighborhood of the graph of $x_2$ in $X_R$, considered as a scheme.)
We also define
$\Gr_{G,X}\widetilde\times\Gr_{G,X}$ as the ind-scheme that represents the
functor
$$R\mapsto\left\{(\mathcal F_1,\mathcal F,\nu_1,\eta,x_1,x_2)\left|
\begin{aligned}
\;&(x_1,x_2)\in X^2(R)\\
&\mathcal F_1,\,\mathcal F\;\text{ $G$-bundles on }X_R\\
&\nu_1\text{ trivialization of $\mathcal F_1$ on }X_R\smallsetminus x_1\\
&\eta:\mathcal F_1|_{X_R\smallsetminus x_2}\xrightarrow\sim
\mathcal F|_{X_R\smallsetminus x_2}
\text{ isomorphism}\end{aligned}\right\}\right.\Bigr/\text{isomorphism.}$$
We then get a diagram
\begin{equation}
\label{eqn:convolution-X2}
\Gr_{G,X}\times\Gr_{G,X}\xleftarrow{\,p}\widetilde{\Gr_{G,X}\times\Gr_{G,X}}
\xrightarrow{q\,}\Gr_{G,X}\widetilde\times\Gr_{G,X}\xrightarrow{m\,}\Gr_{G,X^2}
\end{equation}
over $X^2$ by setting
\begin{align*}
m(\mathcal F_1,\mathcal F,\nu_1,\eta,x_1,x_2)&=
(\mathcal F,\eta\circ\nu_1,x_1,x_2),\\
p(\mathcal F_1,\nu_1,\mu_1,\mathcal F_2,\nu_2,x_1,x_2)&=
((\mathcal F_1,\nu_1,x_1),(\mathcal F_2,\nu_2,x_2)),
\end{align*}
and by defining $q$ as the map
$(\mathcal F_1,\nu_1,\mu_1,\mathcal F_2,\nu_2,x_1,x_2)\mapsto
(\mathcal F_1,\mathcal F,\nu_1,\eta,x_1,x_2)$,
where $\mathcal F$ is obtained by gluing $\mathcal F_1|_{X_R\smallsetminus x_2}$
and $\mathcal F_2|_{\mathcal D_{x_2,R}}$ along the isomorphism
$$\mathcal F_1|_{\mathcal D_{x_2,R}^\times}\xleftarrow[\mu_1]{\sim}
G\times\mathcal D_{x_2,R}^\times\xrightarrow[\nu_2]\sim
\mathcal F_2|_{\mathcal D_{x_2,R}^\times}.$$

\begin{rmk}
 To justify the gluing procedure used here, one cannot simply quote the Beauville--Laszlo theorem, since the point $x_2$ might not be constant. The more general result that we need is discussed in~\cite[Remark~2.3.7 and \S 2.12]{beilinson-drinfeld}.
\end{rmk}

We now explain that $p$ and $q$ are principal bundles for a group scheme over $X^2$. For that,
we define $G_{X,\mathcal O}$ as the group scheme over $X$ that
represents the functor
$$R\mapsto\left\{(x,\mu)\left|
\begin{aligned}
\;&x\in X(R)\\
&\mu\text{ trivialization of $G\times X_R$ on }\mathcal D_{x,R}
\end{aligned}\right\}\right..$$
In the description of the functor that $\Gr_{G,X}$ represents, as in~\S\ref{ss:moduli-Gr} one can replace $(\mathcal F,\nu)$ by a pair $(\mathcal{F}',\nu')$ where
$\mathcal F'$ is a $G$-bundle on $\mathcal D_{x,R}$ and $\nu'$ is a trivialization of $\mathcal{F}'$ on
$\mathcal D_{x,R}^\times$; thus $G_{X,\mathcal O}$ acts on $\Gr_{G,X}$ by twisting
the trivialization; specifically, $\nu'$ gets replaced by $\nu'\circ
\mu^{-1}$. (See also~\cite[\S 3.1]{zhu} for more details about these groups schemes---over arbitrary powers of $X$---and their relation with the Be{\u\i}linson--Drinfeld Grassmannians.)

We consider the second projection $X^2 \to X$, and the pullback $G_{X,\mathcal O}\times_X X^2$ of the group scheme $G_{X,\mathcal{O}}$.
The result
acts on $\widetilde{\Gr_{G,X}\times\Gr_{G,X}}$ by twisting $\mu_1$, which
defines $p$ as a bundle.

In the definition of $\widetilde{\Gr_{G,X}\times\Gr_{G,X}}$, as above one can
replace $(\mathcal F_2,\nu_2)$ by a pair $(\mathcal{F}_2',\nu_2')$ where $\mathcal F'_2$ is a $G$-bundle on
$\mathcal D_{x_2,R}$ and $\nu'_2$ is a trivialization of $\mathcal{F}_2'$ on
$\mathcal D_{x_2,R}^\times$. The group scheme $G_{X,\mathcal O}\times_X X^2$ then
acts on $\widetilde{\Gr_{G,X}\times\Gr_{G,X}}$ by simultaneously twisting
both $\mu_1$ and $\nu'_2$. This action defines $q$ as a principal
bundle.

\subsection{Convolution product and fusion}
\label{ss:conv-fusion}

We go back to our convolution problem, starting this time with
diagram~\eqref{eqn:convolution-X2}. Since $p$ and $q$ are principal bundles, we
can define a convolution product $\star_X$ on
$\Per_{G_{X,\mathcal O}}(\Gr_{G,X},\bk)$ by setting
$$\mathscr M \star_X\mathscr N:=m_*(\mathscr M \, \widetilde\boxtimes \, \mathscr N),$$
where again $\mathscr M \, \widetilde\boxtimes \, \mathscr N$ is defined by the condition that
$$q^*(\mathscr M \, \widetilde\boxtimes \, \mathscr N)=
p^*(\mathscr M\boxtimes\mathscr N).$$
Here $\mathscr M$ and $\mathscr N$ are perverse sheaves on $\Gr_{G,X}$,
and the result $\mathscr M \star_X \mathscr N$ is in
$\Db_c(\Gr_{G,X^2},\bk)$.

\begin{rmk}
\begin{enumerate}
\item
To define the category $\Per_{G_{X,\mathcal O}}(\Gr_{G,X},\bk)$ we use a slight variant of the constructions of Appendix~\ref{sec:appendix}, where algebraic groups are replaced by group schemes over $X$. This does not require any new ingredient: one simply replaces products by fiber products over $X$ everywhere.
The same remarks as in~\S\ref{ss:appendix-Gr} are also in order here: we must consider perverse sheaves supported on a closed finite union of $G_{X,\mathcal{O}}$-orbits, and equivariant under some quotient $(G/H_n)_{X,\mathcal{O}}$. (A more sensible definition of a category of perverse sheaves on $\Gr_{G,X}$ is due to Reich; see~\cite[\S 5.4]{zhu}. These more sophisticated considerations will not be needed here.)
\item
It will follows from Lemma~\ref{lem:convolution-fusion-2} below that in fact $\mathscr M \star_X \mathscr N$ is a perverse sheaf. This perverse sheaf is clearly $G_{X,\mathcal O}$-equivariant, so that this operation indeed defines a functor from $\Per_{G_{X,\mathcal O}}(\Gr_{G,X},\bk) \times \Per_{G_{X,\mathcal O}}(\Gr_{G,X},\bk)$ to $\Per_{G_{X,\mathcal O}}(\Gr_{G,X},\bk)$.
\end{enumerate}
\end{rmk}

For the sake of simplicity,\footnote{The general situation can be
dealt with by putting the torsor of change of coordinates into the picture; see e.g.~\cite[\S 2.1.2]{gaitsgory} or~\cite[Discussion surrounding~(3.1.10)]{zhu} for details.} from now on we restrict to the special case $X=\mathbb A^1$. We can then
use a global coordinate on $X$, which yields a local coordinate at any
point $x\in X$, and therefore allows to identify $\Gr_{G,x}$ with the
affine Grassmannian $\Gr_G$ as we originally defined it. This also leads to
an identification $\Gr_{G,X}=\Gr_G\times X$. We let $\tau:\Gr_{G,X}\to\Gr_G$ be
the projection and define $\tau^\circ:=\tau^*[1] \cong \tau^![-1]$; the shift is introduced
so that $\tau^\circ$ takes a perverse sheaf on $\Gr_G$ to a perverse sheaf
on $\Gr_{G,X}$.


We explained in~\S\ref{ss:BD-Gr} that the restriction of $\Gr_{G,X^2}$ to the diagonal
$\Delta_X$ in $X^2$ is isomorphic to $\Gr_{G,X}$; we may then denote by
$i:\Gr_{G,X}=\Gr_{G,X^2}\bigl|_{\Delta_X}\hookrightarrow\Gr_{G,X^2}$
the closed embedding, and consider the functors $i^\circ:=i^*[-1]$ and $i^\bullet := i^![1]$.

\begin{lem}
\label{lem:convolution-fusion-1}
For $\mathscr F_1$ and $\mathscr F_2$ in $\Per_{\GO}(\Gr_G,\bk)$, we have canonical isomorphisms
$$i^\circ \bigl(\tau^\circ(\mathscr F_1) \star_X \tau^\circ(\mathscr F_2)\bigr) \cong
\tau^\circ(\mathscr F_1\star \mathscr F_2) \cong i^{\bullet} \bigl(\tau^\circ(\mathscr F_1) \star_X \tau^\circ(\mathscr F_2)\bigr).$$
\end{lem}

\begin{proof}
Since the map $m$ in~\eqref{eqn:convolution-X2} is proper, and restricts over the diagonal $\Delta_X$ to the product of the map denoted $m$ in~\eqref{eqn:diagram-convolution} by $\id_{\Delta_X}$, using the base change theorem it suffices to provide canonical isomorphisms
\[
(i')^*(\tau^\circ \mathscr{F}_1 \, \widetilde{\boxtimes} \, \tau^\circ \mathscr{F}_2) \cong (\tau')^\circ(\mathscr{F}_1 \, \widetilde{\boxtimes} \, \mathscr{F}_2)[1], \quad
(i')^!(\tau^\circ \mathscr{F}_1 \, \widetilde{\boxtimes} \, \tau^\circ \mathscr{F}_2) \cong (\tau')^\circ(\mathscr{F}_1 \, \widetilde{\boxtimes} \, \mathscr{F}_2)[-1]
\]
where $i' : (G_{\mathcal{K}} \times^{G_{\mathcal{O}}} \Gr_G) \times \Delta_X \to \Gr_{G,X} \widetilde{\times} \Gr_{G,X}$ is the embedding, $\tau' : (G_{\mathcal{K}} \times^{G_{\mathcal{O}}} \Gr_G) \times \Delta_X \to G_{\mathcal{K}} \times^{G_{\mathcal{O}}} \Gr_G$ is the projection, and $(\tau')^\circ = (\tau')^*[1] \cong (\tau')^![-1]$. The first isomorphism is immediate from the definitions. The proof of the second one is similar, using Remark~\ref{rmk:def-conv-!*}.
\end{proof}

\begin{rmk}
 The isomorphism $i^\circ \bigl(\tau^\circ(\mathscr F_1) \star_X \tau^\circ(\mathscr F_2)\bigr) \cong i^\bullet \bigl(\tau^\circ(\mathscr F_1) \star_X \tau^\circ(\mathscr F_2)\bigr)$ observed in Lemma~\ref{lem:convolution-fusion-1} can also be deduced from more general considerations related to universal local acyclicity; see~\cite[Theorem~A.2.6 and proof of Proposition~5.4.2]{zhu}.
\end{rmk}


We now analyze the convolution diagram over $U=X^2\smallsetminus\Delta_X$:
\begin{equation}
\label{eqn:conv-diagram-away-diag}
\raisebox{24pt}{
{\small
\xymatrix@C=0.4cm{(\Gr_{G,X}\times\Gr_{G,X})|_U&\bigl(\widetilde{\Gr_{G,X}\times\Gr_{G,X}}
\bigl)\bigl|_U\ar[l]_p\ar[r]^q&(\Gr_{G,X}\widetilde\times\Gr_{G,X})|_U
\ar[r]^(.58)\sim_(.58)m\ar[dr]_{\pi\circ m}&\Gr_{G,X^2}|_U
\ar[d]^\wr_\pi\\&&&(\Gr_{G,X}\times\Gr_{G,X})|_U.}}
}
\end{equation}
Here $\pi$ is the isomorphism of~\eqref{eqn:GrBD-away-diag}, defined by 
\[
(\mathcal F,\nu,x_1,x_2) \mapsto
\bigl( (\mathcal F_1,\nu_1,x_1),(\mathcal F_2,\nu_2,x_2) \bigr),
\]
where $\mathcal F_i$ is obtained by gluing the trivial bundle on
$X_R \smallsetminus x_i$ and the bundle $\mathcal F$ on $\mathcal D_{x_i,R}$
using~$\nu$. We note that there exists an isomorphism
$$\bigl(\widetilde{\Gr_{G,X}\times\Gr_{G,X}}\bigl)\bigl|_U\xrightarrow\sim
\Bigl( (\Gr_{G,X}\times\Gr_{G,X}) \times_{X^2} (X\times G_{X,\mathcal O})
\Bigr) \bigl|_U$$
defined by 
\[
(\mathcal F_1,\nu_1,\mu_1,\mathcal F_2,\nu_2,x_1,x_2) \mapsto
\Bigl( ((\mathcal F_1,\nu_1,x_1),(\mathcal F_2,\nu_2,x_2)),
(x_1,(x_2,\mu_1^{-1}\circ\nu_1|_{\mathcal{D}_{x_2,R}}))\Bigr).
\]
Under this identification, the maps $p$ and
$\pi\circ m\circ q$ identify with
$$\xymatrix{(\Gr_{G,X}\times\Gr_{G,X})|_U&
\bigl[(\Gr_{G,X}\times\Gr_{G,X})\underset{X^2}\times(X\times G_{X,\mathcal O})
\bigr]\bigl|_U\ar[l]_-{p_1}\ar[r]^-{a}&(\Gr_{G,X}\times\Gr_{G,X})|_U,}$$
where $p_1$ is the projection on the first factor and $a$ is the
action of $G_{X,\mathcal O}$ on the second copy of $\Gr_{G,X}$.

It follows that if we identify the three spaces on the right-hand side of the
convolution diagram~\eqref{eqn:conv-diagram-away-diag}, then for any
$\mathscr M_1$, $\mathscr M_2$ in $\mathrm{Perv}_{G_{X,\mathcal O}}
(\Gr_{G,X}, \bk)$, the equivariant structure of $\mathscr M_2$
leads to canonical identifications
\begin{equation}
\label{eqn:identification-fusion}
\vcenter{
\xymatrix{(\mathscr M_1 \, \widetilde\boxtimes \, \mathscr M_2)|_U
\ar@{=}[r]\ar@{=}[dr]&(\mathscr M_1 \star_X \mathscr M_2)|_U\ar@{=}[d]\\
&(\mathscr M_1\boxtimes\mathscr M_2)|_U.}
}
\end{equation}


Consider now the open embedding $j:(\Gr_{G,X}\times\Gr_{G,X})|_U \overset{\eqref{eqn:GrBD-away-diag}}{\cong}
\Gr_{G,X^2}|_U\hookrightarrow\Gr_{G,X^2}$. 

\begin{lem}
\label{lem:convolution-fusion-2}
For any
$\mathscr F_1,\mathscr F_2\in\Per_{\GO}(\Gr_G,\bk)$, we have
$$j_{!*}\bigl((\tau^\circ\mathscr F_1\boxtimes\tau^\circ\mathscr F_2)|_U\bigr)
\cong(\tau^\circ\mathscr F_1)\star_X(\tau^\circ\mathscr F_2).$$
\end{lem}

\begin{proof}
We will use the characterization of 
the left-hand side given by~\cite[Corollaire~1.4.24]{bbd}.
In fact, in~\eqref{eqn:identification-fusion} we have already obtained the desired description of $(\tau^\circ\mathscr F_1)\star_X(\tau^\circ\mathscr F_2)$ over $U$. Hence to conclude
it suffices to prove that
\begin{equation}
\label{eqn:conditions-IC-fusion}
i^*\bigl((\tau^\circ\mathscr F_1)\star_X(\tau^\circ\mathscr F_2)\bigr)\in{}^p \hspace{-1pt}
D^{\leq-1}\quad\text{and}\quad i^!\bigl((\tau^\circ\mathscr F_1) \star_X
(\tau^\circ\mathscr F_2)\bigr)\in{}^p \hspace{-1pt} D^{\geq1}.
\end{equation}
However, it follows from Lemma~\ref{lem:convolution-fusion-1} that
\[
i^*\bigl((\tau^\circ\mathscr F_1)\star_X(\tau^\circ\mathscr F_2)\bigr) \cong \tau^\circ (\mathscr{F}_1 \star \mathscr{F}_2)[1].
\]
By Proposition~\ref{prop:exactness-convolution} the right-hand side is concentrated in perverse degree $-1$, proving the first condition in~\eqref{eqn:conditions-IC-fusion}.
The second condition can be checked similarly, using the second isomorphism in Lemma~\ref{lem:convolution-fusion-1}.
\end{proof}

\begin{rmk}
 Once again, Lemma~\ref{lem:convolution-fusion-2} can also be deduced from more general considerations related to universal local acyclicity; see~\cite[Theorem~A.2.6 and proof of Proposition~5.4.2]{zhu}.
\end{rmk}

\subsection{Construction of the commutativity constraint}
\label{ss:construction-commutativity}

Combining Lemma~\ref{lem:convolution-fusion-1} and Lemma~\ref{lem:convolution-fusion-2}, we obtain a canonical isomorphism
\begin{equation}
\label{eqn:convolution-fusion}
\tau^\circ(\mathscr F_1 \star \mathscr F_2) \cong i^\circ j_{!*}\,\bigl((\tau^\circ
\mathscr F_1\boxtimes\tau^\circ\mathscr F_2)|_U\bigr),
\end{equation}
valid for any $\mathscr F_1,\mathscr F_2\in\Per_{\GO}(\Gr_G,\bk)$. In other words,
the convolution product $\mathscr F_1\star \mathscr F_2$ can also be obtained
by a procedure based on the Be{\u\i}linson--Drinfeld Grassmannians $\Gr_{G,X}$ and $\Gr_{G,X^2}$, called the
\emph{fusion product}.

Let $\mathrm{swap}:\Gr_{G,X^2}\to\Gr_{G,X^2}$ be the
automorphism that swaps $x_1$ and $x_2$.
Then we have $(\mathrm{swap}\circ i)=i$. Moreover, $\mathrm{swap}$ stabilizes $\Gr_{G,X^2}|_U$, and under the identification~\eqref{eqn:GrBD-away-diag} the induced automorphism of $(\Gr_{G,X} \times \Gr_{G,X})|_U$ (which we will denote $\mathrm{swap}_U$) swaps the two factors $\Gr_{G,X}$.
Therefore we obtain canonical isomorphisms
\begin{align*}
\tau^\circ(\mathscr F_1 \star \mathscr F_2)
&\cong i^\circ j_{!*}\bigl((\tau^\circ \mathscr F_1\boxtimes\tau^\circ \mathscr F_2)
|_U\bigr)\\[2pt]
&\cong i^\circ \mathrm{swap}^* j_{!*}\bigl((\tau^\circ\mathscr F_1\boxtimes
\tau^\circ\mathscr F_2)|_U\bigr)\\[2pt]
&\cong i^\circ j_{!*} (\mathrm{swap}_U)^* \bigl((\tau^\circ\mathscr F_1\boxtimes
\tau^\circ\mathscr F_2)|_U\bigr)\\[2pt]
&\cong i^\circ j_{!*}\bigl((\tau^\circ \mathscr F_2\boxtimes
\tau^\circ \mathscr F_1)|_U\bigr)\\[2pt]
&=\tau^\circ (\mathscr F_2 \star \mathscr F_1).
\end{align*}
Restricting to a point of $X$, we deduce a canonical isomorphism
\[
\mathscr F_1\star \mathscr F_2\xrightarrow\sim\mathscr F_2 \star \mathscr F_1,
\]
which provides a commutativity constraint for the category $\Per_{\GO}(\Gr_G,\bk)$.

\begin{rmk}\phantomsection
\label{rmk:fusion}
\begin{enumerate}
\item
Later we will modify this commutativity constraint by a sign to make sure that the functor $\F$ sends it to the standard commutativity constraint on vector spaces; see~\S\ref{ss:change-commutativity}.
\item
One may note here that the twisted product
$\Gr_{G,X} \widetilde\times\Gr_{G,X}$, while playing a key role in the proof, is
not involved in the definition of the fusion product, since the maps
$i$ and $j$ only deal with the Be{\u\i}linson--Drinfeld Grassmannians $\Gr_{G,X}$ and
$\Gr_{G,X^2}$. The two points $x_1$ and $x_2$, which are not
interchangeable in the definition of $\Gr_{G,X} \widetilde\times \Gr_{G,X}$,
play the same role in $\Gr_{G,X^2}$. This property is the basis for the
construction of the commutativity constraint.
\item
\label{it:fusion-associativity}
One can describe the associativity constraint considered in~\S\ref{ss:associativity} also in terms of the fusion procedure, using the Be{\u\i}linson--Drinfeld Grassmannian $\Gr_{G,X^3}$ over $X^3$.
\end{enumerate}
\end{rmk}

\section{Further study of the fiber functor}

\subsection{Compatibility of $\F$ with the convolution product}
\label{ss:F-convolution}

In Sections~\ref{sec:convolution}--\ref{sec:convolution-BD}
we have endowed our category $\Per_\GO(\Gr_G,\mathbf k)$ with a convolution
product $\star$, defined either in the ``easy'' way with the convolution diagram~\eqref{eqn:diagram-convolution}
or with the fusion procedure. The latter even allows to define a
commutativity constraint. We now want to show that the functor
$$\F=\coH^\bullet(\Gr_G,\bm ?):\Per_\GO(\Gr_G,\mathbf k)\to\Vect_{\mathbf k}$$
is a fiber functor in the sense of Remark~\ref{rmk:Tannaka}\eqref{it:fiber}; in other words that this is an exact and faithful functor
that maps the convolution product of sheaves to the tensor product of
vector spaces while respecting the associativity, the unit, and the
commutativity constraints of these categories.

The exactness and the faithfulness of $\F$ have already been
proved in Theorem~\ref{thm:fiber-functor}\eqref{it:fiber-functor-2}. 
The goal of this subsection is to prove the following.

\begin{prop}
\label{prop:F-product}
For any $\mathscr{A}_1$, $\mathscr{A}_2$ in $\Per_{\GO}(\Gr_G,\bk)$, there exists a canonical identification
\[
\F(\mathscr A_1 \star \mathscr A_2)=\F(\mathscr A_1)\otimes_{\bk}
\F(\mathscr A_2).
\]
\end{prop}

\begin{proof}
The proof will use the
fusion procedure. 
Recall the setup of Section~\ref{sec:convolution-BD} (in the special case $X=\mathbb{A}^1$), and in particular diagram~\eqref{eqn:convolution-X2}.
Let $\mathscr A_1,\mathscr A_2$ in $\Per_\GO(\Gr_G,\mathbf k)$, and set $\mathscr{B}:=(\tau^\circ\mathscr A_1) \star_X(\tau^\circ\mathscr A_2)$. Then if $f : \Gr_{G,X^2} \to X^2$ is the natural map,
Lemma~\ref{lem:convolution-fusion-1} and~\eqref{eqn:identification-fusion}
%
%
translate to the following properties: for each $k\in\mathbf Z$,\\[-25pt]
\begin{itemize}
\item
the $k$-th cohomology sheaf of the complex $(f_*\mathscr B)|_{\Delta_X}[-2]$ is
locally constant on $\Delta_X$, with stalk
$\coH^k(\Gr_G,\mathscr A_1 \star \mathscr A_2)$;
\item
the $k$-th cohomology sheaf of the complex $(f_*\mathscr B)|_U[-2]$ is
locally constant on $U$, with stalk
$\coH^k(\Gr_G\times\Gr_G,\mathscr A_1\boxtimes\mathscr A_2)$, which identifies with $\bigoplus_{i+j=k} \coH^i(\Gr_G,\mathscr A_1)\otimes \coH^j(\Gr_G,\mathscr A_2)$ by the K\"unneth formula.
\end{itemize}

From there, we will be able to deduce the desired identification
$$\coH^k(\Gr_G,\mathscr A_1 \star \mathscr A_2) \cong \bigoplus_{i+j=k}
\coH^i(\Gr_G,\mathscr A_1)\otimes \coH^j(\Gr_G,\mathscr A_2)$$
as soon as we know that $\mathscr H^{k-2}(f_*\mathscr B)$ is locally
constant on the whole space $X^2$. (Indeed, then this sheaf will be constant, so that we will be able to identify any of its fibers with its global sections canonically.)

We now prove this fact. Set $\widetilde{\mathscr{B}} := (\tau^\circ\mathscr A_1) \, \widetilde{\boxtimes} \, (\tau^\circ\mathscr A_2)$, so that $\mathscr{B}=m_* \widetilde{\mathscr{B}}$. If we
set $\tilde{f}=f\circ m$, we have
$f_*\mathscr B=\tilde{f}_*\widetilde{\mathscr B}$.
For $\lambda,\mu\in X_*(T)^+$, set
\[
\Gr_{G,X}^\lambda=\tau^{-1}(\Gr_G^\lambda),\quad
\Gr_{G,X}^\mu=\tau^{-1}(\Gr_G^\mu),
\]
and define
$\Gr_{G,X}^\lambda \, \widetilde{\times}\,\Gr_{G,X}^\mu \subset \Gr_{G,X} \, \widetilde{\times}\,\Gr_{G,X}$
by the requirement
\[
q^{-1} \bigl( \Gr_{G,X}^\lambda \, \widetilde{\times}\,\Gr_{G,X}^\mu \bigr) =
p^{-1} \bigl( \Gr_{G,X}^\lambda
\times\Gr_{G,X}^\mu\bigr).
\]
(This definition makes sense, since $\Gr_{G,X}^\mu$ is stable under the left action of
$G_{X,\mathcal O}$). Then
$$\widetilde{\mathscr S}=\{\Gr_{G,X}^\lambda\widetilde\times\,\Gr_{G,X}^\mu\mid\lambda, \mu \in X_*(T)^+\}$$
is a stratification of $\Gr_{G,X} \, \widetilde{\times}\,\Gr_{G,X}$, and $\widetilde{\mathscr{B}}$ is
$\widetilde{\mathscr S}$-constructible. To show that the cohomology sheaves
of $\tilde f_*\widetilde{\mathscr B}$ are locally constant, it suffices by
d\'evissage\footnote{More precisely, one uses the following claim: the complexes $\mathscr M$
such that the cohomology sheaves $\mathscr H^k\tilde{f}_*\mathscr M$ are local
systems form a full \emph{triangulated} subcategory of
$\Db_c(\Gr_{G,X}\widetilde\times\,\Gr_{G,X}, \bk)$. To prove this claim, consider a
distinguished triangle $\mathscr M'\to\mathscr M\to\mathscr
M''\xrightarrow{[1]}$ with $\mathscr M'$ and $\mathscr M''$ in the
subcategory. The long exact sequence in cohomology expresses
$\mathscr H^k\tilde f_*\mathscr M$ as an extension
of $\ker(\mathscr H^k\tilde f_*\mathscr M''\to\mathscr H^{k+1}
\tilde f_*\mathscr M')$ by $\coker(\mathscr H^{k-1}\tilde f_*\mathscr
M''\to\mathscr H^k\tilde f_*\mathscr M')$, hence as an extension
of two local systems. Therefore $\mathscr H^k\tilde f_*\mathscr M$
is a local system for each $k$, which means that $\mathscr M$ belongs
to our subcategory.} to check that for each $k\in\mathbf Z$ and each
stratum $S\in\widetilde{\mathscr S}$, the sheaf
$\mathscr H^k\tilde f_*\underline{\mathbf k}_S$ is locally constant.

Let $\widehat\Gr_{G,X^2}$ be the ind-scheme representing the functor
$$R\mapsto\left\{(\mathcal F_1,\nu_1,\mu_1,x_1,x_2)\left|\;\;\;
\begin{aligned}
\;&(x_1,x_2)\in X^2(R)\\
&\mathcal F_1\;\text{ $G$-bundle on }X_R\\
&\nu_1\text{ trivialization of $\mathcal F_1$ on }X_R\smallsetminus x_1\\
&\mu_1\text{ trivialization of $\mathcal{F}_1$ on }\mathcal D_{x_2,R}
\end{aligned}\right\}\right.\Bigr/\text{isomorphism}.$$
There is a
natural map $q':\widehat\Gr_{G,X^2}\to\Gr_{G,X}\times X$, that simply forgets
$\mu_1$. The group scheme $X \times G_{X,\mathcal O}$ acts on $\widehat\Gr_{G,X_2}$
by twisting $\mu_1$, and $q'$ is a bundle for this action. (To justify this, we need to check that a trivialization $\mu_1$ exists for any $(\mathcal{F}_1, \nu_1, x_1)$ in $\Gr_{G,X}(R)$, possibly after base change associated with a faithfully flat extension $R \to R'$. This fact is clear if $x_2 \neq x_1$, and follows from the results recalled in~\S\ref{ss:moduli-Gr} when $x_1=x_2$.)

On the other
hand, $\Gr_{G,X}$ classifies data $(\mathcal F_2,\nu_2,x_2)$ with
$\mathcal F_2$ a $G$-bundle on $\mathcal D_{x_2}$ and $\nu_2$ a
trivialization on $\mathcal D_{x_2}^\times$ (see~\S\ref{ss:moduli-Gr}), so the
group scheme $G_{X,\mathcal O}$ acts on $\Gr_{G,X}$ by twisting $\nu_2$.
We claim that we have an identification
\[
\Gr_{G,X} \, \widetilde{\times} \,\Gr_{G,X}
=
\widehat \Gr_{G,X^2}
\times^{(X \times G_{X,\mathcal O})}
(X \times \Gr_{G,X})
\]
such that the map induced by $q'$ identifies with the map $\Gr_{G,X} \, \widetilde{\times} \,\Gr_{G,X} \to \Gr_{G,X}\times X$ induced by the projection $\Gr_{G,X} \to X$ in the second factor.
In fact, this identification sends an element $(\mathcal{F}_1, \mathcal{F}, \nu_1, \eta, x_1, x_2)$ in $ \bigl( \Gr_{G,X} \, \widetilde{\times} \,\Gr_{G,X} \bigr)(R)$ to the class of the pair
\[
\bigl( (\mathcal{F}_1, \nu_1, \mu_1,x_1,x_2), (\mathcal{F}_2, \nu_2,x_2) \bigr)
\]
where $\mu_1$ is a choice of trivialization of $\mathcal{F}_1$ on $\mathcal{D}_{x_2,R}$, $\mathcal{F}_2$ is the restriction of $\mathcal{F}$ to $\mathcal{D}_{x_2,R}$, and $\nu_1$ is the composition of the isomorphism $\mathcal{F}_1 |_{\mathcal{D}_{x_2,R}^\times} \xrightarrow{\sim} \mathcal{F} |_{\mathcal{D}_{x_2,R}^\times}$ induced by $\eta$ with the trivialization of $\mathcal{F}_1 |_{\mathcal{D}_{x_2,R}^\times}$ induced by $\mu_1$. The inverse map sends the class of $\bigl( (\mathcal{F}_1, \nu_1, \mu_1,x_1,x_2), (\mathcal{F}_2, \nu_2,x_2) \bigr)$ to $(\mathcal{F}_1, \mathcal{F}, \nu_1, \eta, x_1, x_2)$, where $\mathcal{F}$ is the $G$-bundle obtained by gluing $\mathcal{F}_1|_{X \smallsetminus x_2}$ and $\mathcal{F}_2$ using the gluing datum provided by the trivializations $\mu_1$ and $\nu_2$.

These considerations show that the morphism $\Gr_{G,X} \, \widetilde{\times} \,\Gr_{G,X} \to \Gr_{G,X} \times X$ is a locally trivial fibration. Now,
take $\lambda,\mu\in X_*(T)^+$, and set
$S=\Gr_{G,X}^\lambda \,\widetilde{\times}\,\Gr_{G,X}^\mu$. The base change
corresponding to the inclusion $\Gr_{G,X}^\lambda\hookrightarrow\Gr_{G,X}$
and the fiber change corresponding to the inclusion
$\Gr_{G,X}^\mu\hookrightarrow\Gr_{G,X}$ show that the natural map
\[
S = \Gr_{G,X}^\lambda \, \widetilde{\times}\,\Gr_{G,X}^\mu
\to \Gr_{G,X}^\lambda\times X
\]
is a locally trivial fibration with fiber $\Gr_G^\mu$. It follows that
the cohomology sheaves of the pushforward of
$\underline{\mathbf k}_S$ along this map
are locally constant on $\Gr_{G,X}^\lambda\times X$. Further, the projection
$\Gr_{G,X}^\lambda\to X$ is also a locally trivial fibration, and by a last
d\'evissage argument, we conclude that $\tilde{f}_*\underline{\mathbf k}_S$
has locally constant cohomology sheaves, as desired.
\end{proof}

\begin{rmk}
 See~\cite[\S 5.2]{zhu} for a sketch of a different proof of Proposition~\ref{prop:F-product}, based on the use of equivariant cohomology. (This proof does not extend to more general coefficients, since~\cite[Theorem~A.1.10]{zhu} has no analogue for general coefficients.)
\end{rmk}

\subsection{Compatibility with the commutativity constraint}
\label{ss:change-commutativity}

It should be clear (see in particular Remark~\ref{rmk:fusion}\eqref{it:fusion-associativity}) that the identification provided by Proposition~\ref{prop:F-product} sends the associativity constraint on $\Per_{\GO}(\Gr_G,\bk)$ to the natural associativity constraint on $\Vect_\bk$. The situation is slightly more subtle for the commutativity constraint.

By construction,
the fiber functor $\F$ factors through a functor from $\Per_\GO(\Gr_G,\mathbf k)$
to the category $\Vect_{\mathbf k}(\mathbf Z)$ of $\mathbf Z$-graded
$\mathbf k$-vector spaces. We can endow the latter with the usual
structure of tensor category or with the supersymmetric structure;
the difference between the two structures is the definition of the
commutativity constraint, which involves a sign in the super case (see in particular Example~\ref{ex:tannakian-reconstruction}\eqref{it:SVec}).

Recall the notion of even and odd components of $\Gr_G$ from~\S\ref{ss:def-Gr}. It follows in particular from Theorem~\ref{thm:fiber-functor}\eqref{it:fiber-functor-1} that if $\mathscr{A}$ is supported on an even (resp.~odd) component of $\Gr_G$ then $\coH^\bullet(\Gr_G, \mathscr{A})$ is concentrated in even (resp.~odd) degrees. Looking closely at the constructions in~\S\ref{ss:F-convolution}, one can check that the functor $\F$
maps the commutativity constraint
on $\Per_\GO(\Gr_G,\mathbf k)$ defined in~\S\ref{ss:construction-commutativity} to the
supersymmetric commutativity constraint on $\Vect_{\mathbf k}(\mathbf Z)$. (This is related to the fact that the canonical isomorphism $(\mathrm{swap}_U)^* \bigl((\tau^\circ\mathscr F_1\boxtimes
\tau^\circ\mathscr F_2)|_U\bigr) \cong (\tau^\circ\mathscr F_2\boxtimes
\tau^\circ\mathscr F_1)|_U$ involves some signs, since it requires to swap the order in a tensor product of complexes.)

However, to be in a position to apply the Tannakian reconstruction
theorem from Section~\ref{sec:Tannakian}, we need to make sure that $\F$ maps the commutativity constraint
on $\Per_\GO(\Gr_G,\mathbf k)$ to the usual (unsigned) commutativity
constraint on $\Vect_{\mathbf k}(\mathbf Z)$. One solution consists in
altering the commutativity constraint on $\Per_\GO(\Gr_G,\mathbf k)$
by an appropriate sign. 
In fact, due to the change of parity introduced in the functor $\tau^\circ$, one must multiply the isomorphism of~\S\ref{ss:construction-commutativity} by $-1$ for the summands of the perverse sheaves $\mathscr{F}_1$ and $\mathscr{F}_2$ supported on even components of $\Gr_G$. This is the commutativity constraint that we will consider below.

\subsection{Compatibility with the weight functors}
\label{ss:compatibility-weight-functors}

We have noticed in~\S\ref{ss:change-commutativity} that the fiber functor
$\F:\Per_\GO(\Gr_G,\mathbf k)\to\Vect_{\mathbf k}$ in fact factors through
the category $\Vect_{\mathbf k}(\mathbf Z)$ of $\mathbf Z$-graded
$\mathbf k$-vector spaces. We can enhance this result using
the weight functors of~\S\ref{ss:weight-functors}. In fact by Theorem~\ref{thm:fiber-functor}\eqref{it:fiber-functor-1} we have a commutative diagram
$$\xymatrix@R=0.6em@C=7em{&\Vect_{\mathbf k}(X_*(T))
\ar[dd]^{\mathrm{forget}}\\
\Per_\GO(\Gr_G,\mathbf k)\ar[ur]|{\;\bigoplus_\mu \F_\mu\,}
\ar[dr]|{\;\F\;}\\&\Vect_{\mathbf k}}$$
where $\Vect_{\mathbf k}(X_*(T))$ is the category of $X_*(T)$-graded
$\mathbf k$-vector spaces. Recall from Example~\ref{ex:tannakian-reconstruction}\eqref{it:graded-Vect} that the category
$\Vect_{\mathbf k}(X_*(T))$ admits a natural tensor product, with commutativity
and associativity constraints.

\begin{prop}
\label{prop:F-mu-tensor-char-0}
The functor $\bigoplus_\mu \F_\mu$ sends the convolution product $\star$ to the tensor product of $X_*(T)$-graded $\bk$-vector spaces, in a way compatible with the associativity and commutativity constraints.
\end{prop}

\begin{proof}
We need to provide an identification
\begin{equation}
\F_\mu(\mathscr A_1 \star \mathscr A_2)=\bigoplus_{\mu_1+\mu_2=\mu}
\F_{\mu_1}(\mathscr A_1)\otimes_\bk \F_{\mu_2}(\mathscr A_2)
\label{eqn:weight-functor-tensor}
\end{equation}
for each $\mu\in X_*(T)$ and all $\mathscr A_1,\mathscr A_2\in
\Per_\GO(\Gr_G,\mathbf k)$.

Recall how the weight functors $\F_\mu$ are defined (see Remark
\ref{rmk:equiv-def-wf}). We have chosen a maximal
torus and a Borel subgroup $T\subset B\subset G$. Then $T
\subset\GK$ acts on $\Gr_G=\GK/\GO$ with fixed points
$$(\Gr_G)^T=\bigl\{L_\mu : \mu\in X_*(T)\bigr\}.$$
We picked a dominant regular cocharacter $\eta\in X_*(T)$, which provides
a one-parameter
subgroup $\mathbb G_{\mathbf{m}}\subset T$ and a $\mathbf C^\times$-action on
$\Gr_G$ with fixed points $(\Gr_G)^T$. For $\mu\in X_*(T)$, the
attractive variety relative to the fixed point $L_\mu$ is
\[
S_\mu=\bigl\{x\in\Gr_G\bigm| \eta(a) \cdot x \to L_\mu\text{ when }a\to0\bigr\}
\]
(see~the proof of Theorem~\ref{thm:orbits}), and for each
$\mathscr A\in\Per_\GO(\Gr_G,\mathbf k)$,
$$\coH^k_c({S_\mu},\mathscr A)\neq0\ \Longrightarrow\
k=\langle2\rho,\mu\rangle.$$

For $\mu\in X_*(T)$ and $\mathscr A\in\Per_\GO(\Gr_G,\mathbf k)$, we
have $\F_\mu(\mathscr A):=\coH^{\langle2\rho,\mu\rangle}_c(\overline S_\mu,\mathscr A)$.
We get adjunction maps (see
Remark~\ref{rmk:equiv-def-wf})
$$\xymatrix{
&\coH^{\langle2\rho,\mu\rangle}_c(S_\mu,\mathscr A)\ar[d]^\wr\\
\coH^{\langle2\rho,\mu\rangle}(\Gr_G,\mathscr A)\ar[r]&
\coH^{\langle2\rho,\mu\rangle}_c(\overline S_\mu,\mathscr A);}$$ 
moreover for each $k\in\mathbf Z$, there is a decomposition
$$\coH^k(\Gr_G,\mathscr A)=\bigoplus_{\substack{\mu\in X_*(T)\\[2pt]\langle2\rho,\mu\rangle=k}}\F_\mu(\mathscr A).$$

We need to insert this construction in the reasoning in \S8.1.
The various spaces considered in~\S\ref{ss:global-conv} carry an
action of $T$. Specifically, this action
twists $\nu$ in $\Gr_{G,X}$ and $\Gr_{G,X^2}$, and twists $\nu_1$ on
$\widetilde{\Gr_{G,X}\times\Gr_{G,X}}$ and
$\Gr_{G,X}\widetilde{\times}\,\Gr_{G,X}$. The maps $q$ and $m$
in diagram~\eqref{eqn:convolution-X2} and the isomorphism $\pi$
in diagram~\eqref{eqn:conv-diagram-away-diag} are $T$-equivariant.

To each pair $(\mu_1,\mu_2)\in X_*(T)^2$ corresponds a connected
component $\widetilde C_{(\mu_1,\mu_2)}$ of the set of $T$-fixed points in
$\Gr_{G,X} \, \widetilde{\times}\,\Gr_{G,X}$, namely
$$\widetilde C_{(\mu_1,\mu_2)}=\bigl\{([t^{\mu_1},L_{\mu_2}],x,x)\bigm|
x\in X\bigr\}\cup\bigl\{(L_{\mu_1},L_{\mu_2},x_1,x_2)
\bigm|(x_1,x_2)\in U\bigr\}$$
where $[t^{\mu_1},L_{\mu_2}]$ is seen as a point in $G_{\mathcal{K}_x} \times^{G_{\mathcal{O}_x}} \Gr_{G,x}$, identified with the fiber of the twisted product
$\Gr_{G,X} \, \widetilde{\times}\,\Gr_{G,X}$ over a point $(x,x) \in \Delta_X$, and
$(L_{\mu_1},L_{\mu_2})$ is likewise seen as a point in
$\Gr_{G,x_1}\times\Gr_{G,x_2}$, identified with the fiber of $\Gr_{G,X} \, \widetilde{\times}\,\Gr_{G,X}$ over $(x_1,x_2) \in U$ thanks to the
map $\pi\circ m\circ q$ in~\eqref{eqn:conv-diagram-away-diag}.
Moreover the projection $\widetilde C_{\mu_1,\mu_2}\to X^2$ is an isomorphism. (Recall that we have chosen $X=\mathbb{A}^1$, so that we have a canonical identification $\Gr_{G,x} \cong \Gr_G$ for any $x$.)

The map $m:\Gr_{G,X} \, \widetilde{\times}\,\Gr_{G,X}\to
\Gr_{G,X^2}$ glues together along the diagonal $\Delta_X$ the various
connected components $\widetilde C_{(\mu_1,\mu_2)}$ for which $\mu_1+\mu_2$
is the same. Therefore, to each $\mu\in X_*(T)$ corresponds a connected
component
$$C_\mu:=\bigsqcup_{\mu_1+\mu_2=\mu}m\bigl(
\widetilde C_{(\mu_1,\mu_2)}\bigr)$$
of the set of $T$-fixed points in $\Gr_{G,X^2}$.

Our dominant regular cocharacter $\eta\in X_*(T)$ defines a
$\mathbf C^\times$-action on $\Gr_{G,X^2}$.
Denote the attractive variety\footnote{See~\cite[Definition~1.4.2 and Corollary~1.5.3]{drg} for the general construction of the attractor for a $\C^{\times}$-action on a scheme.} around $C_\mu$ by $S_\mu(X^2)$; one can check\footnote{This fact is not automatic (i.e.~it does not follow from the general result~\cite[Theorem~1.6.8]{drg}) because the finite-dimensional pieces of $\Gr_{G,X^2}$ might not be normal. One way to prove this is to first check that $\bigsqcup_{\nu \leq \mu} S_\nu(X^2)$ is closed in $\Gr_{G,X^2}$; then $S_\mu(X^2)$ is the complement of $\bigsqcup_{\nu < \mu} S_\nu(X^2)$ in $\bigsqcup_{\nu \leq \mu} S_\nu(X^2)$.} that $S_\mu(X^2)$ is a locally closed subscheme of $\Gr_{G,X^2}$.
Over a point $(x,x)\in\Delta_X$, the fiber of the map $S_\mu(X^2)\to X^2$
is the semi-infinite orbit $S_\mu$, viewed as a subvariety
of $\Gr_{G,x}$ thanks to the isomorphism $\Gr_{G,X^2}\bigl|_{\Delta_X}=
\Gr_{G,X}$; over a point $(x_1,x_2)\in U$, the fiber of
$S_\mu(X^2)\to X^2$ is the union
$$\bigsqcup_{\mu_1+\mu_2=\mu}S_{\mu_1}\times
S_{\mu_2}\subset\Gr_{G,x_1}\times \Gr_{G,x_2},$$
where we use the isomorphism
$\pi:\Gr_{G,X^2}\bigl|_U\xrightarrow{\sim}\bigl(\Gr_{G,X}\times\Gr_{G,X}\bigr)\bigl|_U$
of~\S\ref{ss:conv-fusion} (see~\cite[Lem\-ma~1.4.9]{drg}).

Consider again $\mathscr B:=(\tau^\circ\mathscr A_1)
\star_X(\tau^\circ\mathscr A_2)$ and
consider the natural maps depicted in the following diagram:
$$\xymatrix{S_\mu(X^2)\ar@/_1.4pc/[rr]_{\tilde{s}_\mu}\ar[r]^{\tilde{s}'_\mu}&
\overline{S_\mu(X^2)}\ar[r]^{\tilde{s}''_\mu}&\Gr_{G,X^2}\ar[d]^f\\
&&X^2.}$$
The stalks of the complex of sheaves $(f\tilde s_\mu)_!(\tilde s_\mu)^*\mathscr B$
can be computed by base change. Using Lemma~\ref{lem:convolution-fusion-1}
and \eqref{eqn:identification-fusion}, and taking into account the shift in the definition of $\tau^{\circ}$, we obtain:
\begin{itemize}
\item
The sheaf $\mathscr H^{k-2}(f\tilde{s}_\mu)_!(\tilde{s}_\mu)^*\mathscr B$ is locally
constant on $\Delta_X$, with stalk
$\coH^k_c(S_\mu,\mathscr A_1\star\mathscr A_2)$, so is
$\F_\mu(\mathscr A_1\star\mathscr A_2)$ if $k=\langle2\rho,\mu\rangle$
and is zero otherwise.
\item
The sheaf $\mathscr H^{k-2}(f\tilde s_\mu)_!(\tilde s_\mu)^*\mathscr B$ is locally
constant on $U$, with stalk isomorphic to
$$\bigoplus_{\mu_1+\mu_2=\mu}\coH^k_c
(S_{\mu_1}\times S_{\mu_2},\mathscr A_1\boxtimes\mathscr A_2),$$
so is isomorphic to
$$\bigoplus_{\mu_1+\mu_2=\mu}\F_{\mu_1}(\mathscr A_1)
\otimes\F_{\mu_2}(\mathscr A_2)$$
if $k=\langle2\rho,\mu\rangle$
and is zero otherwise. 
\end{itemize}
In particular,
$$\mathscr H^{k-2}(f \tilde s_\mu)_!(\tilde s_\mu)^*\mathscr B\neq0\ \Longrightarrow\
k=\langle2\rho,\mu\rangle.$$

Given $\mu\in X_*(T)$, denote
the sheaf $\mathscr H^{\langle2\rho,\mu\rangle-2}\,(f\tilde s''_\mu)_!(\tilde s_\mu'')^*\,\mathscr B$
by $\mathscr L_\mu(\mathscr B)$.
Adjunction yields maps
$$\xymatrix{
&\mathscr H^{\langle2\rho,\mu\rangle-2}(f \tilde s_\mu)_!(\tilde s_\mu)^*\mathscr B\ar[d]\\
\mathscr H^{\langle2\rho,\mu\rangle-2}f_*\mathscr B\ar[r]&
\mathscr L_\mu(\mathscr B).}$$
Since this is true over any point of $X^2$,
the vertical arrow is an isomorphism and the horizontal
arrow is an epimorphism; moreover for each $k\in\mathbf Z$ the projections provide an isomorphism
$$\mathscr H^{k-2}f_*\,\mathscr B\cong
\bigoplus_{\substack{\mu\in X_*(T)\\[2pt]\langle2\rho,\mu\rangle=k}}
\mathscr L_\mu(\mathscr B).$$
We see here that $\mathscr L_\mu(\mathscr B)$ is a direct summand of
the local system $\mathscr H^{k-2}f_*\,\mathscr B$ (see
\S\ref{ss:F-convolution}), so it is a local system itself.
As we saw, its stalk over a point in $\Delta_X$ is
$\F_\mu(\mathscr A_1\star\mathscr A_2)$ and its stalk over a point
in $U$ is $\bigoplus_{\mu_1+\mu_2=\mu} \F_{\mu_1}(\mathscr A_1)\otimes\F_{\mu_2}(\mathscr A_2)$.
We thus obtain the desired identification~\eqref{eqn:weight-functor-tensor}, as in the proof of Proposition~\ref{prop:F-product}.
\end{proof}

\begin{rmk}
\begin{enumerate}
\item
See Proposition~\ref{prop:geometric-restriction} below for a different proof of the compatibility of $\bigoplus_\mu \F_\mu$ with convolution, in a more general context.
\item
Once again, Proposition~\ref{prop:F-mu-tensor-char-0} can be proved in a more elementary way using equivariant cohomology, see~\cite[Proposition~5.3.14]{zhu} (but this proof is specific to the characteristic-$0$ setting).
\end{enumerate}
\end{rmk}

\section{Identification of the dual group}
\label{sec:identification-char-0}

At this point, we have constructed the convolution product $\star$ on
$\Per_{\GO}(\Gr_G,\bk)$, a $\bk$-linear faithful exact functor $\F :\Per_{\GO}(\Gr_G,\bk) \to\Vect_{\mathbf k}$,
an associativity constraint, a commutativity constraint, and a unit
object $U=\IC_0$ such that:
\begin{enumerate}
\item
$\F\circ\star=\otimes\circ(\F\otimes \F)$ and $\F(U)=\bk$;
\item
$\F$ maps the associativity constraint, the commutativity constraint and
the unit constraints of $\Per_{\GO}(\Gr_G,\bk)$ to the corresponding constraints of
$\Vect_{\bk}$;
\item
\label{it:F-dim-1}
If $\F(L)$ has dimension~1, then there exists $L^{-1}$ such that
$L\star L^{-1}\cong U$.
\end{enumerate}
(For~\eqref{it:F-dim-1}, one observes that for $L=\IC_\lambda$ to
satisfy $\dim \F(L)=1$, by Proposition~\ref{prop:dim-weight-spaces} $\lambda$ must be orthogonal to each root $\alpha\in \Delta(G,T)$,
so $\Gr_G^\lambda=\{L_\lambda\}$, and we can take $L^{-1}=\IC_{-\lambda}$
since $-\lambda$ is dominant.)

Tannakian reconstruction (see Theorem~\ref{thm:tannakina-reconstruction}) then gives us an affine group scheme $\widetilde{G}_\bk$ over $\bk$
and an equivalence $\Sat$ which fits in the following
commutative diagram:
$$\xymatrix{\Per_{\GO}(\Gr_G,\bk)\ar[rr]^\sim_{\Sat}\ar[dr]_{\F}&&\Rep_{\mathbf k}
(\widetilde{G}_\bk)\ar[dl]^\omega\\&\Vect_{\mathbf k},&}$$
where $\omega$ is the forgetful functor. We now need to identify
$\widetilde{G}_\bk$.

\begin{rmk}
The group $\widetilde{G}_\bk$ considered here should not be confused with the group $\tilde{G}$ of~\S\ref{ss:proj-embeddings}.
\end{rmk}

\subsection{First step: $\widetilde{G}_\bk$ is a split connected reductive
algebraic group over $\mathbf k$}
\label{ss:identification-char-0-first-step}

\begin{lem}
The affine group scheme $\widetilde{G}_\bk$ is algebraic.
\end{lem}

\begin{proof}
Choose a finite set of generators $\lambda_1, \cdots, \lambda_n$
of the monoid $X_*(T)^+$ of dominant cocharacters. Then for any
nonnegative integral linear combination $\lambda=k_1\lambda_1+\cdots+
k_n\lambda_n$, the sheaf $\IC_\lambda$ appears as a direct summand
of the convolution product 
\[
 \underbrace{\IC_{\lambda_1} \star \cdots \star \IC_{\lambda_1}}_{\text{$k_1$ copies}} \star \cdots \star \underbrace{\IC_{\lambda_n} \star \cdots \star \IC_{\lambda_n}}_{\text{$k_n$ copies}}.
\]
(In fact, this convolution product is a semisimple perverse sheaf by Proposition~\ref{thm:semisimplicity}. Moreover it is easily seen to be supported on $\overline{\Gr_G^\lambda}$, with restriction to $\Gr_G^\lambda$ isomorphic to $\underline{\bk}_{\Gr_G^\lambda}[\dim(\Gr_G^\lambda)]$. Hence it must admit $\IC_\lambda$ as a direct summand.)
Therefore $\mathscr{X}:=\IC_{\lambda_1}\oplus\cdots\oplus
\IC_{\lambda_n}$ is a tensor generator of the category $\Per_{\GO}(\Gr_G,\bk)$; namely, any
object of $\Per_{\GO}(\Gr_G,\bk)$ appears as a subquotient of a direct sum of
tensor powers of $\mathscr{X}$. Thus $\Rep_{\mathbf k}(\widetilde{G}_\bk)$ admits
a tensor generator, which implies that 
$\widetilde{G}_\bk$ is algebraic by Proposition~\ref{prop:properties-G}\eqref{it:properties-G-1}.
\end{proof}

\begin{lem}
\label{lem:tilde-G-connected-char-0}
The affine algebraic group scheme $\widetilde{G}_\bk$ is connected.
\end{lem}

\begin{proof}
If $\lambda$ is a nonzero dominany cocharacter of $T$, then the objects
$\IC_{m\lambda}$ are pairwise non isomorphic for $m\in\mathbf Z_{\geq 0}$ (since they have different supports). It follows
that for any nontrivial object $\mathscr{X}$ in $\Per_{\GO}(\Gr_G,\bk)$, the full subcategory
formed by subquotients of direct sums $\mathscr{X}^{\oplus n}$ cannot be
stable under $\star$. The same property then also holds for the tensor
category $\Rep_{\mathbf k}(\widetilde{G}_\bk)$. This in turn implies
that $\widetilde{G}_\bk$ is connected by Proposition~\ref{prop:properties-G}\eqref{it:properties-G-2}.
\end{proof}

\begin{lem}
\label{lem:reductive-char-0}
The connected affine algebraic group scheme $\widetilde{G}_\bk$ is reductive.
\end{lem}

\begin{proof}
If $\overline{\bk}$ is an algebraic closure of $\bk$, it is clear that $\widetilde{G}_{\overline{\bk}} := \Spec(\overline{\bk}) \times_{\Spec(\bk)} \widetilde{G}_\bk$ is the group scheme provided by Tannakian formalism out of the category $\Per_{\GO}(\Gr_G, \overline{\bk})$. This category is semisimple by Theorem~\ref{thm:semisimplicity}. We conclude using Proposition~\ref{prop:properties-G}\eqref{it:properties-G-3}.
\end{proof}

We now explain the construction of a split maximal torus in $\widetilde{G}_\bk$ (see~\S\ref{ss:overview}).

As in~\S\ref{ss:compatibility-weight-functors}, we denote by $\Vect_{\mathbf k}(X_*(T))$ the category of finite
dimensional $X_*(T)$-graded $\mathbf k$-vector spaces. This is a
monoidal category, and the weight functors provide us with a
factorization of $\F$ as
$$\Per_{\GO}(\Gr_G,\bk) \xrightarrow{\F'}\Vect_{\mathbf k}(X_*(T))\xrightarrow{\text{forget}}\Vect_{\mathbf k},$$
see~\S\ref{ss:compatibility-weight-functors}.
Let $T^\vee_\bk$ be the unique split $\mathbf k$-torus such that
$X^*(T^\vee_\bk)=X_*(T)$; then $\Vect_{\mathbf k}(X_*(T))\cong
\Rep_{\mathbf k}(T^\vee_\bk)$ canonically (see e.g.~\cite[\S I.2.11]{jantzen}), and $\F'$ induces a
functor $\F_{T^\vee_\bk}:\Rep_{\mathbf k}(\widetilde{G}_\bk)\to
\Rep_{\mathbf k}(T^\vee_\bk)$ compatible with the monoidal structures. There is then a commutative diagram
$$\xymatrix@C=1.5cm{\Per_{\GO}(\Gr_G,\bk)\ar[r]^{\F'}\ar[d]^\wr&
\Vect_{\mathbf k}(X_*(T))\ar@{=}[d]\\
\Rep_{\mathbf k}(\widetilde{G}_\bk)\ar[r]^{\F_{T^\vee_\bk}}&
\Rep_{\mathbf k}(T^\vee_\bk)},$$
and the functor $\F_{T^\vee_\bk}$
commutes with the forgetful functors to $\Vect_{\mathbf k}$ and satisfies the conditions of Proposition~\ref{prop:tannakian-morph}.
Hence this functor is induced by a unique morphism
$\varphi:T^\vee_\bk\to\widetilde{G}_\bk$ of algebraic groups.

Each character $\lambda\in X^*(T^\vee_\bk)$ appears in at least one
$\F_{ T^\vee_\bk}(\IC_\mu)$. (One can here e.g.~choose $\mu$ as the dominant $W$-conjugate of $\lambda$ and use Theorem~\ref{thm:orbits} and Proposition~\ref{prop:dim-weight-spaces}.)
It follows that $\varphi$ is an embedding of a closed subgroup, see~\cite[Proposition~2.21(b)]{dm};
so $T^\vee_\bk$ can be considered as a split torus in $\widetilde{G}_\bk$.

Now, for any reductive group $H$ over a field $\FF$, if $\overline{\FF}$ is an algebraic closure of $\FF$ and if we set  $H_{\overline{\FF}}:= \Spec(\overline{\FF}) \times_{\Spec(\FF)} H$, then
\begin{equation}
\label{eqn:rank-dim}
\mathrm{rk}(H) = \dim \Spec \bigl( \Q \otimes_\Z K^0(\Rep_{\overline{\FF}}(H_{\overline{\FF}})) \bigr).
\end{equation}
In fact, the right-hand side admits a basis consisting of classes of induced modules (i.e.~the modules denoted $H^0(\lambda)$ in~\cite[Chap.~II.2]{jantzen}), whose characters are given by the Weyl character formula, see e.g.~\cite[Corollary~II.5.11]{jantzen}. Therefore it identifies with $K_{\Q}/W_H$, where $K_\Q$ is the split $\Q$-torus with character lattice the character lattice of any chosen maximal torus $K_{\overline{\FF}}$ in $H_{\overline{\FF}}$ and $W_H$ is the Weyl group of $H$ with respect to this torus. There is a finite morphism $K_\Q \to K_{\Q}/W_H$, so that this scheme has the same dimension as $K_\Q$, i.e.~has dimension the rank of $H_{\overline{\FF}}$, which by definition is the rank of $H$.

In our case, the functor $\F_{T^\vee_\bk}$ provides a morphism of schemes
\[
T^\vee_\Q \to \Spec \bigl( \Q \otimes_\Z K^0(\Rep_{\overline{\bk}}(\widetilde{G}_{\overline{\bk}})) \bigr),
\]
where $\overline{\bk}$ and $\widetilde{G}_{\overline{\bk}}$ are as in the proof of Lemma~\ref{lem:reductive-char-0}, and $T^\vee_\Q$ is the $\Q$-torus with characters $X_*(T)$.
In view of the description of the
simple objects in $\Per_{\GO}(\Gr_G,\bk)$ (see Section~\ref{sec:semisimplicity}), this morphism identifies the right-hand side with $T^\vee_\Q / W$. We deduce that the rank of $\widetilde{G}_\bk$ is the dimension of $T^\vee_\bk$, i.e.~that
$T^\vee_\bk$ is a maximal torus of~$\widetilde{G}_\bk$.

\subsection{Second step: identification of the root datum of $(\widetilde{G}_\bk,
T^\vee_\bk)$}
\label{ss:identification-second-step}


In view of the general results recalled in~\S\ref{ss:overview}, to finish our determination of the group scheme $\widetilde{G}_\bk$, it only remains to identify the root datum of $(\widetilde{G}_\bk, T^\vee_\bk)$. By the remarks in the proof of Lemma~\ref{lem:reductive-char-0} and the definitions recalled above, for this we can (and shall) assume that $\bk$ is algebraically closed.

We first determine a ``canonical'' Borel subgroup in $\widetilde{G}_\bk$.
Consider the sum $2\rho\in X^*(T)$ of the positive roots of $G$.
Then there exists a (possibly non unique) Borel subgroup $\widetilde{B}\subset\widetilde{G}_\bk$ that
contains $T^\vee_\bk$ and such that $2\rho$ is a dominant coweight for the choice of positive roots of $\widetilde{G}_\bk$ given by the $T^\vee_\bk$-weights in the Lie algebra of $\widetilde{B}$.

\begin{lem}
\label{lem:Borel}
For such a choice of Borel subgroup $\widetilde{B}$, hence of positive roots,
the dominant weights for $T^\vee_\bk$ are
exactly the dominant coweights $X_*(T)^+$ of $T$ (for the choice of the positive roots as the $T$-weights in the Lie algebra of $B$).
\end{lem}

\begin{proof}
Given $\lambda\in X_*(T)^+$ (that is,
dominant for $T\subset B\subset G$),
let $V=\Sat(\IC_\lambda)$ be the simple $\widetilde{G}_\bk$-module
corresponding to the simple object $\IC_\lambda$ of $\Per_{\GO}(\Gr_G,\bk)$. By Proposition~\ref{prop:dim-weight-spaces} the maximal value of
$\langle2\rho,\mu\rangle$ for $\mu$ a weight of $V$ is
obtained for $\mu=\lambda$, and only for this weight.
Therefore $\lambda$ is dominant for $T^\vee_\bk\subset\widetilde{B}\subset
\widetilde{G}_\bk$, and is the highest weight of $V$. 

Conversely,
let $\mu\in X^*(T^\vee_\bk)$ be dominant for $T^\vee_\bk\subset\widetilde{B}\subset\widetilde{G}_\bk$. Let $V$ be the simple $\widetilde{G}_\bk$-module of
highest weight $\mu$. Then $V=\Sat(\IC_\lambda)$ for a unique
$\lambda \in X_*(T)^+$, and by the first
step $\lambda=\mu$. Thus $\mu$ is dominant for $T\subset B\subset G$.
\end{proof}

This claim implies in particular that $\widetilde{B}$ is uniquely determined; that is,
no root of $(\widetilde{G}_\bk,T^\vee_\bk)$ is orthogonal to $2\rho$. From now on we fix this choice of Borel subgroup in $\widetilde{G}_\bk$, and hence of positive roots of $\widetilde{G}_\bk$ with respect to $T^\vee_\bk$. We will denote by $\Delta(\widetilde{G}_\bk,T^\vee_\bk)$ the root system of $\widetilde{G}_\bk$ with respect to $T^\vee_\bk$, by $\Delta_+(\widetilde{G}_\bk,\widetilde{B},T^\vee_\bk)$ the subset of positive roots determined by $\widetilde{B}$, and by $\Delta_{\mathrm{s}}(\widetilde{G}_\bk,\widetilde{B},T^\vee_\bk)$ the corresponding set of simple roots. We use similar notation (with a superscript ``$\vee$'') for coroots, and also for the roots and coroots of $G$. (This is of course consistent with the notation introduced in~\S\ref{ss:def-Gr}.)

\begin{rmk}
\begin{enumerate}
\item
Recall (see~\S\ref{ss:independence-Torel}) that the maximal torus $T^\vee_\bk \subset G^\vee_\bk$ does not depend on any choice. Viewed as a coweight of $T^\vee_\bk$, the element $2\rho$ does not depend on any choice either: it is the only coweight such that the weights of restriction of the action of $G^{\vee}_\bk$ on $\coH^{\bullet}(\Gr_G,\bk)$ to $\bk^\times$ are given by the cohomological grading. Therefore, $\widetilde{B}$ is also canonical in the sense that it does not depend on any choice.
 \item
 In various sources (e.g.~\cite[End of~\S 7]{mv} or~\cite[Discussion after Lem\-ma~5.3.17]{zhu} the ``canonical'' Borel subgroup in $\widetilde{G}_\bk$ is constructed using a ``Pl\"ucker formalism.'' We were not able to find references supporting this construction, hence decided to use a more elementary approach. In any case the two constructions have to produce the same subgroup, see e.g.~\cite[Corollary~5.3.20]{zhu}.
 \item
 Using a construction involving the action of the first Chern class of line bundles on $\Gr_G$, viewed as elements of $\coH^\bullet(\Gr_G, \bk)$ (following ideas of Ginzburg~\cite{ginzburg}) one can ``complete'' the datum of $\widetilde{B}$ and $T^\vee_\bk$ to a canonical pinning on $\widetilde{G}_\bk$; see in particular~\cite{vasserot},~\cite[\S 3.4]{baumann},~\cite[\S 5.3]{yz} or~\cite[Theorem~5.3.23]{zhu}. More precisely, this construction provides a group morphism from the Picard group $\mathrm{Pic}(\Gr_G)$ to the Lie algebra $\mathrm{Lie}(G^\vee_\bk)$, which sends ample line bundles to regular nilpotent elements belonging to $\mathrm{Lie}(\widetilde{B})$. (This property provides another canonical description of $\widetilde{B}$, as the unique Borel subgroup of $G^\vee_\bk$ containing $T^\vee_\bk$ and whose Lie algebra contains these regular nilpotent elements.) In particular, if the derived subgroup of $G$ is quasi-simple then there exists a canonical ample line bundle on $\Gr$ characterized by the fact that its restriction to each connected component is a generator of the corresponding Picard group; decomposing the regular nilpotent element obtained from this line bundle on root spaces we obtain the desired pinning.
 \end{enumerate}
\end{rmk}


Lemma~\ref{lem:Borel} implies that the simple root directions of
$T\subset B\subset G$ are the simple coroot directions of
$T^\vee_\bk\subset\widetilde{B}\subset\widetilde{G}_\bk$:
\begin{equation}
\label{eqn:root-directions}
\bigl\{\mathbf Q_+\cdot\alpha : \alpha \in \Delta^\vee_{\mathrm{s}}(\widetilde{G}_\bk,\widetilde{B},T^\vee_\bk)\bigr\}=
\bigl\{\mathbf Q_+\cdot\beta : \beta \in \Delta_{\mathrm{s}}(G,B,T)\bigr\}.
\end{equation}
(In fact, these sets are the extreme rays of the rational convex polyhedral cone determined by $\{\lambda \in \Q \otimes_\Z X^*(T) \mid \forall \mu \in X_*(T)^+, \ \langle \lambda, \mu \rangle \geq 0\}$.)


\begin{lem}
\label{lem:simple-roots}
We have $\Delta_{\mathrm{s}}(\widetilde{G}_\bk,\widetilde{B},T^\vee_\bk) = \Delta^\vee_{\mathrm{s}}(G,B,T)$ as subsets of $X_*(T) =X^*(T^\vee_\bk)$.
\end{lem}

\begin{proof}
Let $G^\vee_\bk$ be the (connected, split) reductive
$\bk$-group which is Langlands dual to $G$, i.e.~whose root datum is dual to that of $(G,T)$. Then $T^\vee_\bk$ is
also a maximal torus in $G^\vee_\bk$. Choose the positive roots of
$(G^\vee_\bk,T^\vee_\bk)$ as the positive coroots of $T\subset B\subset G$,
so that the dominant weights of $(G^\vee_\bk,T^\vee_\bk)$ are $X_*(T)^+$.

Given $\lambda\in X_*(T)^+$, we can consider the simple $G^\vee_\bk$-module
$V_\lambda(G^\vee_\bk)$ with highest weight $\lambda$, and the simple
$\widetilde{G}_\bk$-module $V_\lambda(\widetilde{G}_\bk)=\Sat(\IC_\lambda)$ with highest
weight $\lambda$. The crucial observation is that these two $T^\vee_\bk$-modules
have the same weights; specifically, the set of weights of both of these
modules is
$$\left\{\mu\in X_*(T)\left|\begin{aligned}
\;&\mu-\lambda\text{ is in the coroot lattice of $(G,T)$}\\
&\text{and }\mu\text{ is in the convex hull of }W\lambda
\end{aligned}\right.\right\},$$
see again Proposition~\ref{prop:dim-weight-spaces}.
(Note however that we do not yet know that these two
$T^\vee_\bk$-modules have the same character.)

We now observe that
$$\{\lambda-\mu\mid\lambda\in X_*(T)^+,\;\mu\text{ a weight of }
V_\lambda(G^\vee_\bk)\}$$
is the $\mathbf N$-span of the positive roots of $(G^\vee_\bk,T^\vee_\bk)$.
The argument just above shows that this is also the $\mathbf N$-span of
the positive roots of $(\widetilde{G}_\bk,T^\vee_\bk)$. Looking at the
indecomposable elements of this monoid, we deduce that the simple roots of $\widetilde{G}_\bk$ are the simple roots
of $G^\vee_\bk$, i.e.~the simple coroots of $G$.
\end{proof}


We can finally conclude.

\begin{thm}
\label{thm:identification-char-0}
The group $\widetilde{G}_\bk$ is Langlands dual to $G$; more precisely the root datum of $\widetilde{G}_\bk$ with respect to $T^\vee_\bk$ is dual to the root datum of $G$ with respect to $T$.
\end{thm}

\begin{proof}
By construction $X^*(T^\vee_\bk)$ is dual to $X^*(T)$. What remains to be proved is that the roots and coroots of $\widetilde{G}_\bk$, together with the canonical bijection between these two sets, coincide with the coroots and roots of $G$, together with their canonical bijection.

Let $\alpha \in \Delta_{\mathrm{s}}(G,B,T)$.
By Lemma~\ref{lem:simple-roots}, the corresponding coroot $\alpha^\vee$ belongs to
$\Delta_{\mathrm{s}}(\widetilde{G}_\bk,\widetilde{B},T^\vee_\bk)$. The coroot
$\widetilde{\alpha}$ of $\widetilde{G}_\bk$ associated with this root is $\mathbf Q_+$-proportional to
a simple root of $T\subset B\subset G$ by~\eqref{eqn:root-directions}.
The conditions
$$\begin{cases}
\langle\widetilde{\alpha},\alpha^\vee\rangle=2,&\\[2pt]
\langle\widetilde{\alpha},\beta^\vee\rangle\leq0&
\text{for $\beta^\vee\in\Delta_{\mathrm{s}}(\widetilde{G}_\bk,\widetilde{B},T^\vee_\bk) \smallsetminus\{\alpha^\vee\}$}
\end{cases}$$
then give $\widetilde{\alpha}=\alpha$.

We thus have an identification
\[
\Delta_{\mathrm{s}}(G,B,T) = \Delta_{\mathrm{s}}^\vee(\widetilde{G}_\bk,\widetilde{B},T^\vee_\bk).
\]
By Lemma~\ref{lem:simple-roots} we also have
\[
\Delta_{\mathrm{s}}(\widetilde{G}_\bk,\widetilde{B},T^\vee_\bk) = \Delta^\vee_{\mathrm{s}}(G,B,T),
\]
and the bijections between simple roots and simple coroots
are the same. We may thus identify the Weyl groups of $G$ and
$\widetilde{G}_\bk$ and extend the above equalities between simple roots/coroots of $\widetilde{G}_\bk$ and coroots/roots of $G$ to equalities between \emph{all} roots and coroots. It is clear from this proof that the bijections between roots and coroots are the same for the two groups, and thus our proof is complete.
\end{proof}

\subsection{Conclusion}

We have finally constructed our canonical equivalence of monoidal categories $\Sat$
which fits in the commutative diagram
$$\xymatrix@R=1.2cm{\Per_{\GO}(\Gr_G,\bk)\ar[rr]^\sim_{\Sat}\ar[dr]|-{\F:=\mathsf{H}^\bullet(\Gr_G,\bm ?)}&&\Rep_{\mathbf k}
(G^\vee_\bk)\ar[dl]|-{\forget}\\&\Vect_{\mathbf k}.&}$$

\newpage

\part{The case of arbitrary coefficients}
\label{pt:arbitrary}

In this part, $\bk$ is an arbitrary commutative Noetherian ring of finite global dimension,\footnote{These assumptions on $\bk$ are needed to have a ``good'' six-functors formalism for derived categories of sheaves of $\bk$-modules, hence to apply the theory of perverse sheaves; see~\cite{ks}.} and we denote by $\Mod_\bk$ the abelian category of finitely generated $\bk$-modules. We continue with the geometric setting of Part~\ref{pt:char-0}: $G$ is a (connected) complex reductive algebraic group, and we consider the affine Grassmannian $\Gr_G$ of $G$. Our main object of study is now the category $\Per_{\GO}(\Gr_G,\bk)$ of $\GO$-equivariant $\bk$-perverse sheaves\footnote{The definition of perverse sheaves in this generality is literally the same as that recalled in~\S\ref{ss:satake-category}. The main difference with the case of fields is that now this subcategory is not stable under Verdier duality in general.} on $\Gr_G$. We will see in~\S\ref{ss:equiv-const} that this category is equivalent to the category $\Per_{\mathscr{S}}(\Gr_G,\bk)$ of $\mathscr{S}$-constructible perverse sheaves (as for fields of characteristic $0$, see Corollary~\ref{cor:equivariance}), but at first we need to distinguish these two categories. 

\section{Convolution and weight functors for general coefficients}

In this section we explain how to modify the definition of the convolution bifunctor, and the proof of its main properties, to treat the case of general coefficients.

\subsection{Weight functors}

Proposition~\ref{prop:weight-functors} still holds in this generality, with the same proof.

\begin{prop}
\label{prop:weight-functors-k}
For $\mathscr A\in\Per_{\mathscr{S}}(\Gr_G,\mathbf k)$, $\mu\in X_*(T)$
and $k\in\mathbf Z$, there exists a canonical isomorphism
$$\coH^k_{T_\mu}(\Gr_G,\mathscr A)\xrightarrow\sim \coH^k_c(S_\mu,\mathscr A),$$
and both terms vanish if $k\neq \langle 2\rho,\mu \rangle$.
\end{prop}

\begin{rmk}
\label{rmk:weight-functors-k}
 The same comments as in Remark~\ref{rmk:weight-functors} apply here also.
\end{rmk}

In view of this fact, as in Section~\ref{sec:weight-functors}, for any $\mu \in X_*(T)$ we denote by
\[
\F_\mu : \Per_{\mathscr{S}}(\Gr_G,\bk) \to \Mod_\bk
\]
the functor defined by
\[
\F_\mu(\mathscr{A}) = \coH^{\langle 2\rho, \mu \rangle}_{T_\mu}(\Gr_G,\mathscr A) \cong \coH^{\langle 2\rho, \mu \rangle}_c(S_\mu,\mathscr A).
\]

\begin{lem}
\label{lem:weight-functors-exact}
For any $\mu \in X_*(T)$, the functor $\F_\mu$ is exact.
\end{lem}

\begin{proof}
Any exact sequence $\mathscr{F}_1 \hookrightarrow \mathscr{F}_2 \twoheadrightarrow \mathscr{F}_3$ in $\Per_{\mathscr{S}}(\Gr,\bk)$ is defined by a distinguished triangle
\[
\mathscr{F}_1 \to \mathscr{F}_2 \to \mathscr{F}_3 \xrightarrow{[1]}
\]
in $\Db_{\mathscr{S}}(\Gr_G,\bk)$. Such a triangle induces a long exact sequence
\[
\cdots \to \coH^{k-1}_c(S_\mu,\mathscr{F}_3) \to \coH^k_c(S_\mu,\mathscr{F}_1) \to \coH^k_c(S_\mu,\mathscr{F}_2) \to \coH^k_c(S_\mu,\mathscr{F}_3) \to \coH^{k+1}_c(S_\mu,\mathscr{F}_1) \to \cdots
\]
in $\Mod_\bk$.
Using the vanishing claim in Proposition~\ref{prop:weight-functors-k} we deduce an exact sequence of $\bk$-modules
\[
\coH^{\langle 2\rho, \mu \rangle}_c(S_\mu,\mathscr{F}_1) \hookrightarrow \coH^{\langle 2\rho, \mu \rangle}_c(S_\mu,\mathscr{F}_2) \twoheadrightarrow \coH^{\langle 2\rho, \mu \rangle}_c(S_\mu,\mathscr{F}_3),
\]
which finishes the proof.
\end{proof}

Then we define the functor
\[
\F : \Per_{\mathscr{S}}(\Gr_G,\mathbf k)\to\Mod_{\mathbf k}
\]
by
\[
\F(\mathscr{A}) = \coH^\bullet(\Gr_G,\mathscr{A}).
\]

The same proof as that of Theorem~\ref{thm:fiber-functor}, together with Lemma~\ref{lem:weight-functors-exact}, gives the following result.

\begin{thm}
\label{thm:fiber-functor-k}
There exists a canonical isomorphism of functors
$$\F\cong\bigoplus_{\mu\in X_*(T)}\F_\mu:
\Per_{\mathscr{S}}(\Gr_G,\mathbf k)\to\Mod_{\mathbf k}.$$
Moreover, $\F$ is exact and faithful.\qed
\end{thm}

\begin{rmk}
\label{rmk:indep-k}
 Using Theorem~\ref{thm:fiber-functor-k} one can also generalize the proof of Lemma~\ref{lem:indep}: the functors $\F_\lambda$ do not depend on the choice of Torel $T \subset B$, up to canonical isomorphism.
\end{rmk}

Below we will also need the following claim (where, as in~\S\ref{ss:def-convolution}, we denote by $\Db_{c,\GO}(\Gr_G,\bk)$ the constructible $\GO$-equivariant derived category).

\begin{lem}
\label{lem:criterion-perv}
For any $\mathscr{F}$ in $\Db_{c,\GO}(\Gr_G,\bk)$, the following conditions are equivalent:
\begin{enumerate}
\item
$\mathscr{F}$ is a perverse sheaf
\item
\label{it:criterion-perv-2}
for any $\mu \in X_*(T)$ and $k \in \Z$ we have
\[
\coH_c^k(S_\mu, \mathscr{F}) \neq 0 \quad \Rightarrow \quad k=\langle 2\rho, \mu \rangle.
\]
\item
\label{it:criterion-perv-3}
for any $\mu \in X_*(T)$ and $k \in \Z$ we have
\[
\coH^k_{S_\mu}(\Gr_G, \mathscr{F}) \neq 0 \quad \Rightarrow \quad k=-\langle 2\rho, \mu \rangle.
\]
\end{enumerate}
\end{lem}

\begin{proof}
If $\mathscr{F}$ is perverse, then the conditions~\eqref{it:criterion-perv-2} and~\eqref{it:criterion-perv-3} hold by Proposition~\ref{prop:weight-functors-k} together with the facts that $\mathscr{F}$ is $G$-equivariant and that $T_{w_0 \mu} = \dot{w}_0 \cdot S_\mu$, where $\dot{w}_0$ is any lift of the longest element $w_0$ of $W$ in $G$.

Now, let us assume that~\eqref{it:criterion-perv-2} holds, and prove that $\mathscr{F}$ is perverse.
Of course we can assume that $\mathscr{F} \neq 0$. Let $n$ be the highest degree for which $\pH^n(\mathscr{F}) \neq 0$ (where $\pH^n(\bm ?)$ is the $n$-th perverse cohomology functor). Then we have a ``truncation triangle''
\[
\mathscr{F}' \to \mathscr{F} \to \pH^n(\mathscr{F})[-n] \xrightarrow{[1]}
\]
where $\mathscr{F}'$ is concentrated in perverse degrees $\leq n-1$. By Theorem~\ref{thm:fiber-functor-k}, there exists $\mu \in X_*(T)$ such that $\F_\mu(\pH^n(\mathscr{F})) \neq 0$. Then Proposition~\ref{prop:weight-functors-k} implies that $\coH_c^k(S_\mu, \mathscr{F}')=0$ if $k \geq n+\langle 2\rho, \mu \rangle$, so that the natural morphism
\[
\coH_c^{n+\langle 2\rho, \mu \rangle}(S_\mu, \mathscr{F}) \to \coH_c^{\langle 2\rho, \mu \rangle}(S_\mu, \pH^n(\mathscr{F}))
\]
is an isomorphism. Since the right-hand side is nonzero by our choice of $\mu$, so is the left-hand side, and then our assumption implies that $n=0$.

If now $m$ is the lowest degree such that $\pH^m(\mathscr{F}) \neq 0$, then similar arguments using the truncation triangle
\[
\pH^n(\mathscr{F})[-m] \to \mathscr{F} \to \mathscr{F}'' \xrightarrow{[1]}
\]
(where $\mathscr{F}''$ is concentrated in perverse degrees $\geq m+1$)
show that $m=0$, which finally proves that $\mathscr{F}$ is perverse.

The fact that~\eqref{it:criterion-perv-3} implies that $\mathscr{F}$ is perverse can be proved similarly using the other description of the functor $\F_\mu$ and the relation between $S_\mu$ and $T_{w_0 \mu}$ noticed at the beginning of the proof.
\end{proof}

More generally, using arguments similar to those in the proof of Lemma~\ref{lem:criterion-perv} one can show the following claim by induction on $\#\{m \in \Z \mid \pH^m(\mathscr{F}) \neq 0\}$.

\begin{lem}
\label{lem:Fmu-non-perverse}
 For any $\mathscr{F}$ in $\Db_{\mathscr{S}}(\Gr_G,\bk)$ and any $n \in \Z$ we have
\[
\coH_c^{n+\langle 2\rho, \mu \rangle}(S_\mu, \mathscr{F}) \cong \F_\mu(\pH^n(\mathscr{F})).
\]
\end{lem}

\subsection{Equivariant and constructible perverse sheaves}
\label{ss:equiv-const}

Now we can prove that Corollary~\ref{cor:equivariance} is still true in this context (but for more serious reasons). By definition, the forgetful functor $\Db_{c,\GO}(\Gr_G,\bk) \to \Db_{\mathscr{S}}(\Gr_G,\bk)$ is t-exact for the perverse t-structures. In the following proposition we consider the restriction of this functor to perverse sheaves.

\begin{prop}
\label{prop:For-k}
The forgetful functor
 \begin{equation*}
\Per_{\GO}(\Gr_G,\bk) \to \Perv
\end{equation*}
is an equivalence of categories.
\end{prop} 


In view of this result, below we will not distinguish the categories $\Per_{\GO}(\Gr_G,\bk)$ and $\Perv$ anymore. In particular, we will now consider $\F$ and $\F_\mu$ as functors from $\Per_{\GO}(\Gr_G,\bk)$ to $\Mod_\bk$.

To explain the proof of Proposition~\ref{prop:For-k} we need to recall a construction from~\cite{vilonen}. Consider some categories $\mathscr{A}$ and $\mathscr{B}$, two functors $F,G : \mathscr{A} \to \mathscr{B}$, and a morphism of functor $\vartheta : F \to G$. Then we define a new category $\mathscr{C}(F,G;\vartheta)$ with:
\begin{itemize}
 \item objects: quadruples $(A,B,m,n)$ with $A$ in $\mathscr{A}$, $B$ in $\mathscr{B}$, and $m : F(A) \to B$, $n : B \to G(A)$ morphisms in $\mathscr{B}$ such that $\vartheta(A) = n \circ m$;
 \item morphisms from $(A,B,m,n)$ to $(A',B',m',n')$: pairs $(f,g)$ where $f : A \to A'$ and $g : B \to B'$ are morphisms in $\mathscr{A}$ and $\mathscr{B}$ respectively, such that both squares in the following diagram commute:
 \[
  \xymatrix@C=1.5cm{
  F(A) \ar[d]_-{F(f)} \ar[r]^-{m} & B \ar[d]_-{g} \ar[r]^-{n} & G(A) \ar[d]^-{G(f)} \\
  F(A') \ar[r]^-{m'} & B' \ar[r]^-{n'} & G(A').
  }
 \]
\end{itemize}
If $\mathscr{A}$, $\mathscr{B}$ are abelian, $F$ is right exact and $G$ is left exact then $\mathscr{C}(F,G;\vartheta)$ is abelian (see~\cite[Proposition~1.1]{vilonen}). In practice we will only consider this situation, but this fact will not play any role in our arguments.

\begin{proof}[Proof of Proposition~{\rm\ref{prop:For-k}}]
First, we claim that the forgetful functor
\begin{equation}
\label{eqn:For-Perv-GO-G}
\Per_{\GO}(\Gr_G,\bk) \to \Per_{\mathscr{S},G}(\Gr_G,\bk)
\end{equation}
is an equivalence of categories. In fact, since $G_{\mathcal{O}}$ is the semi-direct product of $G$ with a pro-unipotent subgroup (namely the kernel of the natural morphism $\GO \to G$),~\cite[Theorem~3.7.3]{bernstein-lunts} shows that the forgetful functor $\Db_{c,\GO}(\Gr_G,\bk) \to \Db_{\mathscr{S},G}(\Gr_G,\bk)$ is fully-faithful. Since the codomain of this functor is generated (as a triangulated category) by the objects of the form $(j_\lambda)_! \underline{\bk}_{\Gr^\lambda_G}$, which belong to its essential image, this functor is also essentially surjective, hence an equivalence. Restricting to perverse sheaves we deduce that~\eqref{eqn:For-Perv-GO-G} is an equivalence as well.

 On the other hand, the forgetful functor $\Per_{\mathscr{S},G}(\Gr_G,\bk) \to \Per_{\mathscr{S}}(\Gr_G,\bk)$ is fully faithful, see~\S\ref{ss:equiv-perv}; hence what we have to prove is the following claim: for any finite closed union of $\GO$-orbits $Z$ and any $\mathscr{S}$-constructible\footnote{Here (and below), by abuse, we still denote by $\mathscr{S}$ the restriction of the stratification $\mathscr{S}$ to $Z$ (or to any locally closed union of strata in $\Gr_G$).} perverse sheaf $\mathscr{F}$ on $Z$, there exists an isomorphism $(p_Z)^*\mathscr{F} \xrightarrow{\sim} (a_Z)^*\mathscr{F}$, where $a_Z,p_Z : G \times Z \to Z$ are the action map and the projection, respectively. In fact, we will prove this property for any \emph{locally closed} union of strata, by induction of the number of strata in $Z$.
 
 We note that the claim is obvious if $Z$ contains only one $\GO$-orbit. (In fact, in this case the category $\Per_{\mathscr{S}}(Z,\bk)$ is equivalent to the category $\Mod_\bk$ via $V \mapsto \underline{V}_Z[\dim Z]$.) Now we consider a general $Z$, choose $\lambda \in X_*(T)^+$ such that $\Gr_G^\lambda \subset Z$ is closed in $Z$, and set $U:=Z \smallsetminus \Gr_G^\lambda$. We denote by $i : \Gr_G^\lambda \to Z$ and $j : U \to Z$ the closed and open embeddings, respectively. We also consider the varieties $\widetilde{Z} := G \times Z$, $\widetilde{U}:= G \times U$, and denote by $\widetilde{i} : G \times \Gr_G^\lambda \to \widetilde{Z}$ and $\widetilde{j} : \widetilde{U} \to \widetilde{Z}$ the closed and open embeddings, respectively. Finally, we denote by $\widetilde{\mathscr{S}}$ the stratification of $\widetilde{Z}$ whose strata are the products $G \times \Gr_G^\mu$ with $\Gr_G^\mu \subset Z$, and also the restriction of this stratification to $\widetilde{U}$.
 
 By induction, we know that the forgetful functor $\Per_{\mathscr{S},G}(U,\bk) \to \Per_{\mathscr{S}}(U,\bk)$ is an equivalence of categories. Now we take $\mathscr{F}$ in $\Per_{\mathscr{S}}(Z,\bk)$, and need to show that there exists an isomorphism $(p_Z)^*\mathscr{F} \xrightarrow{\sim} (a_Z)^*\mathscr{F}$. To check this, we set $\mathscr{A}:= \Per_{\mathscr{S}}(U,\bk)$, and denote by $\mathscr{B}$ the category of $\bk$-local systems on $G$. We consider the functor
 \[
  E := \mathscr{H}^{\langle \lambda,2\rho \rangle+\dim(G)} \bigl( (\widetilde{\sigma}_\lambda)_! (\widetilde{s}_\lambda)^* \bm ? \bigr) : \Per_{\widetilde{\mathscr{S}}}(\widetilde{Z},\bk) \to \mathscr{B},
 \]
where $\widetilde{\sigma}_\lambda : G \times (S_\lambda \cap Z) \to G$ is the projection and $\widetilde{s}_\lambda : G \times (S_\lambda \cap Z) \to \widetilde{Z}$ is the embedding. (Here, the fact that $E$ takes values in local systems rather than more general sheaves follows from the observation that the simple objects in $\Per_{\widetilde{\mathscr{S}}}(\widetilde{Z},\bk)$ are actually $G$-equivariant, so that their images under $E$ are also $G$-equivariant, hence are local systems.) For $\mathscr{G}$ in $\Per_{\widetilde{\mathscr{S}}}(\widetilde{Z},\bk)$, if $g \in G$ the fiber of the complex $(\widetilde{\sigma}_\lambda)_! (\widetilde{s}_\lambda)^* \mathscr{G}$ at $g$ is $R\Gamma_c(g \cdot S_\lambda \cap Z, \mathscr{G}_{|\{g\} \times Z})$; hence this fiber is concentrated in degree $\langle \lambda,2\rho \rangle$ by Remark~\ref{rmk:weight-functors-k} (for the choice of Torel $gTg^{-1} \subset gBg^{-1}$). This implies (as in the proof of Lemma~\ref{lem:weight-functors-exact}) that $E$ is an exact functor.

We then set
\[
 \widehat{F} := {}^p \hspace{-1pt} \widetilde{j}_!(\bm ?), \ \widehat{G} := {}^p \hspace{-1pt} \widetilde{j}_*(\bm ?) : \Per_{\widetilde{\mathscr{S}}}(\widetilde{U},\bk) \to \Per_{\widetilde{\mathscr{S}}}(\widetilde{Z},\bk).
\]
We also denote by $\theta : \widehat{F} \to \widehat{G}$ the natural morphism of functors (provided by adjunction and the fact that $j^* \circ \widehat{F} \cong \id$, or equivalently the fact that $j^* \widehat{G} \cong \id$). Finally, we set $\widetilde{\mathscr{B}} := \Per_{\widetilde{\mathscr{S}}}(G \times \Gr_G^\lambda,\bk)$, which we consider as a full subcategory of $\Per_{\widetilde{\mathscr{S}}}(\widetilde{Z},\bk)$ via the functor $\widetilde{i}_*$. Then we are exactly in the setting of~\cite[Proposition~1.2]{vilonen}, which claims that the functor
\[
 \widetilde{E} : \Per_{\widetilde{\mathscr{S}}}(Z,\bk) \to \mathscr{C}(E \circ \widehat{F}, E \circ \widehat{G} ; E \circ \theta),
\]
sending $\mathscr{G}$ to the quadruple $(\widetilde{j}^*\mathscr{G}, E(\mathscr{G}), m,n)$ where $m : E \circ \widehat{F}(\widetilde{j}^*\mathscr{G}) \to E(\mathscr{G})$ and $n : E(\mathscr{G}) \to E \circ \widehat{G}(\widetilde{j}^*\mathscr{G})$ are the images under $E$ of the adjunction morphisms, is fully faithful.\footnote{In~\cite[Proof of Proposition~A.1]{mv}, the authors claim (without proof) that this functor is in fact an equivalence of categories. Since this fact is not necessary for the proof of Proposition~\ref{prop:For-k}, we will not consider this problem here.}

Now, recall our object $\mathscr{F}$ of $\Per_{\mathscr{S}}(Z,\bk)$. The induction hypothesis provides a canonical isomorphism
\[
 \widetilde{j}^* (p_Z)^* \mathscr{F} [\dim G] \cong (p_U)^* j^* \mathscr{F} [\dim G] \xrightarrow{\sim} (a_U)^* j^* \mathscr{F} [\dim G] \cong \widetilde{j}^* (a_Z)^* \mathscr{F} [\dim G].
\]
On the other hand, for $g \in G$, the fiber of $E \bigl( (p_Z)^* \mathscr{F} [\dim G] \bigr)$, resp.~of $E \bigl( (a_Z)^* \mathscr{F} [\dim G] \bigr)$, at $g$ is $\coH^{\langle \lambda,2\rho\rangle}_c(S_\lambda \cap Z, \mathscr{F})$, resp.~$\coH^{\langle \lambda,2\rho\rangle}_c((g \cdot S_\lambda) \cap Z, \mathscr{F})$. If $k : Z \to \Gr_G$ is the embedding, then we have
\[
 \coH^{\langle \lambda,2\rho\rangle}_c(S_\lambda \cap Z, \mathscr{F}) \cong \coH^{\langle \lambda,2\rho\rangle}_c(S_\lambda, k_! \mathscr{F}) \cong \F_\lambda \bigl( {}^p \hspace{-1pt} \mathscr{H}^0( k_! \mathscr{F}) \bigr)
\]
(by the base change theorem and then Lemma~\ref{lem:Fmu-non-perverse}) and similarly
\[
 \coH^{\langle \lambda,2\rho\rangle}_c((g \cdot S_\lambda) \cap Z, \mathscr{F}) \cong \coH^{\langle \lambda,2\rho\rangle}_c(g \cdot S_\lambda, k_! \mathscr{F}) \cong \F^{gT}_\lambda \bigl( {}^p \hspace{-1pt} \mathscr{H}^0( k_! \mathscr{F}) \bigr),
\]
where $\F^{gT}_\lambda$ denotes the $\lambda$-weight functor constructed using the Torel $gTg^{-1} \subset gBg^{-1}$ (as in~\S\ref{ss:independence-Torel}).
The independence of the functor $\F_\lambda$ on the choice of Torel (see Remark~\ref{rmk:indep-k}) provides a canonical identification between these spaces, and then an isomorphism of local systems $E \bigl( (p_Z)^* \mathscr{F} [\dim G] \bigr) \xrightarrow{\sim} E \bigl( (a_Z)^* \mathscr{F} [\dim G] \bigr)$. The pair of isomorphisms we have constructed provides an isomorphism $\widetilde{E} \bigl( (p_Z)^* \mathscr{F} [\dim G] \bigr) \xrightarrow{\sim} \widetilde{E} \bigl( (a_Z)^* \mathscr{F} [\dim G] \bigr)$. Since $\widetilde{E}$ is fully faithful, we deduce that $(p_Z)^*\mathscr{F}$ and $(a_Z)^*\mathscr{F}$ are (canonically) isomorphic, which shows that $\mathscr{F}$ is $G$-equivariant.
\end{proof}

\subsection{The convolution bifunctor}
\label{ss:convolution-k}

Recall the setting of~\S\ref{ss:def-convolution}. If $\mathscr{F}$ and $\mathscr{G}$ are in $\Per_{\GO}(\Gr_G,\bk)$, the convolution product $\mathscr{F} \star \mathscr{G}$ is again defined by
\[
\mathscr{F} \star \mathscr{G} := m_*(\mathscr{F} \, \widetilde{\boxtimes} \, \mathscr{G}),
\]
but where now $\mathscr{F} \, \widetilde{\boxtimes} \, \mathscr{G}$ is defined by the property that
\[
q^*(\mathscr{F} \, \widetilde{\boxtimes} \, \mathscr{G}) = p^* \bigl( \pH^0(\mathscr{F} \lboxtimes_\bk \mathscr{G}) \bigr),
\]
where $\lboxtimes_\bk$ is now the \emph{derived} external tensor product over $\bk$. The same considerations as in~\S\ref{ss:exactness-convolution} (based on the use of stratified semismall maps) show that $\mathscr{F} \star \mathscr{G}$ is a perverse sheaf.

An associativity constraint for this bifunctor can be constructed as in~\S\ref{ss:associativity}, using the observation that
\[
\pH^0(\mathscr{F}_1 \lboxtimes_\bk \pH^0(\mathscr{F}_2 \lboxtimes_\bk \mathscr{F}_3)) \cong \pH^0(\mathscr{F}_1 \lboxtimes_\bk \mathscr{F}_2 \lboxtimes_\bk \mathscr{F}_3) \cong \pH^0( \pH^0(\mathscr{F}_1 \lboxtimes_\bk \mathscr{F}_2) \lboxtimes_\bk \mathscr{F}_3)
\]
for $\mathscr{F}_1$, $\mathscr{F}_2$, $\mathscr{F}_3$ in $\Per_{\GO}(\Gr_G,\bk)$.

Finally, the same considerations as in Section~\ref{sec:convolution-BD} apply in this generality, and lead to a description of this convolution bifunctor in terms of fusion and to the construction of a commutativity constraint (which we then modify as in~\S\ref{ss:change-commutativity}). In fact, the only change that is required is the replacement of the formula~\eqref{eqn:convolution-fusion} by an isomorphism
\begin{equation}
\label{eqn:convolution-fusion-k}
\tau^\circ(\mathscr F_1 \star \mathscr F_2) \cong i^\circ j_{!*}\,\bigl(\pH^0(\tau^\circ
\mathscr F_1\lboxtimes_\bk\tau^\circ\mathscr F_2)|_U\bigr).
\end{equation}

\subsection{Compatibility with the fiber functor}

In this subsection we study the compatibility of convolution with the functor $\F$ (considered either with values in finitely generated $\bk$-modules, or in $X_*(T)$-graded finitely generated $\bk$-modules). The proof will use the following lemma.

\begin{lem}
\label{lem:tensor-perverse}
If $\F(\mathscr{F})$ or $\F(\mathscr{G})$ is flat over $\bk$, then $\mathscr{F} \lboxtimes_\bk \mathscr{G}$ is perverse.
\end{lem}

\begin{proof}
By Lemma~\ref{lem:criterion-perv} (applied to the group $G \times G$ instead of $G$), it suffices to prove that
\[
\coH^k_c(S_{\nu_1} \times S_{\nu_2}, \mathscr{F} \lboxtimes_\bk \mathscr{G})=0
\]
unless $k=\langle 2\rho, \nu_1 + \nu_2 \rangle$, for any $\nu_1,\nu_2 \in X_*(T)$. However, if for $\mu \in X_*(T)$ we denote by $s_\mu : S_\mu \to \Gr_G$ the embedding, and by $\sigma_\mu : S_\mu \to \mathrm{pt}$ the projection, we have
\[
(\sigma_{\nu_1} \times \sigma_{\nu_2})_! (s_{\nu_1} \times s_{\nu_2})^*(\mathscr{F} \lboxtimes_\bk \mathscr{G}) \cong \bigl((\sigma_{\nu_1})_! (s_{\nu_1})^* \mathscr{F} \bigr) \lotimes_\bk \bigl((\sigma_{\nu_2})_! (s_{\nu_2})^* \mathscr{G} \bigr).
\]
(Here we use the compatibility of external products with $*$-pullback functors, which is easy, and with derived global sections with compact supports, which is proved in~\cite[Theorem~V.10.19]{borel}.)
By Proposition~\ref{prop:weight-functors-k} $(\sigma_{\nu_1})_! (s_{\nu_1})^* \mathscr{F}$ is isomorphic to a $\bk$-module shifted by $[-\langle 2\rho,\nu_1 \rangle]$ and $(\sigma_{\nu_2})_! (s_{\nu_2})^* \mathscr{F}$ is isomorphic to a $\bk$-module shifted by $[-\langle 2\rho,\nu_2 \rangle]$. Moreover, our assumption and Theorem~\ref{thm:fiber-functor-k} imply that one of these $\bk$-modules is flat; hence $(\sigma_{\nu_1} \times \sigma_{\nu_2})_! (s_{\nu_1} \times s_{\nu_2})^*(\mathscr{F} \lboxtimes_\bk \mathscr{G})$ is concentrated in degree $\langle 2\rho, \nu_1 + \nu_2 \rangle$, which finishes the proof.
\end{proof}

Our next task is to define a canonical isomorphism
\[
\F(\mathscr{A}_1 \star \mathscr{A}_2) \cong \F(\mathscr{A}_1) \otimes_\bk \F(\mathscr{A}_2)
\]
for $\mathscr{A}_1, \mathscr{A}_2$ in $\Per_{\GO}(\Gr_G,\bk)$.
The proof is similar to the one explained in~\S\ref{ss:F-convolution}, using the following lemma.

\begin{lem}
\label{lem:cohomology-product-k}
For $\mathscr{A}_1$, $\mathscr{A}_2$ in $\Per_{\GO}(\Gr_G,\bk)$, there exists a canonical isomorphism
 \begin{equation*}
\coH^\bullet \bigl( \Gr_G \times \Gr_G, \pH^0(\mathscr{A}_1 \lboxtimes_\bk \mathscr{A}_2) \bigr) \cong \coH^\bullet(\Gr_G, \mathscr{A}_1) \otimes_\bk \coH^\bullet(\Gr_G, \mathscr{A}_2).
\end{equation*}
\end{lem}

\begin{proof}
First, we construct a natural morphism from the right-hand side to the left-hand side. For this we consider $f \in \coH^n(\Gr_G, \mathscr{A}_1)$, considered as a morphism $\underline{\bk}_{\Gr_G} \to \mathscr{A}_1[n]$, and $g \in \coH^m(\Gr_G, \mathscr{A}_2)$, considered as a morphism $\underline{\bk}_{\Gr_G} \to \mathscr{A}_2[m]$. Then we can consider
\[
 f \lboxtimes_\bk g : \underline{\bk}_{\Gr_G \times \Gr_G} \to \mathscr{A}_1 \lboxtimes_\bk \mathscr{A}_2 [n+m].
\]
Now, since $\mathscr{A}_1 \lboxtimes_\bk \mathscr{A}_2$ is concentrated in nonpositive perverse degrees, we have a canonical (truncation) morphism $\mathscr{A}_1 \lboxtimes_\bk \mathscr{A}_2 \to \pH^0(\mathscr{A}_1 \lboxtimes_\bk \mathscr{A}_2)$. Composing $f \lboxtimes_\bk g$ with the shift of this morphism by $n+m$ provides the desired element of $\coH^{n+m} \bigl( \Gr_G \times \Gr_G, \pH^0(\mathscr{A}_1 \lboxtimes_\bk \mathscr{A}_2) \bigr)$.

We next prove that our morphism is an isomorphism.
If $\F(\mathscr{A}_1)$ 
is projective over $\bk$, then by Lemma~\ref{lem:tensor-perverse} the left-hand side identifies with $\coH^\bullet \bigl( \Gr_G \times \Gr_G, \mathscr{A}_1 \lboxtimes_\bk \mathscr{A}_2 \bigr)$. By the formula~\cite[Theorem~V.10.19]{borel} (already used in the proof of this lemma), we have
\[
 R\Gamma \bigl( \Gr_G \times \Gr_G, \mathscr{A}_1 \lboxtimes_\bk \mathscr{A}_2 \bigr) \cong R\Gamma(\Gr_G, \mathscr{A}_1) \lotimes_\bk R\Gamma(\Gr_G, \mathscr{A}_2).
\]
The cohomology of the left-hand side is $\coH^\bullet \bigl( \Gr_G \times \Gr_G, \mathscr{A}_1 \lboxtimes_\bk \mathscr{A}_2 \bigr)$. Now since $\F(\mathscr{A}_1)$ is projective, $R\Gamma(\Gr_G,\mathscr{A}_1)$ is isomorphic, in the derived category of $\bk$-modules, to its cohomology; it follows that the cohomology of the right-hand side is $\F(\mathscr{A}_1) \otimes_\bk \F(\mathscr{A}_2)$, and then that our morphism is an isomorphism.

To deduce the general case, we observe that by the results of Section~\ref{sec:repr-wt-func} below (see in particular the remarks at the end of~\S\ref{ss:construction-proj} and Proposition~\ref{prop:projective-struct}\eqref{it:projective-struct-3}) there exists an exact sequence
\[
\mathscr{F} \to \mathscr{G} \to \mathscr{A}_1 \to 0
\]
in $\Per_{\GO}(\Gr_G,\bk)$ where $\F(\mathscr{F})$ and $\F(\mathscr{G})$ are free over $\bk$. By right-exactness of the functor $\pH^0(\bm ? \lboxtimes_\bk \mathscr{A}_2)$, we deduce an exact sequence of perverse sheaves
\[
\pH^0(\mathscr{F} \lboxtimes_\bk \mathscr{A}_2) \to \pH^0(\mathscr{G} \lboxtimes_\bk \mathscr{A}_2) \to \pH^0(\mathscr{A}_1 \lboxtimes_\bk \mathscr{A}_2) \to 0.
\]
Since the functor $\coH^\bullet \bigl( \Gr_G \times \Gr_G, \bm ?)$ is exact on $\GO$-equivariant perverse sheaves (by Theorem~\ref{thm:fiber-functor-k} applied to the group $G \times G$), and the case already proven, we deduce an exact sequence
\begin{multline*}
\coH^\bullet(\Gr_G, \mathscr{F}) \otimes_\bk \coH^\bullet(\Gr_G, \mathscr{A}_2) \to \coH^\bullet(\Gr_G, \mathscr{G}) \otimes_\bk \coH^\bullet(\Gr_G, \mathscr{A}_2) \\
\to \coH^\bullet \bigl( \Gr_G \times \Gr_G, \pH^0(\mathscr{A}_1 \lboxtimes_\bk \mathscr{A}_2) \bigr) \to 0.
\end{multline*}
Comparing with the exact sequence obtained by applying the functor $\bm ? \otimes_\bk \coH^\bullet(\Gr_G, \mathscr{A}_2)$ to the exact sequence
\[
\coH^\bullet(\Gr_G,\mathscr{F}) \to \coH^\bullet(\Gr_G,\mathscr{G}) \to \coH^\bullet(\Gr_G,\mathscr{A}_1) \to 0,
\]
we finally deduce that our morphism is an isomorphism in general.
\end{proof}

We also have the following generalization of Proposition~\ref{prop:F-mu-tensor-char-0}, where we denote by $\Mod_\bk(X_*(T))$ the category of finitely generated $X_*(T)$-graded $\bk$-modules.

\begin{prop}
\label{prop:F-mu-tensor}
The functor
\[
\bigoplus_{\mu \in X_*(T)} \F_\mu : \Per_{\GO}(\Gr_G,\bk) \to \Mod_\bk(X_*(T))
\]
sends the convolution product $\star$ to the tensor product of $X_*(T)$-graded $\bk$-modules, in a way compatible with the associativity and commutativity constraints.
\end{prop}


Here again, the proof is similar to that of Proposition~\ref{prop:F-mu-tensor-char-0}, except that now we have to provide a canonical isomorphism
\[
\coH^\bullet_c
\bigl( S_{\mu_1}\times S_{\mu_2}, \pH^0(\mathscr A_1\lboxtimes_\bk \mathscr A_2) \bigr) \cong \coH^\bullet_c(S_{\mu_1}, \mathscr{A}_1) \otimes_\bk \coH^\bullet_c(S_{\mu_2}, \mathscr{A}_2)
\]
for $\mathscr{A}_1, \mathscr{A}_2$ in $\Per_{\GO}(\Gr_G,\bk)$ and $\mu_1, \mu_2 \in X_*(T)$. The proof is completely similar to that of Lemma~\ref{lem:cohomology-product-k}.

\section{Study of standard and costandard sheaves}

\subsection{Definitions}
\label{ss:def-J!-J*}

Recall that for any $\lambda \in X_*(T)^+$ we denote by $j_\lambda : \Gr_G^\lambda \to \Gr_G$ the embedding. We set
\[
\cJ_!(\lambda, \bk) := \pH^0 \bigl( (j_\lambda)_! \underline{\bk}_{\Gr_G^\lambda} [\dim\Gr_G^\lambda] \bigr), \quad
\cJ_*(\lambda, \bk) := \pH^0 \bigl( (j_\lambda)_* \underline{\bk}_{\Gr_G^\lambda} [\dim\Gr_G^\lambda] \bigr).
\]

By adjunction there exists a canonical morphism of complexes
\[
(j_\lambda)_! \underline{\bk}_{\Gr_G^\lambda} [\dim\Gr_G^\lambda] \to (j_\lambda)_* \underline{\bk}_{\Gr_G^\lambda} [\dim\Gr_G^\lambda],
\]
hence a canonical morphism of perverse sheaves
\[
\cJ_!(\lambda, \bk) \to \cJ_*(\lambda, \bk),
\]
and we denote its image (in the abelian category of perverse sheaves) by $\cJ_{!*}(\lambda, \bk)$. It follows from the definition of the perverse t-structure that $(j_\lambda)_! \underline{\bk}_{\Gr_G^\lambda} [\dim\Gr_G^\lambda]$ is concentrated in perverse degrees $\leq 0$, hence that for any perverse sheaf $\mathscr{F}$ we have
\[
\Hom(\cJ_!(\lambda, \bk), \mathscr{F}) \cong \Hom \bigl( (j_\lambda)_!\underline{\bk}_{\Gr_G^\lambda} [\dim\Gr_G^\lambda], \mathscr{F} \bigr).
\]
In particular, using adjunction we see that $\cJ_!(\lambda, \bk)$ has no nonzero morphism to a perverse sheaf supported on $\overline{\Gr_G^\lambda} \smallsetminus \Gr_G^\lambda$. Similar arguments show that $\cJ_*(\lambda, \bk)$ has no nonzero morphism from a perverse sheaf supported on $\overline{\Gr_G^\lambda} \smallsetminus \Gr_G^\lambda$.

If $\bk$ is a field then $\cJ_{!*}(\lambda, \bk)$ is simple, and coincides with the object denoted $\IC_\lambda$ in Section~\ref{sec:semisimplicity}. If moreover $\bk$ has characteristic $0$, then the category $\Per_{\GO}(\Gr_G, \bk)$ is semisimple by Theorem~\ref{thm:semisimplicity}. In view of the properties of $\cJ_!(\lambda, \bk)$ and $\cJ_*(\lambda, \bk)$ recalled above, it follows that the canonical morphisms
\[
\cJ_!(\lambda, \bk) \twoheadrightarrow \cJ_{!*}(\lambda, \bk) \hookrightarrow \cJ_*(\lambda, \bk)
\]
are isomorphisms in this case.

Now we come back to the case of a general Noetherian commutative ring $\bk$ of finite global dimension. In view of the remarks above, the following result is a generalization of Proposition~\ref{prop:dim-weight-spaces}.

\begin{prop}
\label{prop:can-basis-k}
Let $\lambda,\mu\in X_*(T)$ with $\lambda$ dominant. 
Then the $\bk$-module
$\F_\mu \bigl( \cJ_!(\lambda,\bk) \bigr)$, resp.~$\F_\mu \bigl( \cJ_*(\lambda,\bk) \bigr)$,
is free, with a canonical basis parametrized by the irreducible components of $\Gr_G^\lambda\cap S_\mu$, resp.~of $\Gr_G^\lambda\cap T_\mu$.
\end{prop}

\begin{proof}
By Lemma~\ref{lem:Fmu-non-perverse}, we have
\[
\F_\mu(\cJ_!(\lambda,\bk)) \cong \coH^{\langle 2\rho,\mu \rangle}_c \bigl( S_\mu, (j_\lambda)_! \underline{\bk}_{\Gr^\lambda_G}[\langle 2\rho,\lambda \rangle] \bigr).
\]
Using the base change theorem, it is not difficult to check that there exists a canonical isomorphism
\[
\coH^{\langle 2\rho,\mu \rangle}_c \bigl( S_\mu, (j_\lambda)_! \underline{\bk}_{\Gr^\lambda_G}[\langle 2\rho,\lambda \rangle] \bigr) \cong \coH^{\langle 2\rho,\lambda+\mu \rangle}_c(\Gr_G^\lambda \cap S_\mu ; \bk).
\]
Since $\langle 2\rho,\lambda+\mu \rangle = 2\dim(\Gr^\lambda_G \cap S_\mu)$ (see Theorem~\ref{thm:orbits}), the right-hand side is free, with a canonical basis parametrized by irreducible components of $\Gr_G^\lambda \cap S_\mu$.

The case of $\cJ_*(\lambda,\bk)$ is similar, using the description of $\F_\mu$ as $\coH^{\langle 2\rho,\mu \rangle}_{T_\mu}(\Gr_G,\bm ?)$, and Borel--Moore homology instead of cohomology with compact supports.
\end{proof}

\begin{rmk}
 More generally, if $M$ is a finitely generated $\bk$-module, we can consider the perverse sheaves
 \[
\cJ_!(\lambda, M) := \pH^0 \bigl( (j_\lambda)_! \underline{M}_{\Gr_G^\lambda} [\dim\Gr_G^\lambda] \bigr), \quad
\cJ_*(\lambda, M) := \pH^0 \bigl( (j_\lambda)_* \underline{M}_{\Gr_G^\lambda} [\dim\Gr_G^\lambda] \bigr).
\]
Considerations similar to those in the proof of Proposition~\ref{prop:can-basis-k} show that there exist canonical isomorphisms
\begin{equation}
\label{eqn:std-sheaf-mod}
 \F_\mu(\cJ_!(\lambda, M)) \cong \F_\mu(\cJ_!(\lambda, \bk)) \otimes_\bk M, \quad \F_\mu(\cJ_*(\lambda, M)) \cong \F_\mu(\cJ_*(\lambda, \bk)) \otimes_\bk M
\end{equation}
for any $\mu \in X_*(T)$.
\end{rmk}

\subsection{Extension of scalars}

In the following proposition, we denote by
\[
\bk \lotimes_\Z ( \bm ? ) : \Db_{c,\GO}(\Gr_G, \Z) \to \Db_{c,\GO}(\Gr_G,\bk)
\]
the (derived) extension-of-scalars functor. (Note that this functor does not send perverse sheaves to perverse sheaves.) Below we will use this notation also for varieties other than $\Gr_G$.

\begin{prop}
\label{prop:extension-scalars-J!-J*}
For any $\lambda \in X_*(T)^+$, we have a canonical isomorphism
\[
\cJ_!(\lambda, \bk) \cong \bk \lotimes_\Z \cJ_!(\lambda, \Z), \quad \cJ_*(\lambda, \bk) \cong \bk \lotimes_\Z \cJ_*(\lambda, \Z).
\]
\end{prop}

\begin{proof}
As in the proof of Lemma~\ref{lem:tensor-perverse},
for $\mu \in X_*(T)$ we denote by $s_\mu : S_\mu \to \Gr_G$ the embedding, and by $\sigma_\mu : S_\mu \to \mathrm{pt}$ the projection. Then by definition we have
\[
\coH^\bullet_c(S_\mu, \bm ? ) \cong \coH^\bullet((\sigma_\mu)_! (s_\mu)^*( \bm ? )),
\]
where we identify the derived category of $\bk$-sheaves on $\mathrm{pt}$ with the derived category of $\bk$-modules.

By general considerations, 
we have 
\[
\bk \lotimes_\Z (\sigma_\mu)_! (s_\mu)^*( \bm ? ) \cong (\sigma_\mu)_! (s_\mu)^* \bigl( \bk \lotimes_\Z (\bm ?) \bigr).
\]
We apply this isomorphism to $\cJ_!(\lambda, \Z)$. In this case, by Proposition~\ref{prop:weight-functors-k} and Proposition~\ref{prop:can-basis-k} the complex $(\sigma_\mu)_! (s_\mu)^*(\cJ_!(\lambda, \Z))$ is isomorphic to the shift by $[-\langle 2\rho, \mu \rangle]$ of a free $\Z$-module. Hence $\bk \lotimes_\Z (\sigma_\mu)_! (s_\mu)^*(\cJ_!(\lambda, \Z))$ is concentrated in degree $\langle 2\rho,\mu \rangle$ which, by Lemma~\ref{lem:criterion-perv}, shows that $\bk \lotimes_\Z \cJ_!(\lambda, \Z)$ is a perverse sheaf.

Now, we clearly have
\[
(j_\lambda)^! \bigl( \bk \lotimes_\Z \cJ_!(\lambda, \Z) \bigr) \cong \underline{\bk}_{\Gr_G^\lambda}[\dim \Gr_G^\lambda].
\]
By adjunction we deduce a canonical morphism $(j_\lambda)_! \underline{\bk}_{\Gr_G^\lambda}[\dim \Gr_G^\lambda] \to  \bk \lotimes_\Z \cJ_!(\lambda, \Z)$, and then taking the $0$-th perverse cohomology we deduce a canonical morphism
\begin{equation}
\label{eqn:morph-J!}
\cJ_!(\lambda, \bk) \to \bk \lotimes_\Z \cJ_!(\lambda, \Z).
\end{equation}
Using Proposition~\ref{prop:can-basis-k} and the same kind of considerations as above, we see that this morphism induces an isomorphism
\[
\F_\mu \left( \cJ_!(\lambda, \bk) \right) \xrightarrow{\sim} \F_\mu \left( \bk \lotimes_\Z \cJ_!(\lambda, \Z) \right)
\]
for any $\mu \in X_*(T)$. By the faithfulness claim in Theorem~\ref{thm:fiber-functor-k}, this implies that~\eqref{eqn:morph-J!} is an isomorphism, and concludes the proof of the first isomorphism.

The proof of the isomorphism $\cJ_*(\lambda, \bk) \cong \bk \lotimes_\Z \cJ_*(\lambda, \Z)$ is similar.
\end{proof}

\begin{rmk}
 These results imply
 that there exists a canonical isomorphism $\mathbb{D}(\cJ_*(\lambda,\bk)) \cong \cJ_!(\lambda,\bk)$, where $\mathbb{D}$ is the Verdier duality functor on the category $\Db_{c,\GO}(\Gr_G,\bk)$. In fact, by general considerations $\mathbb{D}(\cJ_*(\lambda,\bk))$ is the $0$-th cohomology of $\mathbb{D}((j_\lambda)_* \underline{\Z}_{\Gr_G^\lambda} [\dim \Gr_G^\lambda]) \cong (j_\lambda)_! \underline{\Z}_{\Gr_G^\lambda} [\dim \Gr_G^\lambda]$ for the t-structure $p^+$ of~\cite[\S 3.3]{bbd}. Now consider the truncation triangle for the \emph{usual} perverse t-structure:
 \[
  {}^p\tau_{<0} \bigl( (j_\lambda)_! \underline{\Z}_{\Gr_G^\lambda} [\dim \Gr_G^\lambda] \bigr) \to (j_\lambda)_! \underline{\Z}_{\Gr_G^\lambda} [\dim \Gr_G^\lambda] \to \cJ_!(\lambda,\Z) \xrightarrow{[1]}.
 \]
By definition the left-hand side belongs to ${}^p \hspace{-1pt} \Db_{\mathscr{S}}(\Gr_G,\Z)^{<0}$, hence to ${}^{p^+} \hspace{-1.5pt} \Db_{\mathscr{S}}(\Gr_G,\Z)^{<0}$. On the other hand, the fact that $\bk \lotimes_\Z \cJ_!(\lambda, \Z)$ is a perverse sheaf (see Proposition~\ref{prop:extension-scalars-J!-J*}) shows that $\cJ_!(\lambda,\Z)$ in torsion-free; in view of~\cite[\S 3.3.4]{bbd} this implies that this object belongs to ${}^{p^+} \hspace{-1.5pt} \Db_{\mathscr{S}}(\Gr_G,\Z)^{\geq 0}$. Hence the triangle above is also the truncation triangle for the t-structure $p^+$; in other words we have
\[
 \cJ_!(\lambda,\Z) \cong {}^{p^+} \hspace{-2pt} \mathscr{H}^0 \bigl( (j_\lambda)_! \underline{\Z}_{\Gr_G^\lambda} [\dim \Gr_G^\lambda] \bigr) \cong \mathbb{D}(\cJ_*(\lambda,\bk)).
\]
(See~\cite[Proposition~8.1(c)]{mv} for a proof of this isomorphism which does not refer to the t-structure $p^+$.)
\end{rmk}

\subsection{Relation between integral standard and $\IC$-sheaves}

\begin{lem}
\label{lem:costd-simple-over-Z}
For any $\lambda \in X_*(T)^+$, the canonical surjection
\[
\cJ_!(\lambda, \Z) \to \cJ_{!*}(\lambda,\Z)
\]
is an isomorphism.
\end{lem}

\begin{proof}
The claim amounts to saying that the canonical morphism
\[
\cJ_!(\lambda, \Z) \to \cJ_*(\lambda, \Z)
\]
(see~\S\ref{ss:def-J!-J*}) is injective or, in view of Theorem~\ref{thm:fiber-functor-k}, that for any $\mu \in X_*(T)$ it induces an embedding
\[
\F_\mu(\cJ_!(\lambda, \Z)) \to \F_\mu(\cJ_*(\lambda, \Z)).
\]
However, since the left-hand side is free over $\Z$ by Proposition~\ref{prop:can-basis-k}, it suffices to prove that the induced morphism
\[
\mathbf{Q} \otimes_\Z \F_\mu(\cJ_!(\lambda, \Z)) \to \mathbf{Q} \otimes_\Z \F_\mu(\cJ_*(\lambda, \Z))
\]
is an embedding. By Proposition~\ref{prop:extension-scalars-J!-J*} and its proof, this morphism identifies with the morphism
\[
\F_\mu(\cJ_!(\lambda, \mathbf{Q})) \to \F_\mu(\cJ_*(\lambda, \mathbf{Q}))
\]
induced by the canonical morphism $\cJ_!(\lambda, \mathbf{Q}) \to \cJ_*(\lambda, \mathbf{Q})$. The latter morphism is an isomorphism (see~\S\ref{ss:def-J!-J*}), which concludes the proof.
\end{proof}

\section{Representability of the weight functors}
\label{sec:repr-wt-func}

In Section~\ref{sec:Tannakian}, we presented Deligne and Milne's proof
of (a version of) Tannakian reconstruction for rigid tensor $\mathbf k$-linear abelian
categories, where $\mathbf k$ is a field. One of the key steps is
Proposition~\ref{prop:reconstruction}, which established an equivalence,
for each object $X$ in an abelian $\mathbf k$-linear category $\mathscr C$
endowed with a $\mathbf k$-linear exact faithful functor $\omega : \mathscr{C} \to \Vect_\bk$,
between the abelian subcategory $\langle X\rangle$ generated by $X$
and the category $\Mod_{A_X}$ of modules over an algebra
$A_X\subset\End_{\mathbf k}(\omega(X))$.

The condition that $\mathbf k$ is a field, needed for the proof of
Proposition~\ref{prop:reconstruction}, is too restrictive for our current
setup. Mirkovi\'c and Vilonen choose therefore another approach. Rather
than an equivalence $\langle X\rangle\cong\Mod_{A_X}$ for each object
$X\in\mathscr C$, they produce a Morita equivalence
$\Per_{\GO}(Z,\bk)\cong\Mod_{A_Z(\bk)}$ for each closed subset
$Z\subset\Gr_G$ union of finitely many $\GO$-orbits.
Here $\Per_{\GO}(Z,\bk)$ is the subcategory of $\Per_{\GO}(\Gr_G,\bk)$ consisting of $\GO$-equivariant
perverse sheaves supported on $Z$ and $A_Z(\bk)$ is the (opposite algebra
of the) endomorphism algebra of a projective generator $P_Z(\mathbf k)$ of
$\Per_{\GO}(Z,\bk)$.

The aim of this section is to construct and study the objects $P_Z(\mathbf k)$.
Since the diagram
$$\xymatrix{\Per_{\GO}(Z,\bk)\ar[rr]^{\Hom(P_Z(\mathbf k),\bm?)}\ar[dr]_{\F}
&&\Mod_{A_Z(\mathbf k)}\ar[dl]^{\text{forget}}\\&\Mod_{\bk}&}$$
has to commute, we will choose $P_Z(\bk)$ so that it represents $\F$.

\subsection{Construction of projective objects}
\label{ss:construction-proj}

Let $Z$ be a closed subset of $\Gr_G$, union of finitely many $\GO$-orbits.
For $n\geq0$, we set $\mathcal O_n=\mathcal O/t^{n+1} \mathcal{O}$ and let
$G_{\mathcal O_n}$ be the complex algebraic group which represents the functor $R \mapsto G(R \otimes_\C \mathcal{O}_n)$. We choose (as we may) $n\in \Z_{\geq 0}$ large enough so that the $\GO$-action
on $Z$ factors through $G_{\mathcal O_n}$; then by definition (see~\S\ref{ss:appendix-Gr}) we have $\Db_{c,\GO}(Z,\bk) = \Db_{c,G_{\mathcal O_n}}(Z,\bk)$.

Let $\nu\in X_*(T)$. For any $\mathscr A\in\Per_{\GO}(Z,\bk)$, we have
$$\F_\nu(\mathscr A)=\coH^{\langle2\rho,\nu\rangle}_{T_\nu}(Z,\mathscr A)
=\Hom_{\Db_c(Z,\bk)}(i_!\underline{\bk}_{T_\nu\cap Z}[-\langle2\rho,\nu\rangle],\mathscr A),$$
where $i : T_\nu \cap Z \to Z$ is the embedding.
To represent the functor $\F_\nu$ on the category $\Per_{\GO}(Z,\bk)$,
we need to transform the nonequivariant object $i_!\underline{\bk}_{T_\nu\cap Z}
[-\langle2\rho,\nu\rangle]$ of $\Db_c(Z,\bk)$ into an object of $\Per_{\GO}(Z,\bk)$.
We do this using the ($!$-)induction functor (whose construction is recalled in~\S\ref{ss:appendix-ind}).

Concretely, we consider the commutative diagram
\begin{equation}
\label{eqn:diag-induction}
\vcenter{
\xymatrix{T_\nu\cap Z\ar[d]_i&G_{\mathcal O_n}\times(T_\nu\cap Z)\ar[d]\ar[l]\ar[r]^(.67){\tilde a}&Z\ar@{=}[d]\\Z&G_{\mathcal O_n}\times Z\ar[l]^(.6)p\ar[r]_(.6)a&Z,}
}
\end{equation}
where $a$ is the action map and $p$ is the projection, and we define
\begin{align*}
P_Z(\nu,\mathbf k):&=\pH^0\bigl(a_!p^!i_!\underline{\bk}_{T_\nu\cap Z}[-\langle2\rho,\nu\rangle]\bigr)\\
&\cong\pH^0\bigl(a_!p^*i_!\underline{\bk}_{T_\nu\cap Z}[2\dim(G_{\mathcal O_n})-\langle2\rho,\nu\rangle]\bigr)\\
&\cong\pH^0\bigl(\tilde a_!\underline{\bk}_{G_{\mathcal O_n}\times(T_\nu\cap Z)}[2\dim(G_{\mathcal O_n})-\langle2\rho,\nu\rangle]\bigr),
\end{align*}
the last equality being given by base change along the left (Cartesian) square in~\eqref{eqn:diag-induction}.

\begin{prop}
\label{prop:projective-constr}
The perverse sheaf $P_Z(\nu,\mathbf k)$ is a projective object of
$\Per_{\GO}(Z,\bk)$ that represents the weight functor $\F_\nu$.
\end{prop}

\begin{proof}
We set
$$\mathscr F:=a_!p^!i_!\underline{\bk}_{T_\nu\cap Z}[- \langle 2\rho, \nu \rangle].$$
For any $\mathscr A\in\Per_{\GO}(Z,\bk)$, we have by Lemma~\ref{lem:equiv-ind}
\begin{multline*}
\F_\nu(\mathscr A)=\coH^{\langle 2\rho,\nu \rangle}_{T_\nu}(Z,\mathscr A)
=\Hom_{\Db_c(Z,\bk)}(i_!\underline{\bk}_{T_\nu\cap Z}[-\langle 2\rho,\nu \rangle],\mathscr A)\\
\cong \Hom_{\Db_{c,G_{\mathcal O_n}}(Z,\bk)}(a_!p^!i_!\underline{\bk}_{T_\nu\cap Z}[- \langle 2\rho,\nu \rangle],\mathscr A)
=\Hom_{\Db_{c,G_{\mathcal O_n}}(Z,\bk)}(\mathscr F,\mathscr A).
\end{multline*}

We claim that $\mathscr F$ is concentrated in nonpositive perverse degrees.
Indeed, let $n$ be the largest integer such that $\pH^n(\mathscr F)\neq0$.
The second arrow in the truncation triangle
$${}^p\tau_{<n}\mathscr F\to\mathscr F\to\pH^n(\mathscr F)[-n]\xrightarrow{[1]}$$
is nonzero, so that
$$0\neq\Hom_{\Db_{c,G_{\mathcal O_n}}(Z,\bk)}(\mathscr F,\pH^n(\mathscr F)[-n])=\F_\nu(\pH^n(\mathscr{F})[-n])) = 
\coH^{\langle 2\rho,\nu \rangle-n}_{T_\nu}(Z,\pH^n(\mathscr F));$$
applying Lemma~\ref{lem:criterion-perv}, we deduce that $n=0$, proving our claim.

Our truncation triangle now reads
$${}^p\tau_{<0}\mathscr F\to\mathscr F\to P_Z(\nu,\mathbf k)\xrightarrow{[1]}.$$
For any $\mathscr A\in\Per_{\GO}(Z,\bk)$, we have a long exact sequence
\begin{align*}
\Hom_{\Db_{c,G_{\mathcal O_n}}(Z,\bk)}({}^p\tau_{<0}\mathscr F,\mathscr A[-1])&\to
\Hom_{\Db_{c,G_{\mathcal O_n}}(Z,\bk)}(P_Z(\nu,\mathbf k),\mathscr A)\\&\to
\Hom_{\Db_{c,G_{\mathcal O_n}}(Z,\bk)}(\mathscr F,\mathscr A)\to
\Hom_{\Db_{c,G_{\mathscr O_n}}(Z,\bk)}({}^p\tau_{<0}\mathscr F,\mathscr A).
\end{align*}
By perverse degrees considerations, the first and the last spaces above are
zero; we conclude that we have a canonical isomorphism
$$\F_\nu(\mathscr A)=\Hom_{\Db_{c,G_{\mathcal O_n}}(Z,\bk)}(P_Z(\nu,\mathbf k),\mathscr A).$$ 
Thus $P_Z(\nu,\mathbf k)$ represents the functor $\F_\nu$ on $\Per_{\GO}(Z,\bk)$. Since
the latter is exact (see Lem\-ma~\ref{lem:weight-functors-exact}), $P_Z(\nu,\mathbf k)$ is projective.
\end{proof}

For a fixed $Z$, there are only finitely many $\nu\in X_*(T)$
such that $T_\nu\cap Z\neq\varnothing$ (see
Theorem~\ref{thm:orbits-T}\eqref{it:thm-orbits-T-1}), so that the sum
$$\bigoplus_{\nu\in X_*(T)}P_Z(\nu,\bk)$$
involves finitely many nonzero terms; it therefore defines an object $P_Z(\bk)$ of $\Per_{\GO}(Z,\bk)$. Theorem~\ref{thm:fiber-functor-k} and Proposition~\ref{prop:projective-constr} imply that $P_Z(\bk)$ represents the functor $\F$. Since $\F$ is exact, $P_Z(\bk)$ is projective. Since $\F$ is faithful, $P_Z(\bk)$ is a generator of the category $\Per_{\GO}(Z,\bk)$ (see e.g.\ \cite[chap.~II, \S1]{bass}).
Specifically, for each object $\mathscr A\in\Per_{\GO}(Z,\bk)$, there exists an epimorphism $P_Z(\bk)^n\twoheadrightarrow\mathscr A$ for some $n\geq0$ (because the $\bk$-module $\Hom_{\Per_{\GO}(Z,\bk)}(P_Z(\bk),\mathscr A)$ is finitely generated).

\subsection{Structure of the projective objects}
\label{ss:structure-proj}

Let $Y\subset Z$ be closed subsets of $\Gr_G$, unions of finitely many
$\GO$-orbits. Let $i:Y\to Z$ be the inclusion. The functor ${}^pi^*:=
\pH^0 \bigl( i^* (\bm ?) \bigr)$ maps $\Per_{\GO}(Z,\bk)$ to $\Per_{\GO}(Y,\bk)$ and is the left
adjoint to the inclusion $i_*:\Per_{\GO}(Y,\bk)\to\Per_{\GO}(Z,\bk)$.

\begin{prop}
\label{prop:projective-restr}
There exists a canonical isomorphism $P_Y(\bk)\cong{}^pi^*P_Z(\bk)$ and a canonical surjection
$P_Z(\bk)\twoheadrightarrow P_Y(\bk)$.
\end{prop}

\begin{proof}
Since $P_Z(\bk)$ represents $\F$ on $\Per_{\GO}(Z,\bk)$, its restriction
${}^pi^*P_Z(\bk)$ represents $\F$ on the subcategory $\Per_{\GO}(Y,\bk)$. Since $P_Y(\bk)$
also represents $\F$ on $\Per_{\GO}(Y,\bk)$, we get a canonical isomorphism
${}^pi^*P_Z(\bk)\xrightarrow{\sim} P_Y(\bk)$.

Composing with the adjunction morphism $P_Z(\bk)\to i_*\,{}^pi^*P_Z(\bk)$,
we get a canonical map $u:P_Z(\bk)\to i_*P_Y(\bk)$. Let $f:i_*P_Y(\bk)\to C$
be the cokernel of $u$. As a quotient of $i_*P_Y(\bk)$, the sheaf $C$ is
supported on $Y$, and since $i_*$ is full, $f$ is of the form $i_*g$ for
some map $g:P_Y(\bk)\to C'$ with $C=i_*C'$. Under the adjunction isomorphism
\[ 
\Hom_{\Per_{\GO}(Y,\bk)}(P_Y(\bk),C') \cong
\Hom_{\Per_{\GO}(Y,\bk)}({}^pi^*P_Z(\bk),C') \cong
\Hom_{\Per_{\GO}(Z,\bk)}(P_Z(\bk),i_*C'),
\]
$g$ goes to $(i_*g)\circ u=0$, hence $g=0$, and we conclude that $u$ is
surjective.
\end{proof}

\begin{prop}
\label{prop:projective-struct}
Let $Z$ be a closed subset of $\Gr_G$, union of finitely many $\GO$-orbits.
\begin{enumerate}
\item
\label{it:projective-struct-1}
The object $P_Z(\bk)$ admits a filtration in the abelian category $\Per_{\GO}(Z,\bk)$ parametrized by $\{\lambda \in X_*(T)^+ \mid \Gr_G^\lambda \subset Z\}$ (endowed with any total order refining $\leq$) and with subquotients isomorphic
to
\[
 \F(\cJ_*(\lambda,\bk))\otimes_{\bk}\cJ_!(\lambda,\bk).
\]
\item
\label{it:projective-struct-2}
There exists a canonical isomorphism
$P_Z(\bk)\cong \bk \lotimes_{\Z} P_Z(\Z)$.
\item
\label{it:projective-struct-3}
$\F(P_Z(\Z))$ is a finitely generated free $\Z$-module and we have $\F(P_Z(\bk))=\bk\otimes_{\Z}\F(P_Z(\Z))$.
\end{enumerate}
\end{prop}

\begin{proof}
The proof proceeds by induction on the number of $\GO$-orbits in $Z$.

Let us pick an orbit $\Gr_G^\lambda$ which is open in $Z$,
let $j:\Gr_G^\lambda\to Z$ be the inclusion, and set $Y=Z\smallsetminus
\Gr_G^\lambda$. Our goal is to analyze the kernel $K(\bk)$ of the surjection
constructed in Proposition~\ref{prop:projective-restr}:
\begin{equation}
\label{eqn:kern-prst}
0\to K(\bk)\to P_Z(\bk)\to P_Y(\bk)\to0.
\end{equation}

Let $M$ be a finitely generated $\bk$-module and let $\mathscr M:=\underline{M}_{\Gr_G^\lambda}[\langle2\rho,\lambda\rangle]$ be the shifted constant sheaf with stalk $M$ on $\Gr_G^\lambda$. From the truncation triangle
$$\cJ_*(\lambda,M)\to j_*\mathscr M
\to{}^p\tau_{>0}(j_*\mathscr M)\xrightarrow{[1]}$$
we get an embedding
$$\Ext^i_{\Db_{c,\GO}(Z,\bk)}(P_Y(\bk),\cJ_*(\lambda,M))\hookrightarrow\Ext^i_{\Db_{c,\GO}(Z,\bk)}(P_Y(\bk),j_*\mathscr M)$$
for $i\in\{0,1\}$, because $\Ext^{i-1}_{\Db_{c,\GO}(Z,\bk)}(P_Y(\bk),{}^p\tau_{>0}(j_*\mathscr M))=0$ for (perverse) degree reasons. Since, by adjunction,
$$\Ext^i_{\Db_{c,\GO}(Z,\bk)}(P_Y(\bk),j_*\mathscr M) \cong \Ext^i_{\Db_{c,\GO}(\Gr_G^\lambda,\bk)}(j^*(P_Y(\bk)),\mathscr M)=0,$$
we deduce that $\Ext^i_{\Per_{\GO}(Z,\bk)}(P_Y(\bk),\cJ_*(\lambda,M))=0$ for $i\in\{0,1\}$. Applying the functor
$$\Hom_{\Per_{\GO}(Z,\bk)}(\bm?,\cJ_*(\lambda,M))$$
to the exact sequence \eqref{eqn:kern-prst} and using this vanishing, we get an isomorphism
$$\Hom_{\Per_{\GO}(Z,\bk)}(P_Z(\bk),\cJ_*(\lambda,M))\cong
\Hom_{\Per_{\GO}(Z,\bk)}(K(\bk),\cJ_*(\lambda,M)),$$
and thus
\begin{align*}
\Hom_{\Db_c(\Gr_G^\lambda, \bk)}(j^*K(\bk),\mathscr M)
&\cong\Hom_{\Per_{\GO}(Z,\bk)}(K(\bk),{}^p\tau_{\leq0}(j_*\mathscr M))\\
&=\Hom_{\Per_{\GO}(Z,\bk)}(K(\bk),\cJ_*(\lambda,M))\\
&\cong\Hom_{\Per_{\GO}(Z,\bk)}(P_Z(\bk),\cJ_*(\lambda,M))\\
&\cong\F(\cJ_*(\lambda,M))\\
&\cong\F(\cJ_*(\lambda,\bk))\otimes_{\bk}M\tag*{by \eqref{eqn:std-sheaf-mod}}\\
&\cong\Hom_{\bk}(\F(\cJ_*(\lambda,\bk))^*,M)\tag*{by Proposition~\ref{prop:can-basis-k}.}
\end{align*}
Since $K(\bk)$ is an object of $\Per_{\GO}(Z,\bk)$ and $\Gr_G^\lambda$ is
open in $Z$, the restriction $j^*K(\bk)$ is a shifted local system on
$\Gr_G^\lambda$. From the isomorphism
$$\Hom_{\Db_c(\Gr_G^\lambda, \bk)}(j^*K(\bk),\mathscr M)\cong
\Hom_{\bk}(\F(\cJ_*(\lambda,\bk))^*,M)$$
proved above for any $M$, we deduce that
$$j^*K(\bk)\cong\F(\cJ_*(\lambda,\bk))^*\otimes_{\bk}\underline{\bk}_{\Gr_G^\lambda}[\langle2\rho,\lambda\rangle].$$
The adjunction map $j_!j^*K(\bk)\to K(\bk)$ then gives, after truncation
in nonnegative perverse degrees, a map
$$\alpha:\F(\cJ_*(\lambda,\bk))^*\otimes_{\bk}\cJ_!(\lambda,\bk)\to K(\bk)$$
in $\Per_{\GO}(Z,\bk)$ (see again \eqref{eqn:std-sheaf-mod}).

Since $j^*(\alpha)$ is an isomorphism,
the cokernel $C$ of $\alpha$ is supported on $Y$. Applying the functor
$\Hom_{\Per_{\GO}(Z,\bk)}(\bm?,C)$
to the exact sequence \eqref{eqn:kern-prst}, we get an exact sequence
\begin{multline*}
0\to\Hom_{\Per_{\GO}(Z,\bk)}(P_Y(\bk),C)
\xrightarrow\beta\Hom_{\Per_{\GO}(Z,\bk)}(P_Z(\bk),C)\\
\to\Hom_{\Per_{\GO}(Z,\bk)}(K(\bk),C)
\to\Ext^1_{\Per_{\GO}(Z,\bk)}(P_Y(\bk),C).
\end{multline*}
Since $C$ belongs to $\Per_{\GO}(Y,\bk)$, the map $\beta$ is an
isomorphism between two copies of $\F(C)$. Moreover, using~\eqref{eqn:Ext-heart} we have 
\begin{multline*}
 \Ext^1_{\Per_{\GO}(Z,\bk)}(P_Y(\bk),C) \cong \Ext^1_{\Db_{c,\GO}(Z,\bk)}(P_Y(\bk),C) \\
 \cong \Ext^1_{\Db_{c,\GO}(Y,\bk)}(P_Y(\bk),C) \cong \Ext^1_{\Per_{\GO}(Y,\bk)}(P_Y(\bk),C)=0
\end{multline*}
since $P_Y(\bk)$ is projective in $\Per_{\GO}(Y,\bk)$. It follows that
$\Hom_{\Per_{\GO}(Z,\bk)}(K(\bk),C)=0$, and therefore that $C=0$.
This shows that $\alpha$ is an epimorphism. We will see shortly
that it is in fact an isomorphism.

Let $K'(\bk)$ be the kernel of $\alpha$, so that we have an exact sequence
\begin{equation}
\label{eqn:syz-prst}
0\to K'(\bk)\to\F(\cJ_*(\lambda,\bk))^*\otimes_{\bk}\cJ_!(\lambda,\bk)\xrightarrow\alpha K(\bk)\to0
\end{equation}
in $\Per_{\GO}(Z,\bk)$. As for $C$ above, since $j^*(\alpha)$ is an isomorphism, $K'(\bk)$ is supported on $Y$.

Now we consider the case $\bk=\Z$.
As a consequence of Lemma~\ref{lem:costd-simple-over-Z} (see also the remarks in~\S\ref{ss:def-J!-J*}), the perverse sheaf $\cJ_!(\lambda,\Z)$ does not have any subobject supported on $Y$, and therefore $K'(\Z)=0$. Thus
$$K(\Z)\cong\F(\cJ_*(\lambda,\Z))^*\otimes_{\Z}\cJ_!(\lambda,\Z),$$
and from \eqref{eqn:kern-prst}, we easily get statement~\eqref{it:projective-struct-1} by induction in this case.

We come back to the general case. Since $\cJ_!(\lambda,\bk) \cong \bk \lotimes_\Z \cJ_!(\lambda,\Z)$ (see Proposition~\ref{prop:extension-scalars-J!-J*}), each object
$$\bk \lotimes_\Z \Bigl( \F(\cJ_*(\lambda,\Z))\otimes_{\Z}\cJ_!(\lambda,\Z) \Bigr) \cong \F(\cJ_*(\lambda,\bk))\otimes_{\bk}\cJ_!(\lambda,\bk)$$
is a perverse sheaf. The complex $\bk \lotimes_\Z P_Z(\Z)$ is thus an iterated extension (in the sense of triangulated categories) of perverse sheaves, and is therefore perverse. On the other hand, for each $\mathscr A\in\Per_{\GO}(Z,\bk)$, we have by~\cite[First formula in~(2.6.8)]{ks}
\begin{multline*}
\Hom_{\Per_{\GO}(Z,\bk)}(\bk \lotimes_\Z P_Z(\Z),\mathscr A)=
\Hom_{\Db_{c,\GO}(Z,\Z)}\bigl(P_Z(\Z),R\mathscr H\!om_\bk(\underline{\bk}_Z,\mathscr A)\bigr)\\
=\Hom_{\Per_{\GO}(Z,\Z)}(P_Z(\Z),\mathscr A)=\F(\mathscr A),
\end{multline*}
naturally in $\mathscr A$. Thus $P_Z(\bk)$ and $\bk \lotimes_\Z P_Z(\Z)$ both represent $\F$ on $\Per_{\GO}(Z,\bk)$, and therefore $P_Z(\bk)=\bk \lotimes_\Z P_Z(\Z)$, as claimed in
statement~\eqref{it:projective-struct-2}.

Using this description for $P_Z(\bk)$ and $P_Y(\bk)$ in
\eqref{eqn:kern-prst}, we get $K(\bk)=\bk \lotimes_\Z K(\Z)$. Turning to \eqref{eqn:syz-prst}, we see that $K'(\bk)=\bk \lotimes_\Z K'(\Z)$; and since $K'(\Z)=0$, we eventually get that $K'(\bk)=0$, or in other words that
$$K(\bk)\cong\F(\cJ_*(\lambda,\bk))^*\otimes_{\bk}\cJ_!(\lambda,\bk).$$
This information leads to statement~\eqref{it:projective-struct-1} by induction.

Finally, statement~\eqref{it:projective-struct-3} follows from the discussion above and Proposition~\ref{prop:can-basis-k} (since an extension of free $\Z$-modules is free).
\end{proof}


\subsection{Consequence: highest weight structure}

In this subsection we
assume that $\bk$ is a field. Recall the notion of highest weight category, whose definition is spelled out e.g.~in~\cite[Definition~7.1]{riche-hab}. (These conditions are obvious extensions of those considered in~\cite[\S 3.2]{bgs}, which are inspired by earlier work of Cline--Parshall--Scott~\cite{cps}.) Our goal in this subsection is to prove the following proposition.

\begin{prop}
\label{prop:highest-weight}
The category
$\Per_{\GO}(\Gr_G,\bk)$, together with the ``weight poset'' $(X_*(T)^+, \leq)$, the ``standard objects'' $\{ \cJ_!(\lambda,\bk) : \lambda \in X_*(T)^+\}$ and the ``costandard objects'' $\{ \cJ_*(\lambda,\bk) : \lambda \in X_*(T)^+\}$, is a highest weight category.
\end{prop}

\begin{proof}
 Condition~(1) in~\cite[Definition~7.1]{riche-hab} is obvious, and conditions~(2)--(4) are easily checked using adjunction and the general theory of perverse sheaves. Hence to conclude it suffices to prove that for any $\lambda,\mu \in X_*(T)^+$ we have $\Ext^2_{\Per_{\GO}(\Gr_G,\bk)}(\cJ_!(\lambda,\bk), \cJ_*(\mu,\bk))=0$. And for this it suffices to prove that for any finite closed union of $\GO$-orbits $Z \subset \Gr_G$ containing $\Gr_G^\lambda$ and $\Gr_G^\mu$ we have $\Ext^2_{\Per_{\GO}(Z,\bk)}(\cJ_!(\lambda,\bk), \cJ_*(\mu,\bk))=0$. Before that, let us note that we have
 \begin{equation}
 \label{eqn:Ext1-I!-I*}
  \Ext^1_{\Per_{\GO}(Z,\bk)}(\cJ_!(\lambda,\bk), \cJ_*(\mu,\bk))=0.
 \end{equation}
In fact, using~\eqref{eqn:Ext-heart} we can assume that $Z=\overline{\Gr_G^\lambda} \cup \overline{\Gr_G^\mu}$. Then the vanishing follows from the fact that either $\cJ_!(\lambda,\bk)$ is projective (if $\mu \not > \lambda$) or $\cJ_*(\mu,\bk)$ is injective (if $\lambda \not > \mu$) in $\Per_{\GO}(\overline{\Gr_G^\lambda} \cup \overline{\Gr_G^\mu},\bk)$.

 We denote by $Q_{Z,\lambda}$ the projective cover of the simple object $\cJ_{!*}(\lambda, \bk)$ in the abelian category $\Per_{\GO}(Z,\bk)$. (This category is equivalent to the category of finite-dimensional modules over a finite-dimensional $\bk$-algebra, see~\S\ref{ss:abel-recons} below for details; in particular we can indeed consider projective covers.) We claim that $Q_{Z,\lambda}$ has a filtration with $\cJ_!(\lambda,\bk)$ at the top and with subquotients of the form $\cJ_!(\nu, \bk)$ for some $\nu$'s in $X_*(T)^+$.
 
 This property is true if $\Gr_G^\lambda$ is open in $Z$, since then $Q_{Z,\lambda}=\cJ_!(\lambda,\bk)$ by condition~(3) in the definition of a highest weight category. When $\Gr_G^\lambda$ is not open in $Z$, we proceed along the lines of the proof of Proposition~\ref{prop:projective-struct}. We note that $Q_{Z,\lambda}$ is a direct summand of $P_Z(\bk)$, for the latter is a projective generator of $\Per_{\GO}(Z,\bk)$. Let $\Gr_G^\nu \subset Z$ be an open $\GO$-orbit and set $Y:=Z \smallsetminus \Gr_G^\nu$. The short exact sequence~\eqref{eqn:kern-prst} then induces a short exact sequence
$$0\to K'\to Q_{Z,\lambda}\to Q'_{Y,\lambda}\to0.$$
Here $Q'_{Y,\lambda}:=i_* \pH^0(i^* Q_{Z,\lambda})$, and $K'$ is a direct summand of the sheaf $K(\bk)$ in~\eqref{eqn:kern-prst}, which is a direct sum of copies of $\cJ_!(\nu, \bk)$. 
Since the perverse sheaf $\cJ_!(\nu,\bk)$ is indecomposable,
$K'$ must also be a direct sum of copies of $\cJ_!(\nu, \bk)$. Further, there is no nonzero map $\cJ_!(\nu, \bk)\to\cJ_{!*}(\lambda, \bk)$ (see~\S\ref{ss:def-J!-J*}), so the kernel of the covering map
$$Q_{Z,\lambda}\twoheadrightarrow\cJ_{!*}(\lambda, \bk)$$
contains $K'$, whence a surjective map
$$Q'_{Y,\lambda}\twoheadrightarrow\cJ_{!*}(\lambda, \bk).$$
Moreover $Q'_{Y,\lambda}$ is a direct summand of the term $P_Y$ appearing in~\eqref{eqn:kern-prst}, so is a projective object of $\Per_{\GO}(Y,\bk)$. Lastly, the projectivity of $Q_{Z,\lambda}$ gives a surjective map
$$\Hom_{\Per_{\GO}(Z,\bk)}(Q_{Z,\lambda},Q_{Z,\lambda})\twoheadrightarrow\Hom_{\Per_{\GO}(Z,\bk)}(Q_{Z,\lambda},Q'_{Y,\lambda}),$$
and since by adjunction we can identify
\begin{align*}
\Hom_{\Per_{\GO}(Z,\bk)}(Q_{Z,\lambda},Q'_{Y,\lambda})
&=\Hom_{\Per_{\GO}(Z,\bk)}(Q_{Z,\lambda},i_*\pH^0(i^*Q_{Z,\lambda}))\\
&\cong\Hom_{\Per_{\GO}(Y,\bk)}(\pH^0(i^*Q_{Z,\lambda}),\pH^0(i^*Q_{Z,\lambda}))\\
&=\Hom_{\Per_{\GO}(Y,\bk)}(Q'_{Y,\lambda},Q'_{Y,\lambda}),
\end{align*}
we obtain the existence of a surjective ring homomorphism
$$\End_{\Per_{\GO}(Z,\bk)}(Q_{Z,\lambda})\twoheadrightarrow\End_{\Per_{\GO}(Y,\bk)}(Q'_{Y,\lambda}).$$
Therefore $Q'_{Y,\lambda}$ has a local endomorphism ring, so is indecomposable. We finally conclude that $Q'_{Y,\lambda}$ can be identified with the projective cover $Q_{Y,\lambda}$ of $\cJ_{!*}(\lambda, \bk)$ in $\Per_{\GO}(Y,\bk)$. To sum up, we have a short exact sequence
$$0\to K'\to Q_{Z,\lambda}\to Q_{Y,\lambda}\to0,$$
where $K'$ is a direct sum of copies of $\cJ_!(\nu, \bk)$. Our claim now easily follows by induction on the number of $\GO$-orbits in $Z$.

At this point, we have shown the existence of a short exact sequence
 \[
  0 \to R_{Z,\lambda} \to Q_{Z,\lambda} \to \cJ_!(\lambda, \bk) \to 0
 \]
 such that $R_{Z,\lambda}$ admits a filtration with subquotients of the form $\cJ_!(\nu, \bk)$ for some $\nu$'s in $X_*(T)^+$. We then consider the exact sequence
 \[
  \Ext^1_{\Per_{\GO}(Z,\bk)}(R_{Z,\lambda},\cJ_*(\mu,\bk)) \to \Ext^2_{\Per_{\GO}(Z,\bk)}(\cJ_!(\lambda,\bk), \cJ_*(\mu,\bk)) \to \Ext^2_{\Per_{\GO}(Z,\bk)}(Q_{Z,\lambda}, \cJ_*(\mu, \bk))
 \]
 obtained by applying the functor $\Hom(\bm ?, \cJ_*(\mu,\bk))$ to this exact sequence. Here the first term vanishes because $\Ext^1_{\Per_{\GO}(Z,\bk)}(\cJ_!(\nu,\bk), \cJ_*(\mu,\bk))=0$ for any $\nu$, and the third term vanishes because $Q_\lambda$ is projective. We deduce the desired vanishing, and finally the proposition.
\end{proof}

\section{Construction of the group scheme}
\label{sec:construction}

In this section, we construct an affine $\bk$-group scheme $\widetilde G_{\bk}$
and an equivalence of monoidal categories $\Sat$ from $\Per_{\GO}(\Gr_G,\bk)$ to the category $\Rep_{\bk}(\widetilde G_{\bk})$ of representations of this group scheme on finitely generated $\bk$-modules. Along the way, we will show that the function algebra $\Z\left[\widetilde G_{\Z}\right]$ is a free $\Z$-module and that $\bk\left[\widetilde G_{\bk}\right] \cong \bk \otimes_{\Z} \Z\left[\widetilde G_{\Z}\right]$. (These facts will play a key role in Section~\ref{sec:identification} below.)

\subsection{Abelian reconstruction}
\label{ss:abel-recons}

Let us recall the following variant of Gabriel and Mitchell's theorem. Here we will denote by $\mathrm{mod}_\bk$ the category of \textit{all} (i.e.~not necessarily finitely generated) $\bk$-modules.

\begin{prop}
\label{prop:gabriel-mitchell}
Let $\mathscr C$ be a $\bk$-linear abelian category. Let $P$ be a projective object and let $A=\End_{\mathscr C}(P)$. Let $\mathrm{Modfp}_P$ be the full subcategory of $\mathscr C$ consisting of those objects that admit a presentation of the form $P_1\to P_0\to M\to0$, where $P_1$ and $P_0$ are direct sums of finitely many copies of $P$. Let also $\mathrm{Modfp}^r_A$ be the category of finitely presented right $A$-modules.
\begin{enumerate}
\item
\label{it:gm-1}
The functor $G=\Hom_{\mathscr C}(P,\bm?)$ defines an equivalence of categories from $\mathrm{Modfp}_P$ to $\mathrm{Modfp}^r_A$.
\item
\label{it:gm-2}
The endomorphism ring of the functor $G : \mathrm{Modfp}_P \to \mathrm{mod}_\bk$ is canonically isomorphic to $A^{\mathrm{op}}$.
\end{enumerate}
\end{prop}

\begin{proof}
Statement~\eqref{it:gm-1} is proved as in \cite[Proposition~II.2.5]{auslander-reiten-smalo}. 
The proof of~\eqref{it:gm-2} is similar to that of the corresponding claim in Proposition~\ref{prop:reconstruction}.
\end{proof}

Let $Z$ be a closed subset of $\Gr_G$, union of finitely many $\GO$-orbits. As we mentioned at the end of~\S\ref{ss:construction-proj}, each object in $\Per_{\GO}(Z,\bk)$ is a quotient of a module $P_Z(\bk)^n$, so each object $\mathscr A\in\Per_{\GO}(Z,\bk)$ admits a presentation of the form $P_1\to P_0\to\mathscr A\to0$ with $P_1$ and $P_0$ isomorphic to direct sums of finitely many copies of $P_Z(\bk)$. Moreover, the ring
\[
 A_Z(\bk):=\End_{\Per_{\GO}(Z,\bk)}(P_Z(\bk))^{\mathrm{op}}
\]
is a finitely generated $\bk$-module, hence is left Noetherian, so that each finitely generated left $A_Z(\bk)$-module is finitely presented. In the present situation, Proposition~\ref{prop:gabriel-mitchell} thus states that the functor $\F=\Hom_{\Per_{\GO}(Z,\bk)}(P_Z(\bk),\bm?)$ induces an equivalence of categories $\Sat_Z$, as depicted on the following diagram:
$$\xymatrix{\Per_{\GO}(Z,\bk)\ar[dr]_{\F}\ar[rr]_-{\sim}^-{\Sat_Z}&&\Mod_{A_Z(\bk)}\ar[dl]^{\text{forget}}\\&\Mod_{\bk}.&}$$

Let $i:Y\hookrightarrow Z$ be the inclusion of a closed subset, union of (finitely many) $\GO$-orbits. The perverse restriction functor
\[
{}^pi^* = \pH^0 \bigl(i^* (\bm ?) \bigr) : \Per_{\GO}(Z,\bk)\to\Per_{\GO}(Y,\bk)
\]
is left adjoint to the extension-by-zero functor $i_*:\Per_{\GO}(Y,\bk)\to\Per_{\GO}(Z,\bk)$. Further, this functor sends $P_Z(\bk)$ to $P_Y(\bk)$ (see Proposition~\ref{prop:projective-restr}) and thus induces a morphism of algebras $f_Y^Z$ from $A_Z(\bk)=\End_{\Per_{\GO}(Z,\bk)}(P_Z(\bk))^{\mathrm{op}}$ to $A_Y(\bk)=\End_{\Per_{\GO}(Y,\bk)}(P_Y(\bk))^{\mathrm{op}}$. By functoriality and adjointness, for each $\mathscr A\in\Per_{\GO}(Y,\bk)$, the action of an element $a\in A_Z(\bk)$ on
$$\Sat_Z(\mathscr A)=\Hom_{\Per_{\GO}(Z,\bk)}(P_Z(\bk),i_*\mathscr A)$$
coincides with the action of $f_Y^Z(a)\in A_Y(\bk)$ on
$$\Sat_Y(\mathscr A)=\Hom_{\Per_{\GO}(Y,\bk)}(P_Y(\bk),\mathscr A).$$
As a consequence, the diagram
$$\xymatrix{\Per_{\GO}(Y,\bk)\ar[d]_{i_*}\ar[rr]^{\Sat_Y}&&\Mod_{A_Y(\bk)}\ar[d]^{(f_Y^Z)^*}\\\Per_{\GO}(Z,\bk)\ar[dr]_{\F}\ar[rr]^{\Sat_Z}&&\Mod_{A_Z(\bk)}\ar[dl]^{\text{forget}}\\&\Mod_{\bk}&}$$
commutes, where $(f_Y^Z)^*$ is the restriction-of-scalars functor associated with $f_Y^Z$.

Since $A_Z(\bk) \cong \F(P_Z(\bk))$ is a finitely generated free $\bk$-module (see Proposition~\ref{prop:projective-struct}\eqref{it:projective-struct-3}), the same dictionary as the one set up in~\S\ref{ss:alg-coalg} can be used in the present context. Namely, we may  endow the dual $\bk$-module
\[
B_Z(\bk):=\Hom_\bk(A_Z(\bk),\bk)
\]
with the structure of a $\bk$-coalgebra and identify the category $\Mod_{A_Z(\bk)}$ with the category $\Comod_{B_Z(\bk)}$ of right $B_Z(\bk)$-comodules that are finitely generated over $\bk$. The dual of the algebra map $f_Y^Z:A_Z(\bk)\to A_Y(\bk)$ is a coalgebra map $B_Y(\bk)\to B_Z(\bk)$, and we can consider the limit $B(\bk)$ of the directed system of coalgebras thus constructed (over the poset of closed finite unions of $\GO$-orbits under inclusion).

\begin{prop}
\label{prop:base-change-B}
The $\Z$-module $B(\Z)$ is free, and we have a canonical isomorphism of $\bk$-coalgebras $B(\bk) \cong \bk \otimes_{\Z} B(\Z)$.
\end{prop}

\begin{proof}
The freeness assertion follows from Proposition~\ref{prop:projective-struct}\eqref{it:projective-struct-1} and its proof. The second assertion follows directly from Proposition~\ref{prop:projective-struct}\eqref{it:projective-struct-3}.
\end{proof}

We eventually get an equivalence of abelian categories $\Sat$ and a commutative diagram
$$\xymatrix{\Per_{\GO}(\Gr_G,\bk)\ar[rr]_-{\sim}^{\Sat}\ar[dr]_{\F}&&\Comod_{B(\bk)}\ar[dl]^{\text{forget}}\\&\Mod_{\bk}.&}$$

\subsection{Tannakian reconstruction}
\label{ss:tannaka-recons}

We now want to endow $B(\bk)$ with the structure of a Hopf algebra, and upgrade $\Sat$ to an equivalence of monoidal categories.

For $\lambda\in X_*(T)^+$, we set $Z_\lambda:=\overline{\Gr_G^\lambda}$ and we shorten the notation $P_{Z_\lambda}(\bk)$, $A_{Z_\lambda}(\bk)$ and $B_{Z_\lambda}(\bk)$ to respectively $P_\lambda(\bk)$, $A_\lambda(\bk)$ and $B_\lambda(\bk)$. We note that for $\lambda,\mu \in X_*(T)^+$, the perverse sheaf $\mathscr A \star \mathscr B$ belongs to $\Per_{\GO}(Z_{\lambda+\mu},\bk)$ whenever $\mathscr A\in\Per_{\GO}(Z_\lambda,\bk)$ and $\mathscr B \in\Per_{\GO}(Z_\mu,\bk)$.

An element $a\in A_{\lambda+\mu}(\bk)$ defines an endomorphism of the bifunctor
$$\Hom_{\Per_{\GO}(Z_{\lambda+\mu},\bk)}(P_{\lambda+\mu}(\bk),\bm? \star \bm?):\Per_{\GO}(Z_\lambda,\bk)\times\Per_{\GO}(Z_\mu,\bk)\to\Mod_{\bk}.$$
Now since $\F$ is a tensor functor, we have a canonical isomorphism of bifunctors
\begin{align*}
\Hom_{\Per_{\GO}(Z_{\lambda+\mu},\bk)}(P_{\lambda+\mu}(\bk),\bm? \star \bm?)&\cong\F(\bm? \star \bm?)\\&\cong\F(\bm?)\otimes_{\bk}\F(\bm?)\\&\cong\Hom_{\Per_{\GO}(Z_\lambda,\bk)}(P_\lambda(\bk),\bm?)\otimes_{\bk}\Hom_{\Per_{\GO}(Z_\mu,\bk)}(P_\mu(\bk),\bm?).
\end{align*}
By an immediate generalization of Proposition~\ref{prop:gabriel-mitchell}\eqref{it:gm-2}, our element $a$ thus defines an element of the ring $A_\lambda(\bk)\otimes_{\bk}A_\mu(\bk)$. This leads to 
a ring homomorphism
$$A_{\lambda+\mu}(\bk)\to A_\lambda(\bk)\otimes_{\bk}A_\mu(\bk).$$
Dualizing, we get a coalgebra map
$$B_\lambda(\bk)\otimes_{\bk}B_\mu(\bk)\to B_{\lambda+\mu}(\bk).$$
Taking the limit of these maps over $\lambda$ and $\mu$, this construction provides
a multiplication map on $B(\bk)$, which can be seen to be associative and commutative.

On the other hand, it is clear that $B_0(\bk)=\bk$, so that the natural morphism $B_0(\bk) \to B(\bk)$ defines a canonical element in $B(\bk)$ which is easily seen to be a unit. Altogether, we have thus constructed
a bialgebra structure on $B(\bk)$. Since our construction is based on natural transformations of functors, the functor $\Sat$ is easily seen to be compatible with the monoidal structures.

If we set
\[
\widetilde{G}_\bk := \Spec(B(\bk)),
\]
then the bialgebra structure on $B(\bk)$ translates to a structure of monoid scheme on $\widetilde{G}_\bk$. To conclude, what remains to show is that $B(\bk)$ admits an antipode, or in other words that $\widetilde{G}_\bk$ is a group scheme. Since, by Proposition~\ref{prop:base-change-B}, we have
\begin{equation}
\label{eqn:base-change-tG}
\widetilde{G}_\bk \cong \Spec(\bk) \times_{\Spec(\Z)} \widetilde{G}_\Z,
\end{equation}
it suffices to prove this when $\bk=\Z$. This will be done in Proposition~\ref{prop:tG-group} below.

\begin{lem}
\label{lem:comodules-rank-1}
Assume that $\bk=\Z$. If $M$ is an object of $\Per_{\GO}(\Gr_G,\Z)$ such that $\F(M)$ is free of rank $1$, then there exists $M^*$ in $\Per_{\GO}(\Gr_G,\Z)$ such that $M \star M^*$ is the unit object.
\end{lem}

\begin{proof}
Consider the object $\Q \lotimes_\Z M \in \Per_{\GO}(\Gr_G,\Q)$. This object is such that $\F(\Q \lotimes_\Z M)$ has dimension $1$, where here $\F$ means the tensor functor for coefficients $\Q$; as noticed at the beginning of Section~\ref{sec:identification-char-0}, this implies that $\Q \lotimes_\Z M \cong \cJ_{!*}(\lambda, {\Q})$ for some $\lambda \in X_*(T)^+$ orthogonal to all the roots of $G$, i.e.~such that $\Gr^\lambda_G=\{L_\lambda\}$.

By the results in~\S\ref{ss:abel-recons}, we have an embedding
\[
f:\Hom_{\Per_{\GO}(\Gr_G,\Z)}(M, \cJ_{!*}(\lambda,{\Z}))\hookrightarrow\Hom_\Z(\F(M), \F(\cJ_{!*}(\lambda,{\Z})))
\]
whose image is the set of all the $B(\Z)$-comodule maps. Since $\F(M)\cong\Z\cong\F(\cJ_{!*}(\lambda,{\Z}))$, the codomain of $f$ is a free $\Z$-module of rank~$1$. Therefore $\Hom_{\Per_{\GO}(\Gr_G,\Z)}(M, \cJ_{!*}(\lambda,{\Z}))$ is either $0$ or a free $\Z$-module of rank $1$; since
\[
\Q \otimes_\Z \Hom_{\Per_{\GO}(\Gr_G,\Z)}(M, \cJ_{!*}(\lambda,{\Z})) \cong \Hom_{\Per_{\GO}(\Gr_G,\Q)}(\Q \lotimes_\Z M, \cJ_{!*}(\lambda,{\Q}))=\Q,
\]
it is in fact free of rank $1$. We see moreover that the cokernel of $f$ is either $0$ or a cyclic group. Now if a nonzero multiple of a $\Z$-linear map $f:\F(M) \to \F(\cJ_{!*}(\lambda,{\Z}))$ is a morphism of $B(\Z)$-comodules, then the map $f$ itself is a morphism of comodules, because $\F(\cJ_{!*}(\lambda,{\Z}))\otimes_{\Z}B(\Z)$ is torsion-free. The cokernel of $f$ is therefore torsion-free, hence is zero. In other words, $f$ is an isomorphism, and any map in $\Hom_\Z(\F(M), \F(\cJ_{!*}(\lambda,{\Z})))$ is a $B(\Z)$-comodule map.

The image by $f^{-1}$ of an isomorphism of $\Z$-modules $\F(M) \xrightarrow{\sim} \F(\cJ_{!*}(\lambda,{\Z}))$ is thus an isomorphism $M \xrightarrow{\sim} \cJ_{!*}(\lambda, {\Z})$. One can then take $M^*:=\cJ_{!*}(-\lambda, {\Z})$.
\end{proof}

\begin{prop}
\label{prop:tG-group}
The monoid scheme $\widetilde{G}_\Z$ is a group scheme.
\end{prop}

\begin{proof}
First, we remark that if $M$ is a right $B(\Z)$-comodule which is free of rank $1$ over $\Z$, then Lemma~\ref{lem:comodules-rank-1} implies that $M$ is invertible in the monoidal category of $\widetilde{G}_\Z$-modules, hence that $\widetilde{G}_\Z(R)$ acts by invertible endomorphisms on $R \otimes_\Z M$, for any $\Z$-algebra $R$. As in the case of fields (see the proof of Theorem~\ref{thm:tannakina-reconstruction}), this implies the same claim for any right $B(\Z)$-comodule which is free of finite rank. Then, consider an arbitrary object $M$ in $\Comod_{B(\Z)}$. By~\cite[Proposition~3]{serre2}, there exist right $B(\Z)$-comodules $M'$ and $M''$ which are free of finite rank over $\Z$ and an exact sequence of $B(\Z)$-comodules
\[
 M'' \to M' \to M \to 0.
\]
Then for any $\Z$-algebra $R$ we have an exact sequence
\[
 R \otimes_\Z M'' \to R \otimes_\Z M' \to R \otimes_\Z M \to 0.
\]
Any element of $\widetilde{G}_\Z(R)$ acts on $R \otimes_\Z M''$ and $R \otimes_\Z M'$ by invertible endomorphisms by the case treated above; the 5-lemma implies that the same claim holds also for $M$.
This implies the proposition since the statement in Remark~\ref{rmk:Tannaka}\eqref{it:G-end-fiber-functor} holds in our present setting, see~\cite[Chap.~II, Scholie 3.1.1(3)]{sr}.
\end{proof}

%

\section{Identification of the group scheme}
\label{sec:identification}

\subsection{Statement and overview of the proof}
\label{ss:identification-statement}

In Section~\ref{sec:construction} we have constructed an affine $\bk$-group scheme $\widetilde{G}_\bk$ and an equivalence of monoidal categories
\[
 \Per_{\GO}(\Gr_G,\bk) \xrightarrow{\sim} \Rep_\bk(\widetilde{G}_\bk).
\]
Our goal now is to identify $\widetilde{G}_\bk$. 
To state this result we need some terminology. Recall that:
\begin{itemize}
\item
a \emph{reductive group} over a scheme $S$ is a smooth affine group scheme over $S$ all of whose geometric fibers are connected reductive algebraic groups; see~\cite[Expos\'e~XIX, D\'efinition~2.7]{sga3};
\item
a \emph{split torus} over $S$ is a group scheme which is isomorphic to a finite product of copies of the multiplicative group $\mathbb{G}_{\mathbf{m},S}$;
\item
a \emph{split maximal torus} of a  group scheme $H$ over $S$ is a closed subgroup scheme $K$ of $H$ which is a split torus and such that for any geometric fiber $\overline{s}$, the morphism $K_{\overline{s}} \to H_{\overline{s}}$ identifies $K_{\overline{s}}$ with a maximal torus of $H_{\overline{s}}$; see~\cite[Expos\'e~XIX, p.~10]{sga3}.
\end{itemize}
When $S=\mathrm{Spec}(\Z)$, it is known that a reductive group $H$ over $\mathrm{Spec}(\Z)$ which admits a split maximal torus is determined, up to isomorphism, by the root datum of $\Spec(\C) \times_{\mathrm{Spec}(\Z)} H$; see~\cite[Expos\'e XXIII, Corollaire~5.4]{sga3}. For such a group, if $\bk$ is an algebraically closed field, the root datum of $\Spec(\bk) \times_{\mathrm{Spec}(\Z)} H$ does not depend on $\bk$, and will be called the root datum of $H$.

When $\bk=\Z$, the answer to our question is provided by the following theorem.

\begin{thm}
\label{thm:description-gp}
The group scheme $\widetilde{G}_\Z$ is the unique reductive group over $\Z$ which admits a split torus and whose root datum is dual to that of $G$.
\end{thm}

In fact, below we will prove a slightly more precise result: we will construct a maximal torus of $\widetilde{G}_\Z$ whose group of characters identifies with $X_*(T)$, and show that the root datum of $\widetilde{G}_\Z$ with respect to this maximal torus is dual to the root datum of $(G,T)$.
For a general $\bk$, since $\widetilde{G}_\bk \cong \Spec(\bk) \times_{\Spec(\Z)} \widetilde{G}_\Z$ (see~\eqref{eqn:base-change-tG}), Theorem~\ref{thm:description-gp} determines
$\widetilde{G}_\bk$ also up to isomorphism. 

When $\bk$ is a field of characteristic $0$, this description\footnote{Note that in this setting there are two different groups that we have denoted $\widetilde{G}_\bk$: the one constructed in Section~\ref{sec:identification-char-0} using Tannakian reconstruction, and the one constructed ``by hand'' in Section~\ref{sec:construction}. These two groups are canonically identified thanks to~\cite[Theorem~X.1.2]{milne}.} has already been proved in Theorem~\ref{thm:identification-char-0}; this special case will play an important role in the proof below. In fact, a result of Prasad--Yu~\cite[Proposition~1.5]{py} ensures that a flat affine group scheme $H$ over $\Z$ such that $\Spec(\bk) \times_{\Spec(\Z)} H$ is a connected reductive group for any algebraically closed\footnote{As stated in~\cite{py}, the claim requires this property rather when $\bk$ is either $\Q$ or a finite field $\FF_p$. But an affine group scheme over a field is reductive iff its base change to an algebraic closure of the field is reductive; this follows from the fact that smoothness can be checked on this base change, see~\cite[Remark~6.30(2)]{gw}, and similarly for connectedness, see Footnote~\ref{fn:connectedness}.} field $\bk$, whose dimension is independent of $\bk$, is necessarily reductive. Hence what remains to be done is:
\begin{enumerate}
\item
\label{it:overview-k-1}
construct a subgroup scheme of $\widetilde{G}_\Z$ which is a split torus;
\item
\label{it:overview-k-2}
check that for an algebraic closure $\bk$ of a finite field, the group scheme $\widetilde{G}_\bk$ is reductive;
\item
\label{it:overview-k-3}
show that the base change to $\bk$ of our $\Z$-torus is a maximal torus of $\widetilde{G}_\bk$;
\item
\label{it:overview-k-4}
and finally, show that $\widetilde{G}_\bk$ has the appropriate root datum with respect to this maximal torus.
\end{enumerate}

Here~\eqref{it:overview-k-1} will be easy, and based on the same arguments as for fields of characteristic $0$, see~\S\ref{ss:identification-first}. The proof of~\eqref{it:overview-k-2}--\eqref{it:overview-k-4} will rely on another result of Prasad--Yu~\cite[Theorem~1.2]{py} which, in our setting, characterizes reductive group schemes over $\Z_p$ in terms of properties of their base change to $\Q_p$ and to an algebraic closure of $\FF_p$. (More precisely, this result will be needed to show that $\widetilde{G}_\bk$ is reduced; the other properties will be checked directly.)



\subsection{First properties}
\label{ss:identification-first}

For any $\bk$,
by construction $\widetilde{G}_\bk$ is an affine group scheme over $\bk$. Moreover, this group scheme is flat over $\bk$ by Proposition~\ref{prop:base-change-B}.


\begin{lem}
\label{lem:G-algebraic-connected}
If $\bk$ is a field, the group scheme $\widetilde{G}_\bk$ is algebraic and connected.
\end{lem}

\begin{proof}\footnote{This proof was suggested to us by G.~Williamson.}
By Proposition~\ref{prop:properties-G}\eqref{it:properties-G-1}, to prove that $\widetilde{G}_\bk$ is algebraic we need to exhibit a tensor generator of the category $\Rep_{\bk}(\widetilde{G}_\bk) \cong \Per_{\GO}(\Gr_G,\bk)$. 
By Proposition~\ref{prop:highest-weight}, the category $\Per_{\GO}(\Gr_G,\bk)$ has a natural highest weight structure.
Hence we can consider the \emph{tilting} objects in this category, namely those which admit both a filtration with subquotients of the form $\cJ_!(\lambda,\bk)$, and a filtration with subquotients of the form $\cJ_*(\lambda,\bk)$; see e.g.~\cite[\S 7.5]{riche-hab}. If we denote by $\mathrm{Tilt}_{\GO}(\Gr_G,\bk)$ the full subcategory of $\Per_{\GO}(\Gr_G,\bk)$ consisting of the tilting objects, then the indecomposable objects in $\mathrm{Tilt}_{\GO}(\Gr_G,\bk)$ are parametrized by $X_*(T)^+$ (see e.g.~\cite[Theorem~7.14]{riche-hab}), and the natural functor
\[
K^{\mathrm{b}} \mathrm{Tilt}_{\GO}(\Gr_G,\bk) \to \Db \Per_{\GO}(\Gr_G,\bk)
\]
is an equivalence of categories (see~\cite[Proposition~7.17]{riche-hab}). In particular, any object of $\Per_{\GO}(\Gr_G,\bk)$ is a subquotient of a tilting object.

Now, it is known that the subcategory $\mathrm{Tilt}_{\GO}(\Gr_G,\bk)$ is stable under the convolution bifunctor $\star$. In fact, consider the ``parity sheaves'' $\{\mathscr{E}_\lambda : \lambda \in X_*(T)^+\}$ in~$\Db_{\mathscr{S}}(\Gr_G,\bk)$ in the sense of Juteau--Mautner--Williamson~\cite{jmw} (for the constant pariversity). It follows from~\cite[Proposition~3.3]{jmw2} that if these objects are perverse, then they coincide with the tilting objects in $\Perv \cong \Per_{\GO}(\Gr_G,\bk)$. The fact that they are indeed perverse is proved in~\cite{jmw2} under certain technical conditions on $\mathrm{char}(\bk)$, and in~\cite[Corollary~1.6]{mr} under the assumption that $\mathrm{char}(\bk)$ is good for $G$.\footnote{Recall that a prime number $p$ is called \emph{bad} for $G$ if $p = 2$ and $\Delta(G,T)$ has a component not of type $\mathbf{A}$, or if $p = 3$ and $\Delta(G,T)$ has a component of type $\mathbf{E}$, $\mathbf{F}$ of $\mathbf{G}$, or finally if $p = 5$ and $\Delta(G,T)$ has a component of type $\mathbf{E}_8$. A prime number is called \emph{good} for $G$ if it is not bad for $G$.} This settles the question in this case, since convolution preserves parity complexes; see~\cite[Theorem~1.5]{jmw2}. The proof that $\mathrm{Tilt}_{\GO}(\Gr_G,\bk)$ is stable under convolution for a general field $\bk$ will appear in~\cite{bmrr}.

Finally we can conclude: if $(\lambda_1, \cdots, \lambda_n)$ is a finite generating subset of the monoid $X_*(T)^+$, and if $\mathscr{T}_i$ is the indecomposable tilting object attached to $\lambda_i$ for any $i \in \{1, \cdots, n\}$, then by support considerations we see that any indecomposable tilting object in $\Per_{\GO}(\Gr_G,\bk)$ is a direct summand of a tensor power of $\mathscr{T}_1 \oplus \cdots \oplus \mathscr{T}_n$, and therefore that $\mathscr{T}_1 \oplus \cdots \oplus \mathscr{T}_n$ is a tensor generator of the category $\Per_{\GO}(\Gr_G,\bk)$.

Once we know that $\widetilde{G}_\bk$ is algebraic, the fact that it is connected follows from Proposition~\ref{prop:properties-G}\eqref{it:properties-G-2}, using the same considerations as in the proof of Lemma~\ref{lem:tilde-G-connected-char-0}.
\end{proof}

\begin{rmk}\phantomsection
\label{rmk:group-field}
\begin{enumerate}
 \item 
 The algebraicity claim in Lemma~\ref{lem:G-algebraic-connected} is not proved in this way in~\cite{mv}. In fact, in order to apply the results of~\cite{py} we only need to know that the \emph{reduced} subgroup $(\widetilde{G}_\bk)_{\mathrm{red}}$ is of finite type, when $\bk$ is an algebraic closure of a finite field. The proof of this claim in the published version of~\cite{mv} is incomplete, but the authors have recently added in the arXiv version of their paper an appendix explaining how to fill this gap.
In any case, the prior knowledge of the fact that $\widetilde{G}_\bk$ is algebraic will allow us to simplify some later steps of the proof.
\item
\label{it:connected}
The fact that $\widetilde{G}_\bk$ is connected implies that $(\widetilde{G}_\bk)_{\mathrm{red}}$ is connected; see~\cite[\S 6.6]{waterhouse}.
\end{enumerate}
\end{rmk}

\begin{lem}
\label{lem:dim-geom-fiber}
If $\bk$ is an algebraic closure of a finite field, then the dimension of $\widetilde{G}_{\bk}$ is at most the dimension of the reductive $\bk$-group with root datum dual to that of $G$ (i.e.~$\dim(G)$).
\end{lem}

\begin{proof}
This property follows from the general fact that the dimension of fibers of a flat morphism of finite presentation is a lower semicontinuous function (on the target), see~\cite[Tag 0D4H]{stacks-project}. In more ``down to earth'' terms, one can argue as follows.
Let $p$ be the characteristic of $\bk$, and set
$d:=\dim(\widetilde{G}_{\bk})$. Then there exist $d$ algebraically independent functions $f_1, \cdots, f_d$ in $\bk \left[\widetilde{G}_{\bk} \right]$ (see e.g.~\cite[Theorem~5.22]{gw}). Since
\[
\bk \left[\widetilde{G}_{\bk} \right] = \bk \otimes_{\FF_p} \FF_p \left[\widetilde{G}_{\FF_p} \right],
\]
there exists a finite field $\FF \subset \bk$ such that each $f_i$ belongs to
\[
\FF \otimes_{\FF_p} \FF_p \left[\widetilde{G}_{\FF_p} \right] \cong \FF \left[\widetilde{G}_{\FF} \right].
\]
Let $\mathbf{O}$ be a finite extension of $\Z_p$ with residue field $\FF$. Then since
\[
\FF \left[\widetilde{G}_{\FF} \right] = \FF \otimes_{\mathbf{O}} \mathbf{O} \left[\widetilde{G}_{\mathbf{O}} \right],
\]
each $f_i$ can be lifted to a function $\tilde{f}_i \in \mathbf{O} \left[\widetilde{G}_{\mathbf{O}} \right]$. Since $\mathbf{O} \left[\widetilde{G}_{\mathbf{O}} \right]$ is torsion-free, the collection $\tilde{f}_1, \cdots, \tilde{f}_d$ does not satisfy any algebraic equation with coefficients in $\mathbf{O}$. Finally, since $\mathbf{O} \left[\widetilde{G}_{\mathbf{O}} \right]$ is free over $\mathbf{O}$, if $\mathbf{K}$ is the fraction field of $\mathbf{O}$ this collection is algebraically independent in
\[
\mathbf{K} \otimes_{\mathbf{O}} \mathbf{O} \left[\widetilde{G}_{\mathbf{O}} \right] \cong \mathbf{K} \left[\widetilde{G}_{\mathbf{K}} \right].
\]
Hence $d$ is at most $\dim(\widetilde{G}_{\mathbf{K}})$. We conclude using the fact that $\dim(\widetilde{G}_{\mathbf{K}})$ is the dimension of the split reductive $\bk$-group with root datum dual to that of $G$, see Theorem~\ref{thm:identification-char-0}.
\end{proof}

We finish this subsection with the following remark, valid for any ring $\bk$. We denote by $T^\vee_\bk$ the split $\bk$-torus whose group of characters is $X_*(T)$. Then, as in the case of fields of characteristic $0$ (see~\S\ref{ss:identification-char-0-first-step}), the weight functors define a canonical functor $\Per_{\GO}(\Gr_G,\bk) \to \Rep(T^\vee_\bk)$ sending convolution to tensor product and the functor $\F$ to the natural forgetful functor. In view of~\cite[Theorem~X.1.2]{milne} (compare with Proposition~\ref{prop:reconstruction-morph} and Proposition~\ref{prop:tannakian-morph}), this defines for any $\bk$-algebra $\bk'$ a group morphism $T^\vee_{\bk}(\bk') \to \widetilde{G}_\bk(\bk')$, or in other words a $\bk$-group scheme morphism $T^\vee_\bk \to \widetilde{G}_\bk$. Again as in the characteristic-$0$ case, for any $\lambda \in X_*(T)$ the free rank-$1$ $T^\vee_\bk$-module defined by $\lambda$ appears as a direct summand of the image of an object of $\Per_{\GO}(\Gr_G,\bk)$; considering matrix coefficients we deduce that $\lambda$ belongs to the image of the associated morphism $\bk [\widetilde{G}_{\bk} ] \to \bk[T^\vee_\bk]$. This shows that this morphism is surjective, i.e.~that the morphism $T^\vee_\bk \to \widetilde{G}_\bk$ is a closed embedding.


\subsection{Study of the group $( \widetilde{G}_{\bk} )_{\mathrm{red}}$ for $\bk$ an algebraic closure of a finite field}
\label{ss:identification-study}

In this subsection we fix a prime number $p$ and assume that $\bk$ is an algebraic closure of $\FF_p$. We study in detail the algebraic $\bk$-group scheme\footnote{Recall that if $H$ is a group scheme over a field $\FF$, the associated reduced scheme $H_{\mathrm{red}}$ is not necessarily a closed subgroup. But this is true if $\FF$ is perfect, which is the case here; see~\cite[\S VI.6]{milne}.} $( \widetilde{G}_{\bk} )_{\mathrm{red}}$.
Recall that this group is connected; see Remark~\ref{rmk:group-field}\eqref{it:connected}.
We also remark that the embedding $T^\vee_\bk \to \widetilde{G}_\bk$ factors through an embedding $T^\vee_\bk \to ( \widetilde{G}_\bk )_{\mathrm{red}}$ since $T^\vee_\bk$ is reduced. The goal of this subsection is to prove the following proposition.

\begin{prop}
\label{prop:identification-red}
The group scheme $( \widetilde{G}_{\bk} )_{\mathrm{red}}$ is a connected reductive group, $T^\vee_\bk$ is a maximal torus of this group, and the root datum of $( \widetilde{G}_{\bk} )_{\mathrm{red}}$ with respect to $T^\vee_\bk$ is dual to that of $(G,T)$.
\end{prop}

Note that Proposition~\ref{prop:identification-red} is sufficient to complete the program outlined in~\S\ref{ss:identification-statement}. Indeed,
once this result is proved, we will know that the group scheme $\widetilde{G}_{\Z_p}$ over $\Z_p$ satisfies the following conditions:
\begin{itemize}
\item
$\widetilde{G}_{\Z_p}$ is affine and flat over $\Z_p$ (see~\S\ref{ss:identification-first});
\item
the generic fiber $\widetilde{G}_{\Q_p} = \Spec(\Q_p) \times_{\Spec(\Z_p)} \widetilde{G}_{\Z_p}$ is connected and smooth over $\Q_p$ (see Theorem~\ref{thm:identification-char-0});
\item
the reduced geometric special fiber $(\widetilde{G}_{\bk})_{\mathrm{red}} = (\Spec(\bk) \times_{\Spec(\Z_p)} \widetilde{G}_{\Z_p})_{\mathrm{red}}$ is of finite type over $\bk$ (see Lemma~\ref{lem:G-algebraic-connected}) and its identity component $(\widetilde{G}_{\bk})_{\mathrm{red}}^\circ$ is a reductive group of the same dimension as $\widetilde{G}_{\Q_p}$ (see Proposition~\ref{prop:identification-red}).
\end{itemize}
In the terminology of~\cite{py}, this means that $\widetilde{G}_{\Z_p}$ is quasi-reductive. We will also know that
\begin{itemize}
 \item the root data of $\widetilde{G}_{\Q_p}$ and $(\widetilde{G}_{\bk})_{\mathrm{red}}^\circ$ coincide.
\end{itemize}
By~\cite[Theorem~1.2]{py}, it will follow that $\widetilde{G}_{\Z_p}$ is a reductive group over $\Z_p$.
This will imply in particular that $\widetilde{G}_\bk$ is reduced, hence that in Proposition~\ref{prop:identification-red} we can omit the subscript ``$\mathrm{red}$,'' and thus will finally prove the properties~\eqref{it:overview-k-2}--\eqref{it:overview-k-4} of~\S\ref{ss:identification-statement}.

The proof of Proposition~\ref{prop:identification-red} will be based on the same ideas as in
Section~\ref{sec:identification-char-0}, but with many additional difficulties.
We need some preparatory lemmas. 
We denote by $R$ the quotient of $( \widetilde{G}_{\bk} )_{\mathrm{red}}$ by its unipotent radical. Then the composition $T^\vee_\bk \to ( \widetilde{G}_{\bk} )_{\mathrm{red}} \to R$ is injective, so that we can also consider $T^\vee_\bk$ as a closed subgroup of $R$. 

\begin{lem}
\label{lem:identification-R-torus}
$T^\vee_\bk$ is a maximal torus of $R$.
\end{lem}

\begin{proof}
First, by~\cite[Corollary~III.3.6.4]{dg}, $\widetilde{G}_{\bk}$ is isomorphic, as a scheme, to the product (over $\Spec(\bk)$) of $( \widetilde{G}_{\bk} )_{\mathrm{red}}$ with a scheme of the form $\Spec(\bk[X_1, \cdots, X_r]/(X_1^{p^{n_1}}, \cdots, X_r^{p^{n_r}}))$ for some positive integers $n_1, \cdots, n_r$. It follows\footnote{See also~\cite[Corollary~VI.10.2]{milne} for a direct proof of this fact.} that for some $n$, the $n$-th Frobenius morphism $\mathrm{Fr}^n_{\widetilde{G}_{\bk}} : \widetilde{G}_{\bk} \to (\widetilde{G}_{\bk})^{(n)}$ (see e.g.~\cite[\S I.9.2]{jantzen}) factors through $\bigl( ( \widetilde{G}_{\bk} )_{\mathrm{red}} \bigr)^{(n)}$. Hence we can consider the diagram
\[
 \xymatrix@C=2cm{
 \widetilde{G}_{\bk} \ar[r]^-{\mathrm{Fr}^n_{\widetilde{G}_{\bk}}} \ar@{-->}[rd] & \bigl( \widetilde{G}_{\bk} \bigr)^{(n)} \\
 (\widetilde{G}_{\bk})_{\mathrm{red}} \ar@{^{(}->}[u] \ar[d] \ar[r]^-{\mathrm{Fr}^n_{(\widetilde{G}_{\bk})_{\mathrm{red}}}} & \bigl( (\widetilde{G}_{\bk})_{\mathrm{red}} \bigr)^{(n)} \ar@{^{(}->}[u] \ar[d] \\
 R \ar[r]^-{\mathrm{Fr}^n_R} & R^{(n)}.
 }
\]

Now, consider a simple representation $V$ of $R^{(n)}$, seen as a (simple) representation of the group $\bigl( ( \widetilde{G}_{\bk} )_{\mathrm{red}} \bigr)^{(n)}$. Our factorization above allows to see $V$ as a representation of $\widetilde{G}_\bk$. This representation is simple: in fact, its restriction to $( \widetilde{G}_{\bk} )_{\mathrm{red}}$ is simply the twist of $V$ by $\mathrm{Fr}^n_{(\widetilde{G}_{\bk})_{\mathrm{red}}}$, hence it is simple
by~\cite[Proposition~I.9.5]{jantzen}. 
In this way
 we obtain an injective ring morphism
\[
\Q \otimes_\Z K^0(\Rep_{\bk}(R^{(n)})) \hookrightarrow \Q \otimes_\Z K^0(\Rep_\bk(\widetilde{G}_\bk)).
\]
By~\cite[Theorem~5.22(3)]{gw}, this shows that
\[
\dim \Spec(\Q \otimes_\Z K^0(\Rep_\bk(\widetilde{G}_\bk))) \geq \dim \Spec (\Q \otimes_\Z K^0(\Rep_{\bk}(R^{(n)}))).
\]
Here, by~\eqref{eqn:rank-dim} the right-hand side is equal to $\mathrm{rk}(R^{(n)}) = \mathrm{rk}(R)$, and the left-hand side is equal to $\dim(T)$ (by the same considerations as in the characteristic-$0$ case, see~\S\ref{ss:identification-char-0-first-step}). Hence this inequality means that $\mathrm{rk}(R) \leq \dim(T^\vee_\bk)$, hence that $T^\vee_\bk$ is a maximal torus in $R$.
\end{proof}

Now we choose a Borel subgroup $\widetilde{B}$ of $R$ containing $T^\vee_\bk$ for which the sum $2\rho$ of the positive roots of $G$ is a dominant cocharacter (for the choice of positive roots given by the $T^\vee_\bk$-weights in the Lie algebra of $\widetilde{B}$). We then use the same notation as in~\S\ref{ss:identification-second-step} for roots and coroots of $G$ and $R$.

\begin{lem}
\label{lem:identification-pos-roots}
The set of dominant weights of $(R,T^\vee_\bk)$ relative to the system of positive roots $\Delta_+(R, \widetilde{B}, T^\vee_\bk)$ is $X_*(T)^+ \subset X_*(T)=X^*(T^\vee_\bk)$.
\end{lem}

\begin{proof}
For $\lambda \in X_*(T)$ a dominant weight relative to the system of positive roots $\Delta_+(R, \widetilde{B}, T^\vee_\bk)$, we denote by $L^R(\lambda)$ the corresponding simple $R$-module. Let $n$ be as in the proof of Lem\-ma~\ref{lem:identification-R-torus}. Then for $\lambda$ as above, the action of $R$ on $L^R(p^n \lambda)$ factors through an action of $R^{(n)}$ by Steinberg's theorem (see~\cite[Proposition~II.3.16]{jantzen}), hence this module determines a simple $\widetilde{G}_\bk$-module (see the proof of Lemma~\ref{lem:identification-R-torus}). The action of $T^\vee_\bk$ on this module is then determined by the character of the $R$-module $L^R(p^n \lambda)$.

On the other hand,
let $\mu \in X_*(T)^+$ be the dominant coweight of $G$ such that the simple perverse sheaf corresponding to $L^R(p^n \lambda)$ is $\cJ_{!*}(\mu,\bk)$.
Then we can write in the Grothendieck group of $\Per_{\GO}(\Gr_G,\bk)$
\[
[\cJ_{!*}(\mu,\bk)] = [\cJ_!(\mu, \bk)] + \sum_{\substack{\nu \in X_*(T)^+\\ \nu < \mu}} c_{\mu,\nu} \cdot [\cJ_!(\nu,\bk)]
\]
for some coefficients $c_{\mu,\nu} \in \Z$. This gives rise to a second way of expressing the action of $T^\vee_\bk$ on this $\widetilde{G}_\bk$-module using Proposition~\ref{prop:can-basis-k}. In particular, since the highest weight of $L^R(p^n \lambda)$ (considered as an $R$-module, with the choice of positive roots determined by $\widetilde{B}$) is a weight for which the function $\langle 2\rho, \bm ? \rangle$ attains its maximum, we must have $\mu=p^n \lambda$, so that $p^n \lambda$ belongs to $X_*(T)^+$, and finally $\lambda \in X_*(T)^+$.

On the other hand, let $\lambda \in X_*(T)^+$. Consider the simple $\widetilde{G}_\bk$-module $L^{\widetilde{G}_\bk}(\lambda)$ corresponding to the simple perverse sheaf $\cJ_{!*}(\lambda,\bk)$. The $T_\bk^\vee$-weights of this module, or equivalently of its restriction to $(\widetilde{G}_\bk)_{\mathrm{red}}$, can be estimated as above using Proposition~\ref{prop:can-basis-k}; in particular $\lambda$ is a weight of this module. Hence there exists a composition factor $M$ of the $(\widetilde{G}_\bk)_{\mathrm{red}}$-module $L^{\widetilde{G}_\bk}(\lambda)$ which admits $\lambda$ as a $T^\vee_\bk$-weight. Since $M$ is simple, the $(\widetilde{G}_\bk)_{\mathrm{red}}$-action factors through an $R$-action. Considering once again the values of the function $\langle 2\rho, \bm ? \rangle$, we see that $\lambda$ must be the highest weight of $M$, and thus that $\lambda$ is dominant with respect to the system of positive roots $\Delta_+(R, \widetilde{B}, T^\vee_\bk)$.
\end{proof}

As for~\eqref{eqn:root-directions}, Lemma~\ref{lem:identification-pos-roots} implies that
\begin{equation}
\label{eqn:dir-roots-coroots}
\bigl\{\mathbf Q_+\cdot\alpha : \alpha \in \Delta^\vee_{\mathrm{s}}(R,\widetilde{B},T^\vee_\bk)\bigr\}=
\bigl\{\mathbf Q_+\cdot\beta : \beta \in \Delta_{\mathrm{s}}(G,B,T)\bigr\}.
\end{equation}

\begin{lem}
\label{lem:identification-inclusion}
We have $\Z \cdot \Delta(R,T^\vee_\bk) \subset \Z \cdot \Delta^\vee(G,T)$ (in $X_*(T) = X^*(T^\vee_\bk)$).
\end{lem}

\begin{proof}
Recall that the connected components of $\Gr_G$ are in a natural bijection with the quotient $X_*(T) / \Z\Delta^\vee(G,T)$, see~\S\ref{ss:def-Gr}. Let $Z \subset T^\vee_\bk$ be the (scheme-theoretic) intersection of the kernels of all the elements in $\Z\Delta^\vee(G,T)$, so that $Z$ is a diagonalisable group scheme with $X^*(Z) \cong X_*(T) / \Z\Delta^\vee(G,T)$. Then any object of $\Per_{\GO}(\Gr_G,\bk)$ is naturally graded by the group of characters of $Z$, in a way compatible with 
the functor
\[
\Per_{\GO}(\Gr_G,\bk) \cong \Rep_\bk(\widetilde{G}_\bk) \to \Rep_\bk(Z)
\]
(where the second arrow is the forgetful functor). In particular, for any $\chi \in X^*(Z)$, the subspace of the left regular representation $\bk[\widetilde{G}_\bk]$ consisting of the functions $f$ satisfying $f(z^{-1}g)=\chi(z) f(g)$ is stable under the action of $\widetilde{G}_\bk$; hence $Z$ is a central subgroup of $\widetilde{G}_\bk$, and then its image in $R$ is central also. We deduce that all the roots of $(R,T^\vee_\bk)$ restrict trivially to $Z$, i.e.~that the morphism $X_*(T) \to X_*(T) / \Z\Delta^\vee(G,T)$ factors through $X_*(T) / \Z \cdot \Delta(R,T^\vee_\bk)$, whence the claim.
\end{proof}

%
%
%

\begin{lem}
\label{lem:Weyl-groups}
The Weyl groups of $(G,T)$ and of $(R,T^\vee_\bk)$, seen as groups of automorphisms of $X_*(T)$, together with their subsets of simple reflections, coincide.
\end{lem}

\begin{proof}
Recall that the Weyl group of $(G,T)$ is denoted by $W$. We also denote by $S \subset W$ the subset of simple reflections (i.e.~the reflections associated with simple roots).
We will denote by $W'$ the Weyl group of $(R,T^\vee_\bk)$, and by $S' \subset W'$ the subset of simple reflections. We fix $n$ as in the proof of Lemma~\ref{lem:identification-R-torus}.

For $\lambda \in X_*(T)^+$, we can recover the orbit $W' \cdot (p^n \lambda)$ as the set of extremal points of the convex polytope consisting of the convex hull of the weights of the simple $R$-module $L^R(p^n \lambda)$. Using the same considerations as in the proof of Lemma~\ref{lem:identification-R-torus} we see that this set coincides with the orbit $W\cdot (p^n \lambda)$, so that $W' \cdot \lambda = W \cdot \lambda$.

Now, we define an element of $X_*(T)^+$ to be \emph{regular} if its orbit under $W'$ (or equivalently under $W$) has the maximal possible cardinality, or equivalently if it is not orthogonal to any simple root of $(G,T)$, or equivalently if it is not orthogonal to any simple coroot of $(R,T^\vee_\bk)$. Then for $\lambda$ regular, we can recover the subset $\{s \cdot \lambda : s \in S\} \subset W \cdot \lambda$ as the subset consisting of elements $\mu$ such that the segment joining $\lambda$ to $\mu$ is also extremal in the convex hull of $W \cdot \lambda$. A similar description applies for $\{s' \cdot \lambda : s' \in S'\}$, from which we deduce that
\[
\{s \cdot \lambda : s \in S\} = \{s' \cdot \lambda : s' \in S'\}.
\]
This implies that $S=S'$: in fact if $s \in S$, then for any $\lambda \in X_*(T)^+$ regular there exists $s' \in S'$ such that $s \cdot \lambda = s' \cdot \lambda$, and $s'$ does not depend on $\lambda$ because the direction of $\lambda - s' \cdot \lambda$ is the line generated by the coroot of $G$ associated with $s$ and also the line generated by the root of $R$ associated with $s'$; then we have $s=s'$.

Finally, once we know that $S=S'$ we deduce that $W=W'$, since $W$, resp.~$W'$, is generated by $S$, resp.~$S'$.
\end{proof}

\begin{lem}
\label{lem:identification-inclusion-2}
We have $\Z \Delta(G,T) \subset \Z \Delta^\vee(R,T^\vee_\bk)$ in $X^*(T) = X_*(T^\vee_\bk)$. Moreover, if this inclusion is an equality the root datum of $(R,T^\vee_\bk)$ is dual to that of $(G,T)$.
\end{lem}

\begin{proof}
Let $\alpha \in \Delta_{\mathrm{s}}(G,B,T)$. By~\eqref{eqn:dir-roots-coroots}, we know that there exists $a \in \Q_+ \smallsetminus \{0\}$ such that $a \alpha \in \Delta^\vee_{\mathrm{s}}(R,\widetilde{B}, T^\vee_\bk)$. We can also consider the coroot $\alpha^\vee$ of $(G,T)$ associated with the root $\alpha$, and the root $(a \alpha)^\wedge$ of $(R,T^\vee_\bk)$ associated with the coroot $a \alpha$. By Lemma~\ref{lem:Weyl-groups}, we have
\[
\id - \langle \alpha^\vee,\bm? \rangle \alpha = \id - \langle (a \alpha)^\wedge, \bm? \rangle (a \alpha)
\]
as automorphisms of $X^*(T)=X_*(T^\vee_\bk)$; it follows that $(a \alpha)^\wedge = \frac{1}{a} \alpha^\vee$. On the other hand, Lemma~\ref{lem:identification-inclusion} shows that $(a \alpha)^\wedge \in \Z \Delta^\vee(G,T)$; hence $\frac{1}{a} \in \Z$, and $\alpha = \frac{1}{a}(a \alpha) \in \Z \Delta^\vee(R,T_\bk^\vee)$.

If the inclusion $\Z \Delta(G,T) \subset \Z \Delta^\vee(R,T^\vee_\bk)$ is an equality, then with the notation used above we must have $a=1$ for any $\alpha$; then $\Delta_{\mathrm{s}}(R,\widetilde{B},T^\vee_\bk) = \Delta^\vee_{\mathrm{s}}(G,B,T)$ and $\Delta^\vee_{\mathrm{s}}(R,\widetilde{B},T^\vee_\bk) = \Delta_{\mathrm{s}}(G,B,T)$, and the canonical bijections between simple roots and coroots of $R$ and of $G$ coincide. Taking orbits under the Weyl groups, it follows that $\Delta(R,T^\vee_\bk) = \Delta^\vee(G,T)$ and $\Delta^\vee(R,T^\vee_\bk) = \Delta(G,T)$, in a way compatible with the bijections between roots and coroots.
\end{proof}

\begin{lem}
\label{lem:identification-adjoint}
If $G$ is semisimple of adjoint type, then Proposition~{\rm \ref{prop:identification-red}} holds.
\end{lem}

\begin{proof}
If $G$ is semisimple of adjoint type, then $\Z \Delta(G,T)=X^*(T)$. It follows that the inclusion in Lemma~\ref{lem:identification-inclusion-2} is an equality, and then that the root datum of $R$ with respect to $T^\vee_\bk$ is dual to that of $(G,T)$. 

Then we conclude as follows:
of course we have $\dim(R) \leq \dim \bigl( ( \widetilde{G}_{\bk} )_{\mathrm{red}} \bigr) = \dim(\widetilde{G}_\bk)$. Lem\-ma~\ref{lem:dim-geom-fiber} and our determination of $\Delta(R,T^\vee_\bk)$ imply that this inequality is in fact an equality, so that $( \widetilde{G}_{\bk} )_{\mathrm{red}}=R$, and then the claim follows from our identification of the root datum of $R$.
\end{proof}


\begin{lem}
\label{lem:identification-ss}
If $G$ is semisimple, then Proposition~{\rm \ref{prop:identification-red}} holds.
\end{lem}

\begin{proof}
We assume that $G$ is semisimple.
Now that the claim is known if $G$ is of adjoint type (see Lemma~\ref{lem:identification-adjoint}),
we will in fact prove directly that $\widetilde{G}_\bk$ is a semisimple group with maximal torus $T^\vee_\bk$ and root datum dual to that of $(G,T)$.

Let $G_{\mathrm{ad}}$ be the adjoint quotient of $G$, and let $T_{\mathrm{ad}}$ be the image of $T$ in $G_{\mathrm{ad}}$. Then we can consider the group scheme $(\widetilde{G_{\mathrm{ad}}})_{\bk}$ constructed in Section~\ref{sec:construction} starting from the group $G_{\mathrm{ad}}$. By Lemma~\ref{lem:identification-adjoint} and the remarks following Proposition~\ref{prop:identification-red}, we know that $(\widetilde{G_{\mathrm{ad}}})_{\bk}$ is semisimple with root datum dual to that of $(G_{\ad}, T_{\ad})$. The morphism $G \to G_\ad$ induces a closed embedding $\Gr_G \hookrightarrow \Gr_{G_\ad}$, which then defines a group scheme morphism $(\widetilde{G_\ad})_\bk \to \widetilde{G}_\bk$ via Tannakian formalism.

The connected components of $\Gr_{G_{\mathrm{ad}}}$ are parametrized by $X_*(T_{\ad})/\Z\Delta^\vee(G_{\ad},T_{\ad})$, and $\Gr_G$ is the union of those corresponding to elements in the subset $X_*(T)/ \Z\Delta^\vee(G_{\ad},T_{\ad})$. (Here $\Delta^\vee(G_{\ad},T_{\ad})$ is included in $X_*(T)$, and identifies with $\Delta^\vee(G,T)$.) Hence if $Z \subset (T_\ad)^\vee_\bk$ is the (scheme-theoretic) intersection of the kernels of the elements of $X_*(T)$, so that $Z$ is a diagonalisable $\bk$-group scheme with $X^*(Z) \cong X_*(T_{\ad})/X_*(T)$, then any object $\mathscr{F}$ of $\Per_{G_{\ad,\mathcal{O}}}(\Gr_{G_\ad},\bk)$ admits a canonical grading $\mathscr{F}=\bigoplus_{\chi \in X^*(Z)} \mathscr{F}_\chi$, and using the equivalence of Proposition~\ref{prop:For-k} we see that $\Per_{\GO}(\Gr_G,\bk)$ identifies with the full subcategory of objects $\mathscr{F}$ such that $\mathscr{F}_\chi=0$ for $\chi \neq 1$. This means that $\widetilde{G}_\bk$ is the quotient of $(\widetilde{G_{\mathrm{ad}}})_{\bk}$ by the finite central subgroup scheme $Z$. Hence $\widetilde{G}_\bk$ is semisimple, and its root datum is dual to that of $(G,T)$.
\end{proof}

Finally, we conclude the proof of Proposition~\ref{prop:identification-red} with the following lemma.

\begin{lem}
\label{lem:identification-red}
Proposition~{\rm \ref{prop:identification-red}} holds for a general reductive group $G$.
\end{lem}

We will give two proofs of this lemma: the first one is a slightly expanded version of the proof given in~\cite{mv}, and the second one is new (to the best of our knowledge).

\begin{proof}[First proof of Lemma~{\rm \ref{lem:identification-red}}]
Here also, we will prove directly that $\widetilde{G}_\bk$ is reduced and reductive, and compute its root datum.

Let $Z(G)$ be the center of $G$, and set $H:=Z(G)^\circ$. Then $H$ is a torus and $G/H$ is a semisimple group; in particular the group $\widetilde{H}_\bk$ constructed as for $G$ is the $\bk$-torus dual to $H$, and $\widetilde{G/H}_\bk$ is the semisimple group dual to $G/H$.

The natural maps $H \hookrightarrow G$ and $G \twoheadrightarrow G/H$ induce morphisms
\[
\xymatrix{
\Gr_H \xrightarrow{i} \Gr_G \xrightarrow{\pi} \Gr_{G/H},
}
\]
which exhibit $\Gr_G$ as a trivial cover of $\Gr_{G/H}$ with fiber $\Gr_H$. (In fact, if we choose a lattice $Y \subset X_*(T)$ such that the composition $Y \hookrightarrow X_*(T) \twoheadrightarrow X_*(T/H)$ is an isomorphism, then $\Gr_{G/H}$ identifies with the union of the connected components of $\Gr_G$ corresponding to elements in $Y / \Z\Delta^{\vee}(G,T) \subset X_*(T) / \Z\Delta^{\vee}(G,T)$. Note that there exists an isomorphism of varieties $\Gr_G \cong \Gr_H \times \Gr_{G/H}$, but that in general such an isomorphism cannot be chosen to be compatible---in any reasonable sense---with the construction of the convolution product.) We have associated exact functors
\begin{equation}
\label{eqn:ses-tesnor-categories}
\Per_{H_{\mathcal{O}}}(\Gr_H, \bk) \xrightarrow{i_*} \Per_{\GO}(\Gr_G,\bk) \xrightarrow{\pi_*} \Per_{(G/H)_\mathcal{O}}(\Gr_{G/H},\bk),
\end{equation}
where $i_*$ is fully faithful and $\pi_*$ is essentially surjective. (Here we use Proposition~\ref{prop:For-k} to make sense of the functor $i_*$ as a functor between the categories of \emph{equivariant} perverse sheaves.) These functors are compatible with the monoidal structures and forgetful functors, hence induce group scheme morphisms
\begin{equation}
\label{eqn:ses-identification}
\widetilde{G/H}_\bk \to \widetilde{G}_\bk \to \widetilde{H}_\bk
\end{equation}
via Tannakian formalism.
If $\mathscr{F}$ is in $\Per_{\GO}(\Gr_G,\bk)$ and if we set $\mathscr{F}_0:=i_* \pH^0( i^* \mathscr{F})$, then $\mathscr{F}_0$ is a subobject of $\mathscr{F}$ and $\pi_*(\mathscr{F}_0) \subset \pi_* \mathscr{F}$ is the largest subobject isomorphic to a direct sum of copies of the unit object. This shows that~\eqref{eqn:ses-tesnor-categories} is an exact sequence of tensor categories in the sense of~\cite[Definition~3.7]{bn}; in view of~\cite[Remark~3.13]{bn} we deduce that~\eqref{eqn:ses-identification} is an exact sequence of $\bk$-group schemes. (Here the fact that the first morphism is a closed embedding can be seen using~\cite[Proposition~2.21(b)]{dm}, and the fact that the second morphism is a quotient morphism in the sense of~\cite[\S 15.1]{waterhouse} or~\cite[\S VII.7]{milne} follows from~\cite[Proposition~2.21(a)]{dm}; however exactness at the middle term is less obvious, in particular since it is not clear a priori that $\widetilde{G/H}_\bk$ is a normal subgroup. In fact, the property stated right after~\eqref{eqn:ses-identification} essentially guarantees this.)

We have just proved that $\widetilde{G}_\bk$ is an extension of $\widetilde{H}_\bk$ 
by $\widetilde{G/H}_\bk$.
Since both of these group schemes are smooth, by~\cite[Proposition~VII.10.1]{milne} this implies that $\widetilde{G}_\bk$ is also a smooth group, i.e.~that~\eqref{eqn:ses-identification} is an extension of $\bk$-algebraic groups in the ``traditional'' sense of e.g.~\cite{humphreys-gps}. The unipotent radical of $\widetilde{G}_\bk$ has trivial image in the torus $\widetilde{H}_\bk$, hence is included in $\widetilde{G/H}_\bk$; since the latter group is semisimple it follows that this unipotent radical is trivial, i.e.~that $\widetilde{G}_\bk$ is reductive.


Since $\widetilde{H}_\bk$ is commutative, $\widetilde{G/H}_\bk$ contains the derived subgroup of $\widetilde{G}_\bk$; and since $\widetilde{G/H}_\bk$ is semisimple it coincides with the derived subgroup of $\widetilde{G}_\bk$. The torus $(T/H)^\vee_\bk$ dual to $T/H$ embeds naturally in $T^\vee_\bk$, and identifies with a maximal torus in $\widetilde{G/H}_\bk$; hence the associated embedding $X_* \bigl( (T/H)^\vee_\bk \bigr) \hookrightarrow X_*(T^\vee_\bk)$ induces an isomorphism
\[
 \Z \Delta^\vee \bigl( \widetilde{G/H}_\bk, (T/H)^\vee_\bk \bigr) \xrightarrow{\sim} \Z \Delta^\vee \bigl( \widetilde{G}_\bk, T^\vee_\bk \bigr).
\]
On the other hand, in terms of $G$ this embedding identifies with the morphism $X^*(T/H) \to X^*(T)$ induced by the quotient morphism $T \to T/H$; hence it induces an isomorphism
\[
 \Z \Delta(G/H, T/H) \xrightarrow{\sim} \Z \Delta(G,T).
\]
Since the embedding
$\Z \Delta(G/H,T/H) \subset \Z \Delta^\vee \bigl( \widetilde{G/H}_\bk,(T/H)^\vee_\bk \bigr)$ of Lemma~\ref{lem:identification-inclusion-2} is known to be an equality, we deduce that the embedding $\Z \Delta(G,T) \subset \Z \Delta^\vee \bigl( \widetilde{G}_\bk, T^\vee_\bk \bigr)$ is an equality also, hence by Lemma~\ref{lem:identification-inclusion-2} that the root datum of $(\widetilde{G}_\bk, T^\vee_\bk)$ is dual to that of $(G,T)$.
\end{proof}

\begin{proof}[Second proof of Lemma~{\rm \ref{lem:identification-red}}]
 We again set $H=Z(G)^\circ$, and consider the quotient $G/H$ and the closed embedding $\widetilde{G/H}_\bk \hookrightarrow \widetilde{G}_\bk$. Since $\widetilde{G/H}_\bk$ is known to be reduced this embedding factors through $(\widetilde{G}_\bk)_{\mathrm{red}}$, and since $\widetilde{G/H}_\bk$ is semisimple the composition with the quotient morphism $(\widetilde{G}_\bk)_{\mathrm{red}} \to R$ is injective; hence $\widetilde{G/H}_\bk$ can (and will) be considered as a closed subgroup of $R$. Consider the subspaces
 \[
  \mathrm{Lie}(\widetilde{G/H}_\bk), \, \mathrm{Lie}(T^\vee_\bk) \subset \mathrm{Lie}(R),
 \]
 where $\mathrm{Lie}(\bm ?)$ means the Lie algebra.
 We have
 \begin{multline*}
  \mathrm{Lie}(\widetilde{G/H}_\bk) \cap \mathrm{Lie}(T^\vee_\bk) = \{x \in \mathrm{Lie}(\widetilde{G/H}_\bk) \mid \forall t \in T^\vee_\bk, \, t \cdot x=x\} \\
  \subset \{x \in \mathrm{Lie}(\widetilde{G/H}_\bk) \mid \forall t \in (T/H)^\vee_\bk, \, t \cdot x=x\} = \mathrm{Lie}((T/H)^\vee_\bk)
 \end{multline*}
(where the $\bk$-torus $(T/H)^\vee_\bk$ dual to $T/H$ is seen as a closed subgroup of $T^\vee_\bk$, and as the maximal torus of $\widetilde{G/H}_\bk$). We deduce that
\[
 \dim(\mathrm{Lie}(R)) \geq \dim \big(\mathrm{Lie}(\widetilde{G/H}_\bk) + \mathrm{Lie}(T^\vee_\bk) \bigr) \geq \dim(G/H)+\dim(H) = \dim(G).
\]
Since the left-hand side coincides with $\dim(R)$ (see~\cite[\S 12.2]{waterhouse}), which is at most $\dim((\widetilde{G}_\bk)_{\mathrm{red}})$, using Lemma~\ref{lem:dim-geom-fiber} we deduce that all the inequalities above are equalities. In particular, $(\widetilde{G}_\bk)_{\mathrm{red}} = R$ is reductive, and we have
\[
 \# \Delta(R,T^\vee_\bk) = \# \Delta \bigl( \widetilde{G/H}_\bk, (T/H)^\vee_\bk \bigr) = \# \Delta(G,T).
\]
This formula, together with Lemma~\ref{lem:Weyl-groups}, implies that if $\mathscr{D}(R)$ is the derived subgroup of $R$ we have
\[
 \dim(\mathscr{D}(R)) = \# \Delta(R,T^\vee_\bk) + \# \Delta_{\mathrm{s}}(R,T^\vee_\bk) = \dim(\widetilde{G/H}_\bk).
\]
Since $\widetilde{G/H}_\bk$ is semisimple it is included in $\mathscr{D}(R)$, which is connected (see~\cite[Theorem~10.2]{waterhouse}); hence this equality implies that $\widetilde{G/H}_\bk = \mathscr{D}(R)$.

Once this equality is known, we can conclude essentially as in the last part of the first proof: the embedding $X_* \bigl( (T/H)^\vee_\bk \bigr) \hookrightarrow X_*(T^\vee_\bk)$ induces an isomorphism $\Z \Delta^\vee \bigl( \widetilde{G/H}_\bk, (T/H)^\vee_\bk \bigr) \xrightarrow{\sim} \Z\Delta^\vee(R,T^\vee_\bk)$ and an isomorphism $\Z\Delta(G/H,T/H) \xrightarrow{\sim} \Z\Delta(G,T)$, which shows that the embedding
$\Z \Delta(G,T) \subset \Z \Delta^\vee(R,T^\vee_\bk)$ of Lemma~\ref{lem:identification-inclusion-2} is an equality, and then that the root datum of $(R,T^\vee_\bk)$ is dual to that of $(G,T)$. Since $R=(\widetilde{G}_\bk)_{\mathrm{red}}$, this concludes the proof of Lemma~\ref{lem:identification-red}.
\end{proof}

\begin{rmk}
From the point of view of Geometric Representation Theory, the most interesting case of the geometric Satake equivalence is when $\bk$ is an algebraically closed field.
As explained above, for this special case the results of~\cite{py} are required only to justify that the group scheme $\widetilde{G}_\bk$ is reduced. It would be desirable to find a direct justification for this fact (but we were not able to do so).
\end{rmk}

\section{Complement: restriction to a Levi subgroup}
\label{sec:Levi}

In this subsection we construct a geometric counterpart of the functor of restriction to a Levi subgroup, following~\cite[\S\S5.3.27--31]{beilinson-drinfeld}. This construction plays a key role in various applications of the geometric Satake equivalence, see e.g.~\cite{bg, small2}.

\subsection{The geometric restriction functor}

Let $P \subset G$ be a parabolic subgroup containing $B$, and let $L \subset P$ be the Levi factor containing $T$. If $B_L=B \cap L$, then $B_L$ is a Borel subgroup of $L$, and $P$ is determined by the subset $\Delta_{\mathrm{s}}(L,B_L,T) \subset \Delta_{\mathrm{s}}(G,B,T)$.

The embedding $L \hookrightarrow G$ induces a closed embedding $\Gr_L \hookrightarrow \Gr_G$, whose image identifies with the fixed points $(\Gr_G)^{Z(L)^\circ}$ (where $Z(L) \subset L$ is the center of $L$, and $Z(L)^\circ$ is the identity component of $ Z(L)$). In fact, choose a dominant cocharacter $\eta \in X_*(T)$ which is orthogonal to the simple roots in $\Delta_{\mathrm{s}}(L,B_L,T)$, but not to any other simple root. Then (the image of) $\Gr_L$ identifies with $(\Gr_G)^{\eta(\C^\times)}$. We will denote by $\mathscr{S}_L$ the stratification of $\Gr_L$ by $L_\mathcal{O}$-orbits.

The connected components of the affine Grassmannian $\Gr_L$ are in a canonical bijection with the quotient $X_*(T) / \Z \Delta^\vee(L,T)$; see~\S\ref{ss:def-Gr}. If $c$ belongs to this quotient, then we denote by $\Gr_L^c$ the corresponding connected component of $\Gr_L$ and we set
\begin{align*}
 S_c &:= \left\{x \in \Gr_G \left| \ \lim_{a \to 0} (\eta(a) \cdot x) \in \Gr_L^c \right. \right\}; \\
 T_c &:= \left\{x \in \Gr_G \left| \ \lim_{a \to \infty} (\eta(a) \cdot x) \in \Gr_L^c \right. \right\}.
\end{align*}
If $N_P \subset P$ is the unipotent radical and $N_P^- \subset G$ is the unipotent radical of the parabolic subgroup of $G$ which is opposite to $P$ with respect to $T$, then we have
\[
 S_c = (N_P)_{\mathcal{K}} \cdot \Gr_L^c, \qquad T_c = (N_P^-)_{\mathcal{K}} \cdot \Gr_L^c.
\]
We will denote by
\[
\Gr_G \xleftarrow{s_c} S_c \xrightarrow{\sigma_c} \Gr_L^c, \quad \Gr_G \xleftarrow{t_c} T_c \xrightarrow{\tau_c} \Gr_L^c
\]
the natural maps. 

If $\rho_L$ is the half sum of the positive roots of $L$ determined by $B_L$, then for any $\lambda \in \Delta^\vee(L,T)$ we have $\langle 2\rho-2\rho_L, \lambda \rangle = 0$. It follows that the pairing $\langle 2\rho-2\rho_L,c \rangle$ makes sense for $c \in X_*(T) / \Z \Delta^\vee(L,T)$.

\begin{lem}
\label{lem:rest-levi}
 For any $c \in X_*(T) / \Z \Delta^\vee(L,T)$ and any $\mathscr{F}$ in $\Per_{\GO}(\Gr_G,\bk)$, there exists a canonical isomorphism
 \[
  (\tau_c)_* (t_c)^! \mathscr{F} \xrightarrow{\sim} (\sigma_c)_! (s_c)^* \mathscr{F}
 \]
 in $\Db_{\mathscr{S}_L}(\Gr_L,\bk)$.
 Moreover, this complex is concentrated in perverse degree $\langle 2\rho-2\rho_L,c \rangle$.
\end{lem}

\begin{proof}
 As in the case $L=T$ (see Proposition~\ref{prop:weight-functors-k}), the isomorphism follows from Braden's hyperbolic localization theorem~\cite[Theorem~1]{braden}. If, for $\lambda \in X_*(T)$, we denote by $S_\lambda^L, T_\lambda^L \subset \Gr_L$ the semi-infinite orbits for the group $L$, then for any $\lambda \in c$ the base change isomorphism provides a canonical isomorphism
 \[
  \coH^{\bullet}_c \left( S_\lambda^L, (\sigma_c)_! (s_c)^* \mathscr{F} \right) \cong \coH^\bullet_c(S_\lambda,\mathscr{F}).
 \]
 By Lemma~\ref{lem:criterion-perv}, this implies that $(\sigma_c)_! (s_c)^* \mathscr{F} [-\langle 2\rho-2\rho_L,c \rangle]$ is a perverse sheaf, and finishes the proof.
\end{proof}

In view of this lemma, for $c \in X_*(T) / \Z \Delta^\vee(L,T)$ we consider the functor
\[
\F_c := (\sigma_c)_! (s_c)^*(\bm ?)[-\langle 2\rho-2\rho_L,c \rangle] : \Per_{\GO}(\Gr_G,\bk) \to \Per_{L_{\mathcal{O}}}(\Gr_L,\bk).
\]
We also set
 \[
  \mathsf{R}^G_L := \bigoplus_{c \in X_*(T) / \Z \Delta^\vee(L,T)} \F_c : \Per_{\GO}(\Gr_G,\bk) \to \Per_{L_{\mathcal{O}}}(\Gr_L,\bk).
 \]
 The arguments of Lemma~\ref{lem:rest-levi} provide, for any $\lambda \in X_*(T)$, a canonical isomorphism
 \begin{equation}
 \label{eqn:rest-levi-torus}
 \F^L_\lambda \circ \mathsf{R}^G_L \xrightarrow{\sim} \F_\lambda
 \end{equation}
 (where $\F_\lambda^L$ is the $\lambda$-weight functor for the group $L$). In particular, summing over $\lambda$ and using Theorem~\ref{thm:fiber-functor-k} we deduce a canonical isomorphism of functors.
 \[
 \F^L \circ \mathsf{R}^G_L \cong \F
 \]
 where $\F^L:=\coH^\bullet(\Gr_L, \bm ?)$.

\begin{prop}
\label{prop:geometric-restriction}
 The functor $ \mathsf{R}^G_L$
 sends the convolution product on $\Per_{\GO}(\Gr_G,\bk)$ to the convolution product on $\Per_{L_{\mathcal{O}}}(\Gr_L,\bk)$, in a way compatible with associativity and commutativity constraints. 
 \end{prop}

\begin{proof}
 Recall the objects considered in Section~\ref{sec:convolution-BD}. As in the proof of Proposition~\ref{prop:F-mu-tensor-char-0} (which was only concerned with the case $L=T$)
 we can consider ``relative'' versions $S_c(X) \subset \Gr_{G,X}$, $S_c(X^2) \subset \Gr_{G,X^2}$ of the varieties $S_c$, and denote the corresponding embeddings and projections by
 \begin{align*}
  \tilde{s}_c : S_c(X) \to \Gr_{G,X}, &\quad \tilde{\sigma}_c : S_c(X) \to \Gr^c_{L,X}, \\
  \tilde{s}_c^2 : S_c(X^2) \to \Gr_{G,X^2}, &\quad \tilde{\sigma}_c^2 : S_c(X^2) \to \Gr^c_{L,X^2},
 \end{align*}
 where $\Gr^c_{L,X}$ and $\Gr_{L,X^2}^c$ are the connected components of $\Gr_{L,X}$ and $\Gr_{L,X^2}$ defined by $c$.
 Here, for $x \in X$, the fiber of $S_c(X^2)$ over $(x,x) \in X^2$ is canonically identified with $S_c$, and the fiber over $(x_1,x_2)$ with $x_1 \neq x_2$ is canonically identified with $\bigsqcup_{c_1+c_2=c} S_{c_1} \times S_{c_2}$.
 
 Now, consider the diagram
 \begin{equation}
 \label{eqn:diag-rest-Levi}
 \vcenter{
  \xymatrix@C=1.5cm{
  (\Gr_{G,X} \times \Gr_{G,X})|_{U} \ar[r]^-{j} & \Gr_{G,X^2} & \Gr_{G,X} \ar[l]_-{i} \\
  \bigsqcup_{c_1+c_2=c} \bigl( S_{c_1}(X) \times S_{c_2}(X) \bigr)|_U \ar[u]^-{(\tilde{s}^2_c)|_U} \ar[d]_-{(\tilde{\sigma}^2_c)|_U} \ar[r]^-{j_c} & S_c(X^2) \ar[u]_-{\tilde{s}^2_c} \ar[d]^-{\tilde{\sigma}^2_c} & S_c(X) \ar[u]_-{\tilde{s}_c} \ar[d]^-{\tilde{\sigma}_c} \ar[l]_-{i_c} \\
  \bigsqcup_{c_1+c_2=c} \bigl( \Gr_{L,X}^{c_1} \times \Gr_{L,X}^{c_2} \bigr)|_U \ar[r]^-{j_L^c} & \Gr^c_{L,X^2} & \Gr_{L,X}^c \ar[l]_-{i_L^c},
  }
  }
 \end{equation}
 where $i_c$ and $j_c$ are the restrictions of $i$ and $j$. All the squares in this diagram are Cartesian by~\cite[Lemma~1.4.9]{drg}. Moreover, $(\tilde{s}^2_c)|_U$ identifies with the restriction to $U$ of the disjoint union of inclusions $\tilde{s}_{c_1} \times \tilde{s}_{c_2}$, and similarly for $\tilde{\sigma}^2_c$.
 
We fix $\mathscr{A}_1, \mathscr{A}_2$ in $\Per_{\GO}(\Gr_G,\bk)$. Then by~\eqref{eqn:convolution-fusion-k} we have
\[
\tau^\circ(\mathscr A_1 \star \mathscr A_2) \cong i^\circ j_{!*}\bigl( \pH^0(\tau^\circ
\mathscr A_1\lboxtimes_\bk \tau^\circ\mathscr A_2)|_U\bigr).
\]
We set
\[
\tilde{\F}_c := (\tilde{\sigma}_c)_! (\tilde{s}_c)^*(\bm ?)[-\langle 2\rho-2\rho_L,c \rangle], \quad \tilde{\F}_c^2 := (\tilde{\sigma}^2_c)_! (\tilde{s}^2_c)^*(\bm ?)[-\langle 2\rho-2\rho_L,c \rangle].
\]
Then on the one hand we have
\begin{equation}
\label{eqn:tildeF-1}
\tilde{\F}_c \bigl( \tau^\circ(\mathscr A_1 \star \mathscr A_2) \bigr) \cong (\tau_L)^\circ(\F_c(\mathscr A_1 \star \mathscr A_2)),
\end{equation}
and on the other hand we have
\[
\tilde{F}_c \Bigl( i^\circ j_{!*}\bigl( \pH^0 (\tau^\circ
\mathscr A_1\lboxtimes_\bk \tau^\circ\mathscr A_2)|_U\bigr) \Bigr) \cong (i^c_L)^\circ \Bigl( \tilde{F}_c^2 \circ j_{!*}\bigl( \pH^0(\tau^\circ
\mathscr A_1\lboxtimes_\bk \tau^\circ\mathscr A_2)|_U\bigr)  \Bigr)
\]
by the base change theorem. We claim that
\begin{multline}
\label{eqn:tildeF-2}
\tilde{F}_c^2 \circ j_{!*}\bigl( \pH^0(\tau^\circ
\mathscr A_1\lboxtimes_\bk \tau^\circ\mathscr A_2)|_U\bigr) \\
\cong (j^c_L)_{!*}\left( \bigoplus_{c_1+c_2=c} \pH^0 \bigl( (\tau_L)^\circ
\F_{c_1}(\mathscr A_1) \lboxtimes_\bk (\tau_L)^\circ \F_{c_2}(\mathscr A_2) \bigr)|_U \right).
\end{multline}
In fact, to check this it suffices to prove that the left-hand side satisfies the properties~\eqref{eqn:characterization-IC} which characterize the right-hand side. The isomorphism over $U$ follows from the base change theorem applied in the left-hand side of diagram~\eqref{eqn:diag-rest-Levi} and the description above of the maps $(\tilde{s}^2_c)|_U$ and $(\tilde{\sigma}^2_c)|_U$. The restriction of our complex to the inverse image of $X$ is computed in~\eqref{eqn:tildeF-1}, and satisfies the required property. Finally, the co-restriction to the inverse image of $X$ can be computed similarly, using the other description of the functors $\tilde{\F}_c$ and $\tilde{\F}^2_c$ provided by Braden's theorem.\footnote{Here we need to apply Braden's theorem on a finite-dimensional subvariety of $\Gr_{G,X^2}$. Since such a variety is not necessarily normal, the proof in~\cite{braden} does not apply in this context. The more general form of this result that we need is proved in~\cite{drg}.} Finally, comparing~\eqref{eqn:tildeF-1} and~\eqref{eqn:tildeF-2} and using the isomorphism~\eqref{eqn:convolution-fusion-k} for $L$, we obtain a canonical isomorphism
\[
(\tau_L)^\circ(\F_c(\mathscr{A}_1 \star \mathscr{A}_2)) \cong  \bigoplus_{c_1+c_2=c} (\tau_L)^\circ( \F_{c_1}(\mathscr{A}_1) \star \F_{c_2}(\mathscr{A}_2)).
\]
Restricting to a point in $x$ and then summing over $c$, we deduce the wished-for isomorphism
\[
\mathsf{R}^G_L(\mathscr{A}_1 \star \mathscr{A}_2) \cong \mathsf{R}^G_L(\mathscr{A}_1) \star \mathsf{R}^G_L(\mathscr{A}_2).
\]
The proof of compatibility with the constraints is left to the reader.
\end{proof}

\subsection{Description of the induced morphism of group schemes}

The results of Section~\ref{sec:construction} provide canonical equivalences of monoidal categories
\[
\Per_{\GO}(\Gr_G,\bk) \cong \Rep_\bk(\widetilde{G}_\bk), \qquad \Per_{L_{\mathcal{O}}}(\Gr_L,\bk) \cong \Rep_\bk(\widetilde{L}_\bk).
\]
In view of~\cite[Theorem~X.1.2]{milne}, the functor $\mathsf{R}^G_L$ defines a $\bk$-group scheme morphism
\[
\varphi_L^G : \widetilde{L}_\bk \to \widetilde{G}_\bk.
\]
The isomorphisms~\eqref{eqn:rest-levi-torus} show that the composition of $\varphi_L^G$ with the canonical embedding $T^\vee_\bk \to \widetilde{L}_\bk$ (see~\S\ref{ss:identification-first}) is the canonical morphism $T^\vee_\bk \to \widetilde{G}_\bk$.

\begin{prop}
The morphism $\varphi_L^G$ is a closed embedding, which induces an isomorphism between $\widetilde{L}_\bk$ and the Levi subgroup\footnote{See~\cite[Expos\'e~XXVI, \S 1.7]{sga3} for the notion of Levi subgroup of a reductive group over a base scheme.} of $\widetilde{G}_\bk$ containing $T^\vee_\bk$ whose roots are the coroots of $L$.
\end{prop}

\begin{proof}
First, we assume that $\bk$ is a field. In this case, by~\cite[Proposition~2.21(b)]{dm}, to prove that $\varphi_L^G$ is a closed embedding it suffices to prove that any object of $\Per_{L_{\mathcal{O}}}(\Gr_L,\bk)$ is a subquotient of an object in the essential image of $\mathsf{R}^G_L$. However, as in the proof of Lemma~\ref{lem:G-algebraic-connected}, any object of $\Per_{L_{\mathcal{O}}}(\Gr_L,\bk)$ is a subquotient of a tilting object. Now the functor $\mathsf{R}^G_L$ sends tilting objects of $\Per_{\GO}(\Gr_G,\bk)$ to tilting objects of $\Per_{L_{\mathcal{O}}}(\Gr_L,\bk)$. (In the case $\mathrm{char}(\bk)$ is good for $G$, this fact follows from~\cite[Theorem~1.6]{jmw2} and the results of~\cite[\S 1.5]{mr}; the general case is treated in~\cite{bmrr}.) Moreover, it is not difficult to check that if $\lambda \in X_*(T)$ is dominant for $L$, then the indecomposable tilting object in $\Per_{L_{\mathcal{O}}}(\Gr_L,\bk)$ labelled by $\lambda$ is a direct summand of the image under $\mathsf{R}^G_L$ of the indecomposable tilting object in $\Per_{\GO}(\Gr_G,\bk)$ labelled by the unique $W$-conjugate of $\lambda$ belonging to $X_*(T)^+$. It follows that any tilting object in $\Per_{L_{\mathcal{O}}}(\Gr_L,\bk)$ is a direct summand of an object in the essential image of $\mathsf{R}^G_L$, which finishes the proof of the fact that $\varphi_L^G$ is a closed embedding.

Once this fact is established, we note that since $\varphi_L^G$ intertwines the canonical morphisms $T^\vee_\bk \to \widetilde{L}_\bk$ and $T^\vee_\bk \to \widetilde{G}_\bk$, it must induce, for any $\alpha \in \Delta_{\mathrm{s}}^\vee(L,B_L,T)$, an isomorphism between the root subgroup of $\widetilde{L}_\bk$ associated with $\alpha$ and the root subgroup of $\widetilde{G}_\bk$ associated with $\alpha$. Now the group $\widetilde{L}_\bk$, resp.~the Levi subgroup $\widetilde{L}_\bk'$ of $\widetilde{G}_\bk$ containing $T^\vee_\bk$ whose roots are the coroots of $L$, is generated by $T^\vee_\bk$ and these subgroups. We deduce that the image of $\varphi_L^G$ is $\widetilde{L}_\bk'$, or in other words that $\varphi_L^G$ induces an isomorphism between $\widetilde{L}_\bk$ and $\widetilde{L}_\bk'$.

Now we treat the case $\bk=\Z$. Consider the morphism $(\varphi_L^G)^* : \Z[\widetilde{G}_\Z] \to \Z[\widetilde{L}_\Z]$. If $C$ is the cokernel of this morphism, then $C$ is a finitely generated $\Z[\widetilde{G}_\Z]$-module which satisfies $C \otimes_\Z \mathbf{F}=0$ for any field $\mathbf{F}$. By~\cite[Claim $(*)$ in the proof of Lemma~1.4.1]{br}, it follows that $C=0$, i.e.~that $(\varphi_L^G)^*$ is surjective, and hence that $\varphi_L^G$ is a closed embedding. It is easily checked, using similar arguments, that the image of $\varphi_L^G$ satisfies condition (b) in~\cite[Expos\'e~XXVI, Proposition~1.6(ii)]{sga3} (for the parabolic subgroup containing $T^\vee_\Z$ and whose roots are $\Delta^\vee_+(L,B_L,T) \sqcup (-\Delta_+^\vee(G,B,T)$). By the unicity claim in this statement, it follows that this image is the Levi subgroup of $\widetilde{G}_\Z$ containing $T^\vee_\Z$ whose roots are the coroots of $L$.

Finally, the general case follows from the case $\bk=\Z$ by base change.
\end{proof}

\newpage

\appendix

\section{Equivariant perverse sheaves}
\label{sec:appendix}

\subsection{Equivariant perverse sheaves}
\label{ss:equiv-perv}

Let $X$ be a complex algebraic variety, let $H$ be a connected\footnote{This assumption is crucial; in case $H$ is disconnected, only the first definition of equivariant perverse sheaves has favorable properties.} algebraic group acting on $X$, and consider a commutative Noetherian ring of finite global dimension $\bk$.
Let
\[
a,p : H \times X \to X, \quad e:X \to H \times X
\]
be the maps defined by
\[
p(g,x)=x, \quad a(g,x) = g \cdot x, \quad e(x)=(1,x).
\]
Let also $p_{23} : H \times H \times X \to H \times X$ be the projection on the last two components, and $m : H \times H \to H$ be the multiplication map.

Let $\mathscr{T}$ be a 
stratification of $X$ whose strata are stable under the $H$-action. Then there are at least 3 ``reasonable'' definitions of the category of $\mathscr{T}$-constructible $H$-equivariant perverse sheaves on $X$:
\begin{enumerate}
\item
the heart $\Per^\#_{\mathscr{T},H}(X,\bk)$ of the perverse t-structure on the $\mathscr{T}$-constructible equivariant derived category $\Db_{\mathscr{T},H}(X,\bk)$ in the sense of Bernstein--Lunts, see~\cite[\S 5]{bernstein-lunts};
\item
the category $\Per^\flat_{\mathscr{T},H}(X,\bk)$ whose objects are pairs $(\mathscr{F},\vartheta)$ where $\mathscr{F} \in \Per_{\mathscr{T}}(X,\bk)$ and $\vartheta : a^* \mathscr{F} \to p^* \mathscr{F}$ is an isomorphism such that
\begin{equation}
\label{eqn:cocycle}
e^*(\vartheta) = \id_{\mathscr{F}} \quad \text{and} \quad (m \times \id_X)^*(\vartheta) = (p_{23})^*(\vartheta) \circ (\id_H \times a )^*(\vartheta),
\end{equation}
and whose morphisms from $(\mathscr{F},\vartheta)$ to $(\mathscr{F}',\vartheta')$ are morphisms $f : \mathscr{F} \to \mathscr{F}'$ in $\Per_{\mathscr{T}}(X,\bk)$ such that the following diagram commutes:
\[
\xymatrix@C=1.5cm{
a^* \mathscr{F} \ar[d]_-{a^*(f)} \ar[r]^-{\vartheta} & p^* \mathscr{F} \ar[d]^-{p^*(f)} \\
a^* \mathscr{F}' \ar[r]^-{\vartheta'} & p^* \mathscr{F}';
}
\]
\item
the full subcategory $\Per_{\mathscr{T},H}(X,\bk)$ of $\Per_{\mathscr{T}}(X,\bk)$ consisting of objects $\mathscr{F}$ such that there exists an isomorphism $p^*\mathscr{F} \cong a^* \mathscr{F}$.
\end{enumerate}

There exists an obvious forgetful functor $\Per^\flat_{\mathscr{T},H}(X,\bk) \to \Per_{\mathscr{T},H}(X,\bk)$. Next, we will define a canonical functor 
\begin{equation}
\label{eqn:forgPerH}
\Per^\#_{\mathscr{T},H}(X,\bk) \to \Per^\flat_{\mathscr{T},H}(X,\bk).
\end{equation}
For this we need the following observation. We denote by $\mathrm{For}_H : \Db_{\mathscr{T},H}(X,\bk) \to \Db_{\mathscr{T}}(X,\bk)$ the forgetful functor. The morphism $p$ is a $\phi$-morphism of varieties in the sense of~\cite[\S 0.1]{bernstein-lunts}, where $\phi$ is the unique morphism $H \to \{1\}$ and where $H$ acts on $H \times X$ via left multiplication on the first factor. Therefore, this map defines a functor
\begin{equation}
\label{eqn:def-p^*}
 p^* : \Db_{\mathscr{T}}(X,\bk) \to \Db_{\widetilde{\mathscr{T}},H}(H \times X, \bk)
\end{equation}
(where $\widetilde{\mathscr{T}}$ is the stratification of $H \times X$ whose strata are the subvarieties $H \times S$ with $S \in \mathscr{T}$),
see~\cite[\S 6.5]{bernstein-lunts}.

\begin{lem}
\label{lem:equiv-isom}
 For any $\mathscr{F}$ in $\Db_{\mathscr{T},H}(X,\bk)$, there exists a canonical isomorphism
 \[
  a^* \mathscr{F} \xrightarrow{\sim} p^* \mathrm{For}_H(\mathscr{F})
 \]
 in $\Db_{\widetilde{\mathscr{T}},H}(H \times X,\bk)$.
\end{lem}

\begin{proof}
In view of~\cite[\S 6.6, Item 5]{bernstein-lunts}, the functor~\eqref{eqn:def-p^*} is an equivalence of categories, whose quasi-inverse is the composition $e^* \circ \mathrm{For}_H$ (where we also denote by $\mathrm{For}_H$ the forgetful functor $\Db_{\widetilde{\mathscr{T}},H}(H \times X, \bk) \to \Db_{\widetilde{\mathscr{T}}}(H \times X, \bk)$). Therefore, to define an isomorphism as in the lemma it suffices to construct an isomorphism
\[
 e^* \circ \mathrm{For}_H (a^* \mathscr{F}) \xrightarrow{\sim} \mathrm{For}_H(\mathscr{F}).
\]
In fact, such an isomorphism is clear from the facts that $a^*$ commutes with forgetful functors in the obvious way and that $a \circ e=\id_X$.
\end{proof}

If $\mathscr{F}$ is in $\Per^\#_{\mathscr{T},H}(X,\bk)$, applying the forgetful functor to the isomorphism of Lemma~\ref{lem:equiv-isom} we obtain a canonical isomorphism $\vartheta : a^* \mathrm{For}_H(\mathscr{F}) \xrightarrow{\sim} p^* \mathrm{For}_H(\mathscr{F})$ in $\Db_{\widetilde{\mathscr{T}}}(H \times X, \bk)$. We leave it to the reader to check that this isomorphism satisfies the conditions~\eqref{eqn:cocycle}; then the pair $(\mathrm{For}_H(\mathscr{F}), \vartheta)$ defines an object of $\Per^\flat_{\mathscr{T},H}(X,\bk)$. This construction provides the whished-for functor~\eqref{eqn:forgPerH}.

The following result is well known, but not explicitly proved in the literature to the best of our knowledge (except for a very brief treatment in~\cite[Appendix~A]{mv1}).

\begin{prop}
\label{prop:equiv-perv}
The forgetful functors
\[
\Per^\#_{\mathscr{T},H}(X,\bk) \to \Per^\flat_{\mathscr{T},H}(X,\bk) \to \Per_{\mathscr{T},H}(X,\bk)
\]
are equivalences of categories.
\end{prop}

In view of this proposition, in the body of these notes we identify the three categories above, and denote them by $\Per_{\mathscr{T},H}(X,\bk)$.

In the proof of this proposition we will use the fact (see~\cite[Th\'eor\`eme~3.2.4]{bbd}) that perverse sheaves form a stack for the smooth topology. In our particular case, if $\pi : P \to X$ is a smooth resolution (in the sense of \cite{bernstein-lunts}), $\mathscr{U}$ denotes the stratification on $P$ whose strata are the subsets $\pi^{-1}(S)$ for $S \in \mathscr{T}$, $\mathscr{V}$ denotes the stratification on $P/H$ whose strata are the subsets $q(U)$ with $U \in \mathscr{U}$ (where $q : P \to P/H$ is the projection), and if
\[
r_1,r_2 : P \times_{P/H} P \to P, \quad r_{12},r_{23}, r_{13} : P \times_{P/H} P \times_{P/H} P \to P \times_{P/H} P
\] 
are the natural projections, this means that the category $\Per_{\mathscr{V}}(P/H,\bk)$ is equivalent, via the functor $q^*$, to the category whose objects are pairs $(\mathscr{F},\sigma)$ where $\mathscr{F} \in \Per_{\mathscr{U}}(P,\bk)[-\dim(H)]$ and $\sigma : (r_1)^* \mathscr{F} \xrightarrow{\sim} (r_2)^* \mathscr{F}$ is an isomorphism such that $(r_{23})^*(\sigma) \circ (r_{12})^*(\sigma) = (r_{13})^*(\sigma)$, and whose morphisms $(\mathscr{F},\sigma) \to (\mathscr{F}',\sigma')$ are morphisms 
$f \in \Hom_{\Db_{\mathscr{U}}(P,\bk)}(\mathscr{F}, \mathscr{F}')$
such that $(r_2)^* (f) \circ \sigma = \sigma' \circ (r_1)^*(f)$. 

With this result at hand we can give the proof of Proposition~\ref{prop:equiv-perv}.

\begin{proof}
The second functor is an equivalence by~\cite[\S 4.2.10]{letellier}. Hence what remains to be proved is that the composition $\Per^\#_{\mathscr{T},H}(X,\bk) \to \Per_{\mathscr{T},H}(X,\bk)$ is an equivalence.

Fix a free $H$-space $P$ and a smooth $\dim(X)$-acyclic map $\pi : P \to X$ of relative dimension $d$ (which exist thanks to the results of~\cite[\S 3.1]{bernstein-lunts}), and let $q : P \to P/H$ be the quotient morphism. Then $\Per^\#_{\mathscr{T},H}(X,\bk)$ is (by definition, see~\cite[\S 2.2.4]{bernstein-lunts}) equivalent to the category whose objects are the triples 
 $(\mathscr{F}_P, \mathscr{F}_X, \beta)$ where $\mathscr{F}_P \in \Db_c(P/H,\bk)$, $\mathscr{F}_X \in \Per_{\mathscr{T}}(X,\bk)$ and $\beta : q^*\mathscr{F}_P \xrightarrow{\sim} \pi^* \mathscr{F}_X$ is an isomorphism, and whose morphisms from $(\mathscr{F}_P, \mathscr{F}_X, \beta)$ to $(\mathscr{F}'_P, \mathscr{F}'_X, \beta')$ are the pairs $(f_P,f_X)$ with $f_P : \mathscr{F}_P \to \mathscr{F}_P'$ and $f_X : \mathscr{F}_X \to \mathscr{F}_X'$ compatible (in the natural sense) with $\beta$ and $\beta'$. 

First we show that our functor is faithful. Let $(f_P,f_X) : (\mathscr{F}_P, \mathscr{F}_X, \beta) \to (\mathscr{F}'_P, \mathscr{F}'_X, \beta')$ be a morphism in $\Per^\#_{\mathscr{T},H}(X,\bk)$ such that $f_X=0$. Then by the compatibility of $(f_P,f_X)$ with $\beta$ and $\beta'$ we deduce that $q^*(f_P)=0$. Now it is easily seen that $\mathscr{F}_P$ belongs to $\Per_{\mathscr{V}}(P/H,\bk)[\dim(H)-d]$. Since $q$ is smooth with connected fibers, the functor $q^*$ is fully faithful on perverse sheaves (see~\cite[Proposition~4.2.5]{bbd}); we deduce that $f_P=0$, finishing the proof of faithfulness.

Next we prove that our functor is full. Let $(\mathscr{F}_P, \mathscr{F}_X, \beta)$ and $(\mathscr{F}'_P, \mathscr{F}'_X, \beta')$ be in $\Per^\#_{\mathscr{T},H}(X,\bk)$, and let $f : \mathscr{F}_X \to \mathscr{F}'_X$ be a morphism. To construct a morphism $f_P : \mathscr{F}_P \to \mathscr{F}'_P$ such that $\beta' \circ q^*(f_P) = \pi^*(f) \circ \beta$, we use the stack property recalled above: we remark that the morphism $(\beta')^{-1} \circ \pi^*(f) \circ \beta$ satisfies the descent condition, hence is of the form $q^*(f_P)$ for a unique morphism $f_P : \mathscr{F}_P \to \mathscr{F}'_P$.

Finally, we prove that our functor is essentially surjective. Let $\mathscr{F}$ be in $\Per_{\mathscr{T},H}(X,\bk)$. Then there exists a (unique) isomorphism $\vartheta : a^*(\mathscr{F}) \to p^*(\mathscr{F})$ which satisfies the conditions~\eqref{eqn:cocycle}. Identifying $H \times P$ with $P \times_{P/H} P$ via the morphism $(a,p)$, $(\id_H \times \pi)^*(\vartheta)$ defines an isomorphism $\sigma : (r_1)^* (\pi^* \mathscr{F}) \to (r_2)^* (\pi^* \mathscr{F})$. Identifying $H \times H \times P$ with $P \times_{P/H} P \times_{P/H} P$ via $(g,h,x) \mapsto (ghx,hx,x)$, we see that the second condition in~\eqref{eqn:cocycle} guarantees that $\sigma$ satisfies the descent condition, so that the pair $(\pi^* \mathscr{F},\sigma)$ defines an object $\mathscr{F}_P \in \Db_c(P/H,\bk)$ such that $\pi^* \mathscr{F} \cong q^* \mathscr{F}_P$. Fixing such an isomorphism, we obtain an object of $\Per^\#_{\mathscr{T},H}(X,\bk)$ whose image in $\Per_{\mathscr{T},H}(X,\bk)$ is $\mathscr{F}$.
\end{proof}

\subsection{Induction}
\label{ss:appendix-ind}

Let $X$, $H$ and $\bk$ be as in~\S\ref{ss:equiv-perv}. We consider the constructible derived category $\Db_c(X,\bk)$ of $\bk$-sheaves on $X$, and its $H$-equivariant version $\Db_{c,H}(X,\bk)$. We also denote by
\[
 \mathrm{For}_H : \Db_{c,H}(X,\bk) \to \Db_{c}(X,\bk)
\]
the forgetful functor. Recall that if $H \times X$ is considered as an $H$-variety via left multiplication on the first factor, and if $p : H \times X \to X$ is the projection, then the functor $p^!$ induces an equivalence of categories $\Db_{c}(X,\bk) \to \Db_{c,H}(H \times X,\bk)$, see~\cite[Proposition~2.2.5]{bernstein-lunts}. We consider the functor
\[
 \mathrm{ind}_H : \Db_{c}(X,\bk) \to \Db_{c,H}(X,\bk)
\]
defined by
\[
 \mathrm{ind}_H(\mathscr{F}) = a_! p^!(\mathscr{F}).
\]

\begin{lem}
\label{lem:equiv-ind}
 The functor $\mathrm{ind}_H$ is left adjoint to $\mathrm{For}_H$.
\end{lem}

\begin{proof}
 Let $\mathscr{F}$ in $\Db_{c}(X,\bk)$ and $\mathscr{G}$ in $\Db_{c,H}(X,\bk)$.
 Using first the fact that $p^!$ is an equivalence, then Lemma~\ref{lem:equiv-isom}, and finally adjunction, we obtain canonical isomorphisms
 \begin{multline*}
  \Hom_{\Db_{c}(X,\bk)}(\mathscr{F}, \mathrm{For}_H(\mathscr{G})) \cong \Hom_{\Db_{c,H}(H \times X,\bk)}(p^! \mathscr{F}, p^! \mathrm{For}_H(\mathscr{G})) \\
  \cong \Hom_{\Db_{c,H}(H \times X,\bk)}(p^! \mathscr{F}, a^! \mathscr{G}) \cong \Hom_{\Db_{c,H}(X,\bk)}(a_! p^! \mathscr{F}, \mathscr{G}).
 \end{multline*}
The claim follows.
\end{proof}

\subsection{Convolution}
\label{ss:appendix-conv}

Let $H$ be a complex algebraic group, and let $K \subset H$ be a closed subgroup. Recall that the $K$-bundle given by the quotient morphism $H \to H/K$ is locally trivial for the analytic topology, see~\cite{serre}. (In all the cases we will consider, this morphism is in fact locally trivial for the Zariski topology.) We 
consider the constructible equivariant derived category $\Db_{c,K}(H/K,\bk)$.
This category admits a natural convolution bifunctor, constructed as follows. Consider the diagram
\begin{equation}
\label{eqn:diag-conv-appendix}
H/K \times H/K \xleftarrow{\ p \ } H \times H/K \xrightarrow{\ q \ } H \times^K H/K \xrightarrow{\ m \ } H/K,
\end{equation}
where $H \times^K H/K$ is the quotient of $H \times H/K$ by the action defined by $k \cdot (g,hK) = (gk^{-1}, khK)$ for $k \in K$ and $g,h \in H$, $q$ is the quotient morphism, and the maps $p$ and $m$ are defined by
\[
p(g,hK) = (gK,hK), \quad m([g,hK]) = ghK.
\]
Since $K$ acts freely on $H \times H/K$, by~\cite[Theorem~2.6.2]{bernstein-lunts} the functor $q^*$ induces an equivalence
\[
\Db_{c,K}(H \times^K H/K, \bk) \xrightarrow{\sim} \Db_{c,K \times K}(H \times H/K, \bk)
\]
(where $K$ acts on $H \times^K H/K$ via left multiplication on $H$, and $K \times K$ acts on $H \times H/K$ via $(k_1,k_2) \cdot (g,hK) = (k_1gk_2^{-1}, k_2hK)$). Now, consider some objects $\mathscr{F}_1, \mathscr{F}_2$ in $\Db_{c,K}(H/K, \bk)$. Then $\mathscr{F}_1 \lboxtimes_\bk \mathscr{F}_2$ belongs to $\Db_{c,K \times K}(H/K \times H/K,\bk)$. Since $p$ is a $(K \times K)$-equivariant morphism, $p^*(\mathscr{F}_1 \lboxtimes_\bk \mathscr{F}_2)$ defines 
an object in $\Db_{c,K \times K}(H \times H/K, \bk)$. Hence there exists a unique object $\mathscr{F}_1 \, \widetilde{\boxtimes} \, \mathscr{F}_2$ in $\Db_{c,K}(H \times^K H/K, \bk)$ such that
\begin{equation}
\label{eqn:def-twisted-prod}
q^*(\mathscr{F}_1 \, \widetilde{\boxtimes} \, \mathscr{F}_2) \cong p^*(\mathscr{F}_1 \lboxtimes_\bk \mathscr{F}_2).
\end{equation}
We then set
\[
\mathscr{F}_1 \star \mathscr{F}_2 := m_*(\mathscr{F}_1 \, \widetilde{\boxtimes} \, \mathscr{F}_2).
\]
It is a classical fact that this construction defines a monoidal structure on the category $\Db_{c,K}(H/K, \bk)$ (which does not, in general, restrict to a monoidal structure on $\Per_K(H/K, \bk)$).

\begin{rmk}\phantomsection
\label{rmk:def-conv-!*}
\begin{enumerate}
\item
Since the maps $p$ and $q$ are smooth of relative dimension $\dim(K)$, we have canonical isomorphisms $p^! \cong p^*[\dim(K)]$ and $q^! \cong q^*[\dim(K)]$, so that the condition~\eqref{eqn:def-twisted-prod} can be replaced by $q^!(\mathscr{F}_1 \, \widetilde{\boxtimes} \, \mathscr{F}_2) \cong p^!(\mathscr{F}_1 \lboxtimes_\bk \mathscr{F}_2)$.
\item
In the special case considered for the geometric Satake equivalence, when $\bk$ is not a field one modifies this construction slightly so that it sends pairs of perverse sheaves to perverse sheaves; see~\S\ref{ss:convolution-k}.
\end{enumerate}
\end{rmk}

\subsection{The case of $\Gr_G$}
\label{ss:appendix-Gr}

The main object of study in these notes is the category $\Per_{\GO}(\Gr_G,\bk)$. This setting does not fit exactly in the framework of~\S\S\ref{ss:equiv-perv}--\ref{ss:appendix-conv} because $\GK$ and $\GO$ are not algebraic groups in the usual sense. But the category $\Db_{c,\GO}(\Gr_G,\bk)$ still makes sense, as follows.

For any $n \in \mathbf{Z}_{\geq 1}$, we denote by $H_n \subset \GO$ the kernel of the morphism
\[
G_{\mathcal{O}} \to G_{\mathcal{O} / t^n \mathcal{O}}
\]
induced by the quotient morphism $\mathcal{O} \to \mathcal{O} / t^n \mathcal{O}$. (Here the group scheme $G_{\mathcal{O} / t^n \mathcal{O}}$ is defined in a way similar to $G_\mathcal{O}$.) Note that if $m \geq n$, then $H_m$ is a normal subgroup in $H_n$, and the quotient $H_n/H_m$ is a unipotent group. If $X \subset \Gr_G$ is a closed finite union of $\GO$-orbits, there exists $n \in \mathbf{Z}_{\geq 1}$ such that $H_n$ acts trivially on $X$. Then it makes sense to consider the equivariant derived category $\Db_{c,\GO/H_n}(X, \bk)$. Since $H_n/H_m$ is unipotent for any $m \geq n$, one can check using~\cite[Theorem~3.7.3]{bernstein-lunts} that the functor
\[
\Db_{c,\GO/H_n}(X, \bk) \to \Db_{c,\GO/H_m}(X, \bk)
\]
given by inverse image under the projection $G/H_m \to G/H_n$ is an equivalence of categories. Hence one can define the category $\Db_{c,\GO}(X, \bk)$ to be $\Db_{c,\GO/H_n}(X, \bk)$ for any $n$ such that $H_n$ acts trivially on $X$.

If $X \subset Y \subset \Gr_G$ are closed finite unions of $\GO$-orbits, the direct image under the embedding $X \hookrightarrow Y$ induces a fully-faithful functor $\Db_{c,\GO}(X, \bk) \to \Db_{c,\GO}(Y, \bk)$. Hence we can finally define  $\Db_{c,\GO}(\Gr_G, \bk)$ as the union of the categories $\Db_{c,\GO}(X, \bk)$ for all closed finite unions of $\GO$-orbits $X \subset \Gr_G$.

A construction similar to that of~\S\ref{ss:appendix-conv} produces a convolution bifunctor $\star$ on the category $\Db_{c,\GO}(\Gr_G, \bk)$. More precisely, if $\mathscr{F}_1$ and $\mathscr{F}_2$ are in $\Db_{c,\GO}(\Gr_G, \bk)$, one should choose a closed finite union of $\GO$-orbits $X \subset \Gr_G$ such that $\mathscr{F}_2$ belongs to $\Db_{c,\GO}(X,\bk)$, and $n \in \mathbf{Z}_{\geq 1}$ such that $H_n$ acts trivially on $X$, and replace diagram~\eqref{eqn:diag-conv-appendix} by the similar diagram
\[
\Gr_G \times X \leftarrow \GK/H_n \times X \to (\GK/H_n) \times^{(\GO/H_n)} X \to \Gr_G,
\]
and proceed as before. In the body of the paper, as in~\cite{mv}, to lighten the notation we neglect these technical subtleties.


\end{document}